\documentclass{amsart}
\usepackage[latin1]{inputenc}
\usepackage[english]{babel}
\usepackage{tikz,hyperref}

\usepackage[T1]{fontenc}
\usepackage{amsmath,bbold}
\usepackage{amsfonts}
\usepackage{amssymb}
\usepackage{amsthm,eepic}
\usepackage{mathrsfs}
\usepackage{enumitem}
\usepackage{graphicx}
\graphicspath{{Fig/}}
\usepackage{footnote}
\usepackage{hyperref,tikz}
\hypersetup{colorlinks=true,    linkcolor=blue,
citecolor=red,       filecolor=BrickRed,       urlcolor=darkgreen
}

\renewcommand{\geq}{\geqslant}
\renewcommand{\leq}{\leqslant}

\newtheorem{theorem}{Theorem}
\newtheorem{prop}{Proposition}[section]
\newtheorem{defi}{Definition}

\newtheorem{cor}[prop]{Corollary}
\newtheorem{lemma}{Lemma}[section]

\newcommand{\be}{\begin{equation}}
\newcommand{\ee}{\end{equation}}
\newcommand\inter[1]{\overset{\circ}{#1}}
\newcommand\steq[1]{\stackrel{\text{\rm #1.}}{=}}

\usepackage{pgfplots}

\pgfplotsset{compat=newest}

\colorlet{darkgreen}{green!50!black}
\definecolor{darkseagreen}{rgb}{0.56, 0.74, 0.56}
\definecolor{lightcyan}{rgb}{0.88, 1.0, 1.0}
\definecolor{lightblue}{rgb}{0.68, 0.85, 0.9}
\definecolor{palecerulean}{rgb}{0.61, 0.77, 0.89}
\definecolor{lgreen} {RGB}{180,210,100}
\definecolor{dblue}  {RGB}{20,66,129}
\definecolor{ddblue} {RGB}{11,36,69}
\definecolor{lred}   {RGB}{220,0,0}
\definecolor{nred}   {RGB}{224,0,0}
\definecolor{norange}{RGB}{230,120,20}
\definecolor{nyellow}{RGB}{255,221,0}
\definecolor{ngreen} {RGB}{98,158,31}
\definecolor{dgreen} {RGB}{78,138,21}
\definecolor{nblue}  {RGB}{28,130,185}
\definecolor{jblue}  {RGB}{20,50,100}
\definecolor{Apricot} {RGB}{255, 170, 123} 
\definecolor{dpurple}  {RGB}{53,21,93}

\numberwithin{equation}{section}

\def\R{{\mathbb R}}

\def\N{{\mathbb N}}
\def\C{{\mathbb C}}
\def\Z{{\mathbb Z}}
\def\1{\mathbb{1}}

\def\E{{\mathbb E}}
\def\eps{{\epsilon}}
\def\P{{\mathbb P}}

\def\cal{\mathcal}

\def\Z{\mathbb{Z}}
\def\C{\mathbb{C}}
\def\R{\mathbb{R}}
\def\Q{\mathbb{Q}}
\renewcommand{\epsilon}{\varepsilon}

\usepackage{color}
\definecolor{darkgreen}{rgb}{0,0.4,0}
\definecolor{MyDarkBlue}{rgb}{0,0.08,0.50}
\definecolor{BrickRed}{rgb}{0.65,0.08,0}

\title[Classification of Asymptotic Behaviors of Green Functions]{A Classification of Asymptotic Behaviors of\\ Green Functions of Random Walks in the Quadrant}

\author{Irina Ignatiouk-Robert}
\address{D\'epartement de math\'ematiques, Universit\'e de Cergy-Pontoise, 2 Avenue Adolphe Chauvin
95302 Cergy-Pontoise Cedex, France}
\email{Irina.Ignatiouk@u-cergy.fr}

\date{\today}

\begin{document}

\begin{abstract}
This paper investigates the asymptotic behavior of  Green functions associated to partially homogeneous random walks in the quadrant $\Z_+^2$. There are four possible distributions for the jumps of these processes, depending on the location of the starting point: in the interior, on the two positive axes of the boundary, and at the origin $(0,0)$. 

With mild conditions on the positive jumps of the random walk, which can be unbounded, a complete analysis of the asymptotic behavior of the Green function of the random walk killed at $(0,0)$ is achieved. The main result is that {\em eight} regions of the set of parameters determine completely the possible limiting behaviors of Green functions of these Markov chains. These regions are defined by a  set of relations for several characteristics of the distributions of the jumps.

In the transient case, a  description of the Martin boundary is obtained and in the positive recurrent case, our results  give the  exact limiting behavior of the invariant distribution of a state whose norm goes to infinity along some asymptotic direction in the quadrant. These limit theorems extend  results of the literature obtained, up to now,  essentially for random walks whose jump sizes are either $0$ or $1$ on each coordinate.

Our approach  relies on a combination of  several methods:  probabilistic representations of solutions of analytical equations, Lyapounov functions,  convex analysis,  methods of homogeneous random walks, and  complex analysis arguments. 

\end{abstract}

\maketitle

\tableofcontents

\newpage
 
\section{Introduction}
For a  transient Markov chain  $(Z(n))$ on an infinite countable state space  ${\cal E}$ determining all  possible limits of the associated Martin kernel, {\em the Martin boundary}, is an important and difficult problem in general. By the Poisson-Martin representation theorem, it gives the {\em Martin compactification} of the state space, an integral representation of all non-negative harmonic functions. For an introduction to the theory of Martin boundary for countable Markov chains, see the classical references Doob~\cite{Doob:02} and Dynkin~\cite{Dynkin}.

The characterization of the Martin boundary for homogeneous  random walks in $\Z^d$ has been obtained  in Ney and Spitzer~\cite{Ney-Spitzer}, via a set of technical estimates related to the local central limit theorem.  The Martin boundary  has been also identified for  random walks on free groups, hyperbolic graphs and Cartesian products. See  Woess~\cite{Woess} for a thorough presentation of boundary theory of random walks. 

There are few results for more general Markov chains. For random walks on non-homogeneous trees, the Martin boundary has been obtained in Cartier~\cite{Cartier}.  Doney~\cite{Doney:02} identified  the Martin boundary of homogeneous random walks $(Z(n))$ on $\Z$ killed on the negative half-line of $\Z$. For space-time random walks  $(S(n)){=}(n, Z(n))$  associated to a homogeneous random walk $Z(n)$ on $\Z$ killed on the negative half-line, the Martin boundary is obtained in Alili and Doney~\cite{Alili-Doney}. The proof of these results relies on the one-dimensional structure of these  processes.

The  Martin boundary for partially  homogeneous random walks killed or reflected on a half-space or a cone of $\Z^d$ has been identified with  large deviation techniques, Choquet-Deny theory and  ratio   limit  theorems  of Markov-additive processes. See Ignatiouk~\cite{Ignatiouk:2008,Ignatiouk:2010,Ignatiouk:2020}. 

\subsection*{Random Walks on $\Z_+^2$ }
In this paper we consider a partially homogeneous random walk $(Z(n))$ on $\Z_+^2$ with the following characteristics: the distribution of its jumps is
\begin{enumerate}
\item[] a) $\mu$ in the interior of  $\Z_+^2$;\qquad\qquad b) $\mu_0$ at $0{=}(0,0)$;
\item[] c) $\mu_1$ in $\{0\}{\times}(\Z_+{\setminus}\{0\})$;\hspace{16mm}  d)  $\mu_2$ in $(\Z_+{\setminus}\{0\}){\times}\{0\}$;
\end{enumerate}
The possible negative jumps are either $0$ or ${-}1$ on each coordinate for $\mu$, $\mu_1$ and $\mu_2$. 

When it is transient, the Green function $G$ of the Markov chain is, for $j$, $k{\in}\Z_+^2$,
\[
G(j,k)\steq{def}\sum_{n=0}^{+\infty} \P_j(Z(n){=}k).
\]
The strict hitting time of $0{=}(0,0)$ is denoted as
\[
\tau_0\steq{def}\inf\{k{>}0:Z(k){=}(0,0)\}
\]
and the Green function of the Markov chain  killed at $0$ is, for $j{\in}\Z_+^2$, $k{\in}\Z_+^2{\setminus}\{0\}$,
\[
g(j,k)\steq{def}\sum_{n{\ge}0} \P_j(Z(n)=k,n{<}\tau_0).
\]
\subsection*{Exact Asymptotics of Green Functions}
The Martin Kernel being the ratio of two Green functions, its limiting behavior  can be obtained from  the exact asymptotics of $G(j,k)$ when $k$ goes to infinity.
Ney and Spitzer~\cite{Ney-Spitzer} determines the Martin boundary of homogeneous random walks in $\Z^d$ in this way.

It is easily seen that, for $j$, $k{\in}\Z_+^2{\setminus}\{0\}$,
\[
G(j,k) = g(j,k) + G(j,0) g(0,k)
\text{ and  } 
G(0, k) = G(0,0) g(0,k),
\]
which gives a relation between the asymptotic behaviors of the Green functions $k{\to}G(j,k)$ and $k{\to}g(j,k)$.

This type of asymptotic analysis can also be used to investigate positive recurrent Markov chains. The invariant distribution $\pi$ is represented  with the Green function of the  Markov chain killed at $0$,  for $k{\in}\Z_+^2$,
\be\label{InvGreen}
\pi(k)=\pi(0)\sum_{n=0}^{+\infty} \P_{0}\left(Z(n){=}k,n{<}\tau_0\right)=\pi(0)g(0,k),
\ee
The limiting behavior of $k{\to}g(0,k)$ gives therefore the asymptotic behavior of the invariant distribution of a state $k$ going to infinity.

\subsubsection*{Convergence to Infinity}
For $j{\in}\Z_+^2$, we will investigate the asymptotic behavior of $g(j,k)$  as $k{=}(k_1,k_2){\in}\Z_+^2$ gets large in several ways.
\begin{enumerate}
\item\label{NInf}  With direction $u{\in}{\mathbb S}^1_+{=}\{x{=}(x_1,x_2){\in}\R_+^2:\|x\|{=}x_1^2{+}x_2^2{=}1\}$,
\[
\min(k_1, k_2)\to {+}\infty, \quad\text{ and } \quad \frac{k}{\|k\|}\to u.
\]
\item\label{KMInf} Along the axes.\\
The quantity $k_2{\in}\Z$ is fixed and $k_1{\to}{+}\infty$, and the symmetrical case by exchanging the variables $k_1$ and $k_2$.
\end{enumerate}
The case~\ref{NInf} is the classical set of asymptotic behaviors considered in general. As we will see, the asymptotics of the Green function for the  case~\ref{KMInf} are different from  the case~\ref{NInf}, they exhibit interesting behaviors. They have also been considered in~Kobayashi and Miyazawa~\cite{Kobayashi-Miyazawa-2} for random walks with jumps of size $1$. Note that there is a slight abuse of terminology for~\ref{KMInf} since, strictly speaking, the cases $u{=}(1,0)$ and $u{=}(0,1)$ of~\ref{NInf} are also ``along the axes''.

\subsection*{A Functional Equation}
A  functional equation for  generating functions of the Green function of the Markov chain killed at $0$  plays a central role in our study. It is expressed as, for $j{\in}\Z_+^2$ and $(x,y)$ in a convenient subset of $\C^2$, 
\begin{multline}\label{functional-equation}
\bigl(1{-}P(x,y)\bigr) (H_j(x,y){-}H_j(x,0){-}H_j(0,y)) \\ =  L_j(x,y) + \bigl(\phi_1(x,y){-}1\bigr) H_j(x,0) + \bigl(\phi_2(x,y){-}1\bigr) H_j(0,y)  
\end{multline}
holds, where the quantity $(L_j(x,y))$ is a known function and 
\begin{enumerate}
  \item $(H_j(x,y))$ is the generating function of $(g(j,k),k{\in}\Z_+^2)$;
\item $P(x,y)$ is the generating function of $\mu$ at $(x,y)$, the distribution of the jumps in the interior of $\Z_+^2$. The quantity $Q(x,y){=}xy(1{-}P(x,y))$ is in general referred to as the {\em kernel};
\item For $i{=}1$, $2$, $\phi_i(x,y)$ is the  generating function of $\mu_i$ at $(x,y)$.
\end{enumerate}
In most of studies on the asymptotic behavior of   Green functions or invariant measures of random walks in the quadrant, there is always an associated functional equation similar to equation~\eqref{functional-equation}.
\subsection*{Literature}
We now give a brief presentation of the existing asymptotic results for non-homogeneous random walks in the quadrant.
\begin{enumerate}
\item  Nearest neighbor random walks: Jumps of size $0$ or $1$ on each coordinate.\\
In this case, the kernel $(xy(1{-}P(x,y)))$ is a polynomial of degree at most $2$ in each variable.    In the positive recurrent case, by using methods of complex analysis on elliptic curves,  the asymptotic behavior of the invariant distribution  along lines of $\Z_+^2$ has been obtained in the early work Malyshev~\cite{Malyshev} in 1973. 
Following these ideas, extensions of these results have  been established in  Kurkova and  Malyshev~\cite{Kurkova-Malyshev}, Kurkova and Raschel~\cite{Kurkova-Raschel},   and Li and Zhao~\cite{Li-Zhao-2018}, but, as in the original paper~\cite{Malyshev}, only when positive jumps are of size  $1$. See Fayolle et al.~\cite{F-Y-M-book} for a general presentation and additional references therein.

In the positive recurrent case, Kobayashi and Miyazawa~\cite{Kobayashi-Miyazawa-2}  determines  the exact domain of convergence of the generating function of stationary distribution and obtains the asymptotic behavior of the stationary distribution.

In both cases, \cite{Kobayashi-Miyazawa-2} and~\cite{Malyshev}, the analysis relies in an essential way on the explicit representation of the roots of the quadratic function $(xy(1{-}P(x,y)))$. This is the main limitation of this type of approach. 

\bigskip

\item  For positive recurrent random walks with unbounded positive jumps,  exact asymptotics of the stationary measure have been investigated in Borovkov and Mogulskii~\cite{Borovkov-Mogulskii-2001} with large deviation techniques. The asymptotics are considered for interior directions of ${\mathbb S}^1_+$.  Some constants do not seem to be explicitly determined in the limit results of this reference. In particular, it is not clear how the limiting behavior of the invariant distribution of a state going to infinity depends on the asymptotic direction $u{\in}{\mathbb S}^1_+$.

\item   Positive recurrent random walks with unbounded positive jumps  have also been analyzed in  Kobayashi and Miyazawa~\cite{Kobayashi-Miyazawa} from the point of view of tail asymptotics. For these asymptotics,  a line of $\Z_+^2$ associated to a fixed vector is going to infinity in the sense that its distance  to $(0,0)$ is going to infinity. The quantity considered for the tail asymptotics is the invariant distribution of all states of $\Z_+^2$  above this line.  The exact domain of convergence of the generating function of the invariant distribution  is obtained and, with methods of Markov-additive processes,  exact tail asymptotics with explicit constants are derived. 

\item  Ignatiouk et al.~\cite{I-K-R} has investigated  transient random walks in the quadrant $\Z^2_+$ with unbounded positive jumps,  the size of negative jumps  is not necessarily $1$ but bounded. An additional assumption in~\cite{I-K-R} is that the random walk escapes to the infinity along the horizontal axis and the vertical axis. With methods of local Markov-additive processes and complex analysis arguments, the exact asymptotics of the Green function are obtained. They are expressed in terms of asymptotics of Green functions of random walks in half-plane, i.e. with one boundary removed. Our present paper shows that such a result is wrong in general, both boundaries play a role in fact in several of our convergence results.  
 \end{enumerate}

\subsection*{A Quick Presentation}\ \\
A significant part of our paper is devoted to the definition and the properties of  the classification in eight regions of the space of parameters. It is defined as a set of relations for several characteristics associated to the distributions of the jumps $\mu$, $\mu_i$, $i{\in}\{1,2\}$.  For each region  of this classification,  an investigation of the analyticity properties  of the generating functions functions $H_j$, $j{\in}\Z_2^+$ and the study of the nature of their dominant singularities are achieved.  With these results, the exact asymptotics of the Green function  $(g(j,k),j,k{\in}\N^2)$ are derived. They are stated in Theorems~\ref{theorem3} and~\ref{theorem4} of Section~\ref{Results}.

We now give a sketch of the general method used to obtain these convergence results. Section~\ref{Results}  gives a much more detailed description of the contributions and also of the methods used. 

A first step is of studying some solutions $(x,y)$ of the equation $Q(x,y){=}xy(1{-}P(x,y)){=}0$. This is done by investigating the existence of  a function  $Y(x)$ defined on a subset  of $\C$ such that $(x,Y(x))$ is a zero of $Q$.  By canceling the left-hand side of~\eqref{functional-equation}, it gives a relation between $(H_j(x,0))$ and $(H_j(0,Y(x)))$. An analogue study is done, by exchanging the roles of $x$ and $y$, with $(X(y),y)$. After this step,  an analytic continuation of  $(H_j(0,Y(x)))$  is achieved and, with relation~\eqref{functional-equation} it gives an analytic continuation of $(H_j(x,0))$. This is a difficult point and this is where a  convenient  representation of $(Y(x))$ is crucial.  In our approach,  the key argument in this step is a probabilistic representation of the function $(Y(x))$. Under general conditions, an  expression of the  functions  $(H_j(0,y))$ and $(H_j(x,0))$ is derived, and therefore an expression of  $(H_j(x,y))$,  the generating function of $(g(j,k))$. The last step uses this representation  and, with complex analysis arguments, one can derive the asymptotic behavior of  $(g(j,k))$ when $k$ goes to infinity.

\section{Overview of the Results}\label{Results}  
\subsection*{General Notations}\label{general-notations} 
Throughout the paper, the following notations are used.
\begin{enumerate}
\item For two points $(x_1,y_2)$, $(x_2,y_2){\in}\R^2$, the {\em line segment} in $\R^2$ with the end-points $(x_1,y_2)$ and $(x_2,y_2)$ is denoted by $[(x_1,y_2), \, (x_2,y_2)]$. 
The unit circle of $\R_+^2$ is
\be\label{defS1}
  {\mathbb S}^1_+\steq{def}\{w{=}(u,v){\in}\R_+^2: \|x\|=u^2{+}v^2=1\}.
  \ee
\item For $a{\in}\C$ and $r{>}0$, we denote by $B(a,r)$ the open disk in $\C$ with center $a$ and radius $r$.  A {\em poly-disc} of $\C^2$ is the product of two open discs. 
\item For $r_2{>} r_1 {>} 0$, we let 
\[
C(r_1,r_2) = \{x\in\C:~r_1 < |x| < r_2\} \quad \text{and} \quad \overline{C}(r_1,r_2) = \{x\in\C:~r_1 \leq |x| \leq r_2\}.
\]
\item For a subset $B$ of $[0,{+}x\infty[^2$, we denote
\be\label{OmDef}
 \Omega(B) \steq{def} \{(x,y)\in\C^2: \, (|x|, |y|)\in B\}.
\ee
A set $B\subset[0,{+}\infty[^2$ is {\em logarithmically convex} (resp. {\em strictly logarithmically convex}) if, for any $x$, $y{\in}B$ and $\lambda{\in}[0,1]$, $x^{\lambda}y^{1-\lambda}{\in}B$, resp. $x^{\lambda}y^{1-\lambda}{\in}\inter{B}$ when $x{\ne}y$. We denote by  $\text{LogCH}(B)$ the logarithmic convex hull of $B$ in $[0,{+}\infty[^2$, i.e. the smallest logarithmically convex set of $[0,{+}\infty[^2$ containing $B$. 
\end{enumerate}
To simplify some expressions, we will also use the notations, for a ${\cal C}_2$-function $f$ on $\C^2$,
\begin{align*}
\partial_x f(x,y) &= \frac{\partial f}{\partial x}  (x,y), \quad \partial_y f(x,y) = \frac{\partial f}{\partial y}(x,y), \\
\partial^2_{xx} f(x,y) &= \frac{\partial^2 f}{\partial x^2} (x,y), \quad \partial^2_{xy} f(x,y) = \frac{\partial^2 f}{\partial x \partial y} (x,y), \quad \partial^2_{yy} f(x,y) = \frac{\partial^2 f}{\partial y^2} (x,y).
\end{align*}

We now introduce the non-homogeneous random walks investigated in this paper. 
\subsection{Non-Homogeneous Random walks in $\Z^2_+$: Definitions and Assumptions} 
The process $Z(n){=}(Z_1(n),Z_2(n))$ on $\Z^2_+$ is a Markov chain on $\Z^2_+$ with transition probabilities given 
for $j{=}(j_1,j_2){\in}\Z_+^2$ by 
\be\label{trans_probabilities}
\P_{j}(Z(1)= j{+}k)\steq{def}
\begin{cases} \mu(k) & \text{ if } j_1{>}0 \text{ and } j_2{>}0,\\
\mu_1(k) & \text{ if } j_1{>}0 \text{ and } j_2{=}0,\\
\mu_2(k) & \text{ if } j_2{>}0 \text{ and } j_1{=}0,\\ 
\mu_0(k) & \text{ if }  (j_1,j_2){=}(0,0),
\end{cases} 
\ee
where  $\mu$ is a probability measure on $\Z^2$, and  $\mu_1$, $\mu_2$ and  $\mu_0$ are sub-probability measures (positive measures with total mass  less or equal than $1$) on, respectively, $\Z{\times}\Z_+$, $\Z_+{\times}\Z$ and $\Z^2_+$.

Their  generating functions, defined on their set of convergence in $\C^2$, are denoted by 
\begin{align}
P(x,y) &\steq{def} \sum_{j=(j_1,j_2)\in\Z^2} \mu(j) x^{j_1}y^{j_2}, \label{jump_generating_function_P}\\
\phi_1(x,y)& \steq{def} \sum_{j=(j_1,j_2)\in\Z^2} \mu_1(j) x^{j_1}y^{j_2}, \quad \quad 
\phi_2(x,y) \steq{def} \sum_{j=(j_1,j_2)\in\Z^2} \mu_2(j) x^{j_1}y^{j_2},\label{jump_generating_function_phi1}\\
\phi_0 (x,y)& \steq{def} \sum_{j=(j_1,j_2)\in\Z^2} \mu_0(j) x^{j_1}y^{j_2}.\label{jump_generating_function_phi_k}
\end{align}
The {\em level sets} $D$, $D_1$ and $D_2$  of these generating functions are defined as 
\begin{multline}\label{Ddef}
D \steq{def} \left\{(x,y)\in]0,+\infty[^2~:~P(x,y) \leq 1\right\}\\ \text{ and } D_i \steq{def} \left\{(x,y)\in]0,+\infty[^2~:~ \phi_i(x,y) \leq 1\right\}, \quad i{\in}\{1,2\}.
\end{multline}

\bigskip

\noindent
There are three main assumptions used in our results.

\noindent
{\bf Assumption (A1)} {\em 
\begin{enumerate}[label=(\roman*)]
\item The homogeneous random walk  associated to the distribution $\mu$ is irreducible on $\Z^2$; 
\item The generating function $P$ is finite in a neighborhood of the set $D$ in $\R^2$;
\item The set $D$ has a non-empty interior.
\end{enumerate} 
}

\bigskip
\noindent
    {\bf Assumption (A2)}
    \begin{enumerate} [label=(\roman*)]
\item{\em For $j{=}(j_1,j_2){\in}\Z^2$, $\mu(j_1,j_2){=}0$ if  $j_1{<}{-} 1$ or $j_2 {<}{-}1$.}
    \end{enumerate}
    
\bigskip
\noindent
{\bf Assumption (A3)} {\em 
\begin{enumerate} [label=(\roman*)]
\item The random walk $Z(n){=}(Z_1(n),Z_2(n))$ is irreducible on $\Z^2_+$;
\item The generating functions  $\phi_1$, $\phi_2$ and $\phi_0$ are  finite in a neighborhood of the set $D$;
\item The sets $D{\cap}D_1$ and $D{\cap}D_2$ have a non-empty interior;
\item For $j{=}(j_1,j_2){\in}\Z^2$,
\begin{itemize} 
\item  $\mu_1(j_1,j_2){=} 0$ if $j_1 {<}{-} 1$;
\item $\mu_2(j_1,j_2) {=} 0$  if  $j_2{<}{-} 1$; 
\end{itemize} 
\item There exists $j{=}(j_1,j_2){\in}\Z^2$ with $j_2{>}0$ such that $\mu_1(j_1,j_2) {>} 0$; 
\item There exists $j{=}(j_1,j_2){\in}\Z^2$  with $j_1{>}0$ such that $\mu_2(j_1,j_2) {>} 0$.
\end{enumerate}
}

\medskip

The Markov chain $(Z(n))$ killed at $0$ and its Green function are now introduced. 
\begin{defi}[Killed Markov Chain]\label{Killed}
The return time of the process $(Z(n))$ to the origin $0{=}(0,0)$ is defined by
\[
\tau_0 \steq{def} \inf\{ n \geq 1~: Z(n) = (0,0)\} 
\]
and $(Z_{\tau_0}(n))$ denotes a process with the distribution of the Markov chain $(Z(n))$ killed at $0$. Its Green function is defined by, for $j$, $k{\in}\Z_+^2$,
\be\label{Greeng}
g(j,k) \steq{def} \sum_{n=0}^\infty \P_j(Z_{\tau_0}(n) = k)=\sum_{n=0}^\infty \P_j(Z(n) = k, \; \tau_0 > n). 
\ee
A non-negative function $\varkappa:\Z^2_+{\to}\R_+$ is said to be {\em harmonic} for the Markov chain  $(Z(n))$  killed at $0$ if, for  $j{\in}\Z_+^2$, 
\[
\E_j(\varkappa(Z(1)); \; \tau_0 > 1) = \varkappa(j).
\]
\end{defi}

The next proposition introduces key quantities used to define the different regions which determine the asymptotic behavior of  the Green function $(g(j,k))$. Its proof follows from Lemma~\ref{preliminary-lemma1} and Lemma~\ref{preliminary-lemma2} of Section~\ref{preliminary-section}. Figures~\ref{F0a}, ~\ref{F0b},~\ref{F1} and~\ref{F3} in Section~\ref{FigSec} of the appendix illustrate some of these definitions.

\begin{prop}\label{propx*}
Under the assumptions (A1){-}(A3),
\begin{enumerate} 
\item the sets $D_1$ and $D_2$ are logarithmically convex and  the set $D$  is compact, strictly logarithmically convex and does not intersect with the axes $\{(x,y){\in}\R^2{:} x{=}0\}$ and $\{(x,y){\in}\R^2{:} y{=}0\}$;
\item There exist $x^*_P$, $x^{**}_P{\in}]0,{+}\infty[$ and $x^*$,  $x^{**}\in]0,{+}\infty[$ such that $x^*_P {<} x^{**}_P$, $x^* {<} x^{**}$,  and
\begin{align*}
[x^*_P, x^{**}_P] &= \{x{\in}]0,+\infty[ :~ \exists y{\in}]0,+\infty[, \, (x,y){\in} D\},\\
[x^*,x^{**}] &= \{x{\in}]0,+\infty[ :~ \exists y{\in}]0,+\infty[, \, (x,y){\in} D\cap D_1\}.
\end{align*}
\item There exist functions $Y_1$, $Y_2:[x^*_P, x^{**}_P]{\to}[y^*_P, y^{**}_P]$, such that, for  $x{\in} [x^{*}_P,x^{**}_P]$, $Y_1(x) {\leq} Y_2(x)$ and 
\[
[Y_1(x),Y_2(x)] = \{y{\in}]0,+\infty[~:~(x,y){\in} D\}.
\]
For $x{\in} [x^{*}_P,x^{**}_P]$,  $Y_1(x)$, $Y_2(x)$ are the unique positive solutions of the equation $P(x,y) {=}1$, and  $Y_1(x){=} Y_2(x)$ holds if and only if $x{\in}\{x^{*}_P,x^{**}_P\}$.
\item  There exist $y^*_P, y^{**}_P{\in}]0,+\infty[$ and $y^*, y^{**}{\in}]0,+\infty[$ such that $y^*_P {<} y^{**}_P$, $y^* {<} y^{**}$, and 
\begin{align*}
[y^*_P, y^{**}_P] &= \{y{\in}]0,+\infty[ :~ \exists x{\in}]0,+\infty[, \, (x,y){\in} D\},\\
[y^*,y^{**}] &= \{x{\in}]0,+\infty[ :~ \exists x{\in}]0,+\infty[, \, (x,y){\in} D\cap D_2\}.
\end{align*}
\item   There exist functions $X_1$, $X_2{:} [y^*_P, y^{**}_P]{\to}[x^*_P, x^{**}_P]$ such that, for   $y{\in} [y^{*}_P,y^{**}_P]$, $X_1(y){\leq} X_2(y)$ and 
 \[
 [X_1(y),X_2(y)] =\{x{\in}]0,+\infty[~:~(x,y){\in} D\}.
 \]
 For  $y{\in} [y^{*}_P,y^{**}_P]$, $X_1(y)$, $X_2(y)$ are the unique positive solutions of the equation $P(x,y){=} 1$, and  $X_1(y) {=} X_2(y)$ holds if and only if $y{\in}\{y^*_P, y^{**}_P\}$. 
\end{enumerate}
\end{prop}

The relations $D{\cap}D_1{\subset}D$ and $D{\cap} D_2{\subset}D$ give the inequalities
\be\label{eq-critical-points} 
x^*_P \leq x^* < x^{**} \leq x_P^{**} \quad  \text{and} \quad y^*_P \leq y^* < y^{**} \leq y_P^{**}.
\ee
Note that, since the point $(1,1)$ is an element of $D{\cap} D_1$ and $D{\cap} D_2$, one has also
\be\label{eq2-critical-points} 
x^* \leq 1 \leq x^{**}, \quad  y^* \leq 1 \leq y^{**}. 
\ee

We now define four curves ${\mathcal S}_{11}$ , ${\mathcal S}_{12}$, ${\mathcal S}_{21}$ and ${\mathcal S}_{22}$ on the boundary $\partial D$ of the set $D$. They are also used in the definition of our classification. 
\begin{align}
{\mathcal S}_{11} &\steq{def} \left\{(x,y)\in\partial D:~\partial_x P(x,y) \leq 0,\;\partial_y P(x,y) \leq 0 \right\}, 
\label{eq_S_11}\\ 
{\mathcal S}_{12} &\steq{def} \left\{(x,y)\in\partial D:~\partial_x P(x,y) \leq 0, \; \partial_y P(x,y) \geq 0 \right\},
\label{eq_S_12}\\
{\mathcal S}_{21}  &\steq{def} \left\{(x,y)\in\partial D:~\partial_x P(x,y) \geq 0,\: \partial_y P(x,y) \leq 0 \right\},
\label{eq_S_21}\\
{\mathcal S}_{22}  &\steq{def} \left\{(x,y)\in\partial D:~\partial_x P(x,y) \geq 0, \; \partial_y P(x,y) \geq 0 \right\}.  
\label{eq_S_22}
\end{align}
With Lemma~\ref{preliminary-lemma1} of Section~\ref{preliminary-section}, these curves can be expressed  in terms of the functions $X_1$, $X_2$ and $Y_1$, $Y_2$ of Proposition~\ref{propx*} as follows,
\begin{align}
{\mathcal S}_{11} &= \{(x, Y_1(x)):~ x\in[x^*_P, X_1(y^*_P)]\} = \{ (X_1(y),y):~ y\in[y^*_P, Y_1(x^*_P)]\}, \label{eq2_S_11}\\
{\mathcal S}_{12} &= \{(x, Y_2(x)):~ x\in[x^*_P, X_1(y^{**}_P)]\} = \{ (X_1(y),y):~ y\in[ Y_1(x^*_P), y^{**}_P]\}, \label{eq2_S_12}\\
{\mathcal S}_{21} &= \{(x, Y_1(x)):~ x\in[X_1(y^{*}_P), x^{**}_P]\} = \{ (X_2(y),y):~ y\in[ y^*_P, Y_1(x^{**}_P)]\}, \label{eq2_S_21}\\
{\mathcal S}_{22} &= \{(x, Y_2(x)):~ x\in[X_1(y^{**}_P), x^{**}_P]\} = \{ (X_2(y),y):~ y\in[ Y_1(x^{**}_P), y^{**}_P]\}, \label{eq2_S_22}
\end{align}
and the relations $X_1(y^*_P){=}X_2(y^*_P)$, $X_1(y^{**}_P){=}X_2(y^{**}_P)$, $Y_1(x^*_P){=}Y_2(x^*_P)$ and $Y_1(x^{**}_P){=}Y_2(x^{**}_P)$ hold.

\subsection{A Partition of the Space of Parameters}\label{Part-Sec}
The next proposition shows that there is a partition of eight regions (B$a$), $a{\in}\{0,\ldots,7\}$ for the possible locations of the points $x^{**}$, $Y_1(x^{**})$, $Y_2(x^{**})$, $y^{**}$, $X_1(y^{**})$, $X_2(y^{**})$.  Its proof is given in Section~\ref{section_cases}. As it will be seen, the asymptotic behavior of the Green functions of the Markov chain killed at $0$ depends on the region associated to its parameters. Figures~\ref{F4}, \ref{F7}, \ref{F6}  and~\ref{F8} in Section~\ref{FigSec} of the appendix give an illustration of several situations.

\begin{prop}[Definition of the Classification]\label{cases} Under the assumptions (A1){-}(A3), one and only one of the following cases can occur,
  \begin{enumerate}
  \item[(B0)] \;  $X_1(y^{**}) {<} x^{**} {<} X_2(y^{**})$ and \; $Y_1(x^{**}) {<} y^{**} {<} Y_2(x^{**})$;
\vspace{1mm}
\item[(B1)]\;   $X_1(y^{**}) {<} X_2(y^{**}) = x^{**} {<} x^{**} _P$, \; $Y_1(x^{**}) {<} Y_2(x^{**}) = y^{**}  {<} y^{**}_P$  and 
$(x^{**}, y^{**})\in{\mathcal S}_{22}$;
\vspace{1mm}
\item[(B2)]\;  $X_2(y^{**}) {<} x^{**}$, \; $ Y_2(x^{**}) {<} y^{**}$ and $(x^{**}, Y_2(x^{**})), \, (X_2(y^{**}), y^{**})\in{\mathcal S}_{22}$;
\vspace{1mm}
\item[(B3)] \;   $x^{**} = X_1(y^{**}) {<} x^{**}_P$, \; $ Y_1(x^{**})  {<} Y_2(x^{**}) =  y^{**}$, \; $y^* {\leq} 1 {<} Y_2(x^{**})$   and  $(x^{**}, y^{**}) \in{\mathcal S}_{12}$; 
\vspace{1mm}
\item[(B4)] \;   $x^{**} {<} X_1(y^{**}) {<} x^{**}_P$, \;  $Y_1(x^{**})  {<} Y_2(x^{**}) {<} y^{**}$ , \; $y^*{\leq} 1 {<} Y_2(x^{**})$  and 
\[(x^{**},Y_2(x^{**})),\,(X_1(y^{**}), y^{**}) \in{\mathcal S}_{12};  \]
\item[(B5)] \;   $ y^{**}=Y_1(x^{**}) {<} y^{**}_P$, \;  $X_1(y^{**}) {<} X_2(y^{**}) =  x^{**}$, \; $x^*{\leq} 1 {<} X_2(y^{**})$  and $(x^{**},y^{**})\in{\mathcal S}_{21}$; 
\vspace{1mm}
\item[(B6)] \;   $y^{**} {<} Y_1(x^{**}) {<} y^{**}_P $, \; $X_1(y^{**}) {<} X_2(y^{**}) {<} x^{**}$, \; $x^* {\leq} 1 {<} X_2(y^{**})$  and   
\[
(X_2(y^{**}),y^{**}),\, (x^{**}, Y_1(x^{**}))\in{\mathcal S}_{21};
\]
\item[(B7)] \;   $x^{**} = X_1(y^{**}) = 1$, \; $y^{**} = Y_1(x^{**}) =1$, \; $\partial_x P(1,1) {<} 0$\;  and \; $\partial_y P(1,1) {<} 0$.
\end{enumerate}
\end{prop} 
The cases (B0){-}(B7) have in fact a simple geometrical interpretation. They are determined by the location of the line segments  $[(x^{**}, Y_1(x^{**})), \, (x^{**}, Y_2(x^{**}))]$ and $[(X_1(y^{**}), y^{**}), \, (X_2(y^{**}), y^{**})]$. See Figures~\ref{F3} of Section~\ref{FigSec} of the appendix, where, see Section~\ref{section_cases}, 
\begin{itemize}
  \item the horizontal line segment represents  $[(X_1(y^{**}),y^{**}), (X_2(y^{**}),y^{**})]$;
\item the  vertical \phantom{line segment }"\phantom{ represents }  $[(x^{**}, Y_1(x^{**})), (x^{**},Y_2(x^{**}))]$.
\end{itemize}

The case (B7) corresponds to the case when the Markov chain $(Z(n))$ is transient and escapes to the infinity along the each of the axes  $\{0\}{\times}\N$ and $\N{\times}\{0\}$. See Section~\ref{section-conditions}. In this case, the exact asymptotic of the Green function has been already obtained in the paper \cite{I-K-R} by using methods of Markov-Additive processes. For this reason we will not consider this case.

In the literature, asymptotic results for nearest neighbor random walks in the quadrant are often formulated either under conditions of positive recurrence, see~\cite{F-Y-M-book,Li-Zhao-2018,Malyshev,Miyazawa}, or under conditions of  transience, see \cite{Kurkova-Malyshev}. As it will be seen in Section~\ref{section-conditions}, the classification of Proposition~\ref{cases} is not defined in such a way.  Transience {\em and } recurrence properties  may hold in each of the regions (B$a$), $a{\in}\{3,4,5,6\}$. See Proposition~\ref{recurrence_transience}. The recurrence/transience properties of $(Z(n))$ have in fact a marginal impact in our investigation of the asymptotic behavior of the Green function $(g(j,k))$ of the killed Markov chain. 

\subsection{Convergence Domain and Functional Equation}
For $j{\in}\Z_+^2$, the generating function of $(g(j,k),k{\in}\Z_+^2)$, defined on its  convergence domain in $\C^2$,  is denoted as 
\be\label{HjDef}
H_j(x,y) \steq{def} \sum_{k=(k_1,k_2)\in\Z_+^2{\setminus}\{0\}} g(j,k) x^{k_1}y^{k_2}.
\ee
This is a central set of functions in our analysis. A significant part of our work is devoted to the investigation of their convergence domain and also to determine the nature of the dominant singularities of the functions $x{\to}H_j(x,0)$  and $y{\to}H_j(0,y)$. Once it is done, with Tauberian like theorems and complex analysis arguments,  we will able to derive the asymptotic behavior of the Green function  $(g(j,k),k{\in}\Z_+^2)$ when $k$ goes to infinity. 

Our first important result is that if
\be\label{eq-domain} 
{\Gamma} \steq{def} \{(x,y)\in [0,+\infty[^2:~ x {<} x' \; \text{and} \; y {<} y' \; \text{for some} \; (x',y')\in D\},
    \ee
where $D$ is defined by~\eqref{Ddef},  and 
\begin{align}
x_d &\steq{def}  \begin{cases} x^{**}  &\text{if one of the conditions (B0){-}(B4) holds},\\
X_2(y^{**}) &\text{if either (B5) or (B6) holds},
\end{cases} \label{eq_def_xd}\\
y_d &\steq{def} \begin{cases} y^{**} &\text{if one of the conditions (B0){-}(B2), (B5) or (B6) holds},\\
Y_2(x^{**})  &\text{if either (B3) or (B4) holds}.
\end{cases} \label{eq_def_yd}
\end{align}
then, for  any $j{\in}\Z^2_+$,  the generating function $(x,y){\to}H_j(x,y)$ is analytic on 
\be\label{Omd}
  \Omega_d({\Gamma})\steq{def} \{(x,y)\in \Omega({\Gamma}) :~ |x|{<} x_d, \;  |y| {<} y_d\},
  \ee
In our next result, we will see that the point $x_d$ (resp. $y_d$) is the dominant singularity of the functions $x{\to} H_j(x,0)$ (resp. of the functions $y{\to} H_j(0,y)$) and that the set $\Omega_d({\Gamma})$ is the maximal domain in $\C^2$ where all generating functions $(x,y){\to} H_j(x,y)$, $j{\in}\Z^2_+$, converge. 

For $j{\in}\Z_+^2$, to formulate a key relation for the generating functions $(H_j(x,y))$, we define 
\be\label{Qdef}
Q(x,y) \steq{def} xy(1 - P(x,y))=xy - \sum_{k=(k_1,k_2)\in\Z^2} x^{k_1+1} y^{k_2+1} \mu(k)
\ee
and
\begin{align}
  \psi_1(x,y) &\steq{def} x(1-\phi_1(x,y)) = x - \sum_{k=(k_1,k_2)\in\Z^2} x^{k_1+1} y^{k_2} \mu_1(k),\label{psi1def}\\
  \psi_2(x,y) &\steq{def} y(1-\phi_2(x,y)) = y - \sum_{k=(k_1,k_2)\in\Z^2} x^{k_1} y^{k_2+1} \mu_2(k),\label{psi2def}
\end{align}
\begin{align}
L_j(x,y) &\steq{def}  \begin{cases} x^{j_1}y^{j_2}  -  \P_j(\tau_0 < +\infty) &\text{ if }(j_1,j_2){\not=}(0,0),\\ 
\phi_0(x,y)  -  \P_{(0,0)}(\tau_0 < +\infty) &\text{ if } (j_1,j_2){=}(0,0),
\end{cases} \label{eq-def-Lj} \\
h_j(x,y) &\steq{def} \sum_{k=(k_1,k_2)\in\Z^2_+} g(j, (k_1+1, k_2+1)) x^{k_1} y^{k_2}, \label{series-h}\\
h_{1j}(x) & \steq{def} \sum_{k_1=0}^\infty g(j, (k_1+ 1,0)) x^{k_1}, \quad \text{and} \quad h_{2j}(y) \steq{def} \sum_{k_2=0}^\infty g(j, (0, k_2+1)) y^{k_2}.\label{series-h1-h2}
\end{align}
Under the assumptions (A1){-}(A3), the functions $(x,y) {\to} L_j(x,y)$, $j{\in}\Z^2_+$ are clearly analytic  in a neighborhood  of the set $\Omega(\overline{\Gamma})$, see relation~\eqref{OmDef},  and the functions $(x,y){\to} Q(x,y)$  and $(x,y){\to} \psi_i(x,y)$, $i{=}1$, $2$, can be analytically continued to a neighborhood  of $\Omega(\overline{\Gamma})$.
\begin{theorem}[Convergence Domain and Functional Equation]\label{theorem1} Under the assumptions (A1){-}(A4), for any $j{\in}\Z^2_+$,  the following assertions hold 
\begin{enumerate}[label=\roman*)]
\item  The function $x{\to}h_{1j}(x)$, resp. $y{\to}h_{2j}(y)$,  is analytic in  $B(0, x_d)$, resp. in $B(0, y_d)$.
\item  On the set $\Omega_d({\Gamma})$ the function $(x,y) {\to}h_j(x,y)$ is analytic and the relation
  \be\label{extended-functional-equation} 
Q(x,y) h_j(x,y) = L_j(x,y) + \psi_1(x,y) h_{1j}(x) + \psi_2(x,y) h_{2j}(y)
\ee
holds.
\end{enumerate}
\end{theorem}

When the random walk $(Z(n))$ is positive recurrent, Theorem~\ref{theorem1} has been established in Kobayashi and Miyazawa~\cite{Kobayashi-Miyazawa}. With  Proposition~\ref{cases}, the transience or recurrence properties do not play  a role in our proof of this result.

Section~\ref{proof-theorem1} is devoted to the proof of Theorem~\ref{theorem1}. We give a  sketch of it. 
\begin{enumerate}
\item We first prove that the series~\eqref{series-h1-h2} and~\eqref{series-h} converge in $\Omega(\Theta)$ for a logarithmically convex set $\Theta{\subset}[0,+\infty[^2$ whose boundary contains the line segments $[(x_d,0), (x_d, Y_1(x_d))]$ and $[(0, y_d), (X_1(y_d), y_d)]$  and such that $\inter\Theta{\cap}\inter{D} {\not=} \emptyset$ and $\Theta{\cup} \inter{D} {=}\{ (x,y){\in}  \Gamma : x {<} x_d, y {<} y_d\}$. See Proposition~\ref{upper-bounds} and Lemma~\ref{lemma1-proof-theorem1}. An important ingredient of the proof of this step is the use of  Lyapunov functions. See also Figure~\ref{F5}  of Section~\ref{FigSec} of the appendix.
  \item  The functional equation~\eqref{extended-functional-equation} is established on the set  $\Omega(\Theta)$.
From there, we get that, for $j{\in}\Z_+^2$, the functions $(x,y){\to}Q(x,y) h_j(x,y)$ can be analytically continued to the set $\Omega_d({\Gamma})$. 
\item Since the function $(x,y){\to}h_j(x,y)$ is analytic in $\Omega(\Theta)$ and the function $(x,y){\to} 1/Q(x,y)$ is analytic in $\Omega(\inter{D})$, from these results, we will be able to deduce that the function $(x,y){\to} h_j(x,y)$ can be continued as an analytic function to the set $\Omega_d({\Gamma})$. 
\end{enumerate}

\bigskip

\noindent
{\bf Remark.}\\
Clearly, for $j{\in}\Z_+^2$ and  $x{\ne}0$, $y{\ne}0$,
\[
H_j(x,0) = x h_{1j}(x), \quad H_j(0,y) = y h_{2j}(y), \quad H_j(x,y){-}H_j(x,0){-}H_j(0,y) = xy h_j(x,y),
\]
and the functional equation~\eqref{functional-equation} of the introduction is equivalent to relation~\eqref{extended-functional-equation}. The technical advantage of the formulation~\eqref{extended-functional-equation} is that the functions $\psi_1$, $\psi_2$ and $Q$ are also defined when $x{=}0$ or $y{=}0$.  In some cases, however, the expression~\eqref{functional-equation} is more convenient.

\medskip

\subsection{Singularity Analysis of $H_j$} 
Since the set $\Omega_d({\Gamma})$ can be represented as a union of the poly-discs centered at the  origin of $\C^2$, the above theorem proves that for any $j{\in}\Z^2_+$, the generating function  $H_j(x,y)$ converge in $\Omega_d({\Gamma})$. The results of this section establish that the set  $\Omega_d({\Gamma})$ defined by~\eqref{Omd}  is the exact domain of convergence of the functions $H_j$, $j{\in}\Z^2_+$, i.e. the maximal domain in $\C^2$ where the series \eqref{HjDef} converge, and the dominant singularities of the functions $x{\to}H_j(x,0)$  and $y{\to}H_j(0,y)$  and their nature are identified.

Before stating these results, in order to have straight assertions, we have to fix a (small) technical problem related to irreducibility. It must be noted that, even under the assumptions (A1){-}(A3), the killed Markov chain $(Z_{\tau_0}(n))$, see Definition~\ref{Killed},  is not necessarily irreducible on the space $\Z^2_+{\setminus}\{0\}$.  One may have $g(j,k){=}0$ for some $j$, $k{\in}\Z^2_+{\setminus}\{0\}$. The following lemma solves this problem. 
\begin{lemma}\label{irreductibility-lemma} Under the assumptions (A1){-}(A3), there exist $N_0{\geq}0$ and a finite subset $E_0$ of $\Z^2_+{\setminus}\{(0,0)\}$  such that  for  $j$, $k{\in}\Z^2_+$, 
\begin{align}
g(j,k) = 0 \quad &\text{ if } j{\in}E_0 \text{ and } k{\not\in} E_0,\label{eq-set-E0-1}\\
g(j,k) > 0 \quad &\text{ if } j{\not\in}E_0 \text{ and }\|k\|{\geq} N_0.\label{eq-set-E0-2}
\end{align}
\end{lemma}
Consequently, to  investigate the asymptotics of $g(j,k)$ as $k$ goes to infinity, it is sufficient  to consider a starting point $j$ outside  $E_0$.  The proof of this lemma is given in Section~\ref{pfirr} of the appendix. From now on, the integer $N_0{\geq}0$ and the set $E_0$ satisfying \eqref{eq-set-E0-1} and \eqref{eq-set-E0-2} are fixed. When the killed Markov chain $(Z_{\tau_0}(n))$ is irreducible on $\Z^2_+{\setminus\{(0,0)\}}$, the set $E_0$ is of course empty. 

\bigskip
Proposition~\ref{cases} and the definition of the points $x_d$ and $y_d$ show that the relations, $x_d{\leq}x^{**}{\leq} x^{**}_P$ and $y_d{\leq}y^{**}{\leq} y^{**}_P$ hold.  For $j\in\Z^2_+{\setminus}E_0$, we will prove that  $x_d$, resp.  $y_d$, is the dominant singularity of the function $x{\to}H_j(x,0)$, resp. $y\to H_j(0,y)$.  We will see that the  nature of the singularity $x_d$, resp. $y_d$,  is determined by the cases (B0){-}(B6) and also by several relations between  the quantities $x_d$, $x^{**}$ and $x^{**}_P$, resp.  $y_d$, $y^{**}$ and $y^{**}_P$.  

By relation~\eqref{eq-critical-points}, the points $x^{**}$ and $y^{**}$ are respectively in $[x^*_P, x^{**}_P]$ and $[y^*_P, y^{**}_P]$, and from the definition of the functions $Y_2: [x^*_P, x^{**}_P]{\to}[y^*_P, y^{**}_P]$ and $X_2: [y^*_P, y^{**}_P]{\to}[x^*_P, x^{**}_P]$,  we have
\[
Y_2(x^{**})\leq y^{**}_P \quad \text{and} \quad X_2(y^{**}) \leq x^{**}_P.
\] 
Proposition~\ref{cases} gives that $y^{**} {<} y^{**}_P$ if  one of the cases (B0),(B1), (B5) or (B6)  holds and, similarly,  $x^{**}{<} x^{**}_P$ if one of the cases (B0){-}(B4)  holds. This is summarized as follows. 
\be\label{eq2-def-xd}
x_d{=} \begin{cases} x^{**} {<} x^{**}_P &\text{ if one of the cases (B0),(B1), (B3) or (B4) holds},\\
x^{**}{\leq} x^{**}_P &\text{ with  a possible equality } x^{**}{=}x^{**}_P, \text{ if  (B2) holds},\\
X_2(y^{**}){=} x^{**} {\leq} x^{**}_P&\phantom{ with possible}"\phantom{aaa equality } x^{**}{=} x^{**}_P, \text{ if  (B5) holds},\\
X_2(y^{**}) {<} x^{**} {\leq} x^{**}_P&\text{ if  (B6) holds},
\end{cases} 
\ee
and 
\be\label{eq2-def-yd}
y_d {=} \begin{cases} y^{**}  < y^{**}_P &\text{ if one of the cases (B0),(B1), (B5) or (B6) holds},\\
y^{**}  {\leq} y^{**}_P&\text{ with a possible equality } y^{**}{=} y^{**}_P, \;\text{ if (B2)  holds},\\
Y_2(x^{**}) {=} y^{**}{\leq} y^{**}_P &\phantom{ with possible}"\phantom{aaa equality } y^{**}{=} y^{**}_P,\;\text{ if (B3)  holds},\\ 
Y_2(x^{**}) {<} y^{**}{\leq} y^{**}_P &\text{ if (B4)  holds}.
\end{cases} 
\ee
We now introduce several functions on $\Z_+^2$ which will be used to describe the dependence on the initial state $j{\in}\Z_+^2$ in the asymptotic behavior of   $g(j,k)$ when $k$ goes to infinity. Under convenient conditions,  as we will see, these functions are harmonic for the killed Markov chain at $0$.
\begin{defi}[Functions for the Dependence on the Initial State]
For  $j\in\Z^2_+$, 
\begin{align}
\varkappa_1(j) &\steq{def} L_j(x_d, Y_1(x_d)) + \left(\phi_2(x_d, Y_1(x_d))-1\right) H_j(0, Y_1(x_d)),\label{eq-def-kappa-1}\\
\tilde\varkappa_1(j) &\steq{def} \left. \partial_y\left( \frac{L_j(x, y) + \left(\phi_2(x, y)-1\right) H_j(0, y)}{1-\phi_1(x, y)}\right)\right|_{(x,y)=(x^{**}_P,Y_1(x^{**}_P))}, \label{eq-def-tilde-kappa-1}\\
\varkappa_2(j) &\steq{def} L_j(X_1(y_d), y_d) + \left(\phi_1(X_1(y_d), y_d)-1\right) H_j(X_1(y_d), 0),  \label{eq-def-kappa-2}\\
\tilde\varkappa_2(j) &\steq{def} \left. \partial_x\left( \frac{L_j(x, y) + \left(\phi_1(x, y)-1\right) H_j(x,0)}{1-\phi_2(x, y)}\right)\right|_{(x,y)=(X_1(y^{**}_P), y^{**}_P)}, \label{eq-def-tilde-kappa-2}\\
\varkappa_{(x,y)}(j) &\steq{def}  L_j(x,y) + (\phi_1(x,y)-1) H_j(x,0) + (\phi_2(x,y)-1) H_j(0,x), \label{def-kappa-interior-direction} 
\end{align}
\end{defi}
By Theorem~\ref{theorem1}, and since the functions $\psi_1$,$\psi_2$, $Q$ and $L_j$ are analytic in a neighborhood of the set $\Omega(\overline\Gamma)$, we have 
\begin{itemize}
\item[--] $\varkappa_1$ is well defined on $\Z^2_+$  if $Y_1(x_d){<}x_d$, that is, if one of the cases (B0){-}(B4) holds;
\item[--] $\tilde\varkappa_1$ \phantom{ifwell defined}"\phantom{on $\Z^2_+$ }  if (B2), $x_d{=}x^{**}{=}x^{**}_P$ and  $\phi_1(x^{**}_P, Y_1(x^{**}_P)){<}1$  hold;
\item[--] $\varkappa_2$ \phantom{ifwell defined}"\phantom{on $\Z^2_+$ }  if one the cases (B0){-}(B2), (B5) or (B6) holds;
\item[--] $\tilde\varkappa_2$ \phantom{ifwell defined}"\phantom{on $\Z^2_+$ }  if (B2), $y_d{=}y^{**}{=} y^{**}_P$ and $\phi_2(X_1(y^{**}_P), y^{**}_P){<}1$ hold;
\item  $\varkappa_{(x,y)}$ \phantom{ell defined}"\phantom{on $\Z^2_+$ }  for  $(x,y){\in}\Omega(\overline{\Gamma})$ such that  $|x|{<} x_d$ and $|y|{<} y_d$. 
\end{itemize} 
The following theorem gives a complete description of the nature of the singularity $x_d$ for the function $x{\to}H_j(x,0)$. 
Note that the case (B7) is not considered because, as mentioned before, it has been already investigated in \cite{I-K-R}.
\begin{theorem}[Singularity Analysis of $(H_j)$]\label{theorem2}
  Under the assumptions (A1){-}(A3), the following assertions hold.   
\begin{enumerate}[label=\roman*)]
\item  If   one of the cases (B0),(B1), (B3), (B4) holds or (B2)  and $x_d{<} x^{**}_P$ hold,  then  
\begin{itemize} 
\item[--] there exists $\eps{>}0$ such that, for any $j{\in}\Z^2_+$,   the function $x{\to} H_j(x,0)$  can be  analytically continued to the set $B(0,x_d{+}\eps){\setminus}\{x_d\}$;  
\item[--] the function $\varkappa_1$ of~\eqref{eq-def-kappa-1} is non-negative on $\Z_+^2$,  harmonic for the Markov chain  $(Z(n))$  killed at $0$ and  positive  on the set $\Z^2_+{\setminus}E_0$;
\item[--] for any $j\in\Z^2_+$, 
\be\label{eq1-theorem2} 
\lim_{x\to x_d} (x_d{-}x)H_j(x,0) = a_1 \, \varkappa_1(j),
\ee 
where
\be\label{eq2b-cases-B0-B4-theorem2} 
 a_1 = \left.\left(\frac{d }{dx} \phi_1(x, Y_1(x))\right|_{x=x_d}\right)^{-1}  > 0.  
\ee
\end{itemize} 
\item If (B2) and $x_d{=}x^{**}_P$ hold, then 
\begin{itemize} 
\item[--] there exists $\eps{>}0$ such that, for any $j{\in}\Z^2_+$, the function $x{\to} H_j(x,0)$ can be analytically  continued to $B(0,x^{**}_P{+}\eps){\setminus} [x^{**}_P, x^{**}_P{+}\eps]$;  
\item[--] if $\phi_1(x_d, Y_1(x_d)){=}1$, then the function $\varkappa_1$ is  non-negative $\Z^2_+$, harmonic for the Markov chain  $(Z(n))$  killed at $0$, positive  on the set $\Z^2_+{\setminus}E_0$  and for any $j{\in}\Z^2_+$,  
\be\label{eq-analytical-structure-IIa}
\lim_{x\to x_d} \sqrt{x_d{-}x}\, H_j(x,0) = a_2\,  \varkappa_1(j),
\ee
where the limit is taken in the set $B(0,x^{**}_P{+}\eps){\setminus} [x^{**}_P, x^{**}_P{+}\eps]$ and  
\be\label{eq4a-case-B2-theorem2} 
 a_2 = \left.\left(\partial_y\phi_1(x,y) \sqrt{\partial_x P(x,y)/\partial^2_{yy}P(x,y) } \right)^{-1} \right|_{(x,y) = (x_d, Y_1(x_d))}  > 0;
 \ee
\item[--] if $\phi_1(x_d, Y_1(x_d)){<} 1$, then the function $\tilde\varkappa_1$  of~\eqref{eq-def-tilde-kappa-1} is  non-negative on $\Z^2_+$, harmonic for the Markov chain  $(Z(n))$  killed at $0$,  positive  on the set $\Z^2_+{\setminus} E_0$, and for any $j{\in}\Z^2_+$, 
\be\label{eq-analytical-structure-IIb}
\lim_{x\to x_d} \sqrt{x_d{-}x}\, \frac{d}{dx} \, H_j(x,0) = a_3\,  \tilde\varkappa_1(j), 
\ee
where the limit is taken in the set $B(0,x^{**}_P{+}\eps){\setminus} [x^{**}_P, x^{**}_P{+}\eps]$ and 
\be\label{eq2b-case-B2-theorem2} 
{a}_3 =  \left.\frac{1}{2} \sqrt{\partial_x P(x,y)/\partial^2_{yy}P(x,y) }  \right|_{(x,y) = (x_d, Y_1(x_d))} > 0. 
\ee
\end{itemize}
\item If (B5) and  $x_d{<}x^{**}_P$ hold, then  
\begin{itemize} 
\item[--] there exists $\eps{>}0$ such that, for any $j{\in}\Z^2_+$, the function $x{\to} H_j(x,0)$ can be analytically  continued to the set $B(0, x_d{+}\eps){\setminus}\{x_d\}$;
\item[--] the function $\varkappa_2$  of~\eqref{eq-def-kappa-2} is  non-negative on $\Z^2_+$, harmonic for the Markov chain  $(Z(n))$  killed at $0$, positive  on the set  $\Z^2_+{\setminus}E_0$  and for any $j{\in}\Z^2_+$, 
\be\label{eq-analytical-structure-IV}
\lim_{x\to x_d} (x_d{-}x)^2 H_j(x,0) = a_4 \,  \varkappa_2(j),
\ee
where 
\be\label{eq-analytical-structure-IVb}
a_4 =  (\phi_2(x, y){-}1)\left. \left(\frac{d}{d x} \phi_1(x, Y_1(x))\frac{d }{dy} \phi_2(X_1(y), y) \frac{d }{dx}Y_1(x)\right)^{-1}\right|_{(x,y)=(x_d,y_d)} > 0.
\ee
\end{itemize} 
\item If (B5) and $x_d{=} x^{**}_P$ hold, then  
\begin{itemize} 
\item[--] there exists $\eps{>}0$ such that,  for any $j{\in}\Z^2_+$, the function $x\to H_j(x,0)$ can be continued as an analytic functions to the set $B(0,x^{**}_P{+}\eps){\setminus} [x^{**}_P, x^{**}_P{+}\eps]$; 
\item[--] the function $\varkappa_2$ is   non-negative on $\Z^2_+$, harmonic for the Markov chain  $(Z(n))$  killed at $0$ and  positive  on the set  $\Z^2_+{\setminus} E_0$; 
\item[--] if $\phi_1(x_d, Y_1(x_d)) {=} 1$ then  for any $j{\in}\Z^2_+$, 
\be\label{eq-analytical-structure-Va}
\lim_{x\to x_d} (x_d-x)\, H_j(x,0) = a_5\,  \varkappa_2(j), 
\ee
where the limit  is taken in the set $B(0,x^{**}_P{+}\eps){\setminus} [x^{**}_P, x^{**}_P{+}\eps]$ and 
\be\label{eq-analytical-structure-Vab}
a_5 = \left. (\phi_2(x,y){-}1) \partial^2_{yy}P(x,y)\left( \partial_y\phi_1(x,y) \partial_x P(x,y) \frac{d \phi_2}{dy} (X_1(y), y) \right)^{-1}\right|_{(x,y){=}(x_d,y_d)} > 0.
\ee
\item[--] if $\phi_1(x_d, Y_1(x_d)) {<} 1$ then  for any $j{\in}\Z^2_+$, 
\be\label{eq-analytical-structure-Vb}
\lim_{x\to x_d} \sqrt{x_d-x}\, H_j(x,0) = {a}_6\,  \varkappa_2(j), 
\ee
where the limit  is taken in the set $B(0,x^{**}_P{+}\eps){\setminus} [x^{**}_P, x^{**}_P{+}\eps]$ and 
\be\label{eq-analytical-structure-Vbb}
{a}_6 = \left.(\phi_2(x,y){-}1) \sqrt{\frac{\partial^2_{yy}P(x,y)}{\partial_x P(x,y)}}\left((1{-}\phi_1(x,y)) \frac{d }{dy} \phi_2(X_1(y), y) \right)^{-1} \right|_{(x,y)=(x_d,y_d)}> 0.
\ee
\end{itemize} 
\item  If  (B6)  holds, then 
\begin{itemize} 
\item[--] there exists  $\eps{>}0$ such that, for any $j{\in}\Z^2_+$, the function $x{\to} H_j(x,0)$ can be analytically  continued to the set $B(0, x_d{+}\eps){\setminus}\{x_d\}$;
\item[--] the function $\varkappa_2$ is  non-negative on $\Z^2_+$, harmonic for the Markov chain  $(Z(n))$  killed at $0$, positive   on the set  $\Z^2_+{\setminus}E_0$  and  for any $j{\in}\Z^2_+$, 
\be\label{eq-analytical-structure-III}
\lim_{x\to x_d} (x_d {-} x)H_j(x,0) ~=~  a_7\,  \varkappa_2(j),
\ee
where
\be\label{eq-analytical-structure-IIIb}
a_7 = \left.(\phi_2(x, y){-}1) \left((1{-}\phi_1(x, y))\frac{d  }{dy}\phi_2(X_1(y), y)  \frac{d }{dx} Y_1(x)\right)^{-1}\right|_{(x,y)=(x_d,y_d)}  > 0.
\ee
\end{itemize} 
\item If (B2) holds, then 
\begin{itemize} 
\item[--] the set $\{(x,y){\in}{\mathcal S}_{22}: \; x {<} x_d, \; y {<} y_d\}$ is non-empty; 
\item[--] there exists a neighborhood ${\mathcal V}$ of the set ${\mathcal S}_{22}$ in $\R^2_+$ such that, for any $j{\in}\Z^2_+$,\\ the function $x{\to} (1{-}P(x,y))H_j(x,y)$ can be  analytically  continued to the set $\{(x,y){\in}\Omega({\mathcal V}): |x| {<} x_d,  |y| {<} y_d\}$; 
\item[--] for any $(\hat{x},\hat{y}){\in} \{(x,y){\in}{\mathcal S}_{22}: x{<} x_d, y {<} y_d\}$,  the function $\varkappa_{(\hat{x},\hat{y})}$ is non-negative on $\Z^2_+$, harmonic for the Markov chain  $(Z(n))$  killed at $0$,  positive on the set $\Z^2_+{\setminus}E_0$ and for any $j{\in}\Z^2_+$, 
\be\label{limit-R}
\lim_{\substack{(x,y) \to (\hat{x},\hat{y})\\ (x,y)\in\inter{D}}} (1{-}P(x,y))H_j(x,y) =  \varkappa_{(\hat{x},\hat{y})}(j).
\ee
\end{itemize} 
  \end{enumerate}
\end{theorem}
 By symmetry (it is sufficient to  to exchange  the roles of $x$ and $y$), the analogous results for the assertions i)-v) of Theorem~\ref{theorem2} hold for the functions $y{\to}H_j(0,y)$, \, $j{\in}\Z^2_+{\setminus}E_0$.  
 
 The last assertion of Theorem~\ref{theorem2} is obtained as a consequence of Theorem~\ref{theorem1}. The proofs of the first five assertions are more demanding as it will be seen.

\bigskip

\noindent
{\bf Remark.}\\
In the context of  the functional equation~\eqref{extended-functional-equation}, a classical approach of the literature  consists in finding a suitable analytic function $x{\to}Y(x)$ (resp. $y{\to}X(y)$) satisfying the equation $Q(x,Y(x)) {=} 0$ for any $x$, resp. $Q(X(y),y){=}0$ for any $y$,  in some domain large enough, in order to inject  $y{=}Y(x)$ (resp $x{=}X(y)$) in~\eqref{extended-functional-equation}.

For nearest neighbor  random walks, see Malyshev~\cite{Malyshev},  the equation $Q(x,y){=}0$ is quadratic in $x$ and $y$, its solutions have therefore an explicit form. In the general case, when the jumps of the random walk are unbounded, one clearly cannot find  the functions $Y$ and $X$ in  such a way. See Remark~2.3 of Kobayashi and Miyazawa~\cite{Kobayashi-Miyazawa}. 

\subsection*{Probabilistic Representations of ${\mathbf X_1}$ and ${\mathbf Y_1}$}
In our analysis, a part of the technicalities of the literature related to analytic continuation of the functions $(Y(x))$ and $(X(y))$ mentioned above is avoided via a  probabilistic argument.

We show that   the functions $Y_1{:}[x^*_P, x^{**}_P]{\to}[y^*_P, y^{**}_P]$ and  $X_1{:} [y^*_P, y^{**}_P] {\to} [x^*_P, x^{**}_P]$, introduced in Proposition~\ref{propx*}, have a probabilistic representation. With this result, we are able to prove that there exists  $\eps{>} 0$ such that  the function $X_1$, resp $Y_1$,  can be analytically continued to the set
\[
\{x{\in}\C{:}x^*_P {<} |x| {<} x^{**}_P{+}\eps,  x{\not\in}[x^{**}_P, x^{**}_P{+}\eps[\}, \text{ resp.  }\{y{\in}\C{:} y^*_P {<} |y| {<} y^{**}_P{+}\eps, \; y{\not\in}[y^{**}_P,y^{**}_P{+}\eps[\}.
        \]
       From there,  several important properties of the analytic continuation of  $Y_1$ and $X_1$ are then derived. As a consequence one can  inject $y{=}Y_1(x)$ or $x{=}X_1(y)$  in the equation \eqref{extended-functional-equation}, and then to establish Theorem~\ref{theorem2}. 

\subsection{Asymptotics of the Green function along the Axes} In this section, we investigate the asymptotics of the Green function $g(j,k)$ when one of  coordinate of $k{=}(k_1,k_2){\in}\Z_+^2$ is fixed, i.e.  as $k_1{\to}{+}\infty$ with $k_2$ fixed or  $k_2{\to}{+}\infty$ and $k_1$ fixed. By symmetry it is enough to consider only the first convergence.

We define $\nu_1(0){=}1$ and, for $n{\ge}1$, 
\be\label{eq-def-nu-1}
\nu_1(n) \steq{def}  \frac{1}{(n{-}1)!} \left.\frac{\partial^{n-1}}{\partial y^{n-1}} \left(\frac{\psi_1(x_d, y)}{Q(x_d, y)}\right) \right|_{y=0},  
\ee
where $Q$  and $\psi_1$ are defined by~\eqref{Qdef} and~\eqref{psi1def} respectively. 
As the following theorem shows, the quantity $\nu_1(k_2)$ expresses the dependence on $k_2$  in the limiting behavior of $k_1{\to}g(j,(k_1,k_2)))$ when $k_1{\to}{+}\infty$.

In Section~\ref{section-proof-theorem3},  it is shown that $\nu_1$ is, up to a multiplicative constant, the invariant distribution of a twisted version of a random walk on $\Z\times\Z_+$ obtained by removing the boundary $\{0\}{\times}\Z_+$. It will show in particular that the coefficients $\nu_1(n)$, $n{\in}\Z_+$, are positive. 

 The case (B7), already analyzed, excepted, the following result gives a complete description of all possible cases for the asymptotic behavior of the Green function $g(j,(k_1,k_2))$ as $k_1\to+\infty$ for a fixed $k_2{\in}\Z_+$. 
\begin{theorem}[Asymptotics of Green Function with a Fixed Second Component]\label{theorem3}  Under the assumptions (A1){-}(A3), the following assertions hold. 
\begin{enumerate}  
\item If  either one of the cases (B0),(B1), (B3), (B4) holds or (B2) and $x_d{<} x^{**}_P$ hold,  then for any $j{\in}\Z^2_+{\setminus}E_0$ and $k_2{\in}\Z_+$,  as $k_1{\to}{+}\infty$, 
\be\label{eq1-theorem3} 
g(j, (k_1,k_2))  ~\sim~ a_1 \, \nu_1(k_2) \varkappa_1(j) x_d^{-k_1-1}.
\ee
where  $\varkappa_1(j) > 0$ and $a_1>0$ are defined   respectively by \eqref{eq-def-kappa-1} and~\eqref{eq2b-cases-B0-B4-theorem2}.
\item  If (B2) and $x_d{=} x^{**}_P$  hold, then for any $j{\in}\Z^2_+{\setminus} E_0$ and $k_2{\in}\Z_+$, as $k_1{\to}{+}\infty$, 
 \begin{align}
g(j, (k_1,k_2)) &\sim a_2\,   \nu_1(k_2) \varkappa_1(j) x_d^{- k_1} \left(\sqrt{\pi k_1 x_d}\right)^{-1} \quad\text{if} \quad  \phi_1(x_d, Y_1(x_d)){=}1, \label{eq2-theorem3}\\
 g(j, (k_1,k_2)) &\sim  {a}_3\,   \nu_1(k_2)  \tilde\varkappa_1(j) x_d^{- k_1} \left(k_1\sqrt{\pi k_1 x_d}\right)^{-1}  \quad\text{if} \quad  \phi_1(x_d, Y_1(x_d)){<} 1, \label{eq3-theorem3}
\end{align} 
where  $\varkappa_1(j)>0$, $\tilde\varkappa_1(j)>0$, $a_1>0$ and ${a}_2>0$ are defined respectively by~\eqref{eq-def-kappa-1}, \eqref{eq-def-tilde-kappa-1}, \eqref{eq4a-case-B2-theorem2} and \eqref{eq2b-case-B2-theorem2}.
\item If (B5) holds and $x_d{<}x^{**}_P$,  then for $j{\in}\Z^2_+{\setminus} E_0$ and $k_2{\in}\Z_+$,  as $k_1{\to}{+}\infty$, 
\be\label{eq4-theorem3} 
g(j, (k_1,k_2)) ~\sim~  a_4  \nu_1(k_2)  \varkappa_2(j) k_1 x_d^{-k_1-2},
\ee
where  $\varkappa_2(j)>0$ and $a_4>0$ are defined by~\eqref{eq-def-kappa-2} and~\eqref{eq-analytical-structure-IVb}.
\item If (B5) holds and $x_d{=}x^{**}_P$,  then for  $j{\in}\Z^2_+{\setminus}E_0$ and $k_2{\in}\Z_+$, as $k_1{\to}{+}\infty$,
\begin{align}
 g(j, (k_1,k_2)) &\sim~  a_5\,   \nu_1(k_2)  \varkappa_2(j) x_d^{- k_1-1} \quad \text{if } \quad \phi_1(x_d, Y_1(x_d)){=}1,\label{eq5-theorem3} \\
 g(j, (k_1,k_2)) &\sim~   {a}_6\,   \nu_1(k_2) \varkappa_2(j) x_d^{-k_1} \left(\sqrt{\pi k_1 x_d} \right)^{-1}  \quad \text{if } \quad  \phi_1(x_d, Y_1(x_d)){<} 1,\label{eq6-theorem3} 
\end{align} 
where  $\varkappa_2(j)>0$,  $a_5>0$ and  ${a}_6>0$ are defined respectively by~\eqref{eq-def-kappa-2}, \eqref{eq-analytical-structure-Vab} and~\eqref{eq-analytical-structure-Vbb}. 
\item  If  (B6)  holds,  then for   $j{\in}\Z^2_+{\setminus}E_0$ and $k_2{\in}\Z_+$, as $k_1{\to}{+}\infty$, 
\be\label{eq7-theorem3} 
g(j, (k_1,k_2)) ~\sim~  a_7  \nu_1(k_2) \varkappa_2(j) x_d^{-k_1-1},
\ee
where  $\varkappa_2(j)>0$ and $a_7>0$ are defined respectively by \eqref{eq-def-kappa-2} and~\eqref{eq-analytical-structure-IIIb}. 
\end{enumerate} 
\end{theorem}
For $k_2 = 0$, this result is obtained as a straightforward consequence of Theorem~\ref{theorem2} by using the  Tauberian-like theorem  due to Flajolet and Odlyzko~\cite{FlajoletO}, see Corollary~VI.1 of~\cite{Flajolet-Sedgewick}). To get this result for $k_2 > 0$, we prove that for any $j\in\Z^2_+\setminus E_0$ and $k_2 \in\Z_+$,
\[
\lim_{n\to+\infty} g(j,(n,k_2))/g(0,(n,k_2)) = \nu_1(k_2), \quad k_2\in\Z_+, 
\]
by using a probabilistic representation of the coefficients $\nu_1(k_2)$, $k_2\in\Z_+$. The proof of this result is given in Section~\ref{section-proof-theorem3}. 

\subsection{Asymptotics of the Green function along Directions of ${\mathbb S}^1_+$} 
In this section, we present the asymptotics of the Green function $g(j,(k_1,k_2))$ as $\min\{k_1, k_2\}\to\infty$ and $k/\|k\|{\to}w$ where $w$, the direction,   is an element of ${\mathbb S}^1_+{=}\{(u_1,u_2){\in}\R^2_+:u_1^2{+} u_2^2{=} 1\}$.  

For each of the cases (B0){-}(B6), we will introduce  subsets ${\mathcal W}_0$, ${\mathcal W}_1$ and ${\mathcal W}_2$ of ${\mathbb S}^1_+$  used  to define the partition of the set of directions in ${\mathbb S}^1_+$. This is achieved by the  definitions~\ref{Wdef-B0}, \ref{Wdef-B1}, \ref{Wdef-B2} and~\ref{Wdef-B3-B6} in Section~\ref{region-of-directions}. A critical direction $w_c{=}(u_c,v_c)$ in ${\mathbb S}^1_+$ will play a role in several cases. As it will be seen this partition  will determine a structure of  the asymptotics behavior of the Green function $k{\to}g(j,k)$.

\subsubsection{Diffeomorphism between  ${\mathbb S}^1_+$ and ${\mathcal S}_{22}$}\label{def-homeomorphism-sphere-boundary}
The sets ${\mathbb S}^1_+$ and ${\mathcal S}_{22}$ are defined by~\eqref{defS1} and~\eqref{eq_S_22}. 
Under  Assumptions (A1), see for instance Ney and Spitzer~\cite{Ney-Spitzer}, the Laplace transform
\be\label{function-tilde-P}
(\alpha,\beta) \to \widetilde{P}(\alpha,\beta) \steq{def} P(e^\alpha, e^\beta) = \sum_{k=(k_1,k_2)\in\Z^2} \mu(k) e^{\alpha k_1 + \beta k_2}. 
\ee
is strictly convex, the level set $\widetilde{D} \steq{def} \{(\alpha,\beta){\in} \R^2:\tilde{P}(\alpha, \beta){\leq}1\}$
 is strictly convex and compact, the gradient $\nabla_{\alpha,\beta} \tilde{P}(\alpha,\beta)$ does not vanishes on the boundary  $\partial \widetilde{D}$ of  $\widetilde{D}$, and the function
\be\label{diffeomorphism-log}
(\alpha,\beta)\to \nabla_{\alpha,\beta} \widetilde{P}(\alpha, \beta)/\|\nabla_{\alpha,\beta} \widetilde{P}(\alpha, \beta)\|
\ee
determines a diffeomorphism from $\partial \widetilde{D}$ to the unit circle ${\mathbb S}^1{=}\{(u_1,u_2){\in}\R^2: u_1^2 {+} u_2^2 {=} 1\}$.  

Since for $x{=}e^\alpha$ and $y{=}e^\beta$, one has 
 \be\label{eq2-diffeomorphism-point-direction}
 \partial_\alpha \widetilde{P}(\alpha, \beta) = x\partial_x P(x,y) \quad \text{and} \quad \partial_\beta \tilde{P}(\alpha, \beta) = y\partial_yP(x,y), 
 \ee
for the set  $D$ defined by~\eqref{Ddef},  it follows that the function
\begin{multline}\label{diffeomorphism-point-direction}
(x,y)\to w_D(x,y) = (u_D(x,y), v_D(x,y))  \\ 
~\steq{def}~ \frac{1}{\sqrt{x^2(\partial_xP(x,y))^2 + y^2 (\partial_y P(x,y))^2}} (x\partial_x P(x,y), \, y\partial_y P(x,y)) 
\end{multline}
is a  diffeomorphism from  $\partial D$ to  ${\mathbb S}^1$ and, by  the definition of ${\mathcal S}_{22}$, from  ${\mathcal S}_{22}$ to ${\mathbb S}^1_+$.
We denote by $w{\to}(x_D(w), y_D(w))$, resp. $w{\to} (\alpha_D(w),\beta_D(w))$, the inverse mapping of the function~\eqref{diffeomorphism-point-direction},  resp. of~\eqref{diffeomorphism-log}. 

Remark that, for $w_1{=}(u_1,v_1)$ and  $w_2{=}(u_2,v_2){\in}{\mathbb S}^1_+$, the inequalities $u_1{<} u_2$ and $v_1 {>} v_2$ are equivalent and, since the set $\tilde{D}$ is strictly convex, $\alpha_D(u_1,v_1){<} \alpha_D(u_2,v_2)$, resp. $\beta_D(u_1,v_1){<} \beta_D(u_2,v_2)$,  if and only if  $u_1 {<} u_2$, resp. $v_1 {>} v_2$. Hence, using again \eqref{eq2-diffeomorphism-point-direction} for $x{=}e^\alpha$ and $y{=}e^\beta$, one gets a similar property for the diffeomorphism $w{\to}(x_D(w), y_D(w))$:

\begin{lemma}\label{lemma-diffeomorphism-point-direction} Under Assumption (A1), for any $w_1{=}(u_1,v_1)$, $w_2{=}(u_2,v_2){\in}{\mathbb S}^1_+$, then
\[
x_D(u_1,v_1) < x_D(u_2,v_2) \; \Leftrightarrow \; u_1 < u_2 \; \Leftrightarrow \; v_1 > v_2 \; \Leftrightarrow \; y_D(u_1,v_1) > y_D(u_2,v_2). 
\]
\end{lemma}
\noindent
This elementary property of the diffeomorphism $w{\to}(x_D(w), y_D(w))$ will be useful in the next section. 
\subsubsection{Regions of directions}\label{region-of-directions}
%%% W1 W2
%%% B0 B1
\begin{defi}\label{Wdef-B0}
 If (B0)  holds,  we introduce the vector $w_c{=}(u_c,v_c)\in {\mathbb S}^1_+$ orthogonal to the vector
    \[
    \left(\ln x^{**}{-}\ln X_2(y^{**}),\ln X_2(y^{**}){-} \ln y^{**}\right)
    \]
    and we let 
    \begin{align}\label{eq-def-W1-W2-B0} 
    {\mathcal W}_1 &= \{w{=}(u,v)\in{\mathbb S}^1_+: \, u > u_c\}\\
    {\mathcal W}_2 &= \{w{=}(u,v)\in{\mathbb S}^1_+: \, u < u_c\}\\
    {\mathcal W}_0 &= \emptyset.
    \end{align}
    \end{defi} 
Since in the case (B0), we have $X_2(y_d){>} x_d$ and $Y_2(x_d){>} y_d$, such a vector $w_c{=}(u_c,v_c)\in {\mathbb S}^1_+$ exists, is clearly unique and has positive coordinates. The sets ${\mathcal W}_1$ and ${\mathcal W}_2$ are therefore both non-empty: ${\mathcal W}_1$ is a part of ${\mathbb S}^1_+$  included in the half-plane
\[
\{w{=}(u,v){\in}\R^2:   v_c u  {>} u_c v  \} = \{w{=}(u,v){\in}\R^2:  (\ln(X_2(y_d)){-}\ln(x_d)) u > (\ln(Y_2(x_d)){-}\ln(y_d)) v\}
\]
and contains the vector $(1,0)$, and ${\mathcal W}_2$ is a part of ${\mathbb S}^1_+$ included in the half-plane
\[
\{w{=}(u,v)\in\R^2:   v_c u  {<} u_c v  \} = \{w{=}(u,v)\in\R^2:  (\ln(X_2(y_d)) {-} \ln(x_d)) u < (\ln(Y_2(x_d)){-}\ln(y_d)) v\}
\]
and contains the vector $(0,1)$. 
 \begin{defi}\label{Wdef-B1} If(B1) holds, we take $w_c{=}(u_c,v_c){=}w_D(x_d, Y_2(x_d))$ and define the regions ${\mathcal W}_0$, ${\mathcal W}_1$ and ${\mathcal W}_2$ in the same way as in the case (B0). 
 \end{defi} 
In the case (B1), we have  $x_d{<} x^{**}_P$, $y_d {<} y^{**}_P$ and $(x_d, Y_2(x_d)){=}(X_2(x_d), y_d){\in}{\mathcal S}_{22}$, hence,  similarly to the case (B0),  both coordinates of the vector $w_c$ are positive and the sets of the regions ${\mathcal W}_1$ and ${\mathcal W}_2$ are non-empty: ${\mathcal W}_1$ is a part of ${\mathbb S}^1_+$  included in the half plane $\{w{=}(u,v)\in\R^2:   u_c u { >}  v_c v  \}$  containing the vector $(1,0)$,  and ${\mathcal W}_2$ is a part of ${\mathbb S}^1_+$ included in the half-plane $\{w{=}(u,v)\in\R^2: u_c u  {<} v_c v \}$  containing the vector $(0,1)$.

\begin{defi}\label{Wdef-B2} If (B2) holds, we define
\begin{align}\label{eq-def-W0-W1-W2-B2} 
 {\mathcal W}_0 &= \{w{=}(u,v){\in}{\mathbb S}^1_+ :u_D(X_2(y_d),y_d){<} u{<} u_D(x_d, Y_2(x_d))\},\\
    {\mathcal W}_1 &= \{w{=}(u,v){\in}{\mathbb S}^1_+ : u{>} u_D(x_d, Y_2(x_d))\},\\
    {\mathcal W}_2 &= \{w{=}(u,v){\in}{\mathbb S}^1_+ : u{<} u_D(X_2(y_d),y_d)\}. 
 \end{align}
 \end{defi} 
By Lemma~\ref{lemma-diffeomorphism-point-direction}, for $w{=}(u,v){\in}{\mathbb S}^1_+$, the inequalities $u_D(X_2(y_d),y_d){<} u{<} u_D(x_d, Y_2(x_d))$ are equivalent to the inequalities $x_D(u,v) {<} x_d$ and $y_D(u,v) {<} y_d$, and hence,  the above definition of the sets  ${\mathcal W}_0$, ${\mathcal W}_1$ and ${\mathcal W}_2$ is equivalent to 
\begin{align*}
{\mathcal W}_0 &= \{w_D(x,y) : (x,y)\in{\mathcal S}_{22}, \,x{<} x_d, y{<} y_d\}, \\
{\mathcal W}_1 &= \{w_D(x,y) : (x,y)\in{\mathcal S}_{22}, \,x{>} x_d\}, \\
{\mathcal W}_0 &= \{w_D(x,y) : (x,y)\in{\mathcal S}_{22}, \,y{>} y_d \}.
\end{align*}
In the case (B2), we have $X_2(y_d) {<} x_d {\leq} x^{**}_P$ and $Y_2(x_d) {<} y_d{\leq} y^{**}_P$, hence  the set $\{(x,y){\in}{\mathcal S}_{22}: x {<} x_d, y{<} y_d\}$ is non-empty, and consequently, the set of directions ${\mathcal W}_0$ is also non-empty: this a part of the set ${\mathbb S}^1_+$  included in the intersection of the two half-planes $\{w=(u,v){\in}\R^2:   v_D(X_2(y_d),y_d)  u  >  u_D(X_2(y_d),y_d)  v  \}$ and $\{w{=}(u,v)\in\R^2:   v_D(x_d,Y_2(x_d))  u <  u_D(x_d,Y_2(x_d))  v  \}$. Similar arguments show that in this case, 
\begin{itemize}
\item[--] the set of directions ${\mathcal W}_1$ is empty if and only if $x_d{=} x^{**}_P$, 
\item[--] \ \phantom{the set of}"\phantom{direction} ${\mathcal W}_2$\phantom{is}"\phantom{empty} if and only if $y_d{=} y^{**}_P$. 
\end{itemize} 

\begin{defi}\label{Wdef-B3-B6} If either (B3) or (B4) holds we let $ {\mathcal W}_2 = {\mathcal W}_0 = \emptyset$ and we define ${\mathcal W}_1$ by 
\begin{equation}\label{eq-def-W1-B3-B4} 
    {\mathcal W}_1 = \begin{cases}  {\mathbb S}^1_+\setminus\{(0,1)\}, &\text{if (B3) holds}\\
 {\mathbb S}^1_+   &\text{if (B4) holds}.
 \end{cases} 
 \end{equation} 
 Similarly, if either (B5) or (B6) holds we take $ {\mathcal W}_1{=} {\mathcal W}_0{=} \emptyset$ and we define ${\mathcal W}_2$ by 
 \begin{equation}\label{eq-def-W2-B5-B6} 
    {\mathcal W}_2 = \begin{cases}  {\mathbb S}^1_+\setminus\{(1,0)\}, &\text{if (B5) holds}\\
 {\mathbb S}^1_+   &\text{if (B6) holds}.
 \end{cases} 
 \end{equation} 
 \end{defi} 

We will see that for any $j{\in}\Z^2_+{\setminus}E_0$, the asymptotic of the Green function $g(j,k)$  as $\min\{k_1,k_2\}{\to}{+}\infty$ with $(k_1,k_2)/\|(k_1,k_2)\|{\to}w{\in}{\mathbb S}^1_+$, is 
\begin{itemize}
\item[--] determined  by the simple pole $x_d$ of the function $x{\to} H_j(x,0)$ when $w\in{\mathcal W}_1$; 
\item[--] \ \phantom{determined  by}"\phantom{the simplepole} $y_d$ \phantom{ofthe}"\phantom{function} $y{\to} H_j(0,y)$  when $w\in{\mathcal W}_2$;
\item[--] similar to the  asymptotic of the Green function of the homogeneous random walk on $\Z^2$ associated to the distribution $\mu$ when $w{\in}{\mathcal W}_0$.
\end{itemize} 

\subsubsection*{Twisted Homogeneous Random Walks} 
To formulate convergence results  when $\min\{k_1,k_2\}{\to}{+}\infty$ and $(k_1,k_2)/\|(k_1,k_2)\|{\to}w$, for $w\in{\mathcal W}_0$,   some quantities,  also used in the asymptotics of the Green function of the homogeneous random walk, are now introduced. See Ney and Spitzer~\cite{Ney-Spitzer} or Theorem~25.15 in Woess~\cite{Woess}.
\begin{defi}
For $w{\in}{\mathbb S}^{1}_+$,  we denote by $(S^{w}(n))$ the  homogeneous random walk  on $\Z^2$ with transition probabilities, for
 $m$, $k{=}(k_1,k_2){\in}\Z^2$,
\[
\P_m(S^{w}(1) {=} m{+}k) = (x_D(w))^{k_1} (y_D(w))^{k_2} \mu(k).
\]
%This is the  homogeneous random walk twisted by the harmonic for the homogeneous random walk function $(k_1,k_2) \to ((x_D(w))^{k_1} (y_D(w))^{k_2})$. 
The vector of first moments and the matrix of second moments  of  $(S^{(w)}(n))$  are denoted respectively by ${\mathfrak m}(w){=} ({\mathfrak m}_1(w), {\mathfrak m}_2(w))$ and ${\mathcal Q}(w){=} \bigl( {\mathcal Q}_{i,j}(w) \bigr)_{i,j=1}^2$: for $i$, $j{\in}\{1,2\}$,
\begin{align*}
{\mathfrak m}_i(w) &= \sum_{k\in\Z^2} k_i\, (x_D(w))^{k_1} (y_D(w))^{k_2}\mu(k), \\
{\mathcal Q}_{i,j}(w) &= \sum_{k\in\Z^2} k_i\, k_j \, (x_D(w))^{k_1} (y_D(w))^{k_2}\mu(k).  
\end{align*}
The associated quadratic form at $z{\in}\R^2$ is denoted by $z{\cdot}{\mathcal Q}(w){\cdot}z$. For $w{=}(u,v){\in}{\mathbb S}^1$, we define $w^\perp{=}({-}v,u)$. 
\end{defi}

We can now state our second set of asymptotic results. The proof of this theorem is given in Section~\ref{section-proof-theorem4}. Recall that here and throughout the paper, for $k{=}(k_1,k_2){\in}\Z^2$, we denote $w_k{=}k/\|k\|$. 
\begin{theorem}[Asymptotics of Green Function along Directions of ${\mathbb S}^1_+$]\label{theorem4} Under the assumptions (A1){-}(A3), for any $j\in\Z^2_+{\setminus}E_0$, the following assertions hold:  
\begin{enumerate}[label=\roman*)]

\item If $\min\{k_1,k_2\}{\to}{+}\infty$, then, uniformly with respect to $w_k$ in any compact subset of ${\mathcal W}_1$, 
\be\label{simple-pole-x-asymptotics}
g(j,k) ~\sim~ b_1\varkappa_1(j) x_d^{-k_1-1} (Y_2(x_d))^{-k_2-1},
\ee
where 
\be\label{constant-C1}
b_1 = (\phi_1(x_d, Y_2(x_d)){-}1) \left(\left.\partial_y P(x, y)\right|_{(x,y)=(x_d,Y_2(x_d))} \left.\frac{d}{dx}\phi_1(x,Y_1(x))\right|_{x=x_d} \right)^{-1} > 0;
\ee

\item  If $\min\{k_1,k_2\}{\to}{+}\infty$, then, uniformly with respect to $w_k$ in any  compact subset of ${\mathcal W}_2$,
\be\label{simple-pole-y-asymptotics}
g(j,k) ~ \sim~ b_2 \varkappa_2(j) (X_2(y_d))^{-k_1-1} y_d^{-k_2-1},
\ee
 where 
\be\label{constant-C2}
b_2 = (\phi_2(X_2(y_d), y_d){-}1) \left(\left.\partial_x P(x,y)\right|_{(x,y)=(X_2(y_d),y_d)} \left.\frac{d}{dy}\phi_2(X_1(y), y) \right|_{y=y_d}\right)^{-1} > 0.
\ee
\item\label{G1}  If (B0) holds then, as $\min\{k_1,k_2\}{\to}{+}\infty$ and $w_k{\to}w_c$, 
\be\label{competition-simple-poles-asymptotics}
g(j,k) ~ \sim~ C_1 \varkappa_1(j) x_d^{-k_1-1} (Y_2(x_d))^{-k_2-1}  + C_2 \varkappa_2(j) (X_2(y_d))^{-k_1-1} y_d^{-k_2-1},
\ee 
where  $b_1{>}0$ and $b_2{>}0$ are given respectively by \eqref{constant-C1} and~\eqref{constant-C2}.

\item\label{G2}  If (B3)  and $y_d {<} y^{**}_P$ hold, then as $\min\{k_1,k_2\}{\to}{+}\infty$ and $w_k{\to}(0,1)$, 
\be\label{competition-simple-double-poles-asymptotics}
g(j,k) ~ \sim~ \varkappa_1(j) \left( b_1 x_d^{-k_1-1}   + b_3   (X_2(y_d))^{-k_1-1} y_d^{-1} k_2 \right) y_d^{-k_2-1},
\ee
where  $b_1>0$ is given by \eqref{constant-C1} and
\be\label{constant-C2-case-B3} 
b_3= x_d  (\phi_1(x_d, y_d){-}1) \left.\left(\partial_y P(x,y)\,  \frac{d}{dy} \phi_2(X_1(y), y)\frac{d}{dx} \phi_1(x, Y_1(x)) \frac{d}{dy} X_1(y)\right)^{-1}\right|_{(x,y)=(x_d,y_d)}  > 0.
\ee
\item\label{G3}  If  (B5) and  $x_d{<} x^{**}_P$ hold, then, as $\min\{k_1,k_2\}{\to}{+}\infty$ and $w_k{\to} (1,0)$,
\be\label{competition-double-simple-poles-asymptotics}
g(j,k) ~ \sim~ \varkappa_2(j) x_d^{-k_1-1} \left( b_4  k_1 x_d^{-1} (Y_2(x_d))^{-k_2-2}  + b_2  y_d^{-k_2-1} \right) 
\ee
where $b_2>0$ is given by~\eqref{constant-C2}  and 
\be\label{constant-C1-case-B5} 
b_4 =  y_d  (\phi_2(x_d, y_d){-}1) \left.\left(\partial_x P(x,y) \,  \frac{d}{dx} \phi_1(x, Y_1(x))\frac{d}{dy} \phi_2(X_1(y), y) \frac{d}{dx} Y_1(x)\right)^{-1}\right|_{(x,y)=(x_d,y_d)} > 0.
\ee
\item   if (B2) holds and  $\min\{k_1,k_2\}{\to}{+}\infty$, then,  uniformly with respect to $w_k$ in any compact subset of ${\mathcal W}_0$, 
\be\label{Woess-asymptotics} 
g(j,k) ~\sim~  \varkappa_{(x_D(w_k), y_D(w_k))}(j)\frac{\|{\mathfrak m}(w_k) \| \sqrt{w_k^\perp {\cdot}{\mathcal Q}(w_k){\cdot}w_k^\perp }}{\sqrt{2\pi\|k\|}\,(x_D(w_k))^{k_1} (y_D(w_k))^{k_2}}.
\ee
\end{enumerate}
\end{theorem}

\subsubsection{Missing Asymptotics}\label{MissAss}
With the definition of the regions of directions ${\mathcal W}_0$, ${\mathcal W}_1$ and ${\mathcal W}_2$, the above theorem provides  the asymptotics of the Green function $g(j,k)$ as $\min\{k_1,k_2\}{\to}{+}\infty$ for all possible directions $w$ except the following singular cases:
\begin{itemize} 
\item  (B1) holds and $w{=}w_c{=}w_D(x^{**}, Y_2(x^{**}))$;
\item  (B2) holds and either $w{=}w_D(x^{**}, Y_2(x^{**}))$ or $w{=}w_D(X_2(y^{**}), y^{**})$;
\item  (B3) and $y_d{=}y^{**}{=}y^{**}_P$ hold and  $w{=}(0,1)$;
\item (B5) and $x_d{=}x^{**}{=}x^{**}_P$ hold and  $w=(1,0)$.
\end{itemize}
Getting these missing asymptotics seems to require an additional significant technical effort.  Note the asymptotics  for these singular directions has not been derived even in the case of nearest neighbor jumps.  In Kurkova and Malyshev~\cite{Kurkova-Malyshev} such asymptotic results are obtained but  only along lines of $\Z_+^2$, with a rational direction in particular.

\bigskip

\noindent
{\bf Remarks on Asymptotic Expressions.}\\
To the best of our knowledge,  the asymptotic behaviors~\ref{G2} and~\ref{G3}  have not been established  in the literature, even in the case of nearest neighbor jumps. Note  that the asymptotics of the cases \ref{G1}, \ref{G2} and \ref{G3} of Theorem~\ref{theorem4} are expressed as a sum of two terms.  We now discuss the implications of these asymptotic results. 

\begin{itemize}
\item[Case~\ref{G1}] The asymptotics~\eqref{competition-simple-poles-asymptotics} reflect in fact a competition between the geometric decay determined by the simple pole $x_d{=}x^{**}$ of the function $x{\to}H_j(x,0)$ and the geometric decay determined by the simple pole $y_d{=}y^{**}$ of the function $y\to H_j(0,y)$. In the case when (B0) holds and if $(k_n){=}(k_{1,n},k_{2,n})$ is  a sequence of points  of $\Z^2_+$ going to infinity such that  $k_n/\|k_n\|{\to}w_c$, since $(x_d/X_2(y_d))^{u_c}{=}(y_d/Y_2(x_d))^{v_c}$ and $X_2(y_d){>}x_d$, Theorem~\ref{theorem4} shows in fact that for $j{\in}\Z^2_+{\setminus}E_0$, 
\begin{align}
g(j,k_n) &~ \sim~ C_1 \varkappa_1(j) x_d^{-k_{1,n}-1} (Y_2(x_d))^{-k_{2,n}-1}\quad \text{if }\lim_n \Bigl(k_{1,n} {-}k_{2,n}{u_c}/{v_c} \Bigr)  = {+}\infty\nonumber\\
g(j,k_n) &~ \sim~ C_2 \varkappa_2(j)(X_2(y_d))^{-k_{1,n}-1} y_d^{-k_{2,n}-1} \quad\text{if }\lim_n \Bigl(k_{2,n}{-}k_{2,n}{u_c}/{v_c}  \Bigr)   = {-}\infty\nonumber\\
g(j,k_n) &~ \sim~ \left(C_1 \varkappa_1(j){+}C_2 \left(\frac{x_d}{X_2(y_d)}\right)^\sigma \varkappa_2(j)\right)x_d^{-k_{1,n}-1} (Y_2(x_d))^{-k_{2,n}-1}\label{eq-cor-asymptotics-critical-direction} \\ & \ \hspace{6cm}\text{ if }\lim_n \Bigl(k_{1,n}{-}k_{2,n}{u_c}/{v_c} \Bigr)  = \sigma  \in\R.\nonumber 
\end{align}

The convergence~\eqref{eq-cor-asymptotics-critical-direction} exhibits an interesting phenomenon.
\begin{enumerate}
\item If $u_c/v_c{\in}\R{\setminus}\Q$, then for any $\sigma{\in}\R$, there is a sequence $(k_n){=}(k_{1,n},k_{2,n}){\in}\Z^2_+$ going to infinity such that  $\lim_n k_n/\|k_n\|{\to}w_c$, for which $\sigma{=}\lim_n (k_{1,n}{-}k_{2,n} {u_c}/{v_c} )$ and relation~\eqref{eq-cor-asymptotics-critical-direction} holds. 
\item The rational case, i.e; when $u_c/v_c{=}p/q{\in}\Q$ for some $p$, $q{\in}\N^*$,  relation~\eqref{eq-cor-asymptotics-critical-direction}  holds if and only if for  $n$ large enough, $q k_{1,n}{-}k_{2,n} p$ is constant and, therefore,  $\sigma{\in}\Q$ is of the form $\sigma{=}\tilde{p}/q$  for  some $\tilde{p}{\in}\Z$.
An analogous result has been  established in Ignatiouk et al.~\cite{I-K-R}.
\end{enumerate}

\item[Case~\ref{G2}]
The asymptotics~\eqref{competition-simple-double-poles-asymptotics} reflect a competition between the geometric decay determined by the simple pole $x_d{=}x^{**}$ of the function $x{\to}H_j(x,0)$ and the geometric decay multiplied by a factor $k_2$ determined by the  pole $y_d{=}Y_2(x_d){=} y^{**}$ of the second order  of the function $y{\to} H_j(0,y)$. If (B3) and $y_d{<}y^{**}_P$ hold  and $(k_n){=}(k_{1,n},k_{2,n})$ is  a sequence of points of $\Z^2_+$ going to infinity such that $\|k_n\|{\to}{+}\infty$ and  $k_n/\|k_n\|{\to}(0,1)$, Theorem~\ref{theorem3} shows that for $j{\in}\Z^2_+{\setminus}E_0$, 
\begin{align*}
g(j,k) &\sim \varkappa_1(j) C_1 x_d^{-k_{1,n}-1}  y_d^{-k_{2,n}-1} &&\text{ if }\lim_n k_{2,n}x_d^{k_{1,n}}/(X_2(y_d)))^{k_{1,n}}{=}0;\\ 
g(j,k) &\sim \varkappa_1(j) C_2  (X_2(y_d))^{-k_{1,n}-1} k_{2,n}  y_d^{-k_{2,n}-2} &&\text{ if }\lim_n k_{2,n}x_d^{k_{1,n}}/(X_2(y_d)))^{k_{1,n}}{=}{+}\infty;\\ 
g(j,k) &\sim \varkappa_1(j) \left( C_1    + \sigma C_2  y_d^{-1} \right) x_d^{-k_{1,n}-1} y_d^{-k_{2,n}-1} &&\text{ if }\lim_n k_{2,n} x_d^{k_{1,n}}/(X_2(y_d)))^{k_{1,n}}{=}\sigma{>} 0.
\end{align*}
\item[Case~\ref{G3}]   The asymptotic \eqref{competition-double-simple-poles-asymptotics} reflects a competition between the geometric decay multiplied by a factor $k_1$ determined by the pole $x_d{=}X_2(y_d){=}x^{**}$ of the second order of the function $x\to H_j(x,0)$ and the geometric decay determined by  the  simple pole  $y_d{=}y^{**}$ of the function $y{\to}H_j(0,y)$.

Similar asymptotics hold, by exchanging the roles of $x$ and $y$ in this case when $x_d {<} x^{**}_P$, $\|k_n\|{\to}{+}\infty$ and  $k_n/\|k_n\|{\to}(1,0)$.  
\end{itemize}

\bigskip

\noindent
{\bf Relations with Homogeneous Random Walks.}\\
The asymptotic relation~\eqref{Woess-asymptotics} is similar to   the exact asymptotics of the Green function of the homogeneous random walk $(S(n))$ in $\Z^2$, see Ney and Spitzer~\cite{Ney-Spitzer} and Theorem~25.15 of Woess~\cite{Woess}. The only difference  in fact is that  in the asymptotic obtained by Ney and Spitzer~\cite{Ney-Spitzer},  the function $ j{=} (j_1,j_1){\to}(x_D(w_k))^{j_1}(y_D(w_k))^{j_1}{=}\exp(\alpha_D(w_k) j_1 + \beta_D(w_k) j_2)$,  which is harmonic  for the homogeneous random walk, is replaced in our case by the function $j \to \varkappa_{(x_D(w_k), y_D(w_k))}(j)$, which is harmonic for the killed random walk $(Z_{\tau_0}(n))$.  Analogous results have been obtained in Ignatiouk~\cite{Ignatiouk-2023-cone} for the Green functions of a homogeneous random walk in $\Z^d$ killed outside of an open cone and for the asymptotics along the interior directions of the cone.

\subsection*{Outline of the paper}
Our paper is organized as follows :
\begin{itemize}
\item In Section~\ref{Martin-section}, our results are used to investigate the Martin compactification of $\Z^2_+$ of the killed Markov chain $(Z_{\tau_0}(n))$ and also of the original random walk $(Z(n))$ when it is transient. 
\item In Section~\ref{preliminary-section}, we show that  the points $x^*_P$, $x^*$, $x^{**}$, $x^{**}_P$, $y^*_P$, $y^*$, $y^{**}$, $y^{**}_P$ and the functions $Y_1$, $Y_2:[x^*_P, x^{**}_P]{\to} [y^*_P, y^{**}_P]$ and $X_1$, $X_2: [y^*_P, y^{**}_P]{\to}[x^*_P, x^{**}_P]$ are well defined, and we obtain their first properties;
\item  Proposition~\ref{cases} is proved in Section~\ref{section_cases};
\item  Sections~\ref{proof-theorem1} and \ref{section-proof-theorem4} are respectively devoted to the proofs of Theorems~\ref{theorem1} and \ref{theorem4}. 
\item In Section~\ref{section-conditions}, the conditions (B0){-}(B7) are compared with the conditions of positive recurrence and transience for the random walk $(Z(n))$;
\end{itemize}

\section{Application to the Martin Boundary}\label{Martin-section} 
When the Markov chain $(Z(n))$ is transient, its Green function, for $j$, $k{\in}\Z_+^2$, 
\[
G(j,k) =\sum_{n=0}^\infty \P_j(Z(n)=k),
\]
is related to the Green function $g(j,k)$ of the killed Markov chain at $0$ in the following way, 
\[
G(j,k) = g(j,k) + G(j,0) g(0,k), \quad \text{for $j{\ne}0$}, 
\]
and
\[
G(0, k) = G(0,0) g(0,k). 
\]
The Martin kernels have therefore a simple relation
\be\label{eq1-Martin}
\frac{G(j,k)}{G(0,k)}=\frac{1}{G(0,0)} \frac{g(j,k)}{g(0,k)}+\frac{G(j,0)}{G(0,0)} = \P_0(\tau_0 = +\infty) \frac{g(j,k)}{g(0,k)}+ P_j(\tau_0 < +\infty),
\ee
and the Martin compactification of $\Z^2_+$ for the killed Markov chain $(Z(n))$ at $0$ is homeomorphic to the Martin compactification of the original random walk $(Z(n))$.

In this section, our asymptotic results,  Theorems~\ref{theorem3} and~\ref{theorem4},  are used to obtain the asymptotics of  the Martin kernel  of the killed Markov chain $(Z(n))$ at $0$. We do not assume that the original random walk $(Z(n))$ is positive recurrent or transient. The Martin boundary of $(g(j,k))$ is in fact almost completely described, the four asymptotic cases mentioned in Section~\ref{MissAss} excepted. Throughout this section, we will assume  that the conditions (A1){-}(A3) are satisfied. 

Since by Lemma~\ref{irreductibility-lemma}, $g(j,k) = 0$, for any $j\in E_0$, and $k\in\Z^2_+{\setminus} E_0$,  to investigate the Martin boundary of the killed random walk, it is sufficient to consider $j\in\Z^2_+{\setminus} E_0$. 

\begin{prop}\label{Martin-boundary-B0-B1} Suppose that either (B0) or (B1) holds, and let 
the vector $w_c{=}(u_c,v_c){\in}{\mathbb S}_+^1$ be the critical  direction defined in Section~\ref{region-of-directions} and   $(k_{n}){=}(k_{1,n}, k_{2,n})$ be a sequence of $Z_+^2$ whose norm goes to infinity. Then for any   $j{\in}\Z^2_+{\setminus}E_0$, the following assertions hold. 
\begin{enumerate}
\item If  $\liminf_n {k_{1,n}}/{\|k_{n}\|}{>} u_c$   then  
\be\label{martin-limit-varkappa1}
\lim_{n\to+\infty} \frac{g(j, k_{n})}{g(0, k_{n})} = \frac{\varkappa_1(j)}{\varkappa_1(0)},
\ee
where $\varkappa_1$ is the function defined by relation~\eqref{eq-def-kappa-1}.
\item   If $\limsup_n k_{1,n}/\|k_{n}\|{<} u_c$ then 
\be\label{martin-limit-varkappa2}
\lim_{n\to+\infty} \frac{g(j, k_{n})}{g(0, k_{n})} = \frac{\varkappa_2(j)}{\varkappa_2(0)},
\ee
where $\varkappa_2$ is  the function defined by relation~\eqref{eq-def-kappa-2}.
\item  If (B0) holds and $\lim_n k_n/\|k_n\|{=}w_c$, then 
\[
\lim_{n\to+\infty}\frac{g(j, k_n)}{g(0, k_n)}\\=\frac{b_1 \varkappa_1(j) \lambda_1^{-k_1-1}{+} b_2 \varkappa_2(j)  \lambda_2^{-k_2-1} }{ b_1 \varkappa_1(0) \lambda_1^{-k_1-1}{+} b_2 \varkappa_2(0) \lambda_2^{-k_2-1}},
\]
with $\lambda_1{=}x_d/X_2(y_d)$ and $\lambda_2{=}y_d/Y_2(x_d)$ and $b_1{>} 0$, resp.  $b_2{>}0$,  is given by relation~\eqref{constant-C1}, resp. relation~\eqref{constant-C2}.
\end{enumerate}
\end{prop}
When either (B0) or (B1) holds, Proposition~\ref{Martin-boundary-B0-B1} implies  that the minimal Martin boundary $\partial_m\Z^2_+$, see Proposition~24.4 of~\cite{Woess}, contains two points $\xi_1$ and $\xi_2$, and  any sequence of points $(k_{n})$  converging to infinity and such that  $\liminf_n k_1^{(n)}/\|k^{(n)}\| {>} u_c$ (resp. $\limsup_n k_1^{(n)}/\|k^{(n)}\| {<} u_c$) converges in the Martin compactification of $\Z^2_+$ to  $\xi_1$ (resp. to $\xi_2$).

When (B0) holds, with this  result one gets that  the minimal Martin boundary of $\Z^2_+$ relative for the killed Markov chain at $0$  contains exactly two points, and  with the same arguments as in Theorem~3 of~\cite{I-K-R}, we obtain that  the full Martin boundary is homeomorphic to $\Z\cup\{\pm\infty\}$, resp. $\R\cup\{\pm\infty\}$,  if $u_c/v_c{\in}{\mathbb Q}$, resp. $u_c/v_c{\not\in}{\mathbb Q}$.   By the Poisson-Martin representation theorem,  in this case,  any non-negative harmonic function for the killed Markov chain is therefore of the form $\theta_1 \varkappa_1{+}\theta_2 \varkappa_2$ with for some $\theta_1$, $\theta_2{\in} [0,{+}\infty[$. 

Note that the full Martin boundary is not obtained in the case (B1) since the asymptotics of the Green function $g(j,k)$ along the direction $w_c$ are missing. See Section~\ref{MissAss}. 

For the region~(B2) we have the following proposition. 
\begin{prop}\label{Martin-boundary-B2}
Suppose that  (B2) holds, and let  $(k_{n}){=}(k_{1,n}, k_{2,n})$ be a sequence of $Z_+^2$ whose norm goes to infinity. Then for any  $j{\in}\Z^2_+{\setminus}E_0$, the following assertions hold.
\begin{enumerate}
\item  If $\lim_n k_n/\|k_n\| = w\in{\mathcal W}_0$, then  
\be\label{martin-limit-varkappa-w}
\lim_{n\to+\infty} \frac{g(j, k_n)}{g(0, k_n)} = \frac{\varkappa_{(x_D(w),y_D(w))}(j)}{\varkappa_{(x_D(w),y_D(w))}(0)}
\ee
where $\varkappa_{({\cdot},{\cdot})}$ is  the function defined by~\eqref{def-kappa-interior-direction}.
\item If $x_d{=}x^{**}_P$ and the sequence $(k_{2,n})$ is bounded, then 
\be\label{martin-limit-tilde-varkappa1}
\lim_{n\to{+}\infty} \frac{g(j, k_n)}{g(0, k_n)} = \begin{cases}
\tilde\varkappa_1(j)/\tilde\varkappa_1(0)&\text{if $\phi_1(x^{**}_P, Y_1(x^{**}_P)) < 1$},\\
\varkappa_1(j)/\varkappa_1(0)&\text{if $\phi_1(x^{**}_P, Y_1(x^{**}_P)) =1$.}
\end{cases}
\ee
where $\tilde\varkappa_1$ is  the function defined by~\eqref{eq-def-tilde-kappa-1}.
\item If $x^{**} {<}x^{**}_P$ and $\liminf_n k_{1,n}/\|k_{n}\|{>}u_D(x_d, Y_2(x_d))$, then
  \[
  \lim_{n\to+\infty} \frac{g(j, k_{n})}{g(0, k_{n})} = \frac{\varkappa_1(j)}{\varkappa_1(0)}.
  \]
\item if $y^{**}{=}y^{**}_P$ and the sequence $(k_{1,n})$ is bounded, then 
\be\label{martin-limit-tilde-varkappa2}
\lim_{n\to+\infty} \frac{g(j, k_n)}{g(0, k_n)} = 
\begin{cases}
\tilde\varkappa_2(j)/\tilde\varkappa_2(0) &\text{if $\phi_2(X_1(y^{**}_P), y^{**}_P) < 1$},\\
\varkappa_2(j)/\varkappa_2(0) &\text{if $\phi_2(X_1(y^{**}_P), y^{**}_P) = 1$}
\end{cases}
\ee
where $\tilde\varkappa_2$ is  the function defined by~\eqref{eq-def-tilde-kappa-2}.
\item If  $y^{**}{<}y^{**}_P$ and $\liminf_n k_{2,n}/\|k_{n}\|{>}v_D(x_d, Y_2(x_d)$ then,
  \[
  \lim_{n\to+\infty} \frac{g(j, k_{n})}{g(0, k_{n})} = \frac{\varkappa_2(j)}{\varkappa_2(0)}.
  \]
\end{enumerate} 
\end{prop}
When (B2) holds, this result proves that for any direction $w{\in}{\mathbb S}^1_+$ in the closure $\overline{\mathcal W}_0$ of ${\mathcal W}_0$ there is a point $\xi(w)$ in the Martin boundary of the killed Markov chain, and that if $(k_{n})$ is a sequence of points of $\Z^2_+$ whose norm converges to infinity then, for the convergence in the Martin compactification,
\[
\lim_{n\to+\infty} k_n=
\begin{cases}
\xi(w) &\text{ for } w{\in}{\mathcal W}_0\text{  if } k_{n}/\|k_n\|{\to}w,\\
\xi(w_D(x_d,Y_2(x_d)))&\text{ if }\liminf_n k_{1,n}/\|k_{n}\|{>}u_D(x_d,Y_2(x_d)),\\
\xi(w_D(X_2(y_d), y_d))&\text{ if }\limsup_n k_{1,n}/\|k_{n}\|{<} u_D(X_2(y_d), y_d). 
\end{cases}
\]
As explained before, in the case (B2), also due to the missing asymptotics for the singular directions $w_D(x_d, Y_2(x_d))$ and $w_D(X_2(y_d),y_d)$, the full Martin boundary is not completely determined. We conjecture  that in this case it is homeomorphic to  $\overline{\mathcal W}_0$.

\begin{prop}\label{Martin-boundary-B4} Suppose that  either (B3) and $y_d{<}y^{**}_P$ hold or (B4) holds. Then for any  $j{\in}\Z^2_+{\setminus}E_0$, 
    \[
\lim_{\to+\infty} \frac{g(j, k)}{g(0, k)} = \frac{\varkappa_1(j)}{\varkappa_1(0)},
    \]
\end{prop}
If either (B3) and $y_d{<}y^{**}_P$ hold or (B4) holds,  the full Martin boundary of $\Z^2_+$ of the killed Markov chain  is therefore a single point, and, up to a multiplicative constant, the function $\varkappa_1$  is the unique non-negative harmonic function. 

When the original random walk $(Z(n))$ is transient and the measures $\mu_0$, $\mu_1$ and $\mu_2$ are stochastic,  since the Martin compactification of $\Z^2_+$ of the Markov chain $(Z(n))$ is homeomorphic to the Martin compactification of $\Z^2_+$ of the killed Markov chain, we have the following corollary.

\begin{cor}\label{Martin-boundary-B4p} If the measures $\mu_0$, $\mu_1$ and $\mu_2$ are stochastic and  the Markov chain $(Z(n))$ is transient, then under the hypotheses of Proposition~\ref{Martin-boundary-B4}, for any $j{\in}\Z^2_+$, 
\be\label{eq2-Martin}
\lim_{\|k\|\to+\infty} \frac{G(j,k)}{G(0, k)} = 1.
\ee

\end{cor} 
\begin{proof} Indeed, in this case, using \eqref{eq1-Martin}, Proposition~\ref{Martin-boundary-B4} and Lemma~\ref{irreductibility-lemma}, since the set $E_0$ does not contain the origin, one gets 
\[
\lim_{\|k\|\to+\infty} \frac{G(j,k)}{G(0, k)} = \P_0(\tau_0 = +\infty) \varkappa_1(j) + P_j(\tau_0 < +\infty), \quad \forall j\in\Z^2_+,
\]
and consequently,  by the Poisson-Martin representation theorem, up to a multiplicative constant, for any harmonic function $\varkappa$ for the Markov chain $(Z(n))$, one has 
\[
\varkappa(j)/\varkappa(0) = \P_0(\tau_0 = +\infty) \varkappa_1(j) + P_j(\tau_0 < +\infty), \quad \forall j\in \Z^2_+. 
\]
Since in the case when the measures $\mu_0$, $\mu_1$ and $\mu_2$ are stochastic, the constants functions are harmonic for $(Z(n))$, from this it follows that 
\be\label{eq3-Martin}
\P_0(\tau_0 = +\infty) \varkappa_1(j) + P_j(\tau_0 < +\infty) = 1, \quad \forall j\in \Z^2_+,
\ee
and \eqref{eq2-Martin} holds. 

Remark that the identity \eqref{eq3-Martin} can also be obtained in a straightforward way by using Proposition~\ref{comparison-cor} of Section~\ref{section-conditions}. 
\end{proof} 
By symmetry, similar result can be obtained with $\varkappa_2$ instead of $\varkappa_1$ if either (B5) and $x_d = x^{**}_P$ hold or (B6) holds.

We conclude with the case when either (B3) holds with $y_d{=}y^{**}_P$ or (B5) holds with $x_d{=}x^{**}_P$.
\begin{prop}\label{Martin-boundary-B3} Suppose that (B3) holds with $y_d{=}y^{**}_P$ and let $(k_n)$ be a sequence of points of $\Z_+^2$ whose norm converges to infinity. Then for any  $j\in\Z^2_+{\setminus} E_0$,  the convergence 
  \[
  \lim_{n\to+\infty} \frac{g(j, k_{n})}{g(0, k_{n})} = \frac{\varkappa_1(j)}{\varkappa_1(0)}
  \]
holds in any of the two following cases:
\begin{itemize}
\item  when  $\liminf_n k_{1,n}/\|k_n\|{>}0$;
\item when the sequence $(k_{1,n})$ is bounded and $\lim_n k_{2,n} = +\infty$. 
\end{itemize}
\end{prop} 
When (B3) holds with $y_d{=}y^{**}_P$ there is therefore a point $\xi_0$ in the Martin boundary of $\Z^2_+$ of the killed Markov chain, such that any sequence  of points of $\Z_+^2$   satisfying the conditions of Proposition~\ref{Martin-boundary-B3} converges in the  Martin compactification to $\xi_0$. 

Similar result with the function $\varkappa_2$ instead of $\varkappa_1$ can be obtained when (B5) holds with $x_d{=}x^{**}_P$.

When (B3) holds with $y_d{=}y^{**}_P$, resp.  (B5) holds with $x_d{=}x^{**}_P$, the asymptotics of the Green function $g(j,k)$ when $\|k\|{\to}{+} \infty$ and $k_1/\|k\|{\to}0$, resp. when $\|k\|{\to}{+}\infty$ and $k_2/\|k\|{\to} 0$ are not known.  See Section~\ref{MissAss}. For this reason, the full Martin boundary is not  determined in this case.  We conjecture it is a single point. 

\section{Preliminary results}\label{preliminary-section} 

In the following statement, we investigate the set $D$ defined by \eqref{Ddef}: it is proved that the line segments $[x^*_P,x^{**}_P]$, $[y^*_P, y^{**}_P]$ and the functions $Y_1,Y_2: [x^*_P,x^{**}_P]\to [y^*_P, y^{**}_P]$ and $X_1, X_2: [y^*_P, y^{**}_P]\to [x^*_P,x^{**}_P]$ are well defined, and the first useful for our purpose properties of these functions are obtained. 

\begin{lemma}\label{preliminary-lemma1}  Under the hypotheses (A1), the line segments $[x^*_P,x^{**}_P]$ and $[y^*_P, y^{**}_P]$ and the functions $Y_1,Y_2: [x^*_P,x^{**}_P]\to [y^*_P, y^{**}_P]$ and $X_1, X_2: [y^*_P, y^{**}_P]\to [x^*_P,x^{**}_P]$ are well defined and the following assertions hold: 

\begin{enumerate}[label=\arabic*)]
\item For any $\hat{x}\in[x^*_P, x^{**}_P]$, $Y_1(\hat{x})$ and $Y_2(\hat{x})$ are the only real positive solutions of the equation $P(\hat{x}, y) =1$,
 \[
 Y_1(\hat{x}) < Y_2(\hat{x}), \quad  \partial_y P(\hat{x},Y_1(\hat{x})) < 0 \quad \text{and} \quad  \partial_y P(\hat{x},Y_2(\hat{x})) < 0  \quad \text{if} \quad  x^*_P < \hat{x} < x^{**}_P,
 \]
 and 
 \[
 Y_1(\hat{x}) = Y_2(\hat{x}) \quad \text{and} \quad \partial_y P(\hat{x},Y_1(\hat{x}))= 0 \quad \text{if} \quad \hat{x}\in\{x^*_P, x^{**}_P\}.
 \]
\item For any $\hat{y}\in[y^*_P, y^{**}_P]$, $X_1(\hat{y})$ and $X_2(\hat{y})$ are the only real positive solutions of the equation $P(x,\hat{y})=1$, 
 \[
 X_1(\hat{y}) < X_2(\hat{y}), \quad  \partial_x P(X_1(\hat{y}),\hat{y}) < 0 \quad  \text{and} \quad \partial_x P(X_2(\hat{y}),\hat{y})  \quad \text{if} \quad y^*_P < \hat{y} <  y^{**}_P,
 \]
 and 
\[
 X_1(\hat{y})=X_2(\hat{y}) \quad \text{and} \quad  \partial_x P(X_1(\hat{y}), \hat{y}) = 0 \quad \text{if} \quad  \hat{y}\in\{y^*_P, y^{**}_P\}.
 \]
 
\item The four points $(x^*_P, Y_1(x^*_P))$, $(x^{**}_P, Y_1(x^{**}_P))$, $(X_1(y^*_P), y^*_P)$ and $(X_1(y^{**}_P), y^{**}_P)$ are two by two distinct and moreover $Y_1(x^*_P), Y_1(x^{**}_P)\in]y^*_P, y^{**}_P[$ and $X_1(y^*_P), X_1(y^{**}_P)\in]x^*_P, x^{**}_P[$. 

\item The function $X_1:[y^*_P, Y_1(x^*_P)]{\to}[x^*_P, X_1(y^*_P)]$ is strictly decreasing and its inverse is\\  $Y_1:[x^*_P, X_1(y^*_P)]{\to}[y^*_P, Y_1(x^*_P)]$.

\item The function $X_1:[Y_1(x^*_P), y^{**}_P]{\to}[x^*_P, X_1(y^{**}_P)]$ is strictly increasing and its inverse   is\\  $Y_2: [x^*_P, X_1(y^{**}_P)]{\to}[Y_1(x^*_P), y^{**}_P]$.

\item The function $X_2: [y^*_P, Y_1(x^{**}_P]{\to}[X_1(y^*_P), x^{**}_P]$ is strictly increasing and its inverse is\\ $Y_1:~[X_1(y^*_P), x^{**}_P]{\to}[y^*_P, Y_1(x^{**}_P]$.

\item The function $X_2:[Y_1(x^{**}_P), y^{**}_P]{\to}[X_1(y^{**}_P), x^{**}_P]$  strictly decreasing  and its inverse is\\ $Y_2: [X_1(y^{**}_P), x^{**}_P]{\to} [Y_1(x^{**}_P), y^{**}_P]$.
\end{enumerate}
\end{lemma} 
\begin{proof} To prove this lemma we notice that  under the hypotheses (A1) the jump generating function $\tilde{P}:\R^2\to\R_+$ defined by \eqref{function-tilde-P} 
is strictly convex and finite in a neighborhood of the set $\tilde{D}= \{ (\alpha,\beta)\in\R^2: \tilde{P}(\alpha, \beta)\leq 1\}$. The set $\tilde{D}$ is therefore strictly convex and compact (see for instance, Spitzer~\cite{Spitzer}) and because of the assumption (A1) (iii), it has a non-empty interior. This implies that 

-- each of the sets $\{\alpha\in\R:~\inf_{\beta\in\R} \tilde{P}(\alpha,\beta) \leq 1\} $ and 
$\{\beta\in\R:~\inf_{\alpha\in\R} \tilde{P}(\alpha,\beta)\leq 1\}$
is a line segment with a non-zero length, we will denote them respectively by $[\alpha^*_P,\alpha^{**}_P]$ and $[\beta^*_P,\beta^{**}_P]$; 
 
-- for any $\hat\alpha\in [\alpha^{*}_P,\alpha^{**}_P]$, the set $
[\beta_1(\hat\alpha),\beta_2(\hat\alpha)]=\{\beta\in\R~:~\tilde{P}(\hat\alpha,\beta)\leq 1\}$ is a non-empty line segment , 
the points $\beta_1(\hat\alpha)$ and $\beta_2(\hat\alpha)$  are the only real  solutions of the equation $\tilde{P}(\hat\alpha,\beta) = 1$,  
 \[
 \beta_1(\hat\alpha) <\beta_2(\hat\alpha), \quad  \partial_\beta  \tilde{P}(\hat\alpha, \beta_1(\hat\alpha)) \quad  \text{and} \quad   \partial_\beta \tilde{P}(\hat\alpha, \beta_2(\hat\alpha)) > 0 \quad \text{if $\alpha^*_P <\hat\alpha <  \alpha^{**}_P$},
  \]
  and 
 \[
 \beta_1(\hat\alpha) = \beta_2(\hat\alpha) \quad \text{and} \quad   \partial_\beta \tilde{P}(\hat\alpha,\beta_1(\hat\alpha)) = 0 \quad \text{if  $\hat\alpha \in\{\alpha^*_P, \alpha^{**}_P\}$}: 
 \]
 
 -- for any $\hat\beta\in [\beta^{*}_P,\beta^{**}_P]$, the set $ [\alpha_1(\hat\beta),\alpha_2(\hat\beta)]=\{\alpha\in\R~:~\tilde{P}(\alpha,\hat\beta)\leq 1\}$ 
 is a  non-empty line segment,  the points $\alpha_1(\hat\beta)$ and  $\alpha_2(\hat\beta)$ are the only real  solutions of the equation $\tilde{P}(\alpha,\hat\beta) = 1$, 
\[
 \alpha_1(\hat\beta) < \alpha_2(\hat\beta), \quad  \partial_\alpha  \tilde{P}(\alpha_1(\hat\beta),\beta) < 0 \quad  \text{and} \quad  \partial_\alpha \tilde{P}(\alpha_2(\hat\beta), \hat\beta) > 0 \quad  \text{if $\beta^*_P < \hat\beta < \beta^{**}_P$}, 
\]
and
\[
 \alpha_1(\hat\beta) = \alpha_2(\hat\beta)  \quad  \text{and} \quad   \partial_\alpha \tilde{P}(\alpha_1(\hat\beta), \hat\beta) = 0 \quad  \text{if $\hat\beta\in\{\beta^*_P, \beta^{**}_P\}$},
 \]
 the four points $(\alpha^*_P, \beta_1(\alpha^*_P))$, $(\alpha^{**}_P, \beta_1(\alpha^{**}_P))$, $(\alpha_1(\beta^*_P), \beta^*_P)$ and $(\alpha_1(\beta^{**}_P), \beta^{**}_P)$ are two by two distinct and  $\alpha_1(\beta^*_P), \alpha_1(\beta^{**}_P)\in]\alpha^*_P, \alpha^{**}_P[$ and $\beta_1(\alpha^*_P), \beta_1(\alpha^{**}_P)\in]\beta^*_P, \beta^{**}_P[$. 
\begin{itemize}
\item  $\beta_1{:}~[\alpha^*_P,\alpha^{**}_P]{\to}[\beta^*_P,\beta^{**}_P]$  is strictly convex with a minimum $\beta^*_P$ at $\alpha_1(\beta^*_P)=\alpha_2(\beta^*_P)$;
\item $\beta_2{:}~[\alpha^*_P,\alpha^{**}_P]{\to}[\beta^*_P,\beta^{**}_P]$ is strictly concave with a maximum $\beta^{**}_P$ at $\alpha_1(\beta^{**}_P){=}\alpha_2(\beta^{**}_P)$; 
\item $\alpha_1{:}~[\beta^*_P,\beta^{**}_P]{\to}[\alpha^*_P,\alpha^{**}_P]$  is strictly convex with a minimum $\alpha^*_P$ at $\beta_1(\alpha^*_P){=}\beta_2(\alpha^*_P)$;
\item $\alpha_2{:}~[\beta^*_P,\beta^{**}_P]{\to}[\alpha^*_P,\alpha^{**}_P] $ is strictly concave with a maximum $\alpha^{**}_P$ at$\beta_1(\alpha^{**}_P){=}\beta_2(\alpha^{**}_P)$.
\end{itemize}
\noindent
As a consequence of these properties, one gets that 
\begin{itemize}
\item  $\alpha_1{:}~[\beta^*_P,\beta_1(\alpha^*_P)]{\to}[\alpha^*_P, \alpha_1(\beta^*_P)]$ is strictly decreasing and its inverse function is $\beta_1$;
\item $\alpha_1{:}~[\beta_1(\alpha^*_P), \beta^{**}_P]{\to}[\alpha^*_P, \alpha_1(\beta^{**}_P)]$ is strictly increasing \phantom{andts inverse}"\phantom{function is} $\beta_2$;
\item $\beta_1{:}~[\alpha_1(\beta^*_P), \alpha^{**}_P]{\to}[\beta^*_P, \beta_1(\alpha^{**}_P]$ is strictly increasing  \phantom{and its inverse}"\phantom{function is} $\alpha_2$;
\item $\alpha_2{:}[\beta_1(\alpha^{**}_P), \beta^{**}_P]{\to}[\alpha_1(\beta^{**}_P), \alpha^{**}_P]$ is strictly decreasing \phantom{andts inverse}"\phantom{function is} $\beta_2$.
\end{itemize}
Since the mapping $(\alpha,\beta){\to}(x,y){=}(e^\alpha, e^\beta)$ determines a homeomorphism from $\R^2$ to $]0,+\infty[^2$ and maps the set $\tilde{D}$  to the set $D$,  these assertions prove our lemma with  $x^*_P{=}e^{\alpha^*_P}$, $x^{**}_P{=}e^{\alpha^{**}_P}$, $y^*_P{=}e^{\beta^*_P}$, $y^{**}_P{=}e^{\beta^{**}_P}$,
\[
X_i(y) = e^{\alpha_i(\ln(y))}, \quad \forall y\in [y^*_P,y^{**}_P], \; i\in\{1,2\}, 
\]
and 
\[
Y_i(x) = e^{\beta_i(\ln(x))}, \quad \forall x\in [x^*_P, x^{**}_P], \; i\in\{1,2\}. 
\]
\end{proof}

In the  following lemma, we investigate the sets $D\cap D_1$ and $D\cap D_2$. It is proved that  the line segments $[x^*, x^{**}]$ and $[y^*, y^{**}]$ are well defined.  With this result, we will be able to get the first useful for our purpose properties of the functions $x\to \phi_1(x, Y_1(x))$ and $y\to \phi_2(X_1(y), y)$. 

\begin{lemma}\label{preliminary-lemma2} Under the hypotheses (A1) and (A3), there exist $x^*, x^{**}\in]0,+\infty[$ such that $x^* < x^{**}$ and  
\[
[x^*, x^{**}] = \{x\in]0,+\infty[ :~(x,y)\in D_1\cap D \; \text{for some} \; y > 0\}.
\] 
Moreover, for any $x\in [x^*, x^{**}] $,  there exists  $\tilde{Y}_2(x) \geq Y_1(x)$ such that 
\[
\{y > 0~:~(x,y)\in D_1\cap D\} = [{Y}_1(x), \tilde{Y}_2(x)], 
\] 
 the points $Y_1(x)$ and $\tilde{Y}_2(x)$ are the only real  positive solutions of the equation 
 \be\label{eq0_preliminary_lemma2} 
 \max\{P(x,y), \phi_1(x,y)\} = 1,
 \ee
 and  ${Y}_1(x) < \tilde{Y}_2(x)$ if and only if $x^*< x <x^{**}$. 
\end{lemma} 
\begin{proof} To prove this lemma, we consider  the jump generating  function $\tilde{P}$ defined by \eqref{function-tilde-P}, the jump generating function $\tilde\phi_1:\R^2\to \R_+$, defined for  $(\alpha,\beta)\in\R^2$, by 
\[
\tilde\phi_1(\alpha,\beta) = \sum_{k=(k_1,k_2)\in\Z^2} \exp(\alpha k_1 + \beta k_2) \mu_1(k)
\]
and the sets $\tilde{D}= \{ (\alpha,\beta)\in\R^2: \tilde{P}(\alpha, \beta)\leq 1\}$ and $
\tilde{D}_1 =\{(\alpha,\beta)\in\R^2~: \tilde\phi_1(\alpha,\beta) \leq 1\}$. Recall that for $x= e^\alpha$ and $y=e^\beta$, 
\[
\tilde{P}(\alpha, \beta) = P(x,y), \quad \tilde\phi_1(\alpha, \beta) = \phi_1(x,y),
\]
and that the mapping $(\alpha,\beta)\to (x,y)=(e^\alpha, e^\beta)$ determines a homeomorphism from $\R^2$ to $]0,+\infty[^2$ and maps the set $\tilde{D}\cap \tilde{D}_1$  to the set $D\cap D_1$. Hence, to prove our lemma, it is sufficient to show that under our assumptions, the following assertions hold
\begin{itemize}
\item[--] there exist $\alpha^*, \alpha^{**}\in \R$ such that $\alpha^* < \alpha^{**}$ and  
\[
[\alpha^*, \alpha^{**}] = \{\alpha\in\R : (\alpha,\beta)\in \tilde{D}_1\cap \tilde{D} \; \text{for some} \; \beta\in\R \}.
\] 
\item[--] for any $\alpha\in [\alpha^*, \alpha^{**}] $,  there exists  $\tilde\beta_2(\alpha) \geq \ln Y_1(e^\alpha)$ such that 
\[
\{\beta\in\R~:~(\alpha,\beta)\in \tilde{D}_1\cap \tilde{D}\} = [\ln {Y}_1(e^\alpha), \tilde{\beta}_2(\alpha)], 
\] 
 the points $\beta_1(\alpha) = \ln Y_1(e^\alpha)$ and $\tilde{\beta}_2(\alpha)$ are the only real  positive solutions of the equation 
 \be\label{eq0000_preliminary_lemma2} 
 \max\{\tilde{P}(\alpha,\beta), \tilde\phi_1(\alpha,\beta)\} = 1,
 \ee
 and  $\ln Y_1(e^\alpha) < \tilde{\beta}_2(\alpha)$ if and only if $\alpha^*< \alpha <\alpha^{**}$. 
\end{itemize} 
For this we remark that the set $\tilde{D}\cap \tilde{D}_1$ is  compact convex and has a non empty interior because 
\begin{itemize}
\item[--] the function $\tilde\phi_1$ is convex and consequently also the set  $\tilde{D}_1$ is also convex;
\item[--]  the set $\tilde{D}$ is convex and compact (see the proof of Lemma~\ref{preliminary-lemma1});
\item[--] by Assumption (A3)(iii), the set $D\cap D_1$ has a non-empty interior, and consequently, since the mapping $(\alpha,\beta)\to (x,y)=(e^\alpha, e^\beta)$ determines a homeomorphism from $\R^2$ to $]0,+\infty[^2$ and maps the set $\tilde{D}_1\cap\tilde{D}$ to the set $D_1\cap D$, the set  $\tilde{D}_1\cap\tilde{D}$ has also   a non-empty interior. 
\end{itemize} 
From this, it follows that the following assertions hold:
\begin{itemize} 
\item[--]  there exist $\alpha^* > 0$ and $\alpha^{**} > \alpha^*$ such that 
\[
[\alpha^*, \alpha^{**}] = \{\alpha\in \R~: (\alpha, \beta)\in \tilde{D}_1\cap\tilde{D} \; \text{for some} \; \beta\in\R\}, 
\]
\item[--] for any $\alpha\in[\alpha^*, \alpha^{**}]$, there exists $\tilde\beta_1(\alpha), \tilde\beta_2(\alpha)\in ]0,+\infty[$ such that $\tilde\beta_1(\alpha) \leq \tilde\beta_2(\alpha)$, 
\[
[\tilde\beta_1(\alpha), \tilde\beta_2(\alpha)] = \{\beta \in\R: (\alpha,\beta)\in\tilde{D}\cap \tilde{D}_1\},
\]
and
\[
\max\{\tilde\phi_1(\alpha,  \tilde{\beta}_1(\alpha)), \tilde{P}(\alpha,  \tilde{\beta}_1(\alpha))\}  = \max\{\tilde\phi_1(\alpha,  \tilde{\beta}_2(\alpha)), \tilde{P}(\alpha,  \tilde{\beta}_2(\alpha))\}  = 1,
\]
\item[--]  for any $\alpha\in]\alpha^*, \alpha^{**}[$
\[
\tilde\beta_1(\alpha) < \tilde\beta_2(\alpha).
\]
\end{itemize} 
Hence to complete our proof, it is sufficient  to show that for any $\alpha\in[\alpha^*, \alpha^{**}]$, the following relations hold: 
\be\label{eq111-preliminary-lemma2}
\tilde\beta_1(\alpha) = \ln Y_1(e^\alpha),
\ee
\be\label{eq113-preliminary-lemma2}
\max\{\tilde{P}(\alpha,\beta), \tilde\phi_1(\alpha,\beta)\} < 1 \quad \quad \text{if} \quad \quad \beta_1(\alpha) < \beta < \tilde\beta_2(\alpha), 
\ee
and
\be\label{eq112-preliminary-lemma2}
 \alpha^*< \alpha <\alpha^{**} \quad \text{if} \quad \tilde\beta_1(\alpha) < \tilde{\beta}_2(\alpha).
\ee
Remark moreover that if \eqref{eq113-preliminary-lemma2} holds, then  any $\beta$ such that $\tilde\beta_1(\alpha) < \beta < \tilde{\beta}_2(\alpha)$, the point $(\alpha,\beta)$ belongs to the interior of the set $\tilde{D}\cap\tilde{D}_1$ and consequently, \eqref{eq112-preliminary-lemma2} also holds. Hence, to complete our proof, it is sufficient to prove that  for any $\alpha\in[\alpha^*, \alpha^{**}]$, \eqref{eq111-preliminary-lemma2} and \eqref{eq113-preliminary-lemma2} hold. 
To get these relations, we remark that by Lemma~\ref{preliminary-lemma1} and since $\tilde{D}\cap \tilde{D}_1 \subset \tilde{D}$, one has 
\[
 [\alpha^*,\alpha^{**}] \subset \{\alpha \in\R: (\alpha,\beta)\in\tilde{D} \quad \text{for some} \quad \beta\in\R\} = [\ln x^*_P, \ln x^{**}_P]
\]
and that for any $\alpha\in [\alpha^*,\alpha^{**}]$,
\[
[\tilde\beta_1(\alpha), \tilde\beta_2(\alpha)]  \subset\{\beta\in\R: (\alpha,\beta)\in \tilde{D}\} = [\ln Y_1(e^\alpha), \ln Y_2(e^\alpha)]. 
\]
Since because of Assumption (A3), the function $\beta \to \tilde\phi_1(\alpha,\beta)$ is finite and strictly increasing in a neighborhood of the line segment $[0, \ln Y_2(e^\alpha)]$, one gets therefore that for any $\alpha\in [\alpha^*,\alpha^{**}]$ and $\beta < \tilde\beta_2(\alpha)$, 
\[
\tilde\phi_1(\alpha,\beta) < \tilde\phi_1(\alpha, \tilde\beta_2(\alpha)) \leq 1.
\]
This implies that for any $\alpha\in [\alpha^*,\alpha^{**}]$, \eqref{eq111-preliminary-lemma2} holds.  Moreover, since by Lemma~\ref{preliminary-lemma1}, $\tilde{P}(\alpha, \beta) < 1$ for any $\beta\in]\ln Y_1(e^\alpha), \ln Y_2(e^\alpha)[$, from the last relation it follows that for any  $\beta\in ]\ln Y_1(e^\alpha), \tilde\beta_2(\alpha)[$, \eqref{eq113-preliminary-lemma2} also holds. 
\end{proof}

As a consequence of Lemma~\ref{preliminary-lemma2} we obtain

\begin{cor}\label{preliminary_cor2} Under the assumptions (A1) and (A3),  for any $x\in [x^*_P, x^{**}_P] $,  
\begin{enumerate}[label=(\roman*)]
\item $x\in[x^*,x^{**}]$ if and only if $\phi_1(x, Y_1(x)) \leq 1$; 
\item  $\phi_1(x, Y_1(x)) < 1$ if $x^* < x < x^{**}$;
\item  $\phi_1(x^{**}, Y_1(x^{**})) = 1$ if $x^{**} < x^{**}_P$; 
\item $\phi_1(x^*, Y_1(x^*)) = 1$ if $x^* > x^*_P$.
\end{enumerate}
\end{cor} 
\begin{proof} Indeed, since $P(x, Y_1(x)) = 1$  for any $x\in [x^*_P, x^{**}_P]$, the first assertion of this corollary  follows from Lemma~\ref{preliminary-lemma2}.

To get (ii), we recall that by Lemma~\ref{preliminary-lemma2}, for any $x^* < x < x^{**}$ one has $Y_1(x) < \tilde{Y}_2(x)$ and $\max\{\phi_1(x, \tilde{Y}_2(x)), P(x, \tilde{Y}_2(x))\} = 1$. Since under the hypotheses (A3), for any $x\in[x^*,P, x^{**}_P]$, the function $y\to \phi_1(x,y)$ is finite and strictly increasing n a neighborhood of the line segment $[0, \ln Y_2(e^\alpha)]$, from this it follows that $\phi_1(x, Y_1(x)) < 1$ whenever $x^* < x < x^{**}$. The assertion (ii) is therefore also proved.  

Suppose now that $x^{**} < x^{**}_P$. Then by the first assertion  of our corollary, 
\be\label{eq1_preliminary_cor2}
\phi_1(x, Y_1(x)) > 1 \quad \text{if} \quad x^{**} < x \leq x^{**}_P,
\ee
and by the second assertion, 
\be\label{eq2_preliminary_cor2}
\phi_1(x, Y_1(x)) < 1 \quad \text{if} \quad x^{*} < x < x^{**}.
\ee
Since the function $\phi_1$ is continuous in a neighborhood of the set $D$, the function $Y_1$ is continuous on $[x^*_P, x^{**}_P]$  and $(x, Y_1(x))\in D$ for any $x\in[x^*_P, x^{**}_P]$, the function $x\to \phi_1(x,Y_1(x))$ is therefore continuous on $[x^*_P, x^{**}_P]$, and consequently, relations \eqref{eq1_preliminary_cor2} and \eqref{eq2_preliminary_cor2} prove that $\phi_1(x^{**}, Y_1(x^{**})) = 1$. The assertion (iii) is therefore also proved. The proof of the assertion (iv) is quite similar. 
\end{proof}

\section{The Classification in  Eight Regions}\label{section_cases} 
This section is devoted to the proof of Proposition~\ref{cases} establishing our classification into eight regions in the set of parameters of the random walk $(Z(n))$.

\subsection{The main idea of the proof} 
Remark first of all that the cases (B0){-}(B7) have a simple geometrical interpretation~:~ if we denote by $[(x,y), (\tilde{x},\tilde{y})]$ the line segment in $\R^2$ with the end-points at $(x,y)$ and $(\tilde{x},\tilde{y})$, then by  Lemma~\ref{convexity_lemma1} below,   for any $x\in[x^*_P, x^{**}_P]$ and $y\in [y^*_P, y^{**}_P]$, one and only one of the following assertions holds:
\begin{itemize}
\item the line segments $[(x, Y_1(x)), \, (x, Y_2(x))]$ and $[(X_1(y), y), \, (X_2(y), y)]$ have the common point $(x,y)\in D$,
\item these line segments are disjoint and $(x,y)\not\in D$.
\end{itemize} 
The case  (B0) occurs when the line segments 
 \be\label{line segments**}
 [(x^{**}, Y_1(x^{**})), \, (x^{**}, Y_2(x^{**}))] \quad \text{ and }\quad [(X_1(y^{**}), y^{**}), \, (X_2(y^{**}), y^{**})]
 \ee
 have the common point $(x^{**},y^{**})$ in the interior of the set $D$.  The cases (B1), (B3), (B5) and (B7) occur when these line segments have the common point $(x^{**},y^{**})$ on the boundary of the set $D$:~ 
 \[
 (x^{**}, y^{**})\in \begin{cases} {\mathcal S}_{22} &\text{in the case (B1)},\\
  {\mathcal S}_{12} &\text{in the case (B3)},\\
   {\mathcal S}_{21} &\text{in the case (B5)},\\
  {\mathcal S}_{11} &\text{in the case (B7)}.\\
 \end{cases} 
\] 
The cases (B2), (B4) and (B6)  occur when the line segments  \eqref{line segments**} are disjoint and  $(x^{**}, y^{**})~\not\in~D$. The nearest to $(x^{**}, y^{**})$ vertices of these line segments belong to ${\mathcal S}_{22}$ in the case (B2), to ${\mathcal S}_{12}$ in the case (B4) and to ${\mathcal S}_{21}$ in the case (B6).

The main idea of our proof is the following~:~  \\ 
 First, we  show that  the point $(x^{**},y^{**})$ either belongs to each of the line segments 
 \be\label{sec_cases_line segments**}
 [(x^{**}, Y_1(x^{**}), \, (x^{**}, Y_2(x^{**})] \quad \text{and} \quad [(X_1(y^{**}), y^{**}), \, (X_2(y^{**}), y^{**})]
 \ee
 or does not belong to any of them. With this result we describe all possible cases~: 
 \be\label{eq-cases-(a)- (i) } 
 \begin{matrix}
 \text{(a)} &X_1(y^{**}) < x^{**} < X_2(y^{**}) \quad \text{and}  \quad Y_1(x^{**}) < y^{**} < Y_2(x^{**})\\
 \text{(b)} &X_1(y^{**}) < x^{**} = X_2(y^{**})  \quad \text{  and }  \quad   Y_1(x^{**}) < y^{**} = Y_2(x^{**})   \\
 \text{(c)} &x^{**} > X_2(y^{**})  \quad \text{  and }   \quad  y^{**} > Y_2(x^{**})  \\
 \text{(d)} &x^{**} = X_1(y^{**})  \quad \text{  and }   \quad Y_1(x^{**}) < y^{**} = Y_2(x^{**})\\
 \text{(e)} &x^{**} < X_1(y^{**})  \quad \text{  and }   \quad y^{**} > Y_2(x^{**})\\
 \text{(f)} &X_1(y^{**}) < x^{**}  = X_2(y^{**})  \quad \text{  and }   \quad y^{**} = Y_1(x^{**})\\
 \text{(g)} &x^{**} > X_2(y^{**})  \quad \text{  and }   \quad y^{**} < Y_1(x^{**})\\
 \text{(h)} &x^{**} = X_1(y^{**})  \quad \text{  and }   \quad y^{**} =  Y_1(x^{**})\\
 \text{(i)} &x^{**} < X_1(y^{**})  \quad \text{  and }   \quad y^{**} < Y_1(x^{**}) 
 \end{matrix} 
\ee
 \noindent 
Next, using \eqref{eq2-critical-points}, we prove that the cases (B0) - (B7)  correspond respectively  to the cases (a){-}(h), and the case  i)  never holds. 
 
\subsection{Preliminary results for the proof of Proposition~\ref{cases}} 
 We begin our proof with the following  preliminary results.

\begin{lemma}\label{convexity_lemma1} Under the hypotheses (A1), for any $x\in [x^*_P, x^{**}_P]$ and $y\in [y^*_P, y^{**}_P]$,   the line segments 
$[(x, Y_1(x)), (x,Y_2(x))]$ and $[(X_1(y),y), (X_2(y),y)]$ 
 are disjoint  if and only if 
the point $(x,y)$ does not belong to the set $D$.
\end{lemma}
\begin{proof} Indeed, by Lemma~\ref{preliminary-lemma1}, for any $x\in [x^*_P, x^{**}_P]$ and $y\in [y^*_P, y^{**}_P]$, the line segment $[(x, Y_1(x)), (x,Y_2(x))]$ is the set of all points $(x',y')\in D$ with $x' = x$, and similarly the line segment $[(X_1(y),y), (X_2(y),y)]$ is the set of all points $(x',y')\in D$ with $y' = y$. Hence, if $(x,y)\in D$ then the both line segments $[(x, Y_1(x)), (x,Y_2(x))]$ and $[(X_1(y),y), (X_2(y),y)]$ contain the point $(x,y)$. Conversely, if these line segments are not disjoint, then the point  $(x,y)$ belongs to each of them, and consequently $(x,y)\in~D$. 
\end{proof} 

Since for any $x\in [x^*_P, x^{**}_P]$ and $y\in [y^*_P, y^{**}_P]$, the point $(x,y)$ is the only point that could belong to  the both line segments $[(x, Y_1(x)), (x,Y_2(x))]$ and $[(X_1(y),y), (X_2(y),y)]$,  the above lemma implies the following statement.

\begin{cor}\label{convexity_cor} Under the hypotheses (A1), for any $x\in [x^*_P, x^{**}_P]$ and $y\in [y^*_P, y^{**}_P]$,  if the point $(x,y)$ does not belong to some of the line segments $[(x, Y_1(x)), (x,Y_2(x))]$ or $[(X_1(y),y), (X_2(y),y)]$ then $(x,y)$ neither belongs  to any of them. 
\end{cor} 

By lemma~\ref{preliminary-lemma1}, for any $x\in [x^*_P, x^{**}_P]$, the points $Y_1(x)$ and $(Y_2(x)$ are the only real and positive solutions of the equation $P(x,y) = 1$, and that $P(x,y) < 1$ for any $y\in]Y_1(x), Y_2(x)[$. Hence, for any $x\in [x^*_P, x^{**}_P]$, each of the points $(x,Y_1(x))$ and $(x,Y_2(x))$ belongs to the boundary of the set $D$, and for any $y\in]Y_1(x), Y_2(x)[$, the point $(x,y)$ belongs to the interior $\inter{D}$ of $D$. Similarly, for any $y\in[y^*_P, y^{**}_P]$, each of the points $(X_1(y), y)$ and $(X_2(y),y)$ belongs to the boundary of $D$ and for any $x\in]X_1(y), X_2(y)[$, the point $(x,y)$ belongs to $\inter{D}$. Using Corollary~\ref{convexity_cor}  it  follows the following useful for our purpose property of the set $D$: 

\begin{cor}\label{convexity_cor1} Under the hypotheses (A1), for any $(x,y)\in D$, one and only one of the following assertions holds
\begin{itemize}
\item the point $(x,y)$ belongs to the interior of the set $D$ and in this case, $X_1(y) < x < X_2(y)$ and $Y_1(x) < y < Y_2(x)$;
\item the point $(x,y)$ belongs to the boundary of $D$ and in this case, the point $(x,y)$ is an end point of each of the line segments $[(x, Y_1(x)), (x,Y_2(x))]$ and $[(X_1(y),y), (X_2(y),y)]$. 
\end{itemize} 
\end{cor} 
 
By Lemma~\ref{preliminary-lemma1}, for any $x\in[x^*_P, x^{**}_P]$,  we have
\[
Y_1(x) = Y_2(x) \quad \text{if and only if} \quad x\in\{x^*_P, x^{**}_P\} 
\]
and, for any $y\in[y^*_P, y^{**}_P]$, 
\[
X_1(y) = X_2(y) \quad \text{if and only if} \quad y\in\{y^*_P, y^{**}_P\}, 
\]
and,  the four points $(x^{*}_P, Y_1(x^*_P))$, $(x^{**}_P, Y_1(x^{**}_P))$, $(X_1(y^*_P), y^*_P)$ and $(X_2(y^{**}_P), y^{**}_P)$ are two by two distinct. Hence, the case with $x= X_1(y)=X_2(y)$ and $y =Y_1(x) = Y_2(x)$ is not possible, and as a straightforward consequence of Corollary~\ref{convexity_cor} and Corollary~\ref{convexity_cor1} one gets 

\begin{cor}\label{convexity_cor2} Under the hypotheses (A1), for any $x\in [x^*_P, x^{**}_P]$ and $y\in [y^*_P, y^{**}_P]$, one and only one of the following cases holds: 
\be\label{eq2-cases-(a)- (i)} 
 \begin{matrix}
 \text{(a)} & X_1(y) < x < X_2(y) \quad \text{and}  \quad Y_1(x) < y < Y_2(x)\\
 \text{(b)} &x = X_2(y) > X_1(y)  \quad \text{  and }  \quad y = Y_2(x) > Y_2(x) \\
 \text{(c)}  &x > X_2(y)  \quad \text{  and }   \quad y > Y_2(x)\\
 \text{(d)} & x = X_1(y)  \quad \text{  and }   \quad Y_1(x) < y = Y_2(x)\\
 \text{(e)} &x < X_1(y)  \quad \text{  and }   \quad y > Y_2(x)\\
 \text{(f)}  &X_1(y) < x  = X_2(y)  \quad \text{  and }   \quad y = Y_1(x)\\
 \text{(g)}  &x > X_2(y)  \quad \text{  and }   \quad y < Y_1(x)\\
 \text{(h)} &x = X_1(y)  \quad \text{  and }   \quad y =  Y_1(x)\\
 \text{ (i) }   &x < X_1(y)  \quad \text{  and }   \quad y < Y_1(x) 
 \end{matrix} 
\ee
\end{cor} 
In the case (a) of this statement, the point $(x,y)$ belongs to the interior of the set $D$, in each of the cases (b), (d), (f) and (h), the point $(x,y)$ belongs to the boundary of $D$, and in each of the cases (c), (e), (g) and (i), the point $(x,y)$ does not belong to the set $D$.

When applied with $x=x^{**}$ and $y=y^{**}$, this statement  proves that under the hypotheses (A1), one of the assertions (a){-}(i) of \eqref{eq-cases-(a)- (i) }  holds.

The following statement will be used to investigate the position of the nearest to $(x^{**}, y^{**})$ vertices of the line segments \eqref{sec_cases_line segments**}  when the point $(x^{**},y^{**})$ does not belong to the interior of  $D$. 
\begin{lemma}\label{convexity_lemma3} Under the hypotheses (A1), for any $x\in [x^*_P, x^{**}_P]$ and $y\in [y^*_P, y^{**}_P]$, the following assertions hold~: \\
1) If $x\leq X_1(y)$ and $Y_2(x) \leq y$, then $(x,Y_2(x)), (X_1(y),y) \in{\cal S}_{12}$. \\
2) If $y\leq Y_1(x)$ and $X_2(y) \leq x$, then $(X_2(y),y), \, (x,Y_1(x))\in{\mathcal S}_{21}$.\\
3) If $x \geq X_2(y)$ and $y  \geq X_2(x)$ then $(x,Y_2(x)), (X_2(y),y)\in{\mathcal S}_{22}$.\\
4) If $x\leq X_1(y)$ and $y\leq Y_1(x)$ then $(x, Y_1(x)), \, (X_1(y), y)\in{\mathcal S}_{11}$. 
\end{lemma} 
\begin{proof} Suppose that $x\in [x^*_P, x^{**}_P]$ and $y\in [y^*_P, y^{**}_P]$ and let $x \leq X_1(y)$ and $Y_2(x) \leq y$. Then by Corollary~\ref{convexity_cor2}, either $x < X_1(y)$ and $Y_2(x) < y$, or $x = X_1(y)$ and $Y_2(x) =  y$. In the second case, i.e when $x = X_1(y)$ and $Y_2(x) =  y$, we have $(x,Y_2(x)) = (X_1(y),y)$ and consequently, by the definition of ${\mathcal S}_{12}$, 
\[
(x,Y_2(x)) = (X_1(y),y) \in {\mathcal S}_{12}.
\]

Consider now the case when $x < X_1(y)$ and $Y_2(x) < y$. By the definition of the curves \eqref{eq_S_11} - \eqref{eq_S_22},  we have
\[
(x, Y_2(x))\in {\mathcal S}_{12}\cup{\mathcal S}_{22} \quad \text{and} \quad (X_1(y), y) \in{\mathcal S}_{11}\cup{\mathcal S}_{12}.
\]
If we suppose that $(x,Y_2(x))\in{\mathcal S}_{22}$, then we will get $Y_2(x)\in[Y_1(x^{**}_P), y^{**}_P]$ and consequently, since $Y_2(x) < y \leq y^{**}_P$, we will have also $y\in [Y_1(x^{**}_P), y^{**}_P]$. Since the function $Y_2$ is strictly decreasing on $[Y_1(x^{**}_P), y^{**}_P]$ and since $y > Y_2(x)$,  it  follows that 
\be\label{eq1_convexity_lemma3}
X_2(x) < X_2\circ Y_2(x) = x
\ee
where the last relation holds because by Lemma~\ref{preliminary-lemma1}, the function $X_2:~[Y_1(x^{**}_P), y^{**}_P]\to [X_2(y^{**}_P), x^{**}_P]$ is inverse to the function $Y_2:~  [X_2(y^{**}_P), x^{**}_P]\to [Y_1(x^{**}_P), y^{**}_P]$. Since $X_1(y)\leq X_2(y)$, \eqref{eq1_convexity_lemma3} contradicts the inequality $x < X_1(y)$ and consequently, when $x < X_1(y)$ and $Y_2(x) < y$, the point $(x, Y_2(x))$ belongs to ${\mathcal S}_{12}$. Similar arguments show that in this case,  the point $(X_1(y),y)$ also belongs to ${\mathcal S}_{12}$. The first assertion of Lemma\ref{convexity_lemma3} is therefore proved. 
The proof of the assertions 2)-4) is quite similar.  
\end{proof}

To investigate the cases (d) and (e) we will need the following  preliminary result.

\begin{lemma}\label{convexity_lemma4} If the conditions (A1) - (A3) are satisfied and $(x^{**}, Y_2(x^{**}))\in{\mathcal S}_{12}$, then 
\be\label{eq1-convexity-lemma4}
Y_1(x^{**}) < Y_2(x^{**}) \quad \text{and} \quad 1 < Y_2(x^{**}). 
\ee
\end{lemma} 
\begin{proof} Indeed, under the hypotheses of this lemma, with the definition of ${\mathcal S}_{12}$ and using \eqref{eq-critical-points}, one gets 
\be\label{eq2-convexity-lemma4}
x^*_P \leq x^* < x^{**} \leq X_1(y^{**}_P) < x^{**}_P.
\ee
Hence, in this case, $x^{**}\not= x^*_P$ and $x^{**}\not= x^{**}_P$, and consequently,  by Lemma~\ref{preliminary-lemma1}, $Y_1(x^{**}) < Y_2(x^{**})$. The first relation of \eqref{eq1-convexity-lemma4} is therefore proved. To get the second relation of \eqref{eq1-convexity-lemma4}, we recall that  by Lemma~\ref{preliminary-lemma1}, the function $Y_2$ is strictly increasing on the line segment $[x^*_P,  X_1(y^{**}_P)]$, and we remark that by\eqref{eq2-convexity-lemma4} and  \eqref{eq2-critical-points} the following relations hold :
\[x^*_P \leq 1 \leq x^{**} < X_1(y^{**}_P).
\]
 Hence, 
$
Y_2(1) \leq Y_2(x^{**})$ 
and moreover, 
\be\label{eq4-convexity-lemma4}
Y_2(1) < Y_2(x^{**}) \quad \text{whenever} \quad 1 < x^{**}. 
\ee
Remark now that, with the definition of $Y_1(1)$ and $Y_2(1)$ and since $P(1,1)=1$, we have 
\be\label{eq5-convexity-lemma4}
\text{either } \quad 
Y_1(1) = 1 \leq Y_2(1) \quad \text{or} \quad Y_1(1) \leq 1 = Y_2(1). 
\ee
With these relations and using \eqref{eq4-convexity-lemma4}, one gets 
\[
1 < Y_2(x^{**})  \quad \text{whenever} \quad 1 < x^{**}. 
\]
Now, to complete the proof of our lemma it is sufficient to show that 
\[
1 < Y_2(1) \quad \text{whenever} \quad 1 = x^{**}. 
\]
Suppose that $x^{**}=1$. Then from the first relation of \eqref{eq1-convexity-lemma4} one gets 
\be\label{eq6-convexity-lemma4}
Y_1(1) < Y_2(1),
\ee
and by  Corollary~\ref{preliminary_cor2} and using \eqref{eq2-convexity-lemma4}, we obtain  
\[
\phi(1,Y_1(1)) = 1.
\]
Since the function $y\to \phi(1, y)$ is strictly increasing, from  the last relation and \eqref{eq6-convexity-lemma4} it follows that 
$
\phi_1(1, Y_2(1)) > 1,
$
 and since  under our hypotheses, $\phi_1(1,1) \leq 1$, this proves that  $Y_2(1) \not=1$.  Hence, using again \eqref{eq6-convexity-lemma4}, we conclude  that 
 \[
Y_2(x^{**}) = Y_2(1) > Y_1(1) = 1.
\]
Lemma~\ref{convexity_lemma4} is therefore proved. 
\end{proof} 

To investigate the cases (h) and  i)  the following lemma will be used. 

\begin{lemma}\label{convexity_lemma2} Suppose that the condition (A1)  is satisfied and let $x\in [x^*_P, x^{**}_P]$, $y\in [y^*_P, y^{**}_P]$ and  $(x',y')\in D$ be such that  
\be\label{eq0_convexity_lemma2} 
x' \leq x \leq X_1(y) \quad \text{ and } \quad y' \leq y\leq Y_1(x). 
\ee
Then 
\be\label{eq0a_convexity_lemma2} 
x' = x = X_1(y)\quad \; \text{and} \; \quad y' = y = Y_1(x). 
\ee
\end{lemma} 
\begin{proof} Indeed, suppose that the conditions of our lemma are satisfied and let \eqref{eq0_convexity_lemma2} holds. Then according to the definition of the line segments $[x^*_P, x^{**}_P]$ and $[y^{*}_P, y^{**}_P]$, we have $x^*_P \leq x' \leq x^{**}_P$ and $y^*_P\leq y' \leq y^{**}_P$, and  by Lemma~\ref{convexity_lemma3}, the points $(x,Y_1(x))$ and $(X_1(y),y)$ belong to the set ${S}_{11}$. Using the definition of ${\mathcal S}_{11}$,  it  follows that 
\be\label{eq1-convexity-lemma2} 
x^*_P \leq x' \leq x < X_1(y^*_P) \quad \text{and} \quad 
 y^*_P \leq y' \leq y \leq Y_1(x^*_P). 
\ee
Since by Lemma~\ref{preliminary-lemma1}, the function $Y_1$ is decreasing on the line segment $[x^*_P, X_1(y^*_P)]$ and the function $X_1$ is  decreasing on the line segment $[y^*_P, Y_1(x^*_P)]$,  relations \eqref{eq1-convexity-lemma2}  imply that $
Y_1(x') \geq Y_1(x)$ and $
X_1(y') \geq X_1(y)$, 
and consequently, using \eqref{eq0_convexity_lemma2} one gets 
\be\label{eq3-convexity-lemma2} 
y' \leq y \leq Y_1(x) \leq Y_1(x') \quad \text{and} \quad x' \leq x \leq X_1(y) \leq X_1(y'). 
\ee
Under the hypotheses of our lemma, $(x',y')\in D$ and by Corollary~\ref{convexity_cor} and  the definition of the line segments $[X_1(y'), X_2(y')]$ and $[Y_1(x'), Y_2(x')]$,  we have
\[
X_1(y') \leq x' \leq X_2(y') \quad \text{and} \quad Y_1(x') \leq y' \leq Y_2(x').
\]
When combined with \eqref{eq3-convexity-lemma2} these relations prove \eqref{eq0a_convexity_lemma2}. 
\end{proof}

\subsection{Proof of Proposition~\ref{cases}} Now we are ready to complete the proof of Proposition~\ref{cases}. Under our hypotheses, by Corollary~\ref{convexity_cor2} applied with $x=x^{**}$ and $y=y^{**}$, one and only one of the cases (a)- (i)  of \eqref{eq-cases-(a)- (i) } holds, and remark that the case (a) is equivalent  to the case (B0). 

Remark furthermore that the case (B1) implies (b), and conversely, if the case (b) of \eqref{eq-cases-(a)- (i) } holds, then by Lemma ~\ref{convexity_lemma3} applied with $x=x^{**}$ and $y=y^{**}$, the point $(x^{**},y^{**})$ belongs to the set ${\mathcal S}_{22}$, and  by Lemma~\ref{preliminary-lemma1},  
\[
x^{**} < x^{**}_P \quad \; \text{and} \; \quad y^{**} < y^{**}_P
\]
because in this case, we have $Y_1(x^{**}) < Y_2(x^{**})$ and $X_1(y^{**}) < X_2(y^{**})$,  The case (B1) is therefore equivalent to the case (b). 

Similarly, the case (B2) implies the case (c) of \eqref{eq-cases-(a)- (i) } and conversely, in the case (c), by Lemma~\ref{convexity_lemma3} applied with $x=x^{**}$ and $y=y^{**}$, the points $(x^{**}, Y_2(x^{**}))$ and $(X_2(y^{**}), y^{**})$ belong to the set ${\mathcal S}_{22}$. The case (B2) is therefore equivalent to the case (c) of \eqref{eq-cases-(a)- (i) }.  

And similarly, the case (B3) (resp. (B4)) implies the case (d) (resp (e)), and  conversely, if (d) (resp. (e)) holds, then by Lemma~\ref{convexity_lemma3} applied with $x=x^{**}$ and $y=y^{**}$ the point $(x^{**},y^{**}) = (X_1(y^{**}), y^{**}) = (x^{**}, Y_2(x^{**}))$ belongs to the set  ${\mathcal S}_{12}$ (resp. the points 
  $(x^{**}, Y_2(x^{**}))$ and $(X_1(y^{**}), y^{**})$ belong to the set  ${\mathcal S}_{12}$), and by Lemma~\ref{convexity_lemma4} combined with Lemma~\ref{preliminary-lemma1} and \eqref{eq2-critical-points}, one gets 
\[
Y_1(x^{**}) < Y_2(x^{**}), \quad x^{**} < x^{**}_P \quad \text{and} \quad y^* \leq 1 < Y_2(x^{**}). 
\]
The case (d) is therefore equivalent to the case (B3), and the case (e) is equivalent to (B4). 
 
\noindent 
Similar arguments (it is sufficient to exchange the roles of $x$ and $y$) show that the case (f) is equivalent to (B5) and the case (g) is equivalent to (B6). 
 
\noindent 
Now, to complete the proof of our proposition it is sufficient  to show that the case  (i) of \eqref{eq-cases-(a)- (i) } never holds, and the case (h) is equivalent to (B7). To get this result  we apply Lemma~\ref{convexity_lemma2} with $(x',y')=(1,1)$ and $(x,y)=(x^{**},y^{**})$. 
By relations~\eqref{eq-critical-points} and~\eqref{eq2-critical-points}, we have
\be\label{eq111}
x^*_P\leq x^*\leq 1 \leq x^{**} \leq x^{**}_P \quad \text{and} \quad y^*_P\leq y^* \leq 1 \leq y^{**}\leq y^{**}_P.  
\ee
Hence, when either (h) or  (i)  holds, i.e. if 
\[
X_1(y^{**}) \geq x^{**} \quad \text{and} \quad Y_1(x^{**}) \geq y^{**},
\]
using Lemma~\ref{convexity_lemma2} with $x=x^{**}$, $y= y^{**}$ and $(x',y') = (1,1)$ we obtain   
\be\label{eq222}
(x^{**},y^{**}) = (1,1) \quad \text{and} \quad X_1(1) = Y_1(1) = 1.
\ee
The case  (i)  is therefore impossible and the case (h) is equivalent to \eqref{eq222}. Remark moreover that by Lemma~\ref{convexity_lemma3} applied with $x=x^{**}=1$ and $y=y^{**}=1$, from \eqref{eq222} it follows that the point $(1,1)$ belongs to the curve ${\mathcal S}_{11}$ 
and consequently, by Lemma~\ref{preliminary-lemma1},
\be\label{eq-333}
\partial_x P(1,1) \leq 0 \quad \; \text{and}\; \quad \partial_y P(1,1) \leq 0,
\ee
and notice that by \eqref{eq222} and using \eqref{eq-critical-points} and \eqref{eq2-critical-points},
\[
x^*_P \leq x^* < x^{**} = 1 \quad\; \text{and}\quad \; y^*_P \leq y^* < y^{**} = 1. 
\]
Since the points $(x,y) = (x^*_P, Y_1(x^*_P))$ and $(x',y') =  (X_1(y^*_P), y^*_P)$ are  the only points in ${\mathcal S}_{11}$ for which $\partial_y P(x,y)=0$ and $\partial_x P(x',y') = 0$, the last relations show that $\partial_x P(1,1) \not= 0$ and $\partial_y P(1,1) \not=0$. Using \eqref{eq-333} we conclude therefore that when (h) holds, we have also $\partial_x P(1,1) < 0$ and $\partial_y P(1,1) < 0$ and consequently the case (h) is equivalent to (B7).

\section{The Functional Equation and the Convergence Domain}\label{proof-theorem1} 
\subsection{Sketch of the proof of Theorem~\ref{theorem1}} 
The main ideas of the proof of Theorem~\ref{theorem1} are the following :  By using the method of Lyapunov functions, we first show  that the series 
\be\label{eq-generating-functions} 
H_j(x,y) = \sum_{k=(k_1,k_2)\in\Z^2_+} g(j,k) x^{k_1}y^{k_2}, \quad j\in\Z^2_+ 
\ee
(and consequently, also the series \eqref{series-h} and \eqref{series-h1-h2}) converge on a suitable polycircular set $\Omega(\Theta)$ closed to the points $(x_d,0)$ and $(0,y_d)$. This is a subject of Proposition~\ref{upper-bounds} below. With this preliminary result, we will be able to get the first assertion of Theorem~\ref{theorem1}  and to  introduce on the set $\Omega(\Theta)$, the functional equation \eqref{extended-functional-equation} (see Proposition~\ref{functional-eq-prop}). Next, we show that the functions at the right hand side of \eqref{extended-functional-equation} are analytic in the set $\{(x,y)\in \Omega({\Gamma}) :~ |x| < x_d, \;  |y| < y_d\}$, and we extend in this way, first the function $(x,y)\to R_j(x,y) = Q(x,u)h_j(x,y)$ and next the function $(x,y)\to h_j(x,y)$ as analytic functions to the set $\{(x,y)\in \Omega({\Gamma}) :~ |x| < x_d, \;  |y| < y_d\}$.

\begin{defi}\label{defi-Theta-Omega} 
If  one of the cases (B0){-}(B2) holds, we define  $\Theta$ as the logarithmically convex hull of the union of the two rectangles  $[0,x_d[\times[0,Y_1(x_d)[$ and  $[0, X_1(y_d)[\times[0,y_d[$ (see Figure~\ref{F5})~: 
\be\label{eq_def_Theta_B0_B2}
\Theta =  \text{LogCH}\Bigl\{\bigl([0,x_d[\times[0,Y_1(x_d)[ \bigr) \cup \bigl( [0, X_1(y_d)[\times[0,y_d[ \bigr)  \Bigr\}. 
 \ee 
 In the case when one of the assertions (B3){-}(B6) holds, we let  
 \be\label{eq-def-Theta-B3-B6}
\Theta = [0,x_d[\times [0, y_d[.
\ee
\end{defi}

\begin{prop}\label{upper-bounds} Under the hypotheses (A1){-}(A4),  the  series \eqref{eq-generating-functions}  converges on the set \be\label{Omega} 
\Omega(\Theta) = \{(x,y)\in\C^2~:~(|x|, |y|)\in \Theta\}.
\ee
\end{prop} 
\noindent
In the case when the  random walk $(Z(n))$ is positive recurrent, this statement follows from the results of Miyazawa~\cite{Miyazawa}. In Section~\ref{upper-bounds-section}, we give another  proof of this statement by using the method of Lyapunov functions.  Our proof is valid both for positive recurrent and for transient random walks. 
 
 \bigskip 
 
 \noindent 
Since for any $j\in\Z^2_+$, and $(x,y)\in]0,+\infty[^2$,
 \[
 xyh_j(x,y) \leq H_j(x,y), \quad x h_{1j}(x) = H_j(x,0) \leq H_j(x,y) \quad \text{and} \quad y h_{2j}(y) = H_j(0,y)\leq H_j(x,y), 
 \]
 as a straightforward consequence of Proposition~\ref{upper-bounds}, one gets 

\begin{cor}\label{upper-bounds-cor1} Under the hypotheses (A1){-}(A4), for any $j\in\Z^2$, 

i) the series \eqref{series-h1-h2} converge   (and consequently the functions $h_{1j}$ and $h_{2j}$ are analytic)   respectively in  $B(0, x_d)$ and in $B(0, y_d)$, with $x_d$ and $y_d$ defined respectively by \eqref{eq_def_xd} and \eqref{eq_def_yd};

ii) the series \eqref{series-h} converge on the set $\Omega(\Theta)$ and consequently, the function $h_j$ is analytic in $\Omega(\Theta)$. 
\end{cor} 

With this results, using classical arguments (see Section~\ref{functional-equation-section}) we obtain  
\begin{prop}\label{functional-eq-prop} Under the hypotheses (A1){-}(A4), for any $j\in\Z^2$ and $(x,y)\in\Omega(\Theta)$,  the functional equation \eqref{extended-functional-equation} holds.
\end{prop} 

By the definition of the points $x_d$ and $y_d$, see relations~\eqref{eq_def_xd} and~\eqref{eq_def_yd}, and of the sets $\Theta$ and ${\Gamma}$, when either (B3) or (B4) holds,  we have the relations 
\[
x_d = x^{**}, \quad y_d = Y_2(x_d)\quad \text{and} \quad  \Theta = [0,x_d[\times[0,Y_2(x_d)[ = \{(x,y)\in{\Gamma}: \, x < x_d, \; y<y_d\}, 
\]
and when either (B5) or (B6) holds,
\[
y_d = y^{**}, \quad x_d = X_2(y_d)\quad \text{and} \quad  \Theta = [0, X_2(y_d)[\times[0, x_d[ = \{(x,y)\in{\Gamma}: \, x < x_d, \; y<y_d\}. 
\]
Hence,  when one of the case (B3){-}(B6) occurs, Theorem~\ref{theorem1} follows from Corollary~\ref{upper-bounds-cor1}  and Proposition~\ref{functional-eq-prop}.

When one of the cases (B0){-}(B2) occurs, the first assertion of Theorem~\ref{theorem1} follows from Corollary~\ref{upper-bounds-cor1} as well, and 
to prove the second assertion of Theorem~\ref{theorem1}  the following proposition will be used.

\begin{prop}\label{lemma1-proof-theorem1} If the conditions (A1){-}(A3) are satisfied and one of the assertions (B0){-}(B2) holds, then 
\be\label{eq1-lemma1-proof-theorem11}
\{(x,y)\in{\Gamma}: \, x < x_d, \; y<y_d\} \subset \Theta\cup \{(x,y)\in\inter{D}: \, x < x_d, \; y<y_d\}.
\ee
\end{prop} 

\medskip

With this lemma and using our previous results we will be able to show first that the function $(x,y){\to}h_j(x,y)$ can be continued to the set  $\Omega_d(\gamma)$ as a meromorphic function, and next to show that it is in fact analytic in this set. 

The proof of Theorem~\ref{theorem1} is organized as follows:
\begin{itemize} 
\item[--] Proposition~\ref{upper-bounds} is proved in Section~\ref{upper-bounds-section};
\item[--] Section~\ref{functional-equation-section} is devoted to the proof of Proposition~\ref{functional-eq-prop}; 
\item[--] the proof of Proposition~\ref{lemma1-proof-theorem1} is given in Section~\ref{section-proof- lemma1-proof-theorem1};
\item[--] the proof of Theorem~\ref{theorem1} is completed in Section~\ref{section-proof-theorem1-completed}.
\end{itemize} 
\subsection{Proof of Proposition~\ref{upper-bounds} }\label{upper-bounds-section} 

\subsubsection{The main idea of the proof.} 
To main tool of our proof of this proposition is the method of Lyapunov functions. We show that for any  $(x_1,y_1)\in\inter{D}\cap\inter{D_1}$ and $(x_2,y_2)\in\inter{D}\cap\inter{D_2}$ such that 
\be\label{eq0_upper_bounds} 
x_1 > x_2\quad \text{and} \quad y_1 < y_2,
\ee
the function $f:\Z^2_+\to[0,+\infty[$ defined by 
\be\label{lyapunov_function}
f(k_1,k_2) = x_1^{k_1}y_1^{k_2} + x_2^{k_1}y_2^{k_2}, \quad (k_1,k_2)\in\Z^2_+.
\ee
satisfies 
\be\label{lyapunov_eq}
\E_j\bigl(f(Z(1))\bigr) \leq \theta f(j), \quad \forall j\in\Z^2_+{\setminus} E,
\ee
with some  $0<\theta < 1$ and  for some finite subset $E\subset\Z^2_+$. This is a subject of lemma~\ref{lemma1-upper-bounds} below.    With this result we are able to prove that for any two points $(x_1,y_1)\in\inter{D}\cap\inter{D_1}$ and $(x_2,y_2)\in\inter{D}\cap\inter{D_2}$ satisfying \eqref{eq0_upper_bounds}, the series 
\be\label{eq1_upper_bounds}
\sum_{k= (k_1,k_2)\in\Z^2_+{\setminus} \{(0,0)\}} g(j,k) (x_1^{k_1}y_1^{k_2}  + x_2^{k_1}y_2^{k_2}) 
\ee
converges, this is a subject of Lemma~\ref{lemma2-upper-bounds}. Next, to prove that the series \eqref{eq-generating-functions}  converge for any $(x,y)\in\Theta$, we show that for any $(x,y)\in\Theta$, there are two points $(x_1,y_1)\in\inter{D}\cap\inter{D_1}$ and $(x_2,y_2)\in\inter{D}\cap\inter{D_2}$ satisfying \eqref{eq0_upper_bounds} and such that 
\be\label{eq1a_upper_bounds}
x^{k_1}y^{k_2} \leq x_1^{k_1}y_1^{k_2} + x_2^{k_1}y_2^{k_2}, \quad \forall k=(k_1,k_2)\in\Z^2_+. 
\ee

\subsubsection{Preliminary results for the proof of Proposition~\ref{upper-bounds}.} 
We begin our analysis with the following statements.

\begin{lemma}\label{lemma1-upper-bounds}    Suppose that the conditions (A1){-}(A3) are satisfied and let two points $(x_1,y_1)\in\inter{D}\cap\inter{D_1}$ and $(x_2,y_2)\in\inter{D}\cap\inter{D_2}$ satisfy \eqref{eq0_upper_bounds}. Then  for some finite set $E\subset\Z^2_+$, the function 
$f:\Z^2_+\to ]0,+\infty[$ defined by \eqref{lyapunov_function} satisfies \eqref{lyapunov_eq} with some $0 < \theta < 1$.
\end{lemma} 

\begin{proof} Consider two functions $f_1,f_2:\Z_+^2\to\R$ defined by 
\[
f_1(j_1,j_2) = x_1^{j_1}y_1^{j_2}, \quad \text{and} \quad f_2(j_1,j_2) = x_2^{j_1}y_2^{j_2}, \quad \forall (j_1,j_2)\in\Z_+^2,
\]
and let 
\[
\tilde\theta = \max\{\phi_1(x_1,y_1), \; P(x_1,y_1), \; \phi_2(x_2,y_2),\; P(x_2,y_2)\}.
\]
Then $0 < \tilde\theta < 1$ because  $(x_1,y_1)\in\inter{D}\cap\inter{D_1}$ and $(x_2,y_2)\in\inter{D}\cap\inter{D_2}$, and  moreover, for any $(j=(j_1,j_2)\in\Z_+^2{\setminus} \{0\}$, 
\[
\E_{(j_1,j_2)}(f_1(Z(1)) \leq 
\begin{cases} \tilde\theta f_1(j_1,j_2), &\text{if} \; j_1\not= 0, \\
\phi_2(x_1,y_1) y_1^{j_2},  &\text{if} \; j_1= 0,
\end{cases} 
\]
and 
\[
\E_{(j_1,j_2)}(f_2(Z(1)) \leq \begin{cases} \tilde\theta  f_2(j_1,j_2), &\text{if} \; j_2\not= 0, \\ 
\phi_1(x_2,y_2) x_2^{j_1}&\text{if} \; j_2= 0.
\end{cases} 
\]
For the function $f = f_1 + f_2$, we obtain therefore that for any $j=(j_1,j_2)\in\Z_+^2{\setminus} \{0\}$, 
\be\label{eq2_upper_bounds}
\E_{(j_1,j_2)}(f(Z(1)) \leq \begin{cases} \tilde\theta f(j_1,j_2) &\text{if $j_1 >0$ and $j_2 >0$}\\
\tilde\theta y_2^{j_2} + \phi_2(x_1,y_1) y_1^{j_2} &\text{if $j_1 = 0$ }\\
\tilde\theta x_1^{j_1} + \phi_1(x_2,y_2) x_2^{j_1} &\text{if  $j_2 = 0$}.
\end{cases} 
\ee
Since under  hypotheses of our lemma, $x_1 > x_2> 0$ and $0 < y_1 < y_2$, we have moreover 
\[
\lim_{j_1} (x_2/x_1)^{j_1} = \lim_{j_2\to\infty}(y_1/y_2)^{j_2} = 0,
\]
and consequently, for any $\eps > 0$ there is $N_\eps > 0$ such that for any $(j_1,j_2)\in\Z^2_+$ with $j_2 > N_\eps$ and $j_1 = 0$, 
\[
\tilde\theta y_2^{j_2} + \phi_2(x_1,y_1) y_1^{j_2} \leq (\tilde\theta + \eps) y_2^{j_2} = (\tilde\theta +\eps) f_2(j_1,j_2) \leq (\tilde\theta +\eps) f(j_1,j_2) 
\]
and similarly, for any $(j_1,j_2)\in\Z^2_+$ with $j_1 > N_\eps$ and $j_2 = 0$, 
\[
\tilde\theta x_1^{j_1} + \phi_1(x_2,y_2) x_2^{j_1} \leq (\tilde\theta+\eps) f(j_1,j_2). 
\]
Hence, for $\eps > 0$ such that $\tilde\theta + \eps < 1$, letting $$E =\{(j_1,0)\in\Z_+^2~:~j_1 \leq N_\eps\}\cup \{(0,j_2)\in\Z^2_+~:  \; j_2 \leq N_\eps\}$$ one gets \eqref{lyapunov_eq} with $\theta=\tilde\theta + \eps < 1$. 
\end{proof}

\begin{lemma}\label{lemma2-upper-bounds} Under the hypotheses of Lemma~\ref{lemma1-upper-bounds}, for any $j\in\Z_+^2$,  the series \eqref{eq1_upper_bounds} converges. 
\end{lemma} 
\begin{proof} Indeed, consider two points $(x_1,y_1)\in\inter{D}\cap\inter{D_1}$ and $(x_2,y_2)\in\inter{D}\cap\inter{D_2}$  satisfying \eqref{eq0_upper_bounds} and let  the function 
$f:\Z^2_+\to ]0,+\infty[$ be defined by \eqref{lyapunov_function}. Then by Lemma~\ref{lemma1-upper-bounds}, for some finite set $E\subset\Z^2_+$ and some positive number $\theta < 1$,   \eqref{lyapunov_eq} holds. Without any restriction of generality, we will suppose that $E$ contains the origin $(0,0)$. Denote by $\tau_E$ the first time when the process $(Z(n))$ hits the set $E$~: $$\tau_E = \inf\{n\geq 1:~Z(n)\in E\},$$ and let
\[
g_E(j,k) = \sum_{n=0}^\infty \P_j(Z(n) = k, \; \tau_E \geq n), \quad j\in\Z_+^2{\setminus} E, \; k\in\Z^2_+.
\]
Then by \eqref{lyapunov_eq}, for any $j\in\Z^2_+{\setminus} E$, 
\[
\E_j(f(Z(n)); \; \tau_E \geq n) \leq \theta^n f(j), \quad \forall  n\in\N, 
\]
and consequently,  
\be\label{eq-for-harmonic-functions} 
 \sum_{k\in\Z^2_+} g_E(j,k) \!\!\left(x_1^{k_1}y_1^{k_2} + x_2^{k_1}y_2^{k_2}\right) = 
\sum_{k\in\Z^2_+} g_E(j,k) f(k) = \sum_{n=1}^\infty \E_j\bigl(f(Z(n)); \; \tau_E\geq n\bigr) \leq f(j)(1 - \theta)^{-1}. 
\ee 
Using the identity 
\begin{align*}
g(j,k)  = g_E(j,k) +  \sum_{\ell\in E{\setminus} \{0\}}\sum_{m\in\Z_+^2{\setminus} E} g(j,\ell) \P_\ell(Z(1) = m) g_E(m,k)
\end{align*}
from the above relation it follows that for any $j\in\Z^2_+{\setminus} E$, 
\begin{align*}
\sum_{k=(k_1,k_2)\in\Z^2_+} g(j,k) \bigl(x_1^{k_1}y_1^{k_2} + x_2^{k_1}y_2^{k_2}\bigr) &\leq \frac{1}{1 - \theta}\Bigl( f(j)+  \sum_{\ell\in E{\setminus} \{0\}}\sum_{m\in\Z_+^2{\setminus} E} g(j,\ell) \P_\ell(Z(1) = m) f(m)\Bigr) \\
&\leq \frac{1}{1 - \theta} \Bigl( f(j)+  \sum_{\ell\in E{\setminus} \{0\}} g(j,\ell)  \E_\ell(f(Z(1))) \Bigr).
\end{align*} 
Since the set $E$ is finite and under the hypotheses (A1) and (A3), $\E_\ell(f(Z(1)))  < + \infty$ for any $\ell \in E$, this proves that  the series \eqref{eq1_upper_bounds} converges for any $j\in\Z^2_+{\setminus} E$. To prove that this series converges for $j\in E$, it is sufficient now to notice that for $j\in E$, 
\[
g(j,k) =  \sum_{\ell\in E{\setminus} \{0\}} g(j,\ell) \sum_{m\in\Z_+^2{\setminus} E} \P_\ell(Z(1) = m) g_E (m, k), 
\]
and consequently, 
\begin{align*}
&\sum_{k=(k_1,k_2)\in\Z^2_+}g(j,k) \left(x_1^{k_1}y_1^{k_2} + x_2^{k_1}y_2^{k_2}\right)  \\ 
&\quad \leq \sum_{\ell\in E{\setminus} \{0\}} g(j,\ell) \sum_{m\in\Z_+^2{\setminus} E} \P_\ell(Z(1) = m)  \!\!\!\sum_{k=(k_1,k_2)\in\Z^2_+} g_E(m,k) \left(x_1^{k_1}y_1^{k_2} + x_2^{k_1}y_2^{k_2}\right) < + \infty. 
\end{align*}  
\end{proof}

\begin{lemma}\label{lemma4_upper_bounds} Suppose that the conditions (A1){-}(A3) are satisfied and let $\tilde{x}\in]x^*,x^{**}[$ and $\tilde{y}\in]y^*,y^{**}[$ be such that 
\be\label{eq1_lemma4_upper_bounds}
\tilde{x} > X_1(\tilde{y}) \quad \text{and} \quad \tilde{y} > Y_1(\tilde{x}).
\ee
Then  for any $y_1 > Y_1(\tilde{x})$ closed enough to $Y_1(\tilde{x})$ and any $x_2 > X_1(\tilde{y})$ closed enough to $X_1(\tilde{y})$, the points  $(x_1,y_1) = (\tilde{x}, y_1)$ and $(x_2,y_2) = (x_2, \tilde{y})$ satisfy the conditions of Lemma~\ref{lemma1-upper-bounds}. 
\end{lemma} 
\begin{proof}  Indeed,  by Lemma~\ref{preliminary-lemma2}, for $\tilde{x}\in]x^*,x^{**}[$ and $\tilde{y}\in]y^*,y^{**}[$, each of the  line segments $[Y_1(\tilde{x}), \tilde{Y}_2(\tilde{x})] = \{y > 0:~(\tilde{x},y)\in D\cap D_1\}$ and $[X_1(\tilde{y}), \tilde{X}_2(\tilde{y})] = \{x > 0:~(x,\tilde{y})\in D\cap D_2\}$ has a non-zero length, and for any $y_1$ and $x_2$ such that $Y_1(\tilde{x}) < y_1 <  \tilde{Y}_2(\tilde{x})$ and $X_1(\tilde{y}) < x_2 <  \tilde{X}_2(\tilde{y})$, one has 
\[
(\tilde{x}, y_1)\in\inter{D}\cap\inter{D_1}, \quad \text{and} \quad (x_2, \tilde{y}) \in\inter{D}\cap\inter{D_2}.
\]
Using \eqref{eq1_lemma4_upper_bounds}, this proves that for any $y_1$ and $x_2$ such that $Y_1(\tilde{x}) < y_1 < \min\{\tilde{Y}_2(\tilde{x}), \tilde{y}\}$ and $X_1(\tilde{y}) < x_2 <  \min\{\tilde{x}, \tilde{X}_2(\tilde{y})\}$, the points $(x_1,y_1) = (\tilde{x}, y_1)$ and $(x_2,y_2) = (x_2, \tilde{y})$ satisfy the conditions of Lemma~\ref{lemma1-upper-bounds}. 
\end{proof}

Now we are ready to get 
\begin{lemma}\label{lemma6-upper-bounds}    Suppose that the conditions (A1){-}(A4) are satisfied and let 

\be\label{def-Theta0}
\Theta_0 =  \begin{cases} \bigl([0,x_d[\times[0,Y_1(x_d)[ \bigr) \cup \bigl( [0, X_1(y_d)[\times[0,y_d[ \bigr)   &\text{if either (B0), or (B1) or (B2) holds}\\ 
  [0,x_d[\times [0, Y_2(x_d)[ &\text{if  either (B3) or (B4) holds}\\ 
[0,X_1(y_d)[\times [0, y_d[ &\text{if  either (B5) or (B6) holds}.\\
\end{cases} 
\ee
Then for any $(x,y)\in\Theta_0$, there are two points $(x_1,y_1)$ and $(x_2,y_2)$ satisfying the conditions of Lemma~\ref{lemma1-upper-bounds}    and relations  \eqref{eq1a_upper_bounds}. 
\end{lemma} 
\begin{proof} Suppose first that one of the cases  (B0), (B1) or (B2) occurs and let $(x,y)\in [0, x_d[\times[0,Y_1(x_d)[$. The definitions of the cases (B0), (B1), (B2) and of the points $x_d$ and $y_d$ (see Proposition~\ref{cases} and relations \eqref{eq2-def-xd}, \eqref{eq2-def-yd}),  give
\[
\max\{X_1(y_d), x\}  < x_d = x^{**} \quad \text{and} \quad y < Y_1(x_d) < y_d = y^{**}, 
\]
and by \eqref{eq-critical-points}, 
\[
x^*_P \leq x^* < x^{**}= x_d \quad \text{and} \quad y^*_P \leq y^* < y^{**} = y_d. 
\]
Since  the functions $X_1:[y^*_P,y^{**}_P]\to [x^*_P, x^{**}_P]$ and $Y_1:[x^*_P, x^{**}_P]\to [y^*_P,y^{**}_P]$ are continuous,  it follows that for any $\tilde{x}\in]\max\{x^*, x\}, x_d[$ and $\tilde{y}\in]y^*,y_d[$ closed enough respectively to $x_d$ and $y_d$ one has 
\be\label{eq0a-lemma6-upper-bounds} 
0 \leq x < \tilde{x}, \quad \quad 0\leq y < Y_1(\tilde{x}), 
\ee
\be\label{eq0b-lemma6-upper-bounds} 
X_1(\tilde{y})  < \tilde{x} \quad \text{and} \quad  Y_1(\tilde{x}) < \tilde{y}. 
\ee
By Lemma~\ref{lemma4_upper_bounds}, from \eqref{eq0b-lemma6-upper-bounds} it follows that for  $y_1 > Y_1(\tilde{x})$ and $x_2 > X_1(\tilde{y})$ closed enough respectively to $Y_1(\tilde{x})$ and $X_1(\tilde{y})$, the points $(x_1, y_1) = (\tilde{x}, y_1)$ and $(x_2,y_2) = (x_2 , \tilde{y})$  
satisfy the conditions of Lemma~\ref{lemma1-upper-bounds}, and using \eqref{eq0a-lemma6-upper-bounds} one gets  \eqref{eq1a_upper_bounds}.

 When one of the cases (B0), (B1) or (B2) holds and $(x,y)\in [0, x_d[\times[0,Y_1(x_d)[$, our lemma is therefore proved. For $(x,y)\in [0, X_1(y_d)[\times[0, y_d[$, the proof is quite similar. 

Suppose now that either (B3) or (B4) holds. The definitions of the cases (B3) and (B4) and of the points $x_d$ and $y_d$ (see Proposition~\ref{cases} and relations \eqref{eq2-def-xd}, \eqref{eq2-def-yd}),    and by  \eqref{eq-critical-points}, give
\be\label{eq2-lemma6-upper-bounds}
x^*_P \leq x^* < x^{**} = x_d \leq X_1(y^{**}), 
\ee
\be\label{eq3-lemma6-upper-bounds}
  y^*  \leq 1 < y_d = Y_2(x_d) \leq y^{**}, 
\ee
\be\label{eq1-lemma6-upper-bounds}
 (X_1(y^{**}), y^{**}), \, (x^{**},Y_2(x^{**})) = (x_d,y_d)\in {\mathcal S}_{12}, 
\ee
and 
\be\label{eq4-lemma6-upper-bounds}
Y_1(x_d) < Y_2(x_d).
\ee
From the definition of ${\mathcal S}_{12}$ and relation~ \eqref{eq1-lemma6-upper-bounds}, it follows that 
$
X_1(y^{**}) \leq X_1(y^{**}_P)$ 
and consequently, by  \eqref{eq2-lemma6-upper-bounds},
\be\label{eq5-lemma6-upper-bounds}
x^*_P \leq x^* < x_d \leq X_1(y^{**}) \leq X_1(y^{**}_P).
\ee
Consider now a point $(x,y)\in \Theta_0$. Then by \eqref{def-Theta0} and using \eqref{eq5-lemma6-upper-bounds}, \eqref{eq3-lemma6-upper-bounds} and \eqref{eq4-lemma6-upper-bounds},
\be\label{eq6-lemma6-upper-bounds}
0\leq \max\{x^*, x\} < x_d \leq X_1(y^{**}_P) \quad \text{and} \quad 0\leq \max\{Y_1(x_d), y^*, y \} < y_d = Y_2(x_d) \leq y^{**} 
\ee
Since the function $Y_2$ is strictly increasing on the line segment $[x^*_P, X_1(y^{**}_P)]$ and the functions $Y_1$ and $Y_2$ are continuous on $[x^*_P, x^{**}_P]$, it follows  that for any  $\hat{x}$ and $\tilde{x}$ closed enough to $x_d$ and such that 
\be\label{eq6p-lemma6-upper-bounds}
\max\{x^*,x\} < \hat{x} < \tilde{x} < x^{**} = x_d, 
\ee
one has
\[
\max\{Y_1(\tilde{x}), y^*, y \} < Y_2(\hat{x})< Y_2(\tilde{x}) < Y_2(x_d) = y_d \leq  y^{**}, 
\]
Remark that because of  \eqref{eq5-lemma6-upper-bounds} and \eqref{eq6p-lemma6-upper-bounds}, the points $\hat{x}$ and $\tilde{x}$ belong to the line segment  $[x^*_P, X_1(y^{**}_P)]$. 
Since  the functions $X_1:[ Y_1(x^*_P), y^{**}_P] \to [x^*_P, X_1(y^{**}_P)]$ and $Y_2:  [x^*_P, X_1(y^{**}_P)]\to [ Y_1(x^*_P), y^{**}_P]$ are inverse to each other,  letting   $\tilde{y} = Y_2(\hat{x})$ we get therefore $
X_1(\tilde{y}) = \hat{x}$ and using  the above relations we obtain 
\be\label{eq7-lemma6-upper-bounds}
y^* < \tilde{y} = Y_2(\hat{x}) < y^{**},  \quad Y_1(\tilde{x}) < Y_2(\hat{x}) = \tilde{y}, \quad   X_1(\tilde{y}) = \hat{x} < \tilde{x},
\ee
\be\label{eq8-lemma6-upper-bounds}
x < \hat{x} = X_1(\tilde{y}) \quad \text{and } \quad y < Y_2(\hat{x}) = \tilde{y}. 
\ee
By  Lemma~\ref{lemma4_upper_bounds}, from \eqref{eq6p-lemma6-upper-bounds} and \eqref{eq7-lemma6-upper-bounds} it follows that  for any $y_1 > Y_1(\tilde{x})$ closed enough to $Y_1(\tilde{x})$ and $x_2 > X_1(\tilde{y})$ closed enough to  $X_1(\tilde{y})$, the points 
$(x_1, y_1) = (\tilde{x}, y_1)$ and $(x_2,y_2) = (x_2 , \tilde{y})$ 
satisfy the conditions of Lemma~\ref{lemma1-upper-bounds}, and using moreover \eqref{eq8-lemma6-upper-bounds}, we get that for any $j=(j_1,j_2)\in\Z^2_+$, 
\[
x^{j_1}y^{j_2} \leq (X_1(\tilde{y}))^{j_1} \tilde{y}^{j_2} \leq  x_2^{j_1}y_2^{j_2}  \leq x_1^{j_1}y_1^{j_2} + x_2^{j_1}y_2^{j_2}.
\]
Hence, in the case when either (B3) or (B4) holds, Lemma~\ref{lemma4_upper_bounds}  is also proved.  To prove this lemma the case when either (B5) or (B6) holds, it is sufficient to exchange the roles of $x$ and $y$. 
\end{proof} 

\subsubsection{ Proof of Proposition~\ref{upper-bounds} }\label{subsection_upper_bounds}  
This proposition is a  consequence of  Lemma~\ref{lemma2-upper-bounds} and Lemma~\ref{lemma4_upper_bounds}:~ by Lemma~\ref{lemma4_upper_bounds}, for any $(x,y)\in\Theta_0$, there are two points $(x_1,y_1)$ and $(x_2,y_2)$ for which the conditions of Lemma~\ref{lemma1-upper-bounds}    are satisfied and relations  \eqref{eq1a_upper_bounds} hold. By Lemma Lemma~\ref{lemma2-upper-bounds} and using \eqref{eq1a_upper_bounds}, this proves that 
\[
\sum_{k=(k_1,k_2)\in\Z^2_+} g(j,k)x^{k_1}y^{k_2} \leq \sum_{k=(k_1,k_2)\in\Z^2_+} g(j,k)\left(x_1^{k_1}y_1^{k_2} + x_2^{k_1}y_2^{k_2}\right) < + \infty. 
\]
Hence, for any $(x,y)\in\Theta_0$, the series \eqref{eq-generating-functions}  converge.  Since   the set $\Theta$ is the logarithmically convex hull  of the set $\Theta_0$, and  the domain of convergence of power series with center $0$ is always logarithmically convex, this proves that the series  \eqref{eq-generating-functions}  converge in $\Omega(\Theta)$.

\subsection{Proof of Proposition~\ref{functional-eq-prop}}\label{functional-equation-section} Consider first the case when $j=(j_1,j_2)\not=(0,0)$. 
By Proposition~\ref{upper-bounds}, the series \eqref{eq-generating-functions} 
converge on the set $\Omega(\Theta)$. Hence,  for any  $(x,y) \in \Omega(\Theta)$ with non-zero $x$ and $y$, by the Fubini theorem and using the Markov property, one gets 
\begin{align}
H_j(x,y) &= x^{j_1}y^{j_2} + \sum_{k=(k_1,k_2)\in\Z^2_+}\sum_{n=1}^\infty \P_j(Z(n) = k, \; \tau_0 > n) x^{k_1}y^{k_2} \nonumber \\ 
&=  x^{j_1}y^{j_2} + \sum_{\ell\in\Z^2_+{\setminus} \{(0,0)\}}  g(j,\ell) \E_\ell\left(x^{Z_1(1)}y^{Z_2(1)},\; \tau_0 > 1\right) \label{eq1_functional_equation} 
\end{align} 
Because of Assumptions (A1){-}(A3), for $(x,y) \in \Omega(\Theta)$ with $x\not= 0$ and $y\not= 0$,  we have
\[
\E_\ell\left(x^{Z_1(1)}y^{Z_2(1)},\; \tau_0 > 1\right) =
\begin{cases} x^{\ell_1}y^{\ell_2} P(x,y)  - \mu(-1,-1) \1_{\{(1,1)\}}(\ell_1,\ell_2) & \text{if $\ell_1> 0$ and $\ell_2 >0$,}\\
 x^{\ell_1}\, \phi_1(x,y) - \mu_1(-1,0) \1_{\{(1,0)\}}(\ell_1,\ell_2) & \text{if $\ell_1> 0$ and $\ell_2 = 0$,}\\
 y^{\ell_2}\, \phi_2(x,y) - \mu_2(0,-1) \1_{\{(0,1)\}}(\ell_1,\ell_2)  & \text{if $\ell_1= 0$ and $\ell_2 > 0$.}
 \end{cases} 
\]
Using these relations in \eqref{eq1_functional_equation}, for $(x,y) \in \Omega(\Theta)$ with non-zero $x$ and $y$, one gets therefore 
\begin{align*}
H_j(x,y)  &=  x^{j_1}y^{j_2} - \mu(-1,-1) g(j,(1,1)) - \mu_1(-1,0) g(j, (1,0)) - \mu_2(0,-1) g(j, (0,1)) \\
&+  P(x,y)\hspace{-6mm}\sum_{\substack{\ell=(\ell_1,\ell_2)\in\Z^2_+:\\ \ell_1 > 0, \; \ell_2 > 0}} g(j,\ell) x^{\ell_1}y^{\ell_2} + \phi_1(x,y)\!\!\sum_{\ell_1 = 1}^\infty g(j, (\ell_1,0)) x^{\ell_1}  + \; \phi_2(x,y)\!\!\sum_{\ell_2= 1}^\infty g(j,(0,\ell_2)) y^{\ell_2} 
\end{align*}
or equivalently, 
\[
H_j(x,y) = L_j(x,y) +  xyP(x,y) h_j(x,y)  + x\phi_1(x,y) h_{1j}(x)+ y\phi_2(x,y)h_{2j}(y) \nonumber\\
\]
with
\begin{align*}
L_j(x,y) &=  x^{j_1}y^{j_2} - \mu(-1,-1) g(j,(1,1)) - \mu_1(-1,0) g(j, (1,0)) - \mu_2(0,-1) g(j, (0,1)) \\
&= x^{j_1}y^{j_2}  - \P_j(\tau_0 < +\infty). 
\end{align*} 
Since clearly $H_j(x,y) = xy h_j(x,y) + x h_{1j}(x) + y h_{2j}(y)$, the last relation proves \eqref{extended-functional-equation} for any $(x,y) \in \Omega(\Theta)$ and 
$j=(j_1,j_2)\in\Z^2_+{{\setminus}}\{(0,0)\}$. To get \eqref{extended-functional-equation}  for $(j_1,j_2)=(0,0)$, it is sufficient now to notice that for any $(x,y)\in\Omega(\Theta)$, 
\[
H_{(0,0)}(x,y) = \sum_{(j_1,j_2)\in\Z_+^2{\setminus}\{(0,0)\}} \P_{(0,0)}\bigl(Z(1) = (j_1,j_2)\bigr) H_j(x,y)
\]
and
\begin{align*}
 \sum_{(j_1,j_2)\in\Z_+^2{\setminus}\{(0,0)\}} \hspace{-0.6cm}\P_{(0,0)}\bigl(Z(1) = (j_1,j_2)\bigr) L_j(x,y) &= \E_{(0,0)} \left(x^{Z_1(1)}y^{Z_2(1)}; \; \tau_0 > 1\right) 
 - \P_{(0,0)}(1 < \tau_0 < +\infty) \\
 &= \phi_0(x,y) - \P_{(0,0)}(\tau_0 < +\infty). 
\end{align*}

\subsection{Proof of  Proposition~\ref{lemma1-proof-theorem1}}\label{section-proof- lemma1-proof-theorem1}
Suppose that one of the cases (B0)-(B2) holds. Then by Proposition~\ref{cases} and  the definition of the points $x_d$, $y_d$, the following relations hold 
\be\label{eq3-lemma1-proof-theorem1}
x_d=x^{**} > X_1(y_d), \quad \quad y_d=y^{**} > Y_1(x_d),
\ee
 and the set $\Theta$ is defined as a logarithmically convex hull of the set $$([0, x_d[\times[0, Y_1(x_d)[)\cup([0, X_1(y_d)[\times [0, y_d[).$$ 
 The points $(x_1,y_2) = (x_d, Y_1(x_d))$ and $(x_2,y_2) = (X_1(y_d), y_d)$ are therefore on the boundary of the set $\Theta$ and also on the boundary of the set $D$. Since the set $\Theta$ is logarithmically convex and the set $D$ is strictly logarithmically convex,  it follows that for any $0 < \theta < 1$, the point $(x_\theta, y_\theta)$, with $x_\theta = x_1^\theta x_2^{1-\theta}$ and $y_\theta = y_1^\theta y_2^{1-\theta}$, belongs to the set $\overline{\Theta}\cap\inter{D}$, and consequently, 
 \[
 Y_1(x_\theta) < y_\theta < Y_2(x_\theta). 
 \]
 By the definition of the set $\Theta$, for any $0 < \theta < 1$ and $y\in[0, y_\theta[$, the point $(x_\theta,y)$ is in $\Theta$, and for any $y\in]Y_1(x_\theta), Y_2(x_\theta)[$ and the point $(x_\theta, y)$ is in $\inter{D}$. Consequently, for any $0 < \theta < 1$ and $y\in[0, Y_2(x_\theta)]$, the point $(x_\theta, y)$ belongs to the set $\overline{\Theta}\cup\inter{D}$, or equivalently, that for any $y \geq 0$, 
 \be\label{eq1-lemma1-proof-theorem1}
 y \geq Y_2(x_\theta) \quad \text{whenever} \quad (x_\theta,y)\not\in \overline{\Theta}\cup\inter{D}. 
 \ee
 Consider now  a point $(\hat{x}, \hat{y})\in[0,+\infty[^2$ with $\hat{x} < x_d$ and $\hat{y} < y_d$ and such that $(\hat{x},\hat{y})\not\in (\Theta\cup \inter{D})$. Then $\hat{x} \geq X_1(y_d)$ (because otherwise $ (\hat{x},\hat{y})\in [0, X_1(y_d)[\times[0, y_d)[ \subset \Theta$), and consequently for some $0 < \theta \leq 1$, $\hat{x} = x_\theta = x_d^\theta (X_1(y_d))^{1-\theta}$. Hence, by \eqref{eq1-lemma1-proof-theorem1}, we get $\hat{y} \geq Y_2(\hat{x})$, and with similar arguments (it is sufficient to exchange the roles of $x$ and $y$) we obtain $\hat{x} \geq X_2(\hat{y})$. By  Lemma~\ref{convexity_lemma3}, these inequalities show that $(\hat{x}, Y_2(\hat{x})), (X_2(\hat{y}), \hat{y}) \in{\mathcal S}_{22}$. Hence, if we suppose that  for some $(x',y')\in D$, the inequalities $\hat{x} \leq x'$ and $\hat{y} \leq y'$ hold, then using the similar arguments as in the proof of Lemma~\ref{convexity_lemma2},  we will obtain that  $\hat{x} = x' = X_2(\hat{y})$ and $\hat{y} = y' = Y_2(\hat{x})$. With the definition of the set ${\Gamma}$, this proves that $(\hat{x},\hat{y})\not\in{\Gamma}$.

\subsection{Proof of Theorem~\ref{theorem1}}\label{section-proof-theorem1-completed} Now, we are ready to complete the proof of Theorem~\ref{theorem1}. 
In the cases (B3){-}(B6), Theorem~\ref{theorem1} follows from Corollary~\ref{upper-bounds-cor1}  and Proposition~\ref{functional-eq-prop}, and in the cases  (B0){-}(B2), the first assertion of Theorem~\ref{theorem1} follows from Corollary~\ref{upper-bounds-cor1} as well.To complete the proof of Theorem~\ref{theorem1} we have therefore to prove its second assertion in the cases (B0){-}(B2).  

We know that, for any $j\in\Z^2_+$, the function $(x,y)\to L_j(x,y)$ is analytic in $\Omega({\Gamma})$, the functions $(x,y) \to Q(x,y)= xy(1-P(x,y)$, $(x,y) \to \psi_1(x,y)=x(\phi_1(x,y)-1)$ and $(x,y)\to \psi_2(x,y)=y(1-\phi_2(x,y))$ can be analytically continued  to  the set $\Omega({\Gamma})$, and  by Corollary~\ref{upper-bounds-cor1},  the functions $h_{1j}$ and $h_{2j}$ are analytic respectively in the discs $B(0, x_d)$ and $B(0,y_d)$. Hence for any $j\in\Z^2_+$, the function 
$(x,y) \to L_j(x,y) + \psi_1(x,y) h_{1j}(x) + \psi_2(x,y) h_{2j}(y)$ is analytic in the set $\{(x,y)\in \Omega({\Gamma}):~|x| < x_d, \; |y| < y_d\}$. Since by Proposition~\ref{functional-eq-prop}, on the set $\Omega(\Theta)$, 
\[
Q(x,y) h_j(x,y) = L_j(x,y) + \psi_1(x,y) h_{1j}(x) + \psi_2(x,y) h_{2j}(y), 
\]
and since clearly $\Theta \subset \{(x,y)\in{\Gamma}:~x < x_d, y < y_d\}$, we conclude therefore that the functions $(x,y) \to  Q(x,y) h_j(x,y)$  can be analytically continued to the set  $$\Omega_d(\Gamma) \steq{def} \{(x,y)\in \Omega({\Gamma}):~|x| < x_d, \; |y| < y_d\}$$ and the function $(x,y)\to h_j(x,y)$ can be continued as a meromorphic function 
\[
h_j(x,y) = \frac{L_j(x,y) + \psi_1(x,y) h_{1j}(x) + \psi_2(x,y) h_{2j}(y)}{Q(x,y)}. 
\]
to the set $\Omega_d(\Gamma)$. Since 
\begin{itemize}
\item[--] by Proposition~\ref{upper-bounds},  for any $j\in\Z^2_+$, the function $h_j$ is analytic in $\Omega(\Theta)$;
\item[--] for any $(x,y)\in\Omega(\inter{D})$, by the definition of the set $D$, we have $|P(x,y)| \leq P(|x|,|y|) < 1$ and, consequently, $Q(x,y)\not=0$; 
\item[--] and by Proposition~\ref{lemma1-proof-theorem1}, the set $\Omega_d(\Gamma)$ is included to the union of the open sets $\Omega(\Theta)$ and $\Omega(\inter{D})$,
\end{itemize} 
we conclude therefore that the function $h_j$ can be analytically continued to the set $\Omega_d(\Gamma)$.

\section{Singularity Analysis of Generating Functions}\label{proof-theorem2} 
This section is devoted to the proof of Theorem~\ref{theorem2}. Recall that the case (B7) will not be considered. See Section~\ref{Part-Sec}. Throughout this section we will assume therefore that 

\medskip
\noindent
{\bf Assumption~(A4)} one of the cases (B0){-}(B6) holds. 
\medskip

\noindent

\subsection{The main ideas and the sketch of the proof}\label{sec-sketch-proof-theorem2}  

The main steps of our proof are the following. First, we prove the light version of the assertion (i){-}(v) of this theorem, i.e. we get, in each of the corresponding cases, relations \eqref{eq1-theorem2}, \eqref{eq-analytical-structure-IIb} \eqref{eq-analytical-structure-IIa}, \eqref{eq-analytical-structure-IV}, \eqref{eq-analytical-structure-Va}, \eqref{eq-analytical-structure-Vb}, \eqref{eq-analytical-structure-III}, and \eqref{limit-R} with  positive constants given by \eqref{eq2b-cases-B0-B4-theorem2}, \eqref{eq4a-case-B2-theorem2}, \eqref{eq2b-case-B2-theorem2}, \eqref{eq-analytical-structure-IVb}, \eqref{eq-analytical-structure-Vab}, \eqref{eq-analytical-structure-Vbb}, \eqref{eq-analytical-structure-IIIb}. This is a subject of Propositions~\ref{cases-B0-B4-theorem2} and~\ref{injection-prop2} and Lemma~\ref{lemma-assertion-vi-th2} below.

With these results we will be able to show that in each of the assertions (i){-}(vi) of our theorem, the corresponding function $\varkappa_i$, $\tilde\varkappa_i$ or $\varkappa_{(\hat{x},\hat{y})}$ is non-negative on $\Z^2_+$. 

Second, in Proposition~\ref{harmonic-functions} below, we prove that in each of the assertions (i){-}(vi), the corresponding function $\varkappa_i$, $\tilde\varkappa_i$ or $\varkappa_{(\hat{x},\hat{y})}$ is harmonic for the random walk  $({Z}_{\tau_0}(n))$ and  positive everywhere in the set $\Z^2_+{\setminus} E_0$. With this statement, the proof of our theorem will be completed. 

The first step of the proof of the last assertion of Theorem~\ref{theorem2} is given by Lemma~\ref{lemma-assertion-vi-th2}. This result is obtained as a traightforward consequence of Theorem~\ref{theorem1}.

The first step of the proof of the assertions (i){-}(v) of Theorem~\ref{theorem2} proof is the most difficult. Its main idea is the following: we extend the function $Y_1$ as an analytic function to some domain of $\C$ large enough and we inject next $y= Y_1(x)$ to the functional equation \eqref{extended-functional-equation} in order to get the identity 
\be\label{injection-identity} 
L_j(x, Y_1(x)) + \psi_1(x,Y_1(x))h_{1j}(x) + \psi_2(x,Y_1(x))h_{2j}(Y_1(x)) = 0.   
\ee
Using this result we will be able to extend beyond the point $x_d$ first the function $x\to \psi_1(x,Y_1(x))h_{1j}(x) = (\phi_1(x, Y_1(x)-1)H_j(x,0)$ and next the function $x\to H_j(x,0)$.

Remark that in a difference with nearest neighbor random walks, we have no explicit form of the function $Y_1$. Moreover, by the implicit function theorem, this function can be extended only to some neighborhood of the interval $]x^*_P, x^{**}_P[$ in $\C$ which is clearly not sufficient for our purpose. The first difficulty of our proof is therefore to extend the function $Y_1$ to some sufficient for our analysis domain of $\C$. We perform this first step of our proof by using the probabilistic representation of the function $Y_1: [x^*_P, x^{**}_P]\to [y^*_P, y^{**}_P]$ obtained in~\cite{Ignatiouk:2010}. This is a subject of Proposition~\ref{functions-Y1-X1} below. In this statement, for some $\eps > 0$, we extend the function $Y_1$ as an analytic function to the set
\[
\inter{U}_{\eps} = C(x^*_P, x^{**}_P+\eps) {\setminus} [x^{**}_P, x^{**}_P+\eps[,
\]
and a continuous function on the set 
\be\label{eq-def-U-delta}
U_{\eps} ~=~ C(x^*_P, x^{**}_P+\eps) {\setminus} ]x^{**}_P, x^{**}_P+\eps[ ~=~ \inter{U}_{\eps}  \cup\{x^{**}_P\}.
\ee
It is proved moreover that the extended function $Y_1$ satisfies  the identity 
\be\label{key-identity}
Q(x, Y_1(x)) = 0,
\ee
on  the set $U_\eps$, and that, on the closed annulus $\overline{C}(x^*_P, x^{**}_P)$, the following relations holds 
\be\label{key-point-inequality}
|Y_1(x)| < Y_1(|x|) \quad \text{whenever} \quad x\not= |x|. 
\ee

Another difficulty is that by Theorem~\ref{theorem1}, we know only that the function $y \to h_{2j}(y)$ is analytic in $B(0,y_d)$, and the function $(x,y)\to h_j(x,y)$ is analytic in $\{(x,y)\in \Omega({\Gamma}):~ |x| < x_d, \; |y| < y_d\}$. Hence, we can inject $y=Y_1(x)$ to the functional equation \eqref{extended-functional-equation} only for those $x\in B(0,x_d)$ for which  $|Y_1(x)| < y_d$. To overcome this difficulty, we use the inequality \eqref{key-point-inequality}.
With this inequality, we are able to show that for some $\delta > 0$, the function $x\to (x,Y_1(x))$ maps the annulus $C(x_d-\delta, x_d) = \{x\in\C: x_d - \delta < |x| < x_d\}$ to the set $\{(x,y)\in\Omega({\Gamma}):~|x| < x_D, \, |y| < y_d\}$ where the functions $(x,y)\to H_j(x,y)$, $(x,y)\to L_j(x,y) + \psi_2(x,y)h_{2j}(y)$ and $(x,y)\to \psi_1(x,y) h_{1j}(x)$ are analytic. This is a subject of Proposition~\ref{prop-mapping-injection}  below. 

In this way, on the annulus $C(x_d-\delta, x_d)$,  by letting in the functional equation \eqref{extended-functional-equation} $y=Y_1(x)$ and using \eqref{key-identity} we obtain  the identity \eqref{injection-identity} with analytic in $C(x_d - \delta, x_d)$ functions $x \to \psi_1(x,Y_1(x))h_{1j}(x)$ and $x \to L_j(x, Y_1(x))$ and $x \to \psi_2(x,Y_1(x))h_{2j}(Y_1(x))$. Moreover,  when one of the cases (B0){-}(B4) holds, using again the inequality \eqref{key-point-inequality}, we will be able to show that  for some $\delta > 0$, the function $(x,y) \to L_j(x, Y_1(x)) + \psi_2(x,Y_1(x))h_{2j}(Y_1(x))$ is analytic  in $C(x_d-\delta, x_d+\delta)\cap \inter{U}_\eps$ and  continuous  on $C(x_d-\delta, x_d+\delta)\cap {U}_\eps$. Using this result together with the identity \eqref{injection-identity},    we will extend the function $x \to - \psi_1(x,Y_1(x))h_{1j}(x)=(1-\phi_1(x,Y_1(x))H_j(x,0)$ as an analytic function to the set $C(x_d-\delta, x_d+\delta)\cap \inter{U}_\eps$ and  a continuous function to the set $C(x_d-\delta, x_d+\delta)\cap {U}_\eps$. This is a subject of Corollary~\ref{cor2-proof-theorem2} below. 
 
 Next, in Proposition~\ref{phi-proof-theorem2}, we investigate the function $x\to 1 - \phi_1(x, Y_1(x)))$,  and  finally,  in Proposition~\ref{cases-B0-B4-theorem2} and Proposition~\ref{case-B2-theorem2} below,  when one of the cases (B0){-}(B4) holds, we extend the function $x\to H_j(x,0)$ beyond the point $x_d$. 

In the case when either (B5) or (B6) holds, we begin our analysis by investigating the function $h_{2j}$. With the same arguments as in the previous cases  (it is sufficient to exchange the roles of $x$ and $y$) we obtain that the function $h_{2j}$ can be  extended as an analytic function to the set $B(0, y_d+\delta_0){\setminus}\{y_d\}$. Next we show that for some $\delta > 0$, the function $x\to h_{2j}(Y_1(x))$ is analytic in the set $C(x_d-\delta, x_d+\delta)\cap U_\eps {\setminus}\{x_d\}$. And finally, using again the identity \eqref{injection-identity} we extend first the function $x \to \psi_1(x,Y_1(x))h_{1j}(x)$ and next the function $x\to h_{1j}(x)$ beyond the point $x_d$. This is a subject of Proposition~\ref{injection-prop2} below.

\subsection{Analytic continuation and properties of the function $x\to Y_1(x)$}\label{section-functions-Y1-X1}

We begin our analysis with the following result. 
\begin{prop}\label{functions-Y1-X1}{\bf Analytic continuation of the function $Y_1$}. Suppose the  conditions (A1)  are satisfied and let $\mu(j) = 0$ for all $j\in\Z^2$ with $j_2 < -1$ (remark that we do not need the whole condition (A2) to be satisfied). Then 
\begin{enumerate}[label=\roman*)]
\item the function $Y_1$ is strictly convex on the line segment $[x^*_P, x^{**}_P]$ and for some  $\eps > 0$, it can be extended to the set $U_\eps$  as an analytic function in the set $\inter{U}_\eps$  and a continuous function on the set  $U_\eps$ satisfying  there the identity \eqref{key-identity}.

\item  relation~\eqref{key-point-inequality} holds on the set $\overline{C}(x^*_P, x^{**}_P)$;

\item  for any $\hat{x}\in ]X_1(y^*_P), x^{**}_P[$, the function $x\to (Y_1(x) - Y_1(\hat{x})$ does not vanish in the set $C(\hat{x} - \delta, \hat{x} + \delta){\setminus}\{\hat{x}\}$ for some $\delta > 0$, and has at the point $\hat{x}$ a simple zero with 
\be\label{eq1-injection-prop-2-lemma1}
 \frac{d}{dx} Y_1(\hat{x}) = - \left.\partial_x P(x,y)/\partial_y P(x,y)\right|_{(x,y) = (\hat{x}, Y_1(\hat{x}))}> 0;
\ee

\item  the function $x\to Y_1(x) - Y_1(x^{**}_P)$ does not vanish in the set $C(x^{**}_P - \delta, x^{**}_P + \delta){\setminus} [x^{**}_P, x^{**}_P+\delta[$ for some $\delta > 0$, and as $x\to x^{**}_P$,
\be\label{eq1-functions-Y1}
Y_1(x) - Y_1(x^{**}_P) \sim - c \sqrt{x^{**}_P-x} \quad \; \text{and } \; \quad \frac{d}{dx} Y_1(x)  \sim  \frac{c}{2\sqrt{x^{**}_P-x}} 
\ee
with 
\[
c = \left.\sqrt{\partial_x P(x,y)/\partial^2_{yy}P(x,y)}\, \right|_{(x,y)= (x^{**}_P, Y_1(x^{**}_P))} > 0.
\]
\end{enumerate}
\end{prop} 
\subsubsection{The main ideas and the sketch of the proof of Proposition~\ref{functions-Y1-X1}} 
To prove this result, we first get a probabilistic representation of the function $Y_1$.  As a consequence, we obtain that the function $Y_1$ is strictly convex on the line segment $[x^*_P, x^{**}_P]$ and can be continued as an analytic function to the open annulus $C(x^*_P, x^{**}_P)$ and as a continuous function to the closed annulus $\overline{C}(x^*_P, x^{**}_P)$ satisfying there the inequality \eqref{key-point-inequality} and the identity \eqref{key-identity}.   This is a subject of the results of Subsection~\ref{subsection1_functions_Y1_X1}. 

Next, in Subsection~\ref{subsection2_functions_Y1_X1}, we extend the functions $Y_1$ and $Y_2$ to a neighborhood of $x^{**}_P$ as two branches of a two-valued analytic function having a branching point $x^{**}_P$ and we get \eqref{eq1-functions-Y1}. This result is obtained by using the implicit function theorem and the Morse lemma. 

Finally, in Subsection~\ref{subsection3_functions_Y1_X1}, the proof of Proposition~\ref{functions-Y1-X1} is completed: by using the implicit function theorem and relations  \eqref{key-point-inequality} and \eqref{key-identity} on the circle $\{x\in\C: |x| = x^{**}_P\}$, we extend the function $Y_1$ and the identity \eqref{key-identity} to the whole set $U_\eps = C(x^*_P, x^{**}_P+\eps) {\setminus}]x^{**}_P, x^{**}_P+\eps[$ for some $\eps > 0$,  and  we prove  the two last assertions of Proposition~\ref{functions-Y1-X1}.

\subsubsection{Probabilistic representation of the functions  $Y_1$ its consequences}\label{subsection1_functions_Y1_X1}
Consider  the homogeneous random walk $(S(n)=(S_1(n), S_2(n))$ on $\Z^2$ with transition probabilities 
\[
\P_j(S(1) = k) = \mu(k-j), \quad \forall k,j\in\Z^2,
\]
and the first time $\tau_1$ when the random walk $(S(n))$ hits the set $\Z\times\{0\}$ ~:
\[
\tau_1 = \inf\{n \geq 1~:~ S(n) \in \Z\times\{0\}\},
\]
By Lemma~2.2 of~\cite{Ignatiouk:2010} 
\begin{lemma}\label{probabilistic_representation_lemma1} Under the hypotheses of Proposition~\ref{functions-Y1-X1}, for any $j=(j_1,j_2)\in\Z\times\N^*$ and $x\in[x^*_P, x^{**}_P]$, 
\[
\E_j\left(x^{S_1(\tau_1)}; \; \tau_1 < +\infty\right) = x^{j_1}Y_1(x)^{j_2}.
\]
\end{lemma} 
As a straightforward consequence of this statement one gets the following probabilistic representation of the function $Y_1$ : 
\begin{cor}\label{probabilistic_representation_cor1} Under the hypotheses (A1) and (A2), for any $x\in[x^*_P, x^{**}_P]$, 
\be\label{probabilistic_representation_Y1} 
Y_1(x) = \E_{(0,1)}\left(x^{S_1(\tau_1)}; \; \tau_1 < +\infty\right) 
\ee
\end{cor} 

With this result we get 
\begin{lemma}\label{lemma1_Y1_continued} Under the hypotheses of Proposition~\ref{functions-Y1-X1}, the function $Y_1$ is strictly convex on $[x^*_P, x^{**}_P]$ and can be extended to the set $\overline{C}(x^*_P, x^{**}_P)$ as a function 
\be\label{eq_Y1_continued}
Y_1(x) = \sum_{k_1\in\Z} \P_{(0,1)}(S(\tau_1) = (k_1, 0), \; \tau_1 < +\infty) x^{k_1} 
\ee
which is continuous on $\overline{C}(x^*_P, x^{**}_P)$, analytic in ${C}(x^*_P, x^{**}_P)$ and, on the set $\overline{C}(x^*_P, x^{**}_P)$,  satisfies   the inequality \eqref{key-point-inequality} and the identity \eqref{key-identity}. 
\end{lemma} 
\begin{proof} Indeed, by Corollary~\ref{probabilistic_representation_cor1}, the series 
\be\label{series_Y1}
\sum_{k_1\in\Z} \P_{0,1)}(S(\tau_1) = (k_1, 0), \; \tau_1 < +\infty) x^{k_1} = \E_{(0,1)}\left(x^{S_1(\tau_1)}; \; \tau_1 < +\infty\right) 
\ee
converges and the identity \eqref{eq_Y1_continued} holds for any real $x\in[x^*_P, x^{**}_P]$. Since the coefficients of the series \eqref{series_Y1} are real and non-negative, it follows that this series converges also on $\overline{C}(x^*_P, x^{**}_P)$. Hence, by letting \eqref{eq_Y1_continued} for $x\in \overline{C}(x^*_P, x^{**}_P)$,  the function $Y_1$ can be extended to the set $\overline{C}(x^*_P, x^{**}_P)$ as a function which is continuous on $\overline{C}(x^*_P, x^{**}_P)$ and analytic in ${C}(x^*_P, x^{**}_P)$.  

To get \eqref{key-point-inequality} and to show that the function $Y_1$ is strictly convex on $[x^*_P, x^{**}_P]$, we show that all coefficients of the  series \eqref{series_Y1} are strictly positive. For this we recall that  under our hypotheses, the random walk $(S(n))$ is irreducible on $\Z^2$ and hence, for any $k_1\in\Z$, there are $n\in\N$ and a sequence of points $\ell_0,\ldots, \ell_{n}\in\Z^2$ with $\ell_0 = (0,1)$ and $\sum_{i=0}^n \ell_{i}= (k_1,0)$ such that 
\[
\mu(\ell_i) > 0 \quad \text{for any $i\in\{1,\ldots,n\}$}. 
\]
Moreover, without any restriction of generality, one can assume that for some $j\in\{1,\ldots, n\}$, the second coordinate of each of the points $\ell_1,\ldots,\ell_j$ is either zero or positive, and the second coordinate of each of the points $\ell_{j+1},\ldots, \ell_n$ strictly negative. Then for any $j\in\{0,\ldots,n-1\}$, the second coordinate of the point $\ell_0+\cdots +\ell_n$ 
is strictly positive and consequently, 
\[
 \P_{0,1}(S(\tau_1) = (k_1, 0), \; \tau_1 < +\infty) \geq \prod_{i=1}^n \mu(\ell_i) > 0.
\]
All coefficients of f the  series \eqref{series_Y1} are therefore strictly positive. 

Since  for any $k_1\in\Z$, the function $x\to \P_{(2,0)}(S(\tau_1) = (k_1, 0), \; \tau_1 < +\infty) x^{k_1}$ is convex on $[x^*_P, x^{**}_P]$ and  the function $x\to \P_{(2,0)}(S(\tau_1) = (2, 0), \; \tau_1 < +\infty) x^2$ is strictly convex on $[x^*_P, x^{**}_P]$, this proves that the function $Y_1$ is strictly convex on $[x^*_P, x^{**}_P]$. And using moreover  Proposition P7.5 of~\cite{Spitzer} we get \eqref{key-point-inequality}. 

Remark finally that by relation~\eqref{key-point-inequality}, the function $x\to (x, Y_1(x))$ maps the closed annulus $\overline{C}(x^*_P, x^{**}_P)$ to the set 
\[
\overline{\Gamma}_1 = \{(x,y)\in\C^2:~x\in \overline{C}(x^*_P, x^{**}_P), \; |y| \leq Y_1(|x|)\}. 
\]
Since under the hypotheses of Proposition~\ref{functions-Y1-X1}, the function 
\[
Q(x,y) =  xy - \sum_{k=(k_1,k_2)\in\N^2} \mu(k_1-1,k_2-1) x^{k_1}y^{k_2}
\]
is analytic in a neighborhood of the set $\overline{\Gamma}_1$, and the function $Y_1$ is already extended as an analytic function to the open annulus $C(x^*_P, x^{**}_P)$ and as a continuous function to the closed annulus $\overline{C}(x^*_P, x^{**}_P)$,  it follows that the function $x\to Q(x, Y_1(x))$ is analytic in $C(x^*_P, x^{**}_P)$ and  continuous on $\overline{C}(x^*_P, x^{**}_P)$. Since moreover the identity \eqref{key-identity} holds for any real $x\in[x^*_P, x^{**}_P]$, by the uniqueness of the analytic continuation to the set $C(x^*_P, x^{**}_P)$ and by continuity of the function $x\to Q(x, Y_1(x))$ on $\overline{C}(x^*_P, x^{**}_P)$, this proves that the identity \eqref{key-identity} holds also on the whole set $\overline{C}(x^*_P, x^{**}_P)$. 
\end{proof}

\subsubsection{Analytic continuation of the functions $Y_1$ and $Y_2$ to a neighborhood of the branching point $x^{**}_P$}\label{subsection2_functions_Y1_X1}
Now, we extend the functions $Y_1$ and $Y_2$ to a neighborhood of $x^{**}_P$ as two branches of a two-valued analytic function having a branching point $x^{**}_P$ and we get \eqref{eq1-functions-Y1}. 

\begin{lemma}\label{lemma3_Y1_extended} Under the hypotheses of Proposition~\ref{functions-Y1-X1}, for some  $\eps > 0$ small enough, the functions $Y_1$ and $Y_2$ can be continued to the disk $B(x^{**}_P, \eps)$ as two branches of two-valued analytic in $B(x^{**}_P, \eps){\setminus}\{x^{**}_P\}$ function such that for any $x\in B(x^{**}_P, \eps){\setminus} [x^{**}_P, x^{**}_P+\eps]$, 
\be\label{eq0-lemma3-Y1-extended} 
P(x, Y_1(x)) = P(x,Y_2(x)) = 1, 
\ee
and there is an analytic in $B(0,\sqrt{\eps})$ function  function $F_Y$ such that 
\be\label{eq1_lemma3_Y1_extended} 
Y_1(x) = F_Y(- \sqrt{x^{**}_P-x}),  \quad \quad \quad Y_2(x) = F_Y(\sqrt{x^{**}_P-x}) 
\ee
and 
\be\label{eq2a_lemma3_Y1_extended} 
 \left. \frac{d}{du} F_Y(u)\right|_{u=0} = \left.\sqrt{\frac{\partial_x P(x,y)}{\partial^2_{yy}P(x,y)} }\, \right|_{(x,y)= (x^{**}_P, Y_1(x^{**}_P))} > 0. 
\ee
\end{lemma} 
\begin{proof} By the definition of the point $x^{**}_P$ and the functions $Y_1$, $Y_2$, we have
\[
Y_1(x^{**}_P)) =   Y_2(x^{**}_P),\quad \quad 
 P(x^{**}_P, Y_1(x^{**}_P))  = 1,
\]
and by Lemma~\ref{preliminary-lemma1}, 
\be\label{eq2_lemma3_Y1_extended} 
\left.\partial_y  P(x,y)\right|_{(x,y) = (x^{**}_P, Y_1(x^{**}_P))} = 0, \quad  \text{and} \quad \left.\partial_x  P(x,y)\right|_{(x,y) =(x^{**}_P, Y_1(x^{**}_P))} > 0.
\ee
Hence, by the implicit function theorem, there are  a neighborhood ${\mathcal V}$ of the point $x^{**}_P$  and an analytic in a neighborhood $U$ of $Y_1(x^{**}_P)$ function $y \to \psi(y)$ 
such that for any $(x,y)\in U\times V$, 
\[
P(x,y) = 1 \quad \Leftrightarrow \quad x = \psi(y), 
\]
\[
\psi(Y_1(x^{**}_P)) = x^{**}_P, \quad \quad \quad \left.\frac{d}{dy} \psi(y) \right|_{y= Y_1(x^{**}_P)} = 0,
\]
and 
\[
 \left.\frac{d^2}{dy^2} \psi(y) \right|_{y= Y_1(x^{**}_P)}  = - \left.\frac{\partial^2_{yy} P(x,y)}{\partial_x P(x,y)} \right|_{(x,y)= (x^{**}_P, Y_1(x^{**}_P))} < 0,
\]
where the last relation follows from the second relation of \eqref{eq2_lemma3_Y1_extended}  because  under the hypotheses (A1), the real valued function $y\to P(x^{**}_P,y)$  is strictly convex in a neighborhood of $Y_1(x^{**}_P)$. Using the Morse lemma we conclude therefore that for some neighborhood $\tilde{U}\subset U$ of $Y_1(x^{**}_P)$, there is a $\C$-diffeomorphism $\omega$ from  $\tilde{U}$ to a neighborhood $\tilde{V}$ of $0$, with  
\be\label{eq3_lemma3_Y1_extended} 
\omega(Y_1(x^{**}_P)) = 0 \quad \text{and} \quad \left.\frac{d}{dy}\omega(y)\right|_{y= Y_1(x^{**}_P)} = \left.\sqrt{\frac{\frac{\partial^2}{\partial y^2}P(x,y)}{\partial_x P(x,y)} }\right|_{(x,y)= (x^{**}_P, Y_1(x^{**}_P))} 
\ee
such that $\tilde{V} + x^{**}_P \subset V$ and for any $(x,y)\in (\tilde{V} + x^{**}_P) \times \tilde{U}$, 
\[
P(x,y) = 1 \quad \Leftrightarrow \quad  x = x^{**}_P -   \omega^2(y).
\]
Without any restriction of generality, one can assume that $\omega(\tilde{U}) = \tilde{V} = B(0,\sqrt{\eps})$ with $\eps > 0$ small enough. Then for $x\in B(x^{**}_P,\eps){\setminus}[x^{**}_P, x^{**}_P+\eps]$, and $y\in \tilde{U}$, from the above relation it follows that 
\begin{align*}
P(x,y) = 1 &\quad \Leftrightarrow \quad  \bigl(\omega(y) = \sqrt{x^{**}_P - x} \quad \text{or} \quad \omega(y) = - \sqrt{x^{**}_P - x} \bigr)\\
&\quad \Leftrightarrow \quad  \bigl(y = \omega^{-1}(\sqrt{x^{**}_P - x}) \quad \text{or} \quad y = \omega^{-1}(- \sqrt{x^{**}_P - x} )\bigr)
\end{align*} 
with the square root function $\sqrt{\;}$ analytic in $\C{\setminus} ]-\infty, 0]$. Since the inverse to $\omega$ function $\omega^{-1}$ is analytic in $B(0,\sqrt{\eps})$, and since for real $x\in[x^*_P, x^{**}_P]$ and $y > 0$, 
\[
P(x,y) = 1 \quad \Leftrightarrow \quad  \bigl( y=Y_1(x)\quad \text{or} \quad y=Y_2(x) \bigr)
\]
we conclude therefore that  the functions $Y_1$ and $Y_2$ can be extended to the set $B(x^{**}_P,\eps){\setminus}[x^{**}_P, x^{**}_P+\eps]$ as analytic functions such that either 
\[
Y_1(x) = \omega^{-1}(\sqrt{x^{**}_P - x}) \; \text{ and } \; Y_2(x) = \omega^{-1}(-\sqrt{x^{**}_P - x}),      \quad \forall x\in B(x^{**}_P,\eps){\setminus}[x^{**}_P, x^{**}_P+\eps],
\]
or 
\[
Y_1(x) = \omega^{-1}(-\sqrt{x^{**}_P - x}) \; \text{and} \; 
Y_2(x) = \omega^{-1}(\sqrt{x^{**}_P - x}), \quad \quad \forall x\in B(x^{**}_P,\eps){\setminus}[x^{**}_P, x^{**}_P+\eps].
\]
These relations show that the function $\omega^{-1}$ is real valued on the interval $]-\sqrt{\eps},\sqrt{\eps}[$, and hence,  by the second relation of \eqref{eq3_lemma3_Y1_extended}, it is strictly increasing in a neighborhood of $0$. Since the function $Y_1$ is increasing on the interval $[X_1(y^*), x^{**}_P]$ and the function $Y_2$ is decreasing on the interval $[X_1(y^*), x^{**}_P]$,  it follows that \eqref{eq0-lemma3-Y1-extended}  and \eqref{eq1_lemma3_Y1_extended} hold with $F_Y = \omega^{-1}$.  Relation \eqref{eq2a_lemma3_Y1_extended} follows from \eqref{eq3_lemma3_Y1_extended}. 
\end{proof}

\subsubsection{Analytic continuation of the function $Y_1$ to the set $C(x^{*}_P, x^{**}_P+\eps){\setminus}[x^{**}_P, x^{**}_P+\eps]$}\label{subsection3_functions_Y1_X1}
 Now we extend the function $Y_1$ and the identity \eqref{key-identity}  to the set  $C(x^{*}_P, x^{**}_P+\eps){\setminus}[x^{**}_P, x^{**}_P+\eps]$ for some small $\eps > 0$.  

\begin{lemma}\label{lemma4-Y1} Under the hypotheses of Proposition~\ref{functions-Y1-X1}, for some  $\eps > 0$ small enough, the function $x\to Y_1(x)$  can be analytically continued to the set $
C(x^*_P, x^{**}_P+\eps) {\setminus} [x^{**}_P, x^{**}_P+\eps[$, and satisfies there the identity \eqref{key-identity}. 
\end{lemma} 
\begin{proof} By Lemma~\ref{lemma1_Y1_continued}, the function $Y_1$ is already extended to the closed annulus $\overline{C}(x^*_P, x^{**}_P)$ as an analytic function in ${C}(x^*_P, x^{**}_P)$ and a continuous function on $\overline{C}(x^*_P, x^{**}_P)$ satisfying there the identity \eqref{key-identity}, and that 
 by Lemma~\ref{lemma3_Y1_extended}, for some $\eps >0$, the function $Y_1$  is also already analytically continued to the set $B(x^{**}_P, \eps){\setminus} [x^{**}_P, x^{**}+\eps]$. Hence, to prove this lemma, it is sufficient to show that the function $Y_1$ can be analytically continued to a neighborhood of the set $\{x\in\C:~|x|=x^{**}_P, \; |x|\not= x^{**}_P\}$. In order to get this result, we use the implicit function theorem. 
 
By Lemma~\ref{lemma1_Y1_continued}, $Q(\tilde{x}, Y_1(\tilde{x})) = 0$ for any  point $\tilde{x}\in\C$ with $|x| = x^{**}_P$ and remark that under the hypotheses of Proposition~\ref{functions-Y1-X1},  the function $Q$ is  analytic in a neighborhood of $(\tilde{x}, Y_1(\tilde{x}))$. Hence, by the implicit function theorem,  it is sufficient  to show that for any  such a point $\tilde{x}$, 
\be\label{eq2_lemma2_Y1_extended} 
\partial_y  Q(\tilde{x},Y_1(\tilde{x})) \not= 0 \quad \text{whenever} \quad \tilde{x}\not= x^{**}_P. 
\ee
To get this relation, we let us remark that  according to the definition \eqref{Qdef} of the function $Q$ and using  \eqref{key-point-inequality}, for any  point $\tilde{x}\in\C$ with $|\tilde{x}| = x^{**}_P$ and such that $\tilde{x} \not= x^{**}_P$, one has 
 \begin{multline}\label{eq4_Y1_lemma_continued}
\left| \tilde{x} - \partial_y  Q(\tilde{x},Y_1(\tilde{x}))\right|  \leq \sum_{k_1\in\Z}\sum_{k_2=0}^\infty (k_2+1) \mu(k_1,k_2) (x^{**}_P)^{k_1+1} |Y_1(\tilde{x})|^{k_2} \\
\\< \sum_{k_1\in\Z}\sum_{k_2=0}^\infty (k_2+1) \mu(k_1,k_2)(x^{**}_P)^{k_1+1} Y_1(x^{**}_P)^{k_2} = x^{**}_P - \partial_y  Q(x^{**}_P, Y_1(x^{**}_P))  
 \end{multline}
where 
\[
\partial_y  Q(x^{**}_P, Y_1(x^{**}_P)) = x^{**}_P - x^{**}_P P(x^{**}_P, Y_1(x^{**}_P)) - x^{**}_P Y_1(x^{**}_P) \partial_y Px^{**}_P, Y_1(x^{**}_P)) = 0
\]
because $P(x^{**}_P, Y_1(x^{**}_P)) = 1$ and  by Lemma~\ref{preliminary-lemma1}, $
 \partial_y P(x^{**}_P, Y_1(x^{**}_P)) = 0$. 
Hence, for any $\tilde{x}\in\C$ with $|\tilde{x}| = x^{**}_P$  and such that $\tilde{x}\not= x^{**}_P$, 
\[
\left| \partial_y  Q(\tilde{x},Y_1(\tilde{x})) \right| \geq  |\tilde{x}| - \left| \tilde{x} - \partial_y  Q(\tilde{x},Y_1(\tilde{x}))\right|  = x^{**}_P - \left| \tilde{x} - \partial_y  Q(\tilde{x},Y_1(\tilde{x}))\right|   > 0
\]
and consequently, \eqref{eq2_lemma2_Y1_extended} holds. 
\end{proof}

\subsubsection{Proof of Proposition~\ref{functions-Y1-X1}} 
The  first two assertions of Proposition~\ref{functions-Y1-X1} are proved by Lemma~\ref{lemma1_Y1_continued}, Lemma~\ref{lemma3_Y1_extended} and Lemma~\ref{lemma4-Y1}. 

To prove the third assertion of Proposition~\ref{functions-Y1-X1}, recall that by Lemma~\ref{lemma1_Y1_continued}, the function $Y_1$ is analytic in the annulus $C(x^*_P, x^{**}_P)$ and satisfies there the inequality \eqref{key-point-inequality} and remark that any point $\hat{x}\in ]X_1(y^*_P), x^{**}_P[$ belongs to the annulus 
$C(x^*_P, x^{**}_P)$ because $ x^*_P < X_1(y^*_P) < x^{**}_P$. Hence, for any $\hat{x}\in ]X_1(y^*_P), x^{**}_P[$, the function $x\to (Y_1(x) - Y_1(\hat{x})$ is analytic in a neighborhood of the circle $\{x\in\C :~|x| = \hat{x}$, and by \eqref{key-point-inequality}, the point $\hat{x}$ is an only zero of this function in the circle $\{x\in\C :~|x| = \hat{x}$. Since the zeros of analytic functions are always isolated, this proves that  for any $\hat{x}\in ]X_1(y^*_P), x^{**}_P[$, there is $\delta > 0$ for which  the function $x\to (Y_1(x) - Y_1(\hat{x})$ does not vanish in the set $C(\hat{x} - \delta, \hat{x} + \delta){\setminus}\{\hat{x}\}$. 

Moreover, we have
\begin{itemize}
\item[--] by Lemma~\ref{preliminary-lemma1}, the real valued function $Y_1$ is strictly increasing on the interval $]X_1(y^*_P), x^{**}_P[$, 
\item[--] by Lemma~\ref{lemma1_Y1_continued}, it is also strictly convex on $]X_1(y^*_P), x^{**}_P[$, 
\item[--] and by Lemma~\ref{preliminary-lemma1}, $P(\hat{x},Y_1(\hat{x})) = 1$ and $\partial_y P(\hat{x}, Y_1(\hat{x})) < 0$.
\end{itemize}
 Hence, for any $\hat{x}\in ]X_1(y^*_P), x^{**}_P[$, using the implicit function theorem, one gets \eqref{eq1-injection-prop-2-lemma1}. From this relation, it follows moreover that for any $\hat{x}\in ]X_1(y^*_P), x^{**}_P[$, 
the point the point $\hat{x}$ is a simple zero of the function $x\to Y_1(x) - Y_1(\hat{x})$.  The third assertion of Proposition~\ref{functions-Y1-X1} is therefore also proved.

To get the last assertion of Proposition~\ref{functions-Y1-X1}, recall that by Lemma~\ref{lemma3_Y1_extended} and Lemma~\ref{lemma4-Y1} , the function $Y_1$ was already extended to the set $U_\eps = C(x^*_P, x^{**}+\eps){\setminus}]x^{**}_P, x^{**}_P+\eps[$ as an analytic function in $\inter{U}_\eps = C(x^*_P, x^{**}+\eps){\setminus}[x^{**}_P, x^{**}_P+\eps[$ and a continuous function on $U_\eps$ satisfying there the inequality \eqref{key-point-inequality}. The function $x\to Y_1(x) - Y_1(x^{**}_P)$ is therefore analytic in $\inter{U}_\eps$ and continuous on $U_\eps$, and moreover by \eqref{key-point-inequality}, the point $x^{**}_P$ is its only zero in the circle $\{x\in\C: |x| = x^{**}_P\}$. To complete the proof of the fourth assertion of Proposition~\ref{functions-Y1-X1}, it is therefore sufficient to get \eqref{eq1-functions-Y1} and to show that the point $x^{**}_P$ is an isolated zero of the function $x\to Y_1(x) - Y_1(x^{**}_P)$. For this, we remark that  by Lemma~\ref{lemma3_Y1_extended}, as $x\to x^{**}_P$, 
\be\label{eq10-functions-Y1}
Y_1(x) - Y_1(x^{**}_P)  = F_Y(-\sqrt{x^{**}_P-x}) - F_Y(0)  ~\sim~ - c \sqrt{x^{**}_P-x}  
\ee
and 
\[
\frac{d}{dx} Y_1(x)  \sim  \frac{c}{2\sqrt{x^{**}_P-x}}  
\]
with 
\[
c = \left. \frac{d}{du} F_Y(0)\right|_{u=0}  = \left.\sqrt{\partial_x P(x,y)/\partial^2_{yy}P(x,y)}\, \right|_{(x,y)= (x^{**}_P, Y_1(x^{**}_P))} > 0. 
\]
Relations \eqref{eq1-functions-Y1} are therefore proved, and since the constant $c$ is strictly positive, from \eqref{eq10-functions-Y1} it follows that  the point $x^{**}_P$ is an isolated zero of the function $x\to Y_1(x) - Y_1(x^{**}_P)$.

%%%%%%%%%%%%%%%%%%%%%%%%%%%%%%%%%%%%%%%%%%%%%%

\subsection{The properties of the mapping  $x\to (x,Y_1(x))$ and their consequences}  Remark that by \eqref{key-point-inequality}, the function $x\to (x, Y_1(x))$ maps the closed annulus to the set $\{(x,y)\in\C^2:~ \; x\in C(x^*_P, x^{**}_P), \; |y| \leq Y_1(|x|)\}$, and under the hypotheses (A1){-}(A3), for any $j\in\Z^2_+$, the functions $(x,y) \to \phi_1(x,y)$, $(x,y)\to \psi_2(x,y)$ and $(x,y)\to L_j(x,y)$ are analytic in a neighborhood of this set. Hence, as a straightforward consequence of Proposition~\ref{functions-Y1-X1} one gets 

\begin{cor}\label{cor0-functions-Y1-X1} Under the hypotheses (A1){-}(A3), for some $\delta > 0$, the functions $x\to \phi_1(x, Y_1(x))$, $x\to \psi_2(x,Y_1(x))$ and $x\to L_j(x,Y_1(x))$ for any $j\in\Z^2_+$, are analytic in the set $\inter{U}_\delta$ and continuous on the set $U_\delta$. 
\end{cor}

Another consequence of Proposition~\ref{functions-Y1-X1} is the following property of the mapping $x{\to}(x, Y_1(x))$. 
\begin{prop}\label{prop-mapping-injection} Under the hypotheses  (A1){-}(A4),  for some $\delta > 0$, the function $x\to (x,Y_1(x))$ maps the annulus $C(x_d-\delta, x_d)$ to the set $\{ (x,y) \in \Omega({\Gamma}): \, |x| < x_d, \, |y| < y_d\}$. Moreover,  for any neighborhood $V$ of the set $\Omega(\overline{\Gamma})$ in $\C^2$ and $\hat{y} > Y_1(x_d)$, there is  $\delta > 0$, for which the function $x\to (x,Y_1(x))$ maps the set $C(x_d-\delta, x_d +\delta)\cap U_\eps$ to the set $\{ (x,y) \in V: \, |y| < \hat{y}\}$.
\end{prop} 
\begin{proof} Remark  that by \eqref{key-point-inequality}, for any $x\in ]x^*_P, x^{**}_P[$, 
\[
|Y_1(x)| \leq Y_1(|x|) < Y_2(|x|).
\]
Hence, the function $x\to (x, Y_1(x))$ maps the set $C(x^*_P, x^{**}_P)$ to the set
\[
{\Gamma}_1 = \{(x,y)\in\C^2:~ \; x\in C(x^*_P, x^{**}_P), \; |y| < Y_2(|x|)\}. 
\]
For any $x\in ]x^*_P, x^{**}_P[$, $Y_1(x)$ and $Y_2(x)$ are the only real and positive solution of the equation $P(x,y) = 1$, $Y_1(x) < Y_2(x)$ and for any  $Y_1(x) < y < Y_2(x)$, the point $(x,y)$ belongs to the interior of the set $D=\{(x,y)\in[0,+\infty[^2: P(x,y) \leq 1\}$ (see Lemma~\ref{preliminary-lemma1} for more details). Hence for any $(x,y)\in{\Gamma}_1)$, there is a point $(x',y')\in D$ such that $|x| < x'$ and $|x|< y'$ and consequently, the set ${\Gamma}_1$ is included to the set ${\Gamma}$. The mapping $x\to (x, Y_1(x))$ maps therefore the set $C(x^*_P, x^{**}_P)$ to the set  ${\Gamma}$. 

Consider now the case when one of the cases (B0){-}(B4) holds. In this case, (see Proposition~\ref{cases} and relations \eqref{eq2-def-xd}, \eqref{eq2-def-yd}), we have $Y_1(x_d) < y_d$. Since  the function $x\to (x,Y_1(x))$ maps the annulus $C(x^*_P, x^{**}_P)$ to the set ${\Gamma}$ and 
by Proposition~\ref{functions-Y1-X1}, the function $Y_1$ is continuous on $U_\eps$, it follows that for some $\delta > 0$, the function $x\to (x, Y_1(x))$ maps the annulus $C(x_d-\delta, x)$ to the set $\{ (x,y) \in \Omega({\Gamma}): \, |x| < x_d, \, |y| < y_d\}$. 

When one of the cases (B0){-}(B4) holds, the first assertion of Proposition~\ref{prop-mapping-injection} is therefore proved. 

Consider now the case when  (B5) or (B6) holds. In this case (see Proposition~\ref{cases} and relations \eqref{eq2-def-xd}, \eqref{eq2-def-yd}), the point $(x_d,y_d) = (X_2(y^{**}), y^{**})$ belongs to the curve ${\mathcal S}_{21} = \{(x, Y_1(x)):~ x\in [X_1(y^*_P), x^{**}_P]\} = \{(X_2(y),y):~y\in [y^*_P, Y_1(x^{**}_P)]\}$ and $y^*_P \leq y^* < y_d = y^{**} < y^{**}_P$. Since by lemma~\ref{preliminary-lemma1}, the functions $X_2: [y^*_P, Y_1(x^{**}_P)]\to [X_1(y^*_P), x^{**}_P]$ and $Y_1:[X_1(y^*_P), x^{**}_P] \to [y^*_P, Y_1(x^{**}_P)]$ are strictly increasing and inverse to each other,  it follows that in this case, we have  $x^*_P < X_2(y_d) = x_d$.  Hence, for any $x\in C(x^*_P, x_d)$, using \eqref{key-point-inequality} one gets 
\[
|Y_1(x)| \leq Y_1(|x|) < Y_1(x_d).
\]
Since we have already proved that the function $x\to (x,Y_1(x))$ maps the annulus $C(x^*_P, x^{**}_P)$ to the set ${\Gamma}$, the last relations prove that for $\delta = x_d - x^*_P > 0$, this function maps the annulus $C(x_d - \delta, x_d)$ to the set $\{ (x,y) \in \Omega({\Gamma}): \, |x| < x_d, \, |y| < y_d\}$. 

In the case when one of the cases (B5) or (B6) holds, the first assertion of Proposition~\ref{prop-mapping-injection} is therefore also proved. 

The second assertion of Proposition~\ref{prop-mapping-injection} holds because the function $x\to (x,Y_1(x))$ is continuous in the set $U_\eps$ including the closed annulus $\overline{C}(x^*_P, x^{**}_P)$ and maps  $\overline{C}(x^*_P, x^{**}_P)$ to the closure of the set ${\Gamma}$.

\end{proof} 

Remark  that under the hypotheses (A1)(ii), (A2) and (A3) (ii), (iv), the functions $(x,y) \to Q(x,y)= xy(1-P(x,y)$, $(x,y) \to \psi_1(x,y)=x(\phi_1(x,y)-1)$ and $(x,y)\to \psi_2(x,y)=y(1-\phi_2(x,y))$ can be continued as analytic functions to some neighborhood of the set 
$\Omega(\overline{\Gamma})$.  Hence, from Theorem~\ref{theorem1} it follows  

\begin{cor}\label{upper-bounds-cor2} Under the hypotheses (A1){-}(A4), there is a  neighborhood ${\mathcal V}\subset\C^2$ of the set $\Omega(\overline{\Gamma})$ such that  for any $j\in\Z^2$,
\begin{itemize}
\item the function $(x,y) \to \psi_1(x,y) h_{1j}(x)$ is analytic in the set $\{(x,y)\in {\mathcal V} :~ |x| < x_d\}$ ;
\item the function $(x,y)\to \psi_2(x,y) h_{2j}(y)$ is analytic in the set $\{(x,y)\in {\mathcal V} :~ |y| < y_d\}$ 
\item the function $(x,y) \to R_j(x,y) = Q(x,y) h_j(x,y)$  can be extended as an analytic function to the set  $
\{(x,y)\in {\mathcal V} :~|x| < x_d, \; |y| < y_d\}$ by letting 
\be\label{extended-functional-equation1} 
R_j(x,y) = L_j(x,y) + \psi_1(x,y) h_{1j}(x) + \psi_2(x,y) h_{2j}(y). 
\ee
\end{itemize} 
\end{cor} 

When combined together, Proposition~\ref{prop-mapping-injection}, Corollary~\ref{cor0-functions-Y1-X1} and Corollary~\ref{upper-bounds-cor2} imply the following statement. 

\begin{cor}\label{cor2-proof-theorem2} Under the hypotheses (A1){-}(A4) for some $\delta > 0$, and any $j\in\Z^2_+ $, 

i) the functions $x \to - \psi_1(x, Y_1(x)) h_{1j}(x)=(1 - \phi_1(x, Y_1(x)))H_j(x,0)$ and $x\to \psi_2(x, Y_1(x))h_{2j}(Y_1(x))$ are  analytic in the annulus $C(x_d-\delta,x_d)$ and satisfy there the identity 
\be\label{functional-equation-injection} 
(1 - \phi_1(x, Y_1(x)))H_j(x,0) = L_j(x, Y_1(x)) + \psi_2(x, Y_1(x))h_{2j}(Y_1(x));
\ee 

ii) moreover,  if one of the cases (B0){-}(B4) holds with $x_d  < x^{**}_P$ (i.e. if either one of the cases (B0), (B1), (B3) or (B4) holds or (B2) and $x_d < x^{**}_P$ hold),  then for some $\delta > 0$, the function $x\to L_J(x,Y_1(x)) + \psi_2(x, Y_1(x))h_{2j}(Y_1(x))$ is analytic in the annulus $C(x_d-\delta, x_d+\delta)$ and by \eqref{functional-equation-injection}, the function  $x \to \eta_j(x) = (1 - \phi_1(x, Y_1(x)))H_j(x,0)$ can be extended as an analytic function to $C(x_d-\delta, x_d+\delta)$.

iii) if (B2) holds with $x_d=x^{**}_P$, then for some $\delta > 0$, the function $x\to L_j(x, Y_1(x)) + \psi(x, Y_1(x))h_{2j}(Y_1(x))$ is analytic in the set $C(x_d-\delta, x_d+\delta) \cap \inter{U}_\eps$ and continuous on the set  $C(x_d-\delta, x_d+\delta) \cap {U}_\eps$, and by \eqref{functional-equation-injection},  the function $x \to \eta_j(x) = (1 - \phi_1(x, Y_1(x)))H_j(x,0)$ can be extended to the set $C(x_d-\delta, x_d+\delta) \cap {U}_\eps$ as an analytic function in $C(x_d-\delta, x_d+\delta) \cap \inter{U}_\eps$ and a continuous function on  $C(x_d-\delta, x_d+\delta) \cap {U}_\eps$.
\end{cor} 

With this result,  when one of the cases (B0){-}(B4) holds, the function $x \to (1 - \phi_1(x, Y_1(x)))H_j(x,0)$ is therefore already extended beyond the point $x_d$. To extend the function $x\to H_j(x,0)$ beyond the point $x_d$ we need to investigate the function $x\to \phi_1(x, Y_1(x))$. This is a subject of the next section.

%%%%%%%%%%%%%%%%%%%%%%%%%%%%%%%%%%%%%%%
\subsection{Analytic continuation and properties of the function $x\to \phi_1(x,Y_1(x))$}\label{section-analytic-continuation-functions-phi1-Y1}  
By Corollary~\ref{cor0-functions-Y1-X1}, for some $\delta > 0$, the function $x\to \phi_1(x, Y_1(x))$ is already extended to the set $U_\delta$  as an analytic function in the set $\inter{U}_\delta$  and a continuous function on the set  $U_\delta$. We will need moreover the following properties of the function $x\to \phi_1(x, Y_1(x))$.

\begin{prop}\label{phi-proof-theorem2} Under the hypotheses (A1)- (A3), the following assertions hold: 

i) the function $x\to \phi_1(x,Y_1(x))$ is strictly convex on the line segment $[x^*_P, x^{**}_P]$.  

ii) for any $x\in C(x^*, x^{**})$, 
\be\label{eq10-phi-proof-theorem2}
|\phi_1(x, Y_1(x))| \leq \phi(|x|, Y_1(|x|)) < 1; 
\ee

iii) if $x^{**} < x^{**}_P$, then for some $\hat\delta > 0$, the point $x^{**}$ is an only and simple zero of the function $x\to 1 - \phi_1(x, Y_1(x))$ in the annulus $C(x^*, x^{**} +\hat\delta)$ and  
\be\label{eq20-phi-proof-theorem2}
\left.\frac{d}{dx} \phi_1(x, Y_1(x))\right|_{x=x^{**}} > 0;
\ee

iv) if $x^{**} = x^{**}_P$, then $\phi_1(x^{**}_P, Y_1(x^{**}_P)) \leq 1$;

v)  if $x^{**} = x^{**}_P$ and $\phi_1(x^{**}_P, Y_1(x^{**}_P)) < 1$, then for some $\hat\delta > 0$, the function $x\to 1-\phi_1(x, Y_1(x))$ has no zeros  in $U_{\hat\delta} = C(x^*, x^{**}_P+\hat\delta){\setminus}]x^{**}_P, x^{**}_P+\hat\delta[$;

vi) if $x^{**} = x^{**}_P$ and $\phi_1(x^{**}_P, Y_1(x^{**}_P)) = 1$, then for some $\hat\delta > 0$, the point $x^{**}_P$ is an only zero of the function $x\to 1-\phi_1(x, Y_1(x))$ in $U_{\hat\delta}$ and  as $x\to x^{**}_P$, 
\be\label{eq30-phi-proof-theorem2}
1-\phi_1(x, Y_1(x)) \sim c \sqrt{x^{**}_P-x}  \quad \; \text{with} \; \quad c = \left. \partial_y\phi_1(x,y)\sqrt{\frac{\partial_x P(x,y)}{\partial^2_{yy}P(x,y)} }\, \right|_{(x,y)= (x^{**}_P, Y_1(x^{**}_P))} > 0.
\ee
\end{prop} 

\subsubsection{Outline of the proof of Proposition~\ref{phi-proof-theorem2}} 

To prove this proposition, we first  get a probabilistic representation of the function $x\to \phi_1(x, Y_1(x))$ similar to those of the function $x\to Y_1(x)$. This is a subject of Corollary~\ref{probabilistic_representation_cor2}. In Section~\ref{sec-proof-i-ii-prop-phi-theorem2}, as a straightforward consequence of this result, we get the first and the second assertion of Proposition~\ref{phi-proof-theorem2}. The proof of the third assertion of this statement is given in Section~\ref{sec-proof-iii-prop-phi-theorem2} and the proofs of the last tree assertions are completed in Section~\ref{sec-proof-iv-vi-prop-phi-theorem2}. 

\subsubsection{Probabilistic representation of the function $x\to\phi_1(x,Y_1(x))$} The following  probabilistic representation of the function $x\to \phi_1(x,Y_1(x))$ is another useful for our purpose consequence of Lemma~\ref{probabilistic_representation_lemma1}~:

\begin{cor}\label{probabilistic_representation_cor2} Let  $(\hat{S}(n))$ be a random walk on the half-plane $\Z\times\N$ with transition probabilities 
\be\label{trans_probabilities2}
\P_k(\hat{S}(1) = j) ~=~\begin{cases} \mu(j-k) &\text{for all $j, k=(k_1,k_2)\in\Z\times\N$ with  $k_2 > 0$}\\
\mu_1(j-k) &\text{for  all $j, k=(k_1,k_2)\in\Z\times\N$ with   $k_2 = 0$}\\
\end{cases} 
\ee
and  let $T_1$ be the first time  when the random walk $\bigl(\hat{S}(n)=(\hat{S}_1(n),\hat{S}_2(n))\bigr)$ hits the boundary $\Z\times\{0\}$~:
\[
T_1 = \inf\{n > 0:~ \hat{S}(n) \in \Z\times\{0\}\}.
\] 
Then under the hypotheses (A1) - (A3), for any $x\in \overline{C}(x^*_P, x^{**}_P)$
\be\label{probabilistic_representation_phi_1} 
\phi_1(x,Y_1(x)) = \E_{(0,0)}\left(x^{\hat{S}_1(T_1)}; \; T_1 < +\infty\right).
\ee
\end{cor} 
\begin{proof} To get \eqref{probabilistic_representation_phi_1} from Lemma~\ref{probabilistic_representation_lemma1} it is sufficient to notice that by the Markov property, and according to the definition of the random walks $(S(n))$ and $(\hat{S}(n))$, 
\begin{align*}
\E_{(0,0)}\left(x^{\hat{S}_1(\tau_1)}; \; \tau_1 < +\infty\right)  &= \sum_{k\in\Z\times\{0\}} \mu_1(k) x^{k_1} + \hspace{-3mm}\sum_{k=(k_1,k_2)\in\Z\times\N:~k_2\not= 0} \mu_1(k) \E_k(x^{S_1(T_1)}, \; T_1 < +\infty) \\
&= \sum_{k\in\Z\times\{0\}} \mu_1(k) x^{k_1} + \sum_{k=(k_1,k_2)\in\Z\times\N:~k_2\not= 0} \mu_1(k) x^{j_1}Y_1(x)^{j_2}\\
& = \phi_1(x, Y_1(x)). 
\end{align*} 
\end{proof} 

\subsubsection{Proof of the first and the second assertions of Proposition~\ref{phi-proof-theorem2}}\label{sec-proof-i-ii-prop-phi-theorem2}
With Corollary~\ref{probabilistic_representation_cor2}, by using the same arguments as in the proof of Lemma~\ref{lemma1_Y1_continued}  we get 

\begin{cor}\label{lemma_phi_extended} Under the hypotheses (A1) - (A3), the function $x\to\phi_1(x, Y_1(x))$ is strictly convex on $[x^*_P, x^{**}_P]$ and  satisfies  on the set $\overline{C}(x^*_P, x^{**}_P)$ the following relation 
\be\label{eq1_functions_phi_1}
|\phi_1(x, Y_1(x))| < \phi_1(|x|, Y_1(|x|)), \quad \forall x\not= |x|.
\ee
\end{cor} 

Remark that this statement proves the first assertion of Proposition~\ref{phi-proof-theorem2}. Moreover, since by Corollary~\ref{preliminary_cor2}, 
\be\label{eq1_cor_phi_extended}
\phi_1(x, Y_1(x)) < 1 \quad \text{for any} \quad x\in]x^*, x^{**}[,
\ee
using \eqref{eq1_functions_phi_1} one gets \eqref{eq10-phi-proof-theorem2}, and consequently, the  second assertion of Proposition~\ref{phi-proof-theorem2} also holds. 

\subsubsection{Proof of the third assertion of Proposition~\ref{phi-proof-theorem2}}\label{sec-proof-iii-prop-phi-theorem2}
Suppose now that $x^{**} < x^{**}_P$. In this case, by Corollary~\ref{preliminary_cor2}, 
\be\label{eq11_cor_phi_extended}
\phi_1(x^{**}, Y_1(x^{**})) = 1 \quad 
\text{and } \quad \phi_1(x, Y_1(x)) > 1 \quad \text{for all} \quad x\in]x^{**}, x^{**}_P]. 
\ee
Using this relation together with \eqref{eq10-phi-proof-theorem2} one get that the point $x^{**}$ is an only zero of the function $x\to\phi_1(x, Y_1(x))-1$ in the set $\{x\in\C~:~x^* < |x| \leq x^{**}\}$, and since the function $x\to\phi_1(x, Y_1(x))-1$ is analytic in the neighborhood $\inter{U}_\delta$ of this set,  it  follows that for some $\hat\delta > 0$, the point $x^{**}$ is an only zero of this function in the annulus $C(x^*, x^{**} +\hat\delta)$. Moreover, since the function $x\to\phi_1(x, Y_1(x))$ is strictly convex on the line segment $[x^*_P, x^{**}_P]$ and by \eqref{eq-critical-points}, in the case when $x^{**}< x^{**}_P$, the point $x^{**}$ belongs to the interior of this line segment, from \eqref{eq1_cor_phi_extended} and \eqref{eq11_cor_phi_extended} it follows that 
\[
\left.\frac{d}{dx} \phi_1(x, Y_1(x))\right|_{x= x^{**}} > 0,
\]
and consequently, the point $x^{**}$ is in this case a simple zero of the function $x\to\phi_1(x, Y_1(x))-1$. The third assertion of Proposition~\ref{phi-proof-theorem2} is therefore also proved. 

\subsubsection{Proof of the last three assertions of Proposition~\ref{phi-proof-theorem2}}\label{sec-proof-iv-vi-prop-phi-theorem2}
Suppose now that $x^{**} = x^{**}_P$. Then by  Corollary~\ref{preliminary_cor2}, $\phi_1(x^{**}_P, Y_1(x^{**}_P)) \leq 1$, and consequently, the fourth assertion of Proposition~\ref{phi-proof-theorem2} holds. 

Moreover, if $x^{**}=x^{**}_P$ and $\phi_1(x^{**}_P, Y_1(x^{**}_P)) < 1$, using \eqref{eq1_functions_phi_1} we get 
\[
 |\phi_1(x, Y_1(x))] \leq \phi_1(|x|, Y_1(|x|)) < 1, \quad \text{for any $x\in\C$ with  $x^* < |x| \leq x^{**}_P$}, 
 \]
 and consequently, the function $x\to\phi_1(x, Y_1(x))-1$ has no zeros in the set $\{x\in\C~:~x^* < |x| \leq x^{**}_P\}$. Since  the function $x\to \phi_1(x, Y_1(x))$ is    continuous  on the set  $U_\delta$,  it  follows that, for some $\hat\delta > 0$, it has no zeros in the set $U_{\hat\delta}$. 
 
 An finally, if $x^{**}=x^{**}_P$ and $\phi_1(x^{**}_P, Y_1(x^{**}_P)) = 1$, using again \eqref{eq1_functions_phi_1}  one gets that the point $x^{**}_P$ is an only zero of the function $x\to\phi_1(x, Y_1(x))-1$ in the set $\{x\in\C~:~x^* < |x| \leq x^{**}_P\}$. Moreover, since under our hypotheses, the function $(x,y)\to \phi_1(x,y)$ is analytic in a neighborhood of the point $(x^{**}_P, Y_1(x^{**}_P))$ and since because of Assumption (A3)(iv), 
\[
\partial_y \phi_1(x^{**}_P, Y_1(x^{**}_P)) > 0,  
\]
using \eqref{eq1-functions-Y1} one gets that as $x\to x^{**}_P$, 
 \begin{align*}
 \phi_1(x, Y_1(x)) - 1 &~\sim~ (x-x^{**}_P) \partial_x \phi_1(x^{**}_P, Y_1(x^{**}_P))  - c \sqrt{x^{**}_P-x} \partial_y \phi_1(x^{**}_P, Y_1(x^{**}_P))  \\ 
 &~\sim~  - c \sqrt{x^{**}_P-x} \partial_y \phi_1(x^{**}_P, Y_1(x^{**}_P)) 
 \end{align*}
 with 
\[
c = \left.\sqrt{\partial_x P(x,y)/\partial^2_{yy}P(x,y)}\, \right|_{(x,y)= (x^{**}_P, Y_1(x^{**}_P))} > 0.
\]
This proves that the point $x^{**}_P$ is in this case an isolated zero of the function $x\to\phi_1(x, Y_1(x))-1$ in $U_\delta$. Since this function is continuous in a neighborhood $U_\delta$ of the set $\{x\in\C~:~x^* < |x| \leq x^{**}_P\}$, and has no zeros in $\{x\in\C~:~x^* < |x| \leq x^{**}_P\}$ aside of the point $x^{**}_P$, we conclude therefore that for some $\hat\delta > 0$, the point $x^{**}_P$ is an only zero of the function $x\to\phi_1(x, Y_1(x))-1$ in $U_{\hat\delta}$. The last assertion of 
Proposition~\ref{phi-proof-theorem2} is therefore also proved. 
\bigskip

\subsection{Analytic continuation of the function $x{\to}H_j(x,0)$, cases (B0){-}(B4)} 
In this section, we extend the function $x\to H_j(x,0)$ beyond the point $x_d$ when one of the cases (B0){-}(B4) holds. The cases (B0), (B1), (B3), (B4), and  the case (B2) with $x^{**} < x^{**}_P$ are considered in Proposition~\ref{cases-B0-B4-theorem2}. The case (B2) with $x_d = x^{**}_P$ is considered in Proposition~\ref{case-B2-theorem2}  below.

\begin{prop}\label{cases-B0-B4-theorem2} Suppose that the conditions (A1){-}(A3) are satisfied and let either one of the cases (B0), (B1), (B3), (B4) holds or (B2) and $x^{**} < x^{**}_P$ hold.  Then for some $\delta> 0$ and any $j\in\Z^2_+ $, the function $x\to H_j(x,0)$ can be extended  as an analytic function to the set $B(0, x_d+\delta){\setminus}\{x_d\}$, and \eqref{eq1-theorem2} holds with $\varkappa_1(j)$ defined by \eqref{eq-def-kappa-1} and $a_1>0$ defined by \eqref{eq2b-cases-B0-B4-theorem2}.
 \end{prop} 
 \begin{proof}  Under the hypotheses of this proposition,  by Corollary~\ref{cor2-proof-theorem2}, for some $\delta_1 > 0$, the function $x\to L_j(x, Y_1(x)) + \psi_2(x, Y_1(x)) h_{2j}(Y_1(x))$ is analytic in the annulus 
 $C(x_d-\delta_1, x_d+\delta_1)$, and using the identities  
 \begin{align}
 (1-\phi_1(x, Y_1(x))) H_j(x,0) &=  - \psi_1(x, Y_1(x)) h_{1j}(x) \label{eq1-proof-cases-B0-B4-theorem2}\\ &= L_j(x, Y_1(x)) + \psi_2(x, Y_1(x)) h_{2j}(Y_1(x)) \label{eq2-proof-cases-B0-B4-theorem2}
  \end{align} 
 the function $x\to (1-\phi_1(x, Y_1(x))) H_j(x,0)$ 
 can be analytically continued to the open annulus $C(x_d-\delta_1, x_d+\delta_1)$. Since by Proposition~\ref{phi-proof-theorem2}, for some $\delta_2 > 0$, the function $x\to (1 - \phi_1(x, Y_1(x)))^{-1}$ is analytic in the set $C(x_d -\delta_2, x_d+\delta_2){\setminus}\{x_d\}$ and has at the point $x_d$ a simple pole with 
 \[
 \lim_{x\to x_d} \frac{x_d - x}{1-\phi_1(x, Y_1(x))} = \left.\left(\frac{d}{dx} \phi_1(x, Y_1(x))\right)^{-1}\right|_{x=x_d} > 0,
 \]
 we conclude therefore that the function $x\to H_j(x,0)$ can be extended as an analytic function to the set $B(0, x_d+\delta){\setminus}\{x_d\}$, and that \eqref{eq1-theorem2}  holds with $\varkappa_1(j)$ given by \eqref{eq-def-kappa-1} , and $a_1 > 0$ given by \eqref{eq2b-cases-B0-B4-theorem2}. 
 \end{proof}

 \begin{prop}\label{case-B2-theorem2}  Suppose that the conditions (A1){-}(A3) are satisfied and let (B2) and $x_d=x^{**}_P$ hold. Then the following assertions holds. 
 \begin{itemize} 
 \item[--] If $\phi_1(x^{**}_P, Y_1(x^{**}_P)) = 1$, then for some $\delta > 0$, the function $x\to H_j(x,0)$ can be extended  as an analytic function to the set $B(0, x^{**}_P+\delta){\setminus} [x^{**}, x^{**}_P+\delta[$ satisfying  \eqref{eq-analytical-structure-IIa}  with $\varkappa_1(j)$  defined by \eqref{eq-def-kappa-1} with $x_d=x^{**}_P$ and $a_2>0$ defined by \eqref{eq4a-case-B2-theorem2}.
 \item[--] If $\phi_1(x^{**}_P, Y_1(x^{**}_P)) < 1$, then for some $\delta > 0$, the function $x\to H_j(x,0)$ can be extended  to the set $B(0, x^{**}_P+\delta){\setminus} ]x^{**}_P, x^{**}_P+\delta[$ as an analytic function in $B(0, x^{**}_P+\delta){\setminus} [x^{**}_P, x^{**}_P+\delta[$ and a continuous function on $B(0, x^{**}_P+\delta){\setminus} ]x^{**}_P, x^{**}_P+\delta[$, satisfying \eqref{eq-analytical-structure-IIb} with $\tilde\varkappa_1(j)$ defined by \eqref{eq-def-tilde-kappa-1} 
 and $\tilde{a}_2 > 0$ defined by \eqref{eq2b-case-B2-theorem2}. 
\end{itemize}  
 \end{prop} 
 \begin{proof} Indeed, in this case, by Corollary~\ref{cor2-proof-theorem2},  for some $0 < \delta_1 \leq \eps$, the function $x\to L_j(x, Y_1(x)) + \psi(x, Y_1(x))h_{2j}(Y_1(x))$ is analytic in $C(x_d-\delta_1, x_d+\delta_1) \cap \inter{U}_\eps$ and continuous on  $C(x_d-\delta_1, x_d+\delta_1) \cap {U}_\eps$, and using the identities \eqref{eq1-proof-cases-B0-B4-theorem2} end \eqref{eq2-proof-cases-B0-B4-theorem2}, the function $x \to \eta_j(x)  = (1 - \phi_1(x, Y_1(x)))H_j(x,0)$ can be extended to the set $C(x_d-\delta_1, x_d+\delta_1) \cap {U}_\eps$ as an analytic function in $C(x_d-\delta_1, x_d+\delta_1) \cap \inter{U}_\eps$ and a continuous function on  $C(x_d-\delta_1, x_d+\delta_1) \cap {U}_\eps$. Since by Proposition~\ref{phi-proof-theorem2}, for some $0 < \delta \leq \delta_1$, the function $x \to 1/(1 - \phi_1(x, Y_1(x)))$ is analytic in the set $C(x_d-\delta, x_d+\delta) \cap \inter{U}_\eps$, we get therefore that  the function $x\to H_j(x,0)$ can be extended as an analytic function  to the set $C(x_d-\delta, x_d+\delta) \cap \inter{U}_\eps$ by letting 
 \begin{align} 
 H_j(x,0) &= \eta_j(x) \bigl(1 - \phi_1(x, Y_1(x))\bigr)^{-1} \nonumber\\ &= \left(L_j(x, Y_1(x)) + \psi_2(x, Y_1(x)) h_{2j}(Y_1(x))\right) \bigl(1 - \phi_1(x, Y_1(x))\bigr)^{-1}.\label{eq5-case-B2-theorem2} 
  \end{align} 
Recall moreover that by Proposition~\ref{phi-proof-theorem2},  when  $\phi_1(x^{**}_P, Y_1(x^{**}_P)) < 1$, the function $x \to 1/(1 - \phi_1(x, Y_1(x)))$   is  continuous on $C(x_d-\delta, x_d+\delta) \cap {U}_\eps$, and when $\phi_1(x^{**}_P, Y_1(x^{**}_P)) = 1$, 
 \be\label{eq6-case-B2-theorem2} 
 \lim_{x\to x^{**}_P} \frac{\sqrt{x^{**}_P - x}}{1-\phi_1(x, Y_1(x))} = \left. \left(\partial_y\phi_1(x,y) \sqrt{\partial_x P(x,y)/\partial^2_{yy}P(x,y)} \right)^{-1}  \right|_{(x,y) = (x^{**}_P, Y_1(x^{**}_P))} > 0.
 \ee
 Hence,   in the case when $\phi_1(x^{**}_P, Y_1(x^{**}_P)) < 1$, the function $x\to H_j(x,0)$ can be extended as a continuous function to the set $C(x_d-\delta, x_d+\delta) \cap {U}_\eps$, and  in the case when $\phi_1(x^{**}_P, Y_1(x^{**}_P)) = 1$,  using \eqref{eq5-case-B2-theorem2}  \eqref{eq6-case-B2-theorem2}, one gets \eqref{eq-analytical-structure-IIa} with $\varkappa_1(j)$ defined by \eqref{eq-def-kappa-1}, and $a_2> 0$ given by \eqref{eq4a-case-B2-theorem2}.  
 
Remark finally that  when  $\phi_1(x^{**}_P, Y_1(x^{**}_P)) < 1$,  using \eqref{eq5-case-B2-theorem2} , for $x\in B(0, x^{**}_P)$ closed enough to $x^{**}_P$, one gets 
\begin{align*}
\frac{d}{dx} &H_j(x,0) = \frac{d}{dx} \left(r_j(x, Y_1(x))(1 - \phi_1(x, Y_1(x)))^{-1}\right) \\ 
&=\frac{d}{dx} \left((L_j(x, Y_1(x)) + (\phi_2(x, Y_1(x))-1) H_j(0, Y_1(x)))(1 - \phi_1(x, Y_1(x)))^{-1}\right) \\ 
&= \partial_x \left.\left((L_j(x, y) + (\phi_2(x, y)-1) H_j(0, y))(1 - \phi_1(x, y))^{-1}\right)\right|_{y=Y_1(x)} \\ &\; + \partial_y \left. \left(( L_j(x, Y_1(x)) + (\phi_2(x, Y_1(x))-1) H_j(0, Y_1(x)))(1 - \phi_1(x, Y_1(x)))^{-1}\right)\right|_{y=Y_1(x)} \times\frac{d}{dx} Y_1(x) 
\end{align*} 
where by \eqref{eq1-functions-Y1}, as $x\to x^{**}_P$, 
\[
\frac{d}{dx} Y_1(x)  \sim  \frac{c}{2\sqrt{x^{**}_P-x}}  \quad \; \text{with} \; \quad 
c =  \left.\frac{1}{2} \sqrt{\partial_x P(x,y)/\partial^2_{yy}P(x,y) } \right|_{(x,y) = (x^{**}_P, Y_1(x^{**}_P))} > 0.
\]
Since in the case when $\phi_1(x^{**}_P, Y_1(x^{**}_P)) < 1$, the function $$(x,y) \to  (L_j(x, y) + (\phi_2(x, y)-1) H_j(0, y))(1 - \phi_1(x, y))^{-1}$$ is analytic in a neighborhood of the point  $(x^{**}_P, Y_1(x^{**}_P))$,  it follows \eqref{eq-analytical-structure-IIb} with  $\tilde{a}_2 > 0$ given by \eqref{eq2b-case-B2-theorem2}.
\end{proof}

\subsection{Analytic continuation of the function $x{\to}H_j(x,0)$, cases (B5) and (B6)}

Now we are ready to extend the function $x\to H_j(x,0)$ beyond the point $x_d$ when either (B5) or (B6) holds.  This is a subject of Proposition~\ref{injection-prop2} below. 

\begin{prop}\label{injection-prop2} Under the hypotheses (A1){-}(A3), there is $\delta > 0$ such that for any $j\in \Z^2_+ $,  the following assertions hold 

 i)   If (B6) holds, then  the function $x\to H_j(x,0)$ can be extended as an analytic function to the set $B(0, x_d + \delta){\setminus}\{x_d\}$ and  \eqref{eq-analytical-structure-III} holds with  $a_5>0$  given by \eqref{eq-analytical-structure-IIIb}. 
  
 ii)  If (B5) holds and $x_d =  x^{**} < x^{**}_P$, then  the function $x\to H_j(x,0)$ can be extended as an analytic function to the set $B(0, x_d + \delta){\setminus}\{x_d\}$ and \eqref{eq-analytical-structure-IV} holds with  $a_3>0$ given  by \eqref{eq-analytical-structure-IVb}. 

iii) If (B5), $x_d = x^{**} = x^{**}_P$ and $\phi_1(x^{**}_P, Y_1(x^{**}_P)) = 1$ hold,  the function $x\to H_j(x,0)$ can be extended as an analytic function to the set $B(0, x_d + \delta){\setminus} [x_d, x_d +\delta[$ and  \eqref{eq-analytical-structure-Va} holds with $a_4>0$ given by \eqref{eq-analytical-structure-Vab}. 

iv)   If (B5), $x_d = x^{**} = x^{**}_P$ and $\phi_1(x^{**}_P, Y_1(x^{**}_P)) < 1$ hold,  the function $x\to H_j(x,0)$ can be extended as an analytic function to the set $B(0, x_d + \tilde\delta){\setminus} [x_d, x_d +\tilde\delta[$ and  \eqref{eq-analytical-structure-Vb} holds with $\tilde{a}_4$ given by \eqref{eq-analytical-structure-Vbb}.
\end{prop} 
Recall that under our hypotheses, all functions $(x,y){\to} L_j(x,y)$, $(x,y){\to} \psi_1(x,y)$ and $(x,y){\to}\psi_2(x,y)$ are analytic in some neighborhood ${\mathcal V}$ of the set $\Omega(\overline{\Gamma})$. Throughout this section,  the set ${\mathcal V}$ will be given. 

\subsubsection{Preliminary result} 

We begin the proof of this proposition with the following lemma.

\begin{lemma}\label{lemma-injection-prop2} Under the hypotheses (A1){-}(A3) and if one of the cases (B5) or (B6) holds, there is $\delta > 0$ such that for any $j\in\Z^2_+$, the function $x\to (1-\phi_1(x,Y_1(x))) H_j(0, Y_1(x))$ can be analytically continued to the set $C(x_d-\delta, x_d+\delta)\cap U_\eps \setminus\{x_d\}$.
\end{lemma} 
\begin{proof} Suppose that either (B5) or (B6) holds. Then by Proposition~\ref{cases} and using relations~\eqref{eq2-def-xd}, \eqref{eq2-def-yd}), one gets 
\be\label{eq1-injection2}
 X_1(y_d) < x_d = X_2(y_d) \leq x^{**}, \quad  y^* < y_d = y^{**} = Y_1(x_d) < y^{**}_P \quad \text{and} \quad (x_d,y_d)\in{\mathcal S}_{21}, 
\ee
with $x_d = x^{**}$ in the case $(B5)$, and $x_d < x^{**}$ in the case (B6). Hence, with the same arguments as in the proof of proposition~\ref{cases-B0-B4-theorem2}, (it is sufficient to exchange the roles of $x$ and $y$), one gets that there exists $\delta_0 > 0$ such that for any $j\in\Z^2_+ $,   the functions $h_{1j}$ and $y\to y h_{2j}(y) = H_j(0,y)$ and  can be analytically continued to the set $B(0, y_d+\delta_0){\setminus}\{y_d\}$, and 
  \begin{align}
  \lim_{y\to y_d} (y_d - y) yh_{2j}(y)  &= \lim_{y\to y_d} (y_d - y) H_j(0,y) \nonumber \\ 
  &= c_1 \Bigl(L_j(X_1(y_d), y_d) + (\phi_1(X_1(y_d),y_d))-1) H_j(X_1(y_d),0)\Bigr) \label{injection-relation2} 
  \end{align}
with 
 \be\label{injection-relation2p} 
 c_1 = \left.\left(\frac{d}{dy} \phi_2(X_1(y), y)\right)^{-1}\right|_{y=y_d} > 0. 
  \ee
 Since by \eqref{eq1-injection2}, $y_d < y^{**}_P$, without any restriction of generality, we will suppose throughout our proof that 
 \be\label{eq-supposed} 
 y_d + \delta_0 < y^{**}_P. 
 \ee
Moreover, by using  Proposition~\ref{functions-Y1-X1} and the first assertion of Corollary~\ref{cor2-proof-theorem2} we can  inject $y=Y_1(x)$ to the functional equation \eqref{extended-functional-equation}:
 \begin{itemize}
 \item[--]  by Proposition~\ref{functions-Y1-X1}, for some $\eps > 0$, the function $Y_1$ is already analytically continued to $\inter{U}_\eps$ and extended as a continuous function to the set $U_\eps$, 
 \item[--]  by Corollary~\ref{cor2-proof-theorem2},  there is $\delta_1{\in}]0, \eps[$, such that any $j\in\Z^2_+ $, 
 the functions  $x\to L_j(x, Y_1(x)) + \psi_2(x, Y_1(x))h_{2j}(Y_1(x))$ and $x\to -\psi_1(x, Y_1(x)) h_{1j}(x) = (1 - \phi_1(x, Y_1(x)))H_j(x,0)$ are  analytic in the set $C(x_d-\delta_1,x_d)$ and for any $x\in C(x_d-\delta_1, x_d)$ the following relation holds 
 \be\label{eq100000-proof-theorem2} 
 (1 - \phi_1(x, Y_1(x)))H_j(x,0) = L_j(x, Y_1(x)) + \psi_2(x, Y_1(x))h_{2j}(Y_1(x)). 
 \ee
\end{itemize} 
 To complete the proof of Lemma~\ref{lemma-injection-prop2}, the following steps will be performed: 
\begin{enumerate}
 \item[{\em step 1}:] First we will show that for some $\delta_2 \in]0,\delta_1[$, the function $x\to L_j(x, Y_1(x)) + \psi_2(x, Y_1(x))h_{2j}(Y_1(x))$ is analytic in the set
 \be\label{eq-100001-proof-theorem2} 
 \{x\in C(x_d-\delta_2, x_d+\delta_2)\cap U_\eps :~ Y_1(x) \not= Y_1(x_d)\}.
 \ee
 \item[{\em step 2}:] Next we will prove that for some $\delta_3 \in]0,\delta_2[$, the point $x_d$ is an only zero of the function $x\to Y_1(x)-Y_1(x_d)$ in $C(x_d-\delta_3, x_d +\delta_3)\cap U_\eps$, and we will deduce from our previous result that the function $x\to L_j(x, Y_1(x)) + \psi_2(x, Y_1(x))h_{2j}(Y_1(x))$ is analytic in the set
 \be\label{eq-100002-proof-theorem2} 
 C(x_d-\delta_3, x_d+\delta_3)\cap U_\eps \setminus\{x_d\}.
 \ee
With this result and using the identity \eqref{eq100000-proof-theorem2}  we will be able to complete our proof.
 \end{enumerate}

\noindent
{\em Step 1}: Since because of ~\eqref{eq1-injection2}, $Y_1(x_d) = y_d < y_d + \delta_0$,  by Proposition~\ref{prop-mapping-injection} applied with $\hat{y} = y_d +\delta_0$, we obtain that  for some $0 < \delta_2 < \delta_1$,  the function $x\to (x,Y_1(x))$ maps the set $C(x_d-\delta_2, x_d +\delta_2)\cap U_\eps$  to the set $\{(x,y)\in {\mathcal V} : \, |y| < y_d + \delta_0\}$. Since the function $y\to H_j(0,y)$ is already analytically continued to the set $B(0, y_d+ \delta_0)\setminus\{y_d\}$ and the functions $(x,y){\to}\psi_2(x,y)$ and $(x,y)\to L_j(x,y)$ are analytic in ${\mathcal V}$, it follows that the function $x\to L_j(x, Y_1(x)) + \psi_2(x, Y_1(x)) h_{2j}(Y_1(x))$ is  analytic in the set \eqref{eq-100001-proof-theorem2}. 

 \noindent
 {\em Step 2}: by~\eqref{eq1-injection2} and \eqref{eq-critical-points}, we have 
\[
y^*_P \leq y^* < y_d = y^{**} = Y_1(x_d),
\]
and consequently, $(x_d,y_d)\not= (X_1(y^*_P),y^*_P)$. Since by ~\eqref{eq1-injection2}, $(x_d,y_d)\in{\mathcal S}_{21}$, using the definition of the curve ${\mathcal S}_{21}$ (see \eqref{eq2_S_21})  one gets therefore 
\[
X_1(y^*_P) < x_d \leq x^{**}_P,
\]
 and consequently, using the assertions (ii) and (iv) of Proposition~\ref{functions-Y1-X1}, we conclude  that  for some $0 < \delta_3 < \delta_2$, the point $x_d$ is an only zero of the function $x\to Y_1(x)-Y_1(x_d)$ in $C(x_d-\delta_3, x_d +\delta_3)\cap U_\eps$. Since with our previous result, we have already proved that the function $x\to L_j(x, Y_1(x)) + \psi_2(x, Y_1(x))h_{2j}(Y_1(x))$ is analytic in the set \eqref{eq-100001-proof-theorem2}, this proves that the function $x\to L_j(x, Y_1(x)) + \psi_2(x, Y_1(x))h_{2j}(Y_1(x))$ is analytic in the set \eqref{eq-100002-proof-theorem2}.  Finally, since on the annulus $C(x_d - \delta_1, x_d)$, the identity \eqref{eq100000-proof-theorem2}  holds, this proves that the function $x{\to} - \psi_1(x, Y_1(x)) h_{1j}(x)=(1 - \phi_1(x, Y_1(x)))H_j(x,0)$ can be analytically continued to the set $(C(x_d -\delta_3, x_d+\delta_3)\cap U_\eps){\setminus}\{x_d\}$.
\end{proof} 

\subsubsection{Proof of Proposition~\ref{injection-prop2}} To complete the proof of Proposition~\ref{injection-prop2}, we consider separately all possible cases~:
\begin{itemize}
\item[--] when (B6) holds; 
\item[--] when (B5) holds with $x_d = x^{**} < x^{**}_P$;
\item[--] when (B5) holds with  $x_d = x^{**} = x^{**}_P$ and $\phi_1(x^{**}_P, Y_1(x^{**}_P)) < 1$;
\item[--] when (B5) holds with $x_d = x^{**} = x^{**}_P$ and $\phi_1(x^{**}_P, Y_1(x^{**}_P)) = 1$. 
\end{itemize} 
Suppose first that (B6) holds.  Then by Proposition~\ref{cases}  and the definition of $x_d$ (see \eqref{eq2-def-xd}),  the following relations hold 
\be\label{eq2-proof-injection-prop2}
x^*_P\leq x^* < X_1(x^{**}) < x_d = X_2(y^{**}) < x^{**} \leq x^{**}_P. 
\ee
Hence, by Lemma~\ref{lemma-injection-prop2} and according to the definition of the set $U_\eps$ (see \eqref{eq-def-U-delta}), for some $\hat\delta > 0$, the function $x \to  - \psi_1(x, Y_1(x)) h_{1j}(x) = (1-\phi_1(x, Y_1(x))) H_j(x,0)$ is already analytically continued  to the set $C(x_d-\hat\delta, x_d+\hat\delta){\setminus}\{x_d\}$  .  By Proposition~\ref{phi-proof-theorem2}, for any $x\in C(x^*, x^{**})$,
\[
|\phi_1(x, Y_1(x))| \leq \phi_1(|x|, Y_1(|x|)) < 1. 
\]
Since by \eqref{eq2-proof-injection-prop2}, $x^* < x_d < x^{**}$,  it follows that  for $0 < \delta' < \min\{x^{**} - x_d, x_d - x^*\}$, 
the function $x \to  1/(1-\phi_1(x, Y_1(x))$  is analytic in $C(x_d-\delta', x_d + \delta')$ and 
\be\label{eq99}
1 - \phi_1(x_d, Y_1(x_d)) >  0.  
\ee
Since the function $x \to \eta_j(x) = - \psi_1(x, Y_1(x)) h_{1j}(x) = (1-\phi_1(x, Y_1(x))) H_j(x,0)$ is already extended as an analytic function  to the set $C(x_d-\hat\delta, x_d+\hat\delta){\setminus}\{x_d\}$, it follows that for $\delta = \min\{\delta', \hat\delta\}$, the function the function $x\to H_j(x, 0)$ can be analytically continued to $C(x_d-\delta, x_d+\delta){\setminus}\{x_d\}$ by letting 
\be\label{eq100}
H_j(x,0) = (1-\phi_1(x,Y_1(x)))^{-1}\bigl(L_j(x, Y_1(x)) + \psi_2(x, Y_1(x))h_{2j}(Y_1(x))\bigr). 
\ee
Remark finally that by \eqref{eq1-injection2}, $y_d = y^{**} < y^{**}_P$. Hence, using exactly the same arguments as in the proof of Proposition~\ref{phi-proof-theorem2} (it is sufficient to exchange the roles of $x$ and $y$), one gets $\phi_2(X_1(y_d), y_d) = \phi_2(X_1(y^{**}), y^{**})  = 1$. Since by \eqref{eq1-injection2}, $X_1(y_d) < X_2(y_d) = x_d$ and $Y_1(x_d) = y_d$, and since under our hypotheses (see Assumption (A3)(vi)) the real valued function $x\to \phi_2(x, y_d)$ is strictly increasing,  it follows that 
\be\label{eq101}
\phi_2(x_d, Y_1(x_d))  = \phi_2(x_d, y_d)  = \phi_2(X_2(y_d), y_d)  > \phi_2(X_1(y_d), y_d)  = 1.
\ee
Using these relations together with \eqref{eq100}, \eqref{eq99}, \eqref{injection-relation2} and \eqref{eq1-injection-prop-2-lemma1}, we obtain 
\begin{align*} 
\lim_{x\to x_d} (x_d-x)  &H_j(x,0) = \lim_{x\to x_d}  (x_d-x)  (1-\phi_1(x,Y_1(x)))^{-1} \psi_2(x, Y_1(x))h_{2j}(Y_1(x)) \\ &= \lim_{x\to x_d}  (x_d-x)  (1-\phi_1(x,Y_1(x)))^{-1}(\phi_2(x, Y_1(x))- 1) H_j(0, Y_1(x)) \\
&= a_5  \Bigl(L_j(X_1(y_d), y_d) + (\phi_1(X_1(y_d),y_d)-1) H_j(X_1(y_d),0)\Bigr) = a_5\varkappa_2(j) 
\end{align*} 
with 
\[
a_5 =  \left.(\phi_2(x, y)-1) \left((1 - \phi_1(x, y))\frac{d}{dy} \phi_2(X_1(y), y)  \frac{d}{dx} Y_1(x)\right)^{-1}\right|_{(x,y)=(x_d,y_d)}  > 0.
\]
The first assertion of Proposition~\ref{injection-prop2} is therefore proved. 
 
\bigskip
 Suppose now that (B5) holds and let $x_d = x^{**} < x^{**}_P$. Then by Proposition~\ref{cases} and  the definition of $x_d$ and $y_d$ (see \eqref{eq2-def-xd} and \eqref{eq2-def-yd}), one has  
\[
y_d = y^{**} = Y_1(x^{**}) = Y_1(x_d) \quad \text{and} \quad   x^* < X_1(y_d) < x_d = X_2(y_d) = x^{**} < x^{**}_P.
\]
and by Proposition~\ref{phi-proof-theorem2}, for some $\delta > 0$, the function $x \to  1/(1-\phi_1(x, Y_1(x)))$ is meromorphic  in the annulus $C(x^*, x_d+\delta)$, and has there a unique and simple pole at the point $x_d=x^{**}$ with 
\be\label{eq10-injection-prop2} 
\lim_{x\to x_d} \frac{x_d-x}{1-\phi_1(x, Y_1(x))} = \left. \left(\frac{d}{dx} \phi_1(x, Y_1(x))\right)^{-1}\right|_{x=x_d} > 0; 
\ee
Here, the only difference with the previous case is that the function $x \to  1/(1-\phi_1(x, Y_1(x))$ has a simple pole at $x_d= x^{**}$. Using therefore exactly the same arguments as in the previous case we obtain that for some $\delta > 0$, the function $x\to   H_j(x,0)$ can be extended as an analytic function to the set $C(x_d-\delta, x_d+\delta){\setminus}\{x_d\}$. And using finally \eqref{eq100} together with  \eqref{injection-relation2},    \eqref{eq101}, \eqref{eq1-injection-prop-2-lemma1} and  \eqref{eq10-injection-prop2} we get 
\begin{align*} 
\lim_{x\to x_d} &(x_d-x)^2  H_j(x,0) =  a_3 \Bigl(L_j(X_1(y_d), y_d) + (\phi_1(X_1(y_d),y_d)-1) H_j(X_1(y_d),0)\bigr) = a_3 \varkappa_2(j) 
\end{align*} 
with 
 \begin{align*}
 a_3 &= \left. (\phi_2(x, Y_1(x))-1) \left(\frac{d}{dx} \phi_1(x, Y_1(x))\frac{d}{dy} \phi_2(X_1(y), y) \frac{d}{dx} Y_1(x)\right)^{-1}\right|_{x=x_d, y=y_d} \\
 &=  \left.(\phi_2(x, y)-1) \left(\frac{d}{dx} \phi_1(x, Y_1(x))\frac{d}{dy} \phi_2(X_1(y), y) \frac{d}{dx} Y_1(x)\right)^{-1}\right|_{x=x_d, y=y_d}  > 0. 
 \end{align*} 
 The second  assertion of Proposition~\ref{injection-prop2} is therefore also proved. 
 
 \medskip 
 
 Suppose  now that (B5), $x_d = x^{**} = x^{**}_P$ and $\phi_1(x^{**}_P, Y_1(x^{**}_P)) < 1$ hold. Then by Lemma~\ref{lemma-injection-prop2} and according to the definition of the set $U_\eps$ (see \eqref{eq-def-U-delta}), by using the identity \eqref{eq100000-proof-theorem2}, the function $x \to  - \psi_1(x, Y_1(x)) h_{1j}(x)$ was already analytically continued to the set $C(x^{**}_P - \hat\delta, x^{**}_P + \hat\delta){\setminus} [x^{**}_P, x^{**}_P+\hat\delta[$,  and by  Proposition~\ref{phi-proof-theorem2}, for some $0 < \delta < \hat\delta$, the function  $x\to (1-\phi_1(x, Y_1(x))^{-1}$ is analytic in  $C(x^{**}_P-\delta , x^{**}_P+\delta){\setminus}[x^{**}_P, x^{**}_P+\delta[$
and 
\[
H_j(x, 0) = x h_{1j}(x) = x \eta_j(x) (\psi_1(x, Y_1(x)))^{-1}  =  \eta_j(x) (1-\phi_1(x, Y_1(x))^{-1}, \quad \forall x\in C(x^{**}_P - \delta, x^{**}_P).
\]
 Hence, by using the identity \eqref{eq100}, the function $x\to H_j(x,0)$ can be analytically continued to the set $C(x^{**}_P - \delta, x^{**}_P + \delta){\setminus} [x^{**}_P, x^{**}_P[$, and  using finally \eqref{eq100} together with \eqref{injection-relation2}, \eqref{eq1-functions-Y1} and \eqref{eq1-injection2} one gets  
 \begin{align*} 
\lim_{x\to x^{**}_P} &\sqrt{x^{**}_P-x} H_j(x,0) \\
&=  \lim_{x\to x^{**}_P} \sqrt{x^{**}_P-x}   (1-\phi_1(x, Y_1(x))^{-1} \bigl(L_j(x, Y_1(x)) + \psi_2(x, Y_1(x))h_{2j}(Y_1(x))\bigr) \\
&=  \lim_{x\to x^{**}_P} \sqrt{x^{**}_P-x}   (1-\phi_1(x, Y_1(x))^{-1}  \psi_2(x, Y_1(x))h_{2j}(Y_1(x)) \\ 
&=  \lim_{x\to x^{**}_P} \sqrt{x^{**}_P-x}   (1-\phi_1(x, Y_1(x))^{-1}  (\phi_2(x, Y_1(x))-1) H_j(0, Y_1(x)) \\ 
&= \tilde{a}_4 \Bigl(L_j(X_1(y_d), y_d) + (\phi_1(X_1(y_d),y_d)-1) H_j(X_1(y_d),0)\bigr) = \tilde{a}_4 \varkappa_2(j) 
\end{align*} 
with
\[
\tilde{a}_4 =  (\phi_2(x_d, y_d)-1)\sqrt{\partial^2_{yy}P(x_d,y_d)/\partial_x P(x_d,y_d) } \left((1-\phi_1(x_d, y_d)) \frac{d}{dy} \phi_2(X_1(y_d), y_d) \right)^{-1}   > 0,  
\] 

\medskip 
Consider now the case when (B5), $x_d = x^{**} = x^{**}_P$ and $\phi_1(x^{**}_P, Y_1(x^{**}_P)) = 1$ hold.  In this case, with exactly the same arguments as above one gets that for some $\delta > 0$, the function $x\to H_j(x,0)$ can be continued as an analytic function to the set $C(x^{**}_P - \delta, x^{**}_P + \delta){\setminus} [x^{**}_P, x^{**}_P[$ by using the identity \eqref{eq100}. The only difference is here that now, by Proposition~\ref{phi-proof-theorem2},
\[
\frac{1}{1-\phi_1(x, Y_1(x))} \sim \frac{c_1}{\sqrt{x^{**}_P - x}} 
\]
with 
\[
c_1 = \left.\left(\partial_y\phi_1(x, y)) \sqrt{\partial_x P(x, y)/\partial^2_{yy}P(x, y) } \right)^{-1}\right|_{(x,y)=(x^{**}_P, Y_1(x^{**}_P))} > 0
\]
and consequently, since in this case $x_d = x^{**}_P$ and $y_d = Y_1(x^{**}_P)$, using \eqref{eq100} together with \eqref{injection-relation2} and \eqref{eq1-injection2}, one gets 
\[
\lim_{x\to x_d} (x_d-x) H_j(x,0) = a_4 \Bigl(L_j(X_1(y_d), y_d) + (\phi_1(X_1(y_d),y_d)-1) H_j(X_1(y_d),0)\bigr) 
\]
with
\[
a_4 =   \left.(\phi_2(x, y)-1) \partial^2_{yy}P(x, y)\left(\partial_y\phi_1(x, y) \partial_x P(x, y) \frac{d}{dy} \phi_2(X_1(y), y)  \right)^{-1}\right|_{(x,y)=(x_d,y_d)} > 0.
\] 
Proposition~\ref{injection-prop2} is therefore proved. 

\subsection{Analytic continuation of the function $(x,y)\to (1-P(x,y))H_j(x,y)$} 

The following Lemma is the first step in the proof of the last assertion of Theorem~\ref{theorem2}. 

\begin{lemma}\label{lemma-assertion-vi-th2} If the conditions (A1){-}(A3) are satisfied and  (B2) holds, then 
\begin{itemize}
\item the set $\{(x,y)\in{\mathcal S}_{22}: \; x < x_d, \; y < y_d\}$ is non-empty;
\item there exists a neighborhood ${\mathcal V}$ of the set ${\mathcal S}_{22}$ in $\R^2_+$ such that, for any $j{\in}\Z^2_+$,\\ the function $x{\to} (1{-}P(x,y))H_j(x,y)$ can be  analytically  continued to the set $\{(x,y){\in}\Omega({\mathcal V}): |x| {<} x_d,  |y| {<} y_d\}$; 
\item for any $(\hat{x},\hat{y}){\in} \{(x,y){\in}{\mathcal S}_{22}: x{<} x_d, y {<} y_d\}$,  the function $\varkappa_{(\hat{x},\hat{y})}$ is non-negative on $\Z^2_+$ and for any $j\in\Z^2_+{\setminus} E_0$, 
\be\label{limit-R-lemma}
\lim_{\substack{(x,y) \to (\hat{x},\hat{y})\\ (x,y)\in\inter{D}}} (1-P(x,y))(H_j(x,y)- H_j(x,0) - H_j(0,y)) =  \varkappa_{(\hat{x},\hat{y})}
\ee
\end{itemize} 
\end{lemma}
\begin{proof}  Indeed, by Theorem~\ref{theorem1}, for any $j\in\Z^2_+$ and $(\hat{x},\hat{y})\in {\mathcal S}_{22}$ with $\hat{x} < x_d$ and $\hat{y} < y_d$, there is a neighborhood $V(\hat{x},\hat{y})$ of the point $(\hat{x},\hat{y})$ in $\R^2$, such that for any $j\in\Z^2_+$, the function 
\begin{align*}
(x,y) \to L_j(x,y) + (\phi_1(x,y)-1) &H_j(x,0) + (\phi_2(x,y)-1) H_j(0,y) \\ &= L_j(x,y) + \psi_1(x,y)h_{1j}(x) + \psi_2(x,y) h_{2j}(y)
\end{align*} 
 is analytic in  the polycircular set $\{(x,y)\in\C^2:~ (|x|, |y|)\in V(\hat{x}, \hat{y})\}$, and the function $(x,y) \to (1-P(x,y))(H_j(x,y)- H_j(x,0) - H_j(0,y)) = Q(x,y) h_j(x,y)$ can be continued as an analytic function to $\{(x,y)\in\C^2:~ (|x|, |y|)\in V(\hat{x}, \hat{y})\}$ by letting 
\[
(1-P(x,y))(H_j(x,y)- H_j(x,0) - H_j(0,y)) = L_j(x,y) + (\phi_1(x,y)-1) H_j(x,0) + (\phi_2(x,y)-1) H_j(0,y). 
\]
Hence, for any $j\in\Z^2_+$ the quantity 
\[
\varkappa_{(\hat{x},\hat{y})} = L_j(\hat{x},\hat{y}) + (\phi_1(\hat{x},\hat{y}) -1) H_j(\hat{x},0) + (\phi_2(\hat{x},\hat{y})-1) H_j(0,\hat{y}) 
\]
is well defined and as $(x,y)\to (\hat{x}, \hat{y})$ for $(x,y)\in\inter{D}$, \eqref{limit-R-lemma} holds. Moreover, since by Theorem~\ref{theorem1}, the function $(x,y)\to H_j(x,y)$ is analytic in the set $\Omega_d(\Gamma) = \{(x,y)\in\C^2: (|x|,|y|)\in \Gamma, \; |x| < x_d, \; |y| < y_d\}$, and since the set $\Omega_d(\Gamma)$ is a union of the poly-discs centered at the origin in $\C^2$, the power series 
\[
H_j(x,y) - H_j(x,0) - H_j(0,y) = \sum_{k_1=1}^\infty\sum_{k_2=1}^\infty g(j,k) x^{k_1} y^{k_2}
\]
converge on the set $\Omega_d(\Gamma)$, and consequently, for $(x,y)\in \Gamma$ such that $x < x_d$ and $y < y_d$, 
\[
H_j(x,y) - H_j(x,0) - H_j(0,y)  \geq 0. 
\]
Since for $(x,y)\in \inter{D}$, $P(x,y) < 1$, it follows that for any $j\in\Z^2_+$, the left hand side of \eqref{limit-R-lemma} is non-negative and consequently, the function $\varkappa_{(\hat{x},\hat{y})}$ is non-negative on $\Z^2_+$. 
\end{proof}

\subsection{Harmonic functions $\varkappa_i$, $\tilde\varkappa_i$ and $\varkappa_{(x,y)}$ and their properties }\label{harmonic-functions-section} 
In this section, the second step of the proof of Theorem~\ref{theorem2} is performed (see Section~\ref{sec-sketch-proof-theorem2}). This is a subject of the following statement.

\begin{prop}\label{harmonic-functions} Under the hypotheses (A1) - (A3),  the following assertions hold:~ 

i) The function $\varkappa_1$ defined by \eqref{eq-def-kappa-1}
is non-negative on $\Z^2_+$,  positive on $\Z_+^2{\setminus} E_0$ and harmonic for $(Z_{\tau_0}(n))$ in each of the following cases~:~
\begin{itemize}
\item[--] if one of the assertions (B0), (B1), (B3) or (B4) is  valid;
\item[--] if (B2) holds and $x^{*}  < x^{**}_P$;
\item[--] if (B2) holds, $x^{**}  = x^{**}_P$ and  $\phi_1(x^{**}_P, Y_1(x^{**}_P)) =1$.
\end{itemize} 

ii) The function $\varkappa_2$ defined by \eqref{eq-def-kappa-2}
is non-negative on $\Z^2_+$,  positive on $\Z_+^2{\setminus} E_0$ and harmonic for $(Z_{\tau_0}(n))$ in each of the following cases~:~
\begin{itemize}
\item[--] if one of the assertions (B0), (B1), (B5) or (B6) is  valid;
\item[--] if (B2) holds and $y^{**}  < y^{**}_P$;
\item[--] if (B2) holds, $y^{**}  = y^{**}_P$ and  $\phi_1(X_1(y^{**}_P), y^{**}_P) =1$.
\end{itemize} 

iii)   If  (B2) is  valid with  $x^{**}  = x^{**}_P$ and $\phi_1(x^{**}_P, Y_1(x^{**}_P)) < 1$, then the  function $\tilde\varkappa_1$ defined by \eqref{eq-def-tilde-kappa-1}
is non-negative on $\Z^2_+$,  positive on $\Z_+^2{\setminus} E_0$ and harmonic for $(Z_{\tau_0}(n))$.  

iv)  If (B2) is  valid, then for any $(\hat{x}, \hat{y})\in {\mathcal S}_{22}$ with $\tilde{x} < x_d$ and $\hat{y} < y_d$, the function $\varkappa_{(\hat{x},\hat{y})}$ defined by \eqref{def-kappa-interior-direction} is non-negative on $\Z^2_+$,  positive on $\Z_+^2{\setminus} E_0$ and harmonic for $(Z_{\tau_0}(n))$. 
\end{prop} 
To get this result, we first show there that the left hand sides of \eqref{eq1-theorem2}, \eqref{eq-analytical-structure-IIa}, \eqref{eq-analytical-structure-IIb},  \eqref{eq-analytical-structure-IV}, \eqref{eq-analytical-structure-Va}, \eqref{eq-analytical-structure-Vb}, \eqref{eq-analytical-structure-III} and \eqref{limit-R} are non-negative. With these arguments we conclude that each of the functions defined by \eqref{eq-def-kappa-1},  \eqref{eq-def-tilde-kappa-1},  \eqref{eq-def-kappa-2}, \eqref{def-kappa-interior-direction} is non-negative on $\Z^2_+$.  Next we show that each of these functions (in the corresponding cases i)-iv). of our proposition) is harmonic for  $(Z_{\tau_0}(n))$, and  by using suitable Lyapunov functions, we prove that each of them is not identically zero on the set $\{j\in \Z^2_+; \; \|j\| \geq N_0\}$. With this results and using Lemma~\ref{irreductibility-lemma} we will be able to show that each of the functions defined by \eqref{eq-def-kappa-1},  \eqref{eq-def-tilde-kappa-1},  \eqref{eq-def-kappa-2}, \eqref{def-kappa-interior-direction} (in the corresponding cases i)-iv). of our proposition) is  positive throughout the set $\Z^2_+{\setminus} E_0$.  
 
 \subsubsection{Preliminary estimates.} 

 We begin the proof of this proposition with the following preliminary result

\begin{lemma}\label{preliminary-harmonic-functions-lemma} Suppose that the conditions (A1){-}(A3) are satisfied and let two points $(x_1,y_1)\in\inter{D}\cap\inter{D_1}$ and $(x_2,y_2)\in\inter{D}\cap\inter{D_2}$ satisfy \eqref{eq0_upper_bounds}. Then for some constant $C>0$ (depending on the points $(x_1,y_1)\in\inter{D}\cap\inter{D_1}$ and $(x_2,y_2)\in\inter{D}\cap\inter{D_2}$  but do not depending on $j\in\Z^2_+$), 
\be\label{eq1-preliminary-harmonic-functions-lemma}
\P_j(\tau_0 < +\infty)  \leq C \left(x_1^{j_1}y_1^{j_2} + x_2^{j_1}y_2^{j_2} \right), \quad \forall j\in\Z^2_+
\ee
and 
\be\label{eq2-preliminary-harmonic-functions-lemma}
\sum_{k\in\Z^2_+} g(j,k) \left(x_1^{k_1}y_1^{k_2} + x_2^{k_1}y_2^{k_2}\right) \leq C \left(x_1^{j_1}y_1^{j_2} + x_2^{j_1}y_2^{j_2} \right),  \quad \forall j\in\Z^2_+. 
\ee
\end{lemma}  
\begin{proof} Indeed, if two points $(x_1,y_1)\in\inter{D}\cap\inter{D_1}$ and $(x_2,y_2)\in\inter{D}\cap\inter{D_2}$ satisfy \eqref{eq0_upper_bounds}, then by Lemma~\ref{lemma1-upper-bounds}, for some $0<\theta <1$ and some finite subset $E$ of $\Z^2_+$ such that $(0,0)\in E$ and the function $f:\Z^2_+\to \R_+$ defined by $f(j_1,j_2){=}x_1^{j_1}y_1^{j_2}{+}x_2^{j_1}y_2^{j_2}$, for all $(j_1,j_2){\in}\Z^2_+$ satisfies the inequality 
\[
\E_{j}(f(Z(1)); \; \tau_E > 1) \leq \theta f(j), \quad \forall j=(j_1,j_2)\in\Z^2_+,
\]
with $\tau_E = \inf\{n > 0~:~Z(n)\in E\}$. By Lemma~\ref{lemma2-upper-bounds}, it follows that for any $j\in\Z^2_+$, the series 
\be\label{eq3-preliminary-harmonic-functions-lemma}
\sum_{k\in\Z^2} g(j,k) f(k) = \sum_{k\in\Z^2} g(j,k) \left(x_1^{k_1}y_1^{k_2} + x_2^{k_1}y_2^{k_2}\right)
\ee
converges, and moreover, (see \eqref{eq-for-harmonic-functions} in the proof of Lemma~\ref{lemma2-upper-bounds}), for 
\[
g_E(j,k) = \sum_{n=0}^\infty \P_j(Z(n) = k; \; \tau_E \geq n), \quad j\in\Z^2_+{\setminus} E, \; k\in\Z^2_+, 
\]
the following relation holds :~
\be\label{eq4-preliminary-harmonic-functions-lemma}
\sum_{k\in\Z^2_+} g_E(j,k) f(k) \leq \frac{f(j)}{1-\theta} = \frac{x_1^{j_1}y_1^{j_2} + x_2^{j_1}y_2^{j_2}}{1-\theta}.
\ee
Since  $(0,0)\in E$, from the last relation  it follows in particular that for any $j\in\Z^2_+{\setminus} E$, 
\begin{align}
\P_j(\tau_0 < +\infty) \leq \P_j(\tau_E <+\infty) &\leq \left(\min_{\ell\in E} f(\ell)\right)^{-1} \sum_{k\in\Z^2_+} g_E(j,k) f(k) \nonumber\\
 &\leq \left(\min_{\ell\in E} f(\ell)\right)^{-1}  \frac{1}{1-\theta} \left(x_1^{j_1}y_1^{j_2} + x_2^{j_1}y_2^{j_2}\right) \label{eq5-preliminary-harmonic-functions-lemma}
\end{align} 
and consequently \eqref{eq1-preliminary-harmonic-functions-lemma} is proved. 

Remark finally that  for any $j= (j_1, 0)\in\Z^2_+{\setminus} E$, 
\begin{align*}
\sum_{k\in\Z^2_+} g(j,k) f(k) &= \sum_{k\in\Z^2_+} g_E(j,k) f(k) + \sum_{\ell\in E\setminus\{(0,0)\}} g_E(j,\ell) \sum_{k\in\Z^2_+} g(\ell, k) f(k),\\
&\leq \sum_{k\in\Z^2_+} g_E(j,k) f(k)  ~+~ \P_j(\tau_E < +\infty) \, \max_{\ell\in E} \sum_{k\in\Z^2_+} g(\ell, k) f(k).
\end{align*} 
Since the set $E$ is finite and for any $j\in\Z^2_+$, the series \eqref{eq3-preliminary-harmonic-functions-lemma} converge, this last relation combined with 
 \eqref{eq4-preliminary-harmonic-functions-lemma} and  \eqref{eq5-preliminary-harmonic-functions-lemma} prove  \eqref{eq2-preliminary-harmonic-functions-lemma}. 
\end{proof}

 \begin{lemma}\label{preliminary-harmonic-functions-lemma2} Under the hypotheses (A1){-}(A3), for any $0 < y < y_d$ there are two points $(x_1,y_1)\in\inter{D}\cap\inter{D_1}$ and $(x_2,y_2)\in\inter{D}\cap\inter{D_2}$ satisfying \eqref{eq0_upper_bounds} and such that for some $C > 0$,
 \be\label{eq1-preliminary-harmonic-functions-lemma2} 
 H_j(0,y) \leq C \left(x_1^{k_1}y_1^{k_2} + x_2^{k_1}y_2^{k_2}\right), \quad \forall j\in\Z^2_+.
  \ee 
 \end{lemma} 
\begin{proof} Indeed, if $0< y < y_d$, then, by the definition of the set $\Theta$, the point  $(0, Y_1(x^{**}))$ is in  $\Theta$ and, consequently, by   Lemma~\ref{lemma6-upper-bounds}    there are two points $(x_1,y_1)\in\inter{D}\cap\inter{D_1}$ and $(x_2,y_2)\in\inter{D}\cap\inter{D_2}$  satisfying relations \eqref{eq0_upper_bounds} and  \eqref{eq1_upper_bounds} with $(x,y) = (0,Y_1(x^{**}))$.  By  Lemma~\ref{preliminary-harmonic-functions-lemma},   for these  two points $(x_1,y_1)$ and $(x_2,y_2)$, one gets \eqref{eq2-preliminary-harmonic-functions-lemma}  with $(x,y) = (0,y)$, and consequently  \eqref{eq1-preliminary-harmonic-functions-lemma2} holds. 
\end{proof}

 \subsubsection{Proof of  the first two assertions of Proposition~\ref{harmonic-functions}}  
 Consider first the case when one of the following assertions holds:
\begin{itemize}
\item[--] one of the cases (B0), (B1), (B3), (B4) holds;
\item[--] (B2) and $x_d  {<} x^{**}_P$  hold;
\item[--] (B2), $x_d {=} x^{**}_P$ and $\phi_1(x^{**}_P, Y_1(x^{**}_P)) {=} 1$ hold.
\end{itemize} 
By Proposition~\ref{phi-proof-theorem2}  and  the definition of  $x_d$, we have always $x_d = x^{**}$ and 
\be\label{eq1-harmonic-functions}
\phi_1(x_d, Y_1(x_d)) =1. 
\ee
To prove Proposition~\ref{harmonic-functions} in each of the above cases, we first show that the function $\varkappa_1$ is non-negative and harmonic for  $(Z_{\tau_0}(n))$. Next we will prove that $\varkappa_1(j_1,0) > 0$ for all $j_1> 0$ large enough, and using  finally Lemma~\ref{irreductibility-lemma}, we will conclude that the function $\varkappa_1$ is strictly positive on the set $\Z^2_+{\setminus} E_0$.

Suppose first that either, one of the cases (B0), (B1), (B3), (B4) holds, or (B2) and $x_d  < x^{**}_P$ hold. Then, by Proposition~\ref{cases-B0-B4-theorem2}, for any $j\in\Z^2_+$,   
\[
\lim_{x\to x_d} (x_d-x) H_j(x,0)  = c \varkappa_1(j)
\]
with $c >0$. Since $(x_d-x) H_j(x,0) \geq 0$  for any real $x \in ]0, x^{**}[$,  it follows that $\varkappa_1(j)\geq 0$ for any $j\in\Z^2_+$.
When either one of the cases (B0), (B1), (B3), (B4) occurs or (B2) and $x^{**}  < x^{**}_P$ hold, the function $\varkappa_1$ is therefore non-negative on $\Z^2_+$. 

Suppose now that (B2), $x^{**}  = x^{**}_P$ and $\phi_1(x^{**}_P, Y_1(x^{**}_P)) = 1$ hold. In this case, by Proposition~\ref{case-B2-theorem2}, for any $j\in\Z^2_+$,
\[
\lim_{x\to x_d} \sqrt{x_d-x}\, H_j(x,0) = c\,  \varkappa_1(j) 
\]
with $c>0$. Since $\sqrt{x_d-x}\, H_j(x,0) \geq 0$ for any real $x\in ]0, x_d[$,  it follows that $\varkappa_1(j)\geq 0$ for any $j\in\Z^2_+$. Hence, in this case,    the function $\varkappa_1$ is therefore also non-negative on $\Z^2_+$. 

%%%  
Harmonicity property of the function $\varkappa_1$ for the killed random walk. See Definition~\ref{Killed}. 
For  $j=(j_1,j_2)\in\Z^2_+$, by the definition of the functions $(x,y)\to L_j(x,y)$ and $\varkappa_1:\Z^2_+\to\R_+$,  we have
\[
\varkappa_1(j) = \begin{cases} x_d^{j_1}(Y_1(x_d))^{j_2} - \P_j(\tau < +\infty) + (\phi_2(x_d, Y_1(x_d))-1) H_j(0, Y_1(x_d)) &\text{if $j\not=(0,0)$}, \\
\phi_0(x_d, Y_1(x_d)) - \P_{(0,0)}(\tau < +\infty) + (\phi_2(x_d, Y_1(x_d))-1) H_{(0,0)}(0, Y_1(x_d)) &\text{if $j=(0,0)$}.
\end{cases} 
\]
Remark that for any $j=(j_1,j_2)\in\Z^2_+$, 
\begin{align*}
\E_j\left(\P_{Z(1)}(\tau <+\infty); \; \tau > 1\right) =  &\P_j(\tau < +\infty) - \mu_1(-1,0) \1_{\{(0,1)\}}(j_1,j_2) - \mu_2(0,-1) \1_{\{(1,0)\}}(j_1,j_2) \\
&- \mu(-1,-1) \1_{\{(1,1)\}}(j_1,j_2)  - \mu_0(0,0)\1_{\{(0,0)\}}(j_1,j_2)  
\end{align*} 
and since the function $j \to H_j(0, Y_1(x_d))$ is potential for $(Z_{\tau_0}(n))$ (see the properties of potential for the Markov chains functions in \cite{Woess}),  
\[
\E_j( H_{Z(1)}(0, Y_1(x_d)), \; \tau > 1) = H_j(0, Y_1(x_d)) 
- (Y_1(x_d))^{j_2}  \1_{\{j_1 =0, j_2 > 0\}} (j_1,j_2),  
\]
where for $j,k\in\Z^2$ we denote 
\[
\1_{\{k\}}(j) = \begin{cases} 1 &\text{if $j=k$},\\
0 &\text{otherwise}, 
\end{cases} 
\]
By the definition of the function $Y_1$, we have $P(x_d, Y_1(x_d)) = 1$, and  because of  \eqref{eq1-harmonic-functions}, $\phi_1(x_d, Y_1(x_d)) = 1$. Hence, for any $j\in\Z^2_+$, 
\begin{align*}
\E_j\bigl( x_d^{Z_1(1)} &(Y_1(x_d))^{Z_2(1)}; \; \tau > 1\bigr)  \\
&= \begin{cases} 
\phi_0(x_d, Y_1(x_d)) - \mu_0(0,0)\1_{\{(0,0)\}}(j_1,j_2)  &\text{if $j=(0,0)$},\\  
x_d^{j_1} \phi_1(x_d, Y_1(x_d)) - \mu_1(-1,0)\1_{\{(0,1)\}}(j_1,j_2) &\text{if $j_1 >0$ and $j_2=0$},\\  
(Y_1(x_d))^{j_2}  \phi_2(x_d, Y_1(x_d)) - \mu_2(0,-1) \1_{\{(1,0)\}}(j_1,j_2)  &\text{if $j_1=0$ and $j_2 > 0$}, \\
{x_d^{j_1}(Y_1(x_d))^{j_2} } P(x_d, Y_1(x_d)) - \mu(-1,-1) \1_{\{(1,1)\}}(j_1,j_2) &\text{if $j_1 > 0$, $j_2 > 0$},
\end{cases} \\
&= \begin{cases} 
\phi_0(x_d, Y_1(x_d)) - \mu_0(0,0)\1_{\{(0,0)\}}(j_1,j_2)  &\text{if $j=(0,0)$},\\  
x_d^{j_1}  - \mu_1(-1,0)\1_{\{(0,1)\}}(j_1,j_2) &\text{if $j_1 >0$ and $j_2=0$},\\  
(Y_1(x_d))^{j_2}  \phi_2(x_d, Y_1(x_d)) - \mu_2(0,-1) \1_{\{(1,0)\}}(j_1,j_2)  &\text{if $j_1=0$ and $j_2 > 0$}, \\
{x_d^{j_1}(Y_1(x_d))^{j_2} }  - \mu(-1,-1) \1_{\{(1,1)\}}(j_1,j_2) &\text{if $j_1 > 0$, $j_2 > 0$},
\end{cases} 
\end{align*} 
With these relations,  for any $j=(j_1,j_2)\in\Z^2_+$ with $j_1 > 0$, we get 
\begin{align} 
\E_j(\varkappa_1(Z(1)), \; \tau > 1) &= x_d^{j_1}(Y_1(x_d))^{j_2} - \P_j(\tau < +\infty) + (\phi_2(x_d, Y_1(x_d))-1) H_j(0, Y_1(x_d)), \nonumber\\
&= \varkappa_1(j), \label{eq30-harmonic-functions} 
\end{align} 
 for any $j=(j_1, j_2) \in\Z^2_+$ with $j_1=0$ and $j_2 > 0$, we get
\begin{align} 
\E_j(\varkappa_1(Z(1)) &= {x_d^{j_1}(Y_1(x_d))^{j_2} } \phi_2(x_d, Y_1(x_d))  - \P_j(\tau < +\infty)  \nonumber\\ & \hspace{2cm}+  (\phi_2(x_d, Y_1(x_d))-1) \bigl(-{x_d^{j_1}(Y_1(x_d))^{j_2} } + H_j(0, Y_1(x_d)) \bigr) \nonumber\\
&= \varkappa_1(j). \label{eq31-harmonic-functions} 
\end{align} 
and for $j=(0,0)$, we obtain 
\begin{align} 
\E_j(\varkappa_1(Z(1)) &=  \phi_0(x_d, Y_1(x_d))  - \P_{(0,0)}(\tau < +\infty) +  (\phi_2(x_d, Y_1(x_d))-1) H_{(0,0)}(0, Y_1(x_d)) \nonumber\\
&= \varkappa_1(0,0). \label{eq32-harmonic-functions} 
\end{align} 
Relations \eqref{eq30-harmonic-functions}, \eqref{eq31-harmonic-functions} and \eqref{eq32-harmonic-functions} prove that  function $\varkappa_1$ is harmonic  for $(Z_{\tau_0}(n))$.  

To show that the function $\varkappa_1$ is strictly positive at some point $j=(j_1,j_2)\in\Z^2_+$, we investigate an asymptotic behavior of this function as $j_2 = 0$ and $j_1\to+\infty$. By Proposition~\ref{cases}, in each of the cases (B0){-}(B4), we have $0<Y_1(x_d) < y_d$. Hence, by Lemma~\ref{preliminary-harmonic-functions-lemma2}, there are two points $(x_1,y_1)\in\inter{D}\cap\inter{D_1}$ and $(x_2,y_2)\in\inter{D}\cap\inter{D_2}$ satisfying \eqref{eq0_upper_bounds} for which \eqref{eq1-preliminary-harmonic-functions-lemma2} holds for $y=Y_1(x_d)$ and consequently, by Lemma~\ref{preliminary-harmonic-functions-lemma} we obtain \eqref{eq1-preliminary-harmonic-functions-lemma}. Using these relations for $j=(j_1,j_2)\in\Z^2_+$ with $j_2 = 0$ one gets 
\be\label{eq6-harmonic-functions} 
\left|\varkappa_1(j_1,0) - x_d^{j_1}\right|  \leq C_1  \left(x_1^{j_1}+ x_2^{j_1}\right) , \quad \forall j= (j_1, 0)\in\Z^2_+, 
\ee 
with some do not depending on $j$ constant $C_1>0$. Remark that $0 < x_2 < x_1$ because the points $(x_1,y_1)$ and $(x_2,y_2)$  satisfy  \eqref{eq0_upper_bounds}, and that $x_1 < x_d$ because $(x_1,y_1)\in\inter{D}\cap\inter{D_1}$. 
Hence, from the last relation it follows that 
\[
\varkappa_1(j_1,0) ~\sim~ x_d^{j_1} \quad \text{as $j_1\to\infty$},
\]
and consequently, there is $N_1 > 0$ such that 
\be\label{eq1-varkappa1-strictly-positive}
\varkappa_1(j_1,0) > 0 \quad \text{ for any $j=(j_1,0)\in\Z^2_+$ with $j_1 \geq N_1$}. 
\ee
\medskip

{\em Now we are ready to complete the proof of the first  assertion of Proposition~\ref{harmonic-functions}} :~  By Lemma~\ref{irreductibility-lemma}, there are $N_0 > 0$ and a finite subset $E_0$ of $\Z^2_+$ such that $(0,0)\not\in E_0$ and for any $j\in\Z^2_+{\setminus} E_0$ and $k\in\Z^2_+$ with $\|k\| > N_0$, 
\[
g(j,k) =\sum_{n=0}^\infty \P_j(Z(n) = k, \; \tau_0 > n) > 0. 
\]
Hence, for any $j\in\Z^2_+{\setminus} E_0$ and $k\in\Z^2_+$ with $\|k\| > N_0$, there is $n_{j,k}\in\N$ such that 
\[
\P_j(Z(n_{j,k}) = k, \; \tau_0 > n_{j,k}) > 0. 
\]
For any $j\in\Z^2_+$, $k\in\Z^2_+$ and $n\in\N$, since the function $\varkappa_1$ is  non-negative and harmonic for the killed random walk, one has 
\[
\varkappa_1(j) = \E_j(\varkappa_1(Z(n)), \; \tau_0 > n) \geq \P_j(Z(n) = k, \; \tau_0 > n) \varkappa_1(k). 
\]
Using this relation together with \eqref{eq1-varkappa1-strictly-positive}, for $k = (k_1, 0) \in\Z^2_+$ with $k_1 \geq \max\{N_0, N_1\}$,  one gets that for any $j\in\Z^2_+{\setminus} E_0$ 
\[
\varkappa(j) \geq \P_j(Z(n_{j,k}) = k, \; \tau_0 > n_{j,k}) \varkappa_1(k) > 0,
\]
and consequently, the function $\varkappa_1$ is positive on $\Z^2_+{\setminus} E_0$.

The first  assertion of Proposition~\ref{harmonic-functions} is therefore proved. To get the second assertion of this proposition, it is sufficient to exchange the roles of $x$ and $y$.

 \subsubsection{Proof of  the third  assertion of Proposition~\ref{harmonic-functions}}  
 
Suppose  that (B2) holds and let $x^{**}=x^{**}_P$ and  $\phi_1(x^{**}_P, Y_1(x^{**}_P)) < 1$. By the definition of the point $x_d$, we have here 
\[
x_d = x^{**} = x^{**}_P,
\]
and remark that in this case,
\[
\tilde\varkappa_1(j) = \left. \partial_y\left( \frac{L_j(x, y) + \left(\phi_2(x, y)-1\right) H_j(0, y)}{1-\phi_1(x, y)}\right)\right|_{(x,y)=(x^{**}_P,Y_1(x^{**}_P))} , \quad   \; j\in\Z^2_+.
\]
To show that the function $\tilde\varkappa_1$ is non-negative on $\Z^2_+$, we recall that in this cases, by Proposition~\ref{case-B2-theorem2}, 
\[
\lim_{x\to x_d} \sqrt{x_d-x}\, \frac{d}{dx} \, H_j(x,0) = c\,  \tilde\varkappa_1(j) 
\]
with $c > 0$. Since for real $x\in ]0,x^{**}_P[$, the function $x\to H_j(0, x)$ is  increasing on $]0,x^{**}_P[$, 
 it follows that $\tilde\varkappa_1(j) \geq 0$ for any $j\in\Z^2_+$.  

 Harmonicity property of the function $\tilde\varkappa_1$ for the killed random walk.
 For  $j=(j_1,j_2)\in\Z^2_+$,  we have
\be\label{eq18-(B2)-harmonic-function}
\tilde\varkappa_1(j) = \frac{\partial_y\phi_1((x^{**}_P,Y_1(x^{**}_P)))}{(1-\phi_1(x^{**}_P,Y_1(x^{**}_P)))^2} \varkappa_1(j) + \frac{1}{1-\phi_1((x^{**}_P,Y_1(x^{**}_P)))}  \hat\varkappa_1(j) 
\ee 
with $\varkappa_1(j)$ defined by \eqref{eq-def-kappa-1} and  
\begin{align}
\hat\varkappa_1(j) =  \partial_y L_j(x^{**}_P, Y_1(x^{**}_P))+ &(\phi_2(x^{**}_P, Y_1(x^{**}_P)) - 1) \partial_y H_j(0, Y_1(x^{**}_P)) \label{eq19-(B2)-harmonic-function}\\ &\quad \quad + \partial_y\phi_2((x^{**}_P,Y_1(x^{**}_P))) H_j(0, Y_1(x^{**}_P)), \nonumber 
\end{align}
where  
\[
\partial_y L_j(x^{**}_P, Y_1(x^{**}_P)) = \begin{cases} j_2 \times (x^{**}_P)^{j_1} (Y_1(x^{**}_P))^{j_2-1} &\text{if $j=(j_1,j_2)\in\Z^2_+{\setminus}\{(0,0)\}$,}\\ 
\partial_y\phi_0(x^{**}_P, Y_1(x^{**}_P)) &\text{if $j=(j_1,j_2) =(0,0)$}. 
\end{cases} 
\]
Remark now that for any $j=(j_1,j_2)\in\Z^2_+$, with the same arguments as in the proof of the first assertion of our proposition, one gets 
\be\label{eq20-(B2)-harmonic-function} 
\E_j(\varkappa_1(Z(1)), \; \tau_0 > 1) = 
\varkappa_1(j) + (\phi_1(x^{**}_P, Y_1(x^{**}_P)) - 1) (x^{**}_P)^{j_1}\1_{\{j_1 > 0, j_2=0\}}(j_1,j_2) 
\ee
(here, an only difference with the proof of the first assertion of our proposition is that $\phi_1(x^{**}_P, Y_1(x^{**}_P)) < 1$). Moreover, 
 since the functions $j\to H_j(0, Y_1(x^{**}_P))$ and $j\to \partial_y H_j(0, Y_1(x^{**}_P))$  are potential for $(Z_{\tau_0}(n))$ (see the properties of potential functions for Markov chains  in Woess~\cite{Woess}),  then for any $j=(j_1,j_2)\in\Z^2_+$, one has 
\be\label{eq22-(B2)-harmonic-function} 
\E_j(H_{Z(1)}(0, Y_1(x^{**}_P)), \; \tau_0 > 1) = H_j(0, Y_1(x^{**}_P)) - (Y_1(x^{**}_P))^{j_2} \1_{\{j_1=0, j_2 > 0\}}(j_1,j_2)  
\ee
and
\be\label{eq23-(B2)-harmonic-function} 
\E_j(\partial_y H_{Z(1)}(0, Y_1(x^{**}_P)), \; \tau_0 > 1) = 
\partial H_j(0, Y_1(x^{**}_P))  - j_2\times (Y_1(x^{**}_P))^{j_2-1}\1_{\{j_1=0, j_2 > 0\}}.
\ee
Finally, straightforward calculation shows that 
\begin{align*} 
\E_j(&\partial_y L_{Z(1)}(x^{**}_P, Y_1(x^{**}_P)), \; \tau_0 > 1)  =    \E_j\bigl( Z_2(1) (x^{**}_P)^{Z_1(1)} (Y_1(x^{**}_P))^{Z_2(1)-1}, \; \tau_0 > 1\bigr) \\ 
&= \partial_y \phi_1(x^{**}_P, Y_1(x^{**}_P)) (x^{**}_P)^{j_1} \1_{\{j_1 > 0, j_2 = 0\}}(j_1,j_2)  \\ 
&+ \left( Y_1(x^{**}_P) \partial_y P(x^{**}_P, Y_1(x^{**}_P)) + j_2 P (x^{**}_P, Y_1(x^{**}_P))\right)  j_2 (x^{**}_P)^{j_1} (Y_1(x^{**}_P))^{j_2-1} \\
&+ j_2(Y_1(x^{**}_P))^{j_2-1} \bigl(Y_1(x^{**}_P) \partial_y \phi_2(x^{**}_P, Y_1(x^{**}_P))  + j_2 \phi_2(x^{**}_P, Y_1(x^{**}_P))\bigr) \1_{\{j_1=0, j_1 > 0\}}(j_1,j_2).
\end{align*}
Since $\partial_y P(x^{**}_P, Y_1(x^{**})) = 0$ and $P(x^{**}_P, Y_1(x^{**})) = 1$, from the last relation  it follows that 
\begin{align} 
\E_j(&\partial_y L_{Z(1)}(x^{**}_P, Y_1(x^{**}_P)), \; \tau_0 > 1)  \label{eq21-(B2)-harmonic-function} \\ 
&= \partial_y \phi_1(x^{**}_P, Y_1(x^{**}_P)) (x^{**}_P)^{j_1} \1_{\{j_1 > 0, j_2 = 0\}}(j_1,j_2) +  j_2 (x^{**}_P)^{j_1} (Y_1(x^{**}_P))^{j_2-1} \nonumber\\
&+ j_2(Y_1(x^{**}_P))^{j_2-1} \bigl(Y_1(x^{**}_P) \partial_y \phi_2(x^{**}_P, Y_1(x^{**}_P))  + j_2 \phi_2(x^{**}_P, Y_1(x^{**}_P))\bigr) \1_{\{j_1=0, j_1 > 0\}}(j_1,j_2). \nonumber
\end{align} 
When combined together, relations  \eqref{eq18-(B2)-harmonic-function} - \eqref{eq23-(B2)-harmonic-function} imply that $\E_j(\tilde\varkappa_1(Z(1), \; \tau_0 > 1) = \tilde\varkappa_1(j)$ for any $j\in\Z^2_+$, and hence, the function $\tilde\varkappa_1$ is harmonic for $(Z_{\tau_0}(n))$.

To show that the function $\tilde\varkappa_1$ is strictly positive on $\Z^2_+{\setminus}\{(0,0)\}$, we investigate an asymptotical behavior of this function as $j_2 = 0$ and $j_1\to+\infty$. Remark that for any $j=(j_1,j_2)\in\Z^2_+{\setminus}\{(0,0)\}$ with $j_2 = 0$, 
\be\label{eq1-(B2)-harmonic-functions} 
\tilde\varkappa_1(j) =  C_1 (x^{**}_P)^{j_1} + C_2 \P_j(\tau_0 < +\infty) + C_3H_j(0,Y_1(x^{**}_P)) + C_4 \partial_y H_j(0,Y_1(x^{**}_P))\nonumber 
\ee
with 
\be\label{eq0-(B2)-harmonic-functions} 
C_1 = \frac{\partial_y\phi_1((x^{**}_P,Y_1(x^{**}_P)))}{(1-\phi_1(x^{**}_P,Y_1(x^{**}_P)))^2} 
\ee
and some $C_2, C_3, C_4\in{\mathbb R}$ do not depending on $j_1$. When (B2) and $x^{**}= x^{**}_P$ hold, we have 
\[
x_d = x^{**} = x^{**}_P, \quad \text{and} \quad Y_1(x^{**}_P) < y_d = y^{**}\leq y^{**}_P, 
\]
Hence by Lemma~\ref{preliminary-harmonic-functions-lemma2}, for  $\eps > 0$ such that $Y_1(x^{**}) + \eps < y_d$, there are two points $(x_1,y_1)\in\inter{D}\cap\inter{D_1}$ and $(x_2,y_2)\in\inter{D}\cap\inter{D_2}$ satisfying \eqref{eq0_upper_bounds} for which \eqref{eq1-preliminary-harmonic-functions-lemma2} holds with $y=Y_1(x^{**})+\eps$, and by Lemma~\ref{preliminary-harmonic-functions-lemma} we get \eqref{eq1-preliminary-harmonic-functions-lemma}. Using these relations for $j=(j_1,j_2)\in\Z^2_+$ with $j_2 = 0$ we obtain 
\be\label{eq2-(B2)-harmonic-functions} 
\P_{(j_1,0)} (\tau_0 < +\infty) \leq C_5 \, \frac{x_1^{j_1} + x_2^{j_1}}{1-\theta}, \quad \forall j= (j_1, 0)\in\Z^2_+, 
\ee
and 
\be\label{eq3-(B2)-harmonic-functions} 
H_{(j_1,0)}(0, Y_1(x^{**}_P) + \eps) \leq C_6 (x_1^{j_1} + x_2^{j_1}) 
\ee
with some $C_5 > 0$ and $C_6 > 0$ do not depending on $j_1$. Since the function $y\to H_j(0,y)$ is equal on the disk of its analyticity $B(0, y_d)$ to its power series with the positive coefficients, we have moreover 
\[
\partial_y H_j(0, Y_1(x^{**}_P)) \leq C_7 H_j(0, Y_1(x^{**}_P)+\eps_2) 
\]
with some do not depending on $j_1$ constant $C_7 >0$, and consequently, using \eqref{eq1-(B2)-harmonic-functions}, \eqref{eq2-(B2)-harmonic-functions} and \eqref{eq3-(B2)-harmonic-functions}, we obtain 
\[
\left| \tilde\varkappa_1(j_1,0) -  C_1 (x^{**}_P)^{j_1} \right| \leq C_9 (x_1^{j_1} + x_2^{j_1}), \quad \forall j_1 > 0, 
\]
with some do not depending on $j_1$ constant $C_8 > 0$. Remark finally that the constant $C_1$ defined by \eqref{eq0-(B2)-harmonic-functions} is strictly positive because the function $y\to \phi_1(x^{**}_P, y)$ is convex and strictly increasing on $]0,+\infty[$. Hence,  
with the same argument as in the proof of the first assertion of Proposition~\ref{harmonic-functions} we obtain  
\[
\tilde\varkappa_1(j_1,0) ~\sim~  C_1 (x^{**}_P)^{j_1}  \quad \text{as $j_1 \to+\infty$},
\]
and consequently,  $\tilde\varkappa_1(j) > 0$ for all $j_1 > 0$ large enough. 

The function $\tilde\varkappa_1$ is therefore non-negative, harmonic for  $(Z_{\tau_0}(n))$ and non zero at the points $(j_1,0)$ for all $j_1> 0$ large enough. With this results and using exactly the same arguments as in the proof of the first assertion of our proposition, we conclude that the function $\tilde\varkappa_1$ is  positive everywhere in $\Z^2_+{\setminus} E_0$.

\subsubsection{Proof of  the last assertions of Proposition~\ref{harmonic-functions}}   

Suppose now that (B2) holds, and let a point $(\hat{x}, \hat{y})\in{\mathcal S}_{22}$ be such that $\hat{x} < x_d$ and $\hat{y} < y_d$. For such a point $(\hat{x}, \hat{y})$, by Lemma~\ref{lemma-assertion-vi-th2},  the function $\varkappa_{(\hat{x},\hat{y})}$ is non-negative on $\Z^2_+$. 

\medskip

The proof of the identity 
\[
\E_j(\varkappa_{(\hat{x},\hat{y})}(Z(1)), \; \tau_0 > 1) = \varkappa_{(\hat{x},\hat{y})}(j), \quad \forall j\in\Z^2_+, 
\]
is straightforward, it uses the arguments quite similar to those of the proofs of \eqref{eq30-harmonic-functions}, \eqref{eq31-harmonic-functions} and \eqref{eq32-harmonic-functions}. The function $\varkappa_{(\hat{x},\hat{y})}$  is therefore harmonic for the killed random walk $(Z_{\tau_0}(n))$.

\medskip

To prove that $\varkappa_{(\hat{x},\hat{y})}(j) > 0$ for some $j\in\Z^2_+$ we consider the point $\hat{w} = w_D(\hat{x}, \hat{y}) =(u_D(\hat{x},\hat{y}), v_D(\hat{x},\hat{y}))$ on the unit circle ${\mathbb S}^1$ defined by \eqref{diffeomorphism-point-direction} with $x=\hat{x}$ and $y=\hat{y}$, and we investigate an asymptotical behavior of $\varkappa_{(\hat{x},\hat{y})}(j)$ as $\|j\|\to+\infty$ and $j/\|j\| \to \hat{w}$.

By the definitions of the functions $L_j$ and $\varkappa_{(\hat{x},\hat{y})}$, for $j=(j_1,j_2)\in\Z^2_+{\setminus}\{(0,0)\}$, 
\[
\varkappa_{(\hat{x},\hat{y})}(j)  = \hat{x}^{j_1}\hat{y}^{j_2} - \P_j(\tau_0 < +\infty) + \left(\phi_1(\hat{x},\hat{y})-1\right) H_j(\hat{x}, 0) + \left(\phi_2(\hat{x},\hat{y})-1\right) H_j(0,\hat{y})
\]
Since we assume that $\hat{y} < y_d$,  by Lemma~\ref{preliminary-harmonic-functions-lemma2}, there are two points $(x_1,y_1)\in\inter{D}\cap\inter{D_1}$ and $(x_2,y_2)\in\inter{D}\cap\inter{D_2}$ satisfying \eqref{eq0_upper_bounds} for which \eqref{eq1-preliminary-harmonic-functions-lemma2} holds with $y=\hat{y}$, and by Lemma~\ref{preliminary-harmonic-functions-lemma} we get \eqref{eq1-preliminary-harmonic-functions-lemma}. With the same arguments (it is sufficient to exchange the roles of $x$ and $y$, we get that there are two $(x_3,y_3)\in\inter{D}\cap\inter{D_1}$ and $(x_4,y_4)\in\inter{D}\cap\inter{D_2}$ for which 
\[
H_j(\hat{x},0) \leq C'  \left(x_3^{k_1}y_3^{k_2} + x_4^{k_1}y_4^{k_2}\right), \quad \forall j\in\Z^2_+,
\]
with some do not depending on $j\in\Z^2_+$ constant $C' > 0$. Using the last relation together with  \eqref{eq1-preliminary-harmonic-functions-lemma} and \eqref{eq1-preliminary-harmonic-functions-lemma2}, one gets 
\be\label{eq33-harmonic-functions}
\bigl|\varkappa_{(\hat{x},\hat{y})}(j)  - \hat{x}^{j_1}\hat{y}^{j_2} \bigr| \leq C_1 \sum_{i=1}^4 x_i^{j_1}y_i^{j_2}. 
\ee
with some do not depending on $j$ constant $C_1>0$. For $\hat{w} = w_D(\hat{x}, \hat{y}) =(u_D(\hat{x},\hat{y}), v_D(\hat{x},\hat{y}))$ defined by relation~\eqref{diffeomorphism-point-direction} with $x=\hat{x}$ and $y=\hat{y}$, the point $(\hat{x}, \hat{y})$ is an only point in the set $D$ where the function $(x,y)\to x^{\hat{u}_1} y^{\hat{u_2}}$ achieves its maximum over $D$. Since non of the points $(x_1,y_1), \ldots, (x_4,y_4)$ is equal to $(\hat{x}, \hat{y})$ (recall that the points $(x_1,y_1), \ldots, (x_4,y_4)$ belong to the interior of the set $D$ and the point $(\hat{x}, \hat{y})$ belongs to the boundary of $D$),  it follows  that 
\[
\frac{1}{\hat{x}^{j_1}\hat{y}^{j_2}} \sum_{i=1}^4 x_i^{j_1}y_i^{j_2} \to 0 \quad \text{as} \quad \|j\|\to+\infty \quad \text{and} \quad j/\|j\|\to \hat{u} 
\]
and consequently, using  \eqref{eq33-harmonic-functions}, we get 
\[
\varkappa_{(\hat{x},\hat{y})}(j)  \sim \hat{x}^{j_1}\hat{y}^{j_2} \quad \text{as} \quad \|j\|\to+\infty \quad \text{and} \quad j/\|j\|\to \hat{u}. 
\]
This proves that $\varkappa_{(\hat{x},\hat{y})}(j) > 0$ for any $j\in\Z^2_+$ with $\|j\|$ large enough and $j\|j\|$ closed enough to $\hat{u}$.

\medskip
The function $\varkappa_{(\hat{x},\hat{y})}$ is therefore non-negative, harmonic for  $(Z_{\tau_0}(n))$, and non zero at some points $j=(j_1,j_2)\in\Z^2_+$ with $\|j\| > N_0$. 
Hence,  the similar arguments as in the proof of the first assertion of our proposition, we conclude that the function $\tilde\varkappa_1$ is  positive everywhere on $\Z^2_+{\setminus} E_0$. 

\subsection{Proof of Theorem~\ref{theorem2}} Now we summarize the above results in orther to get Theorem~\ref{theorem2}:
\begin{itemize}
\item[--] The first assertion of Theorem~\ref{theorem2} follows from Proposition~\ref{cases-B0-B4-theorem2} and the first assertion of Proposition~\ref{harmonic-functions}.  
\item[--] The second assertion of Theorem~\ref{theorem2} is a consequence of Proposition~\ref{case-B2-theorem2} and the third assertion of Proposition~\ref{harmonic-functions}. 
\item[--]  Proposition~\ref{injection-prop2}, Proposition~\ref{harmonic-functions} and the second assertion of Proposition~\ref{harmonic-functions} prove the assertions iii)- v) of Theorem~\ref{theorem2}.
\item[--]  The last assertion of our theorem is proved by Lemma~\ref{lemma-assertion-vi-th2} and the last assertion  of Proposition~\ref{harmonic-functions}.  
\end{itemize} 

\section{Asymptotics  along the Axes}\label{section-proof-theorem3}
This section is devoted to the proof of Theorem~\ref{theorem3}.

For $k_2{=}0$, the asymptotics \eqref{eq1-theorem3}-\eqref{eq7-theorem3} follows from Theorem~\ref{theorem2} and the Tauberian-like theorem (see Corollary VI.1  of Flajolet and Sedgevick~\cite{Flajolet-Sedgewick}) in a straightforward way. Hence, to complete the proof of our theorem, i.e. to get the asymptotics \eqref{eq1-theorem3}-\eqref{eq7-theorem3} for $k_2 > 0$,  it is sufficient to show that for any $j\in\Z^2_+{\setminus} E_0$, 
\be\label{eq0-proof-theorem3}
\lim_{n \to+\infty} \frac{g(j, (n,k_2))}{g(j, (n, 0))} = \nu_1(k_2) > 0, \quad \forall k_2\in\Z^2.  
\ee
Before getting this result, let us notice that under the hypotheses (A1){-}(A3), the function $y\to  \psi_1(x_d, y)/Q(x_d, y)$
 is analytic at the origin  $y=0$ because the functions $y\to \psi_1(x_d,y)$ and $y\to Q(x_d,y)$, see definitions~\eqref{psi1def} and~\eqref{Qdef}, are analytic in a neighborhood of the closed disk $\overline{B}(0, Y_2(x_d))$ and with the definition of the function $Q$ (see \eqref{Qdef})
\[
Q(x_d, 0) = {-}\sum_{\substack{k=(k_1,k_2)\in\Z^2: \\ k_2=-1}} x_d^{k_1} \mu(k)  < 0. 
\]
For any given $n\in\N^*$, the quantity \eqref{eq-def-nu-1} is therefore well defined and equal to the $(n-1)$-th coefficient of the Taylor expansion of the function $y\to  \psi_1(x_d, y)/Q(x_d, y)$ in a neighborhood of the origin $y=0$~:
\[
\psi_1(x_d, y)/Q(x_d, y) = \sum_{n=1}^\infty \nu_1(n) y^{n-1}. 
\]
To show that all coefficients $\nu_1(n), \, n\in\N$ are  positive, and to get \eqref{eq0-proof-theorem3} for $k_2 > 0$, we will need  a probabilistic representation of the quantities $\nu_1(n), n\in\N^*$, in terms of the invariant measure of the following Markov chain on $\Z_+$: 
Define the twisted positive measures $\tilde\mu$ on $\Z^2$ and $\tilde\mu_1$ on $\Z\times\Z_+$ by letting 
\[
\tilde\mu(k_1,k_2) = x_d^{k_1} (Y_1(x_d))^{k_2}\mu(k_1,k_2),   \quad \forall (k_1,k_2)\in\Z^2, 
\]
\[
\tilde\mu_1(k_1,k_2) = x_d^{k_1} (Y_1(x_d))^{k_2}\mu_1(k_1,k_2),\quad \forall (k_1,k_2)\in \Z\times\Z_+. 
\]
With the definitions of the point $x_d$ and the function $Y_1$, we have
\[
\tilde\mu(\Z^2) = P(x_d, Y_1(x_d)) = 1 \quad \text{and} \quad \tilde\mu_1(\Z\times\Z_+) = \phi_1(x_d, Y_1(x_d)) \leq 1,
\]
and consequently, the twisted measure $\tilde\mu$ is  stochastic on $\Z^2_+$ and the twisted measure $\tilde\mu_1$ is sub-stochastic on $\Z\times\Z_+$. Consider now a  twisted random walk $(\tilde{Z}^1(n) = (\tilde{Z}^{(1)}_1(n), \tilde{Z}^{(1)}_2(n)))$ on $\Z\times\Z_+$ with transition probabilities 
\[
\P_{(j_1,j_2)}\bigl(\tilde{Z}^{(1)}(1) = (k_1,k_2)\bigr) ~=~ \begin{cases} \tilde\mu(k-j)  &\text{if $j_2 > 0$,}\\
\tilde\mu_1(k-j) &\text{if $j_2 = 0$.}
\end{cases} 
\]
Since the transition probabilities of the twisted random walk $(\tilde{Z}^1(n) = (\tilde{Z}^{(1)}_1(n), \tilde{Z}^{(1)}_2(n)))$  are invariant with respect to the shifts on $(\ell,0)$ for any $\ell\in\Z$, its second component $(\tilde{Z}^{(1)}_2(n))$ is a Markov chain on $\Z_+$ (stochastic if $\phi_1(x_d, Y_1(x_d)) = 1$ and sub-stochastic if $\phi_1(x_d, Y_1(x_d)) < 1$) with transition probabilities 
\[
\P_{j_2}\bigl(\tilde{Z}^{(1)}_2(1) = k_2\bigr) ~=~ \begin{cases} \sum_{k_1=-1}^\infty \tilde\mu(k_1, k_2-j_2)  &\text{if $j_2 > 0$,}\\
\\
\sum_{k_1=-1}^\infty  \tilde\mu_1(k_1, k_2) &\text{if $j_2 = 0$.}
\end{cases} 
\]
Recall that a non-negative measure $\pi_1$ on $\Z_+$ is invariant for  $(\tilde{Z}^{(1)}_2(n))$ if,  for any $\ell{\in}\Z_+$, the  relation 
\be\label{eq1-proof-theorem3}
\pi_1(\ell) = \sum_{\ell'\in\Z_+} \pi_1(\ell') \P_{\ell'}(\tilde{Z}^{(1)}_2(n) = \ell)
\ee
holds. The following lemma relates the vector $\nu_1=(\nu_1(k_2), k_2\in\Z_+)$ with the unique (up to a multiplicative constant) invariant measure for $(\tilde{Z}^{(1)}_2(n))$. 
\begin{lemma}\label{lemma-proof-theorem3} Under the hypotheses (A1){-}(A3),  $\pi_1 = (\pi_1(n) = \nu_1(n)(Y_1(x_d))^n, \, n\in\N)$ is a unique up to the multiplication by constants invariant measure of the Markov chain $(\tilde{Z}^{(1)}_2(n))$.
\end{lemma} 
\begin{proof} We prove this lemma in two steps: First we will show that the Markov chain $(\tilde{Z}^{(1)}_2(n))$ has an invariant measure $\pi_1 = (\pi_1(n), \, n\in\N)$ with $\pi_1(0)=1$, and next we will prove that the generating function $f(y) = \sum_{n=0}^\infty \pi_1(n) (y/Y_1(x_d))^n$ satisfies in a neighborhood of the point $y=0$ in $\C$ the identity 
\[
f(y) = \phi_1(x_d,y) + P(x_d,y) (f(y)-1)
\]
and we will deduce from this identity  that $\pi_1(n) = \nu_1(n)(Y_1(x_d))^n$ for any $n\in\N$. 

To perform the first step of our proof, let us notice that under our hypotheses, the jumps of the Markov chain $(\tilde{Z}^{(1)}_2(n))$ are integrable and moreover, for any non-zero $\ell\in\Z_+$,
\begin{align*}
\E_\ell(\tilde{Z}^{(1)}_2(1)) &= \ell  + \sum_{k_2=-1}^\infty k_2 \sum_{k_1=-1}^\infty x_d^{k_1} (Y_1(x_d))^{k_2}\mu(k_1,k_2) = \ell + Y_1(x_d) \partial_y P(x_d, Y_1(x_d)) 
\end{align*} 
with $\partial_y P(x_d, Y_1(x_d)) \leq 0$. The Markov chain $(\tilde{Z}^{(1)}_2(n))$ is therefore recurrent if the measure $\tilde\mu_1$ is stochastic, i.e. if $\phi_1(x_d, Y_1(x_d)) = 1$, and it is transient if $\phi_1(x_d, Y_1(x_d)) < 1$. 

In the case when the Markov chain $(\tilde{Z}^{(1)}_2(n))$ is recurrent, it has a unique invariant measure $(\pi_1(\ell), \ell\in\Z_+)$ with $\pi_1(0) = 1$ and 
\[
\pi_1(\ell) = \sum_{n=1}^\infty \P_0(\tilde{Z}^{(1)}_2(n) = \ell, \; \tilde{Z}^{(1)}_2(k) \not= 0, \, \forall k < n), \quad \forall \ell > 0. 
\]
Suppose now that the Markov chain $(\tilde{Z}^{(1)}_2(n))$ is transient, and let 
\[
\pi_1^{(k)}(\ell) = \sum_{n=0}^\infty\P_k(\tilde{Z}^{(1)}_2(n) = \ell), \quad k,\ell\in\Z_+.  
\]
Then for any $k, \ell\in\Z_+$ such that $k > \ell$, with a straightforward calculation one gets 
\be\label{eq1-local-invariant-meaure}
\sum_{\ell'\in\Z_+} \pi_1^{(k)}(\ell') \P_{\ell'}(\tilde{Z}^{(1)}_2(1) = \ell) ~=~ \pi_1^{(k)}(\ell).
\ee
Remark moreover that because of Assumption (A2), by the strong Markov property,  
\[
{\pi_1^{(k)}(\ell)}/{\pi_1^{(k)}(0)} = {\pi_1^{(\ell)}(\ell)}/{\pi_1^{(\ell)}(0)}, \quad \forall k \geq \ell, 
\]
and $\P_{\ell'}(\tilde{Z}^{(1)}_2(1) = \ell) = 0$ for all $\ell' > \ell +1$. Using \eqref{eq1-local-invariant-meaure}, it follows that the measure $\pi_1=(\pi_1(\ell), \, \ell\in\Z_+)$ with $\pi_1(0)=1$ and
\[
\pi_1(\ell) = \pi_1^{(\ell)}(\ell)/\pi_1^{(\ell)}(0), \quad \forall \ell\in\N^*,
\]
is invariant for $(\tilde{Z}^{(1)}_2(n))$. The first step of our proof is therefore completed. 

\medskip

Remark now that for any $n\in\Z_+$, using the identity 
\be\label{eq-identity-local-invariant-measure}
\sum_{\ell' = 0}^{\ell+n} \pi_1(\ell')  \P_{\ell'}(\tilde{Z}^{(1)}_2(n) = \ell)  = \pi_1(\ell) 
\ee
with $\ell = 0$, one gets 
\[
 \pi_1(n) \leq \pi_1(0) \left(\P_{n}(\tilde{Z}^{(1)}_2(n) = 0) \right)^{-1} \leq \pi_1(0) \left(\sum_{k_1=-1}^\infty \tilde\mu(k_1,-1) \right)^{-n}
\]
with 
\[
\sum_{k_1=-1}^\infty \tilde\mu(k_1,-1) = \sum_{k_1=-1}^\infty x_d^{k_1} (Y_1(x_d))^{-1}\mu(k_1,-1)  = - Q(x_d, 0) > 0. 
\]
Hence, the generating function $f(y) = \sum_{n=0}^\infty \pi_1(n) (y/Y_1(x_d))^n$ is analytic in a neighborhood of the origin $y=0$ in $\C$, and by \eqref{eq1-proof-theorem3}, for any non zero $y$, satisfies there the identity 
\begin{multline*}
f(y) = f(0) \Bigl( \sum_{(k_1,k_2)\in\Z\times\Z_+} \tilde\mu_1(k_1,k_2) (y/Y_1(x_d))^{k_2}\Bigr) \\ + (f(y) - f(0)) \Bigl( \sum_{(k_1,k_2)\in\Z\times\Z} \tilde\mu(k_1,k_2) (y/Y_1(x_d))^{k_2}\Bigr). 
\end{multline*} 
Since $f(0)=\pi_1(0) = 1$ and with the definition of the twisted measures $\tilde\mu$ and $\tilde\mu_1$, for non-zero $y$, we have
\[
\sum_{(k_1,k_2)\in\Z\times\Z} \tilde\mu(k_1,k_2) \frac{y^{k_2}}{(Y_1(x_d))^{k_2}} = P(x_d, y)\]
and
\[
\sum_{(k_1,k_2)\in\Z\times\Z_+} \tilde\mu_1(k_1,k_2) \frac{y^{k_2}}{(Y_1(x_d))^{k_2}}  = \phi_1(x_d, y),  
\]
this implies that for any non-zero $y\in\C$ with $|y|$ small enough,
\[
(f(y) - 1)(1-P(x_d,y)) = \phi_1(x_d,y) - 1. 
\]
By the definition of the functions $(x,y)\to \psi_1(x,y)$ and $(x,y)\to Q(x,y)$, from the last relation it follows that for any $y\in\C$ with $|y|$ small enough, 
\[
(f(y) - 1) Q(x_d,y) = y\psi_1(x_d,y) 
\]
and consequently, since $Q(x_d,0) \not= 0$, for any $y\in\C$ with $|y|$ small enough,  we have also 
\[
f(y) = 1 + \frac{y\psi_1(x_d,y)}{Q(x_d,y)}. 
\]
From the last relation, by the uniqueness of  the coefficients of the Taylor expansion at the point $y=0$, it follows that 
\[
\pi_1(n)(Y_1(x_d))^{-n} = \nu_1(n), \quad \forall n\in\Z_+, 
\]
and consequently, the measure $\pi_1 = (\nu_1(n)Y_1(x_d))^n, \, n\in \Z_+)$ is the unique invariant measure of the Markov chain $(\tilde{Z}^{(1)}_2(n))$ with $\pi_1(0) = 1$. 
\end{proof} 

As a straightforward consequence of Lemma~\ref{lemma-proof-theorem3} one gets the following property of the coefficients $\nu_1(n), n\in\N$:

\begin{cor}\label{cor-proof-theorem3} Under the hypotheses (A1){-}(A3), all coefficients $\nu_1(n), \, n\in\Z_+$, are  positive and $(\nu_1(n), n\in\Z_+)$ is an only positive solution of the system 
\be\label{eq2-proof-theorem3}
\sum_{\ell_1=1}^{\infty} \mu_1(\ell_1, n) 
+ \sum_{\ell_2=1}^{n+1} \nu_1(\ell_2) \sum_{\ell_1=1}^{\infty} \mu(\ell_1, n-\ell_2) = \nu_1(n), \quad k_2\in\Z_+. 
\ee
\end{cor} 
\begin{proof} Indeed, since by Lemma~\ref{lemma-proof-theorem3}, 
 $\nu_1(n) = \pi_1(n)  (Y_1(x_d))^{-n}$ for any $n\in\Z_+$, the system \eqref{eq2-proof-theorem3} is equivalent to \eqref{eq1-proof-theorem3}. Under our hypotheses, the Markov chain $(\tilde{Z}^{(1)}_2(n))$ is irreducible on $\Z_+$ and $\pi_1(0)=1 \not= 0$.  It follows that   $\pi_1(n) > 0$ for any $n\in\Z_+$, and consequently also $\nu_1(n) = \pi_1(n)  (Y_1(x_d))^{-n}> 0$ for any $n\in\Z_+$. 
\end{proof} 

From this result it follows that to get \eqref{eq0-proof-theorem3}, it is sufficient to show that for any $j\in\Z^2_+{\setminus} E_0$ and $k_2\in\N^*$, the sequence $\bigl(g(j, (n,k_2))/g(j, (n, 0))\bigr)$ converges and the limits $\lim_{n\to\infty} g(j, (n,k_2))/g(j, (n, 0))\bigr)$, $k_2\in\N$, satisfy the system of the equations \eqref{eq2-proof-theorem3}. To prove the convergence of these sequences, the following preliminary results will be needed. 

\begin{cor}\label{cor2-proof-theorem3} Under the hypotheses (A1){-}(A3), if one of the conditions (B0){-}(B6) holds, then  for any $j\in\Z^2_+$
\be\label{eq1a-proof-theorem3}
\lim_n g(j, (k_1+n, 0))/ g(j, (n,0)) = 1, \quad \forall  k_1\in\Z_+, 
\ee
and there are constants $c_j > 0$ and $c'_j > 0$ such that for any $k_1 > 0$, 
\be\label{eq1b-proof-theorem3} 
c_j k_1^\gamma x_d^{-k_1} \leq g(j,(k_1,0)) \leq c'_j k_1^\gamma x_d^{-k_1}, 
\ee
where 
\[
\gamma = \begin{cases} -1/2 &\text{ if $x_d = x^{**}_P$  and either (B2) holds with $\phi_1(x_d, Y_1(x_d)) = 1$,}\\
\, &\text{\quad or (B5) holds with $\phi_1(x_d, Y_1(x_d)) < 1$,},\\
-3/2 &\text{if (B2) holds with $x_d = x^{**}_P$ and $\phi_1(x_d, Y_1(x_d)) < 1$},\\
1 &\text{if (B5) holds with $x_d < x^{**}_P$},\\
0 &\text{otherwise.} 
\end{cases} 
\]
\end{cor} 
\begin{proof} This result is a straightforward consequence of the asymptotics \eqref{eq1-theorem3}-\eqref{eq7-theorem3}  with $k_2 = 0$. These asymptotics 
follow from Theorem~\ref{theorem2} and a Tauberian-like theorem (see Corollary VI.1  of Flajolet and Sedgevick~\cite{Flajolet-Sedgewick}) in a straightforward way.
\end{proof} 

\begin{lemma}\label{lemma2-proof-theorem3}  Under the hypotheses (A1){-}(A3), for any $j\in\Z^2_+$ and $k_2\in\Z_+$ there are three constants $c_1 > 0$, $c_2 > 0$ and $N(k_2) > 0$, such that 
\be\label{eq1-lemma2-proof-theorem3}
c_1 \,g(j, (k_1,0)) \leq g(j,(k_1,k_2)) \leq c_2 \,g(j,(k_1,0)), 
\ee
for any  $k_1\in\Z_+$ such that $k_1 \geq N(k_2)$. 
\end{lemma}
\begin{proof} Remark that for any $j, k=(k_1,k_2)\in\Z^2_+$ and $N\in\N$, by using the Markov property one gets 
\begin{align}
g(j,(k_1,k_2)) &\geq \sum_{n=N}^\infty\P_j(Z(n) = (k_1,k_2), \, Z(N) = (k_1,0), \; \tau_0 > n) \nonumber\\ 
&\geq g(j, (k_1,0)) \P_{(k_1,0)}(Z(N) = (k_1,k_2), \; \tau_0 > 0). \label{eq2-lemma2-proof-theorem3}
\end{align} 
Moreover, similar arguments as in the proof of Lemma~\ref{lemma1_Y1_continued} shows that for any $k_2\in\Z_+$, there is $N\in\N^*$ and a sequence $\bigl(\ell_1^{(0)},\ell_2^{(0)}\bigr), \ldots, \bigl(\ell_1^{(N)},\ell_2^{(N)}\bigr)\in\Z^2$ such that 
\[
\ell^{(0)} + \ldots  + \ell^{(n)} \in \Z\times\Z_+, \quad \forall n\in\{0,\ldots, N\}, 
\] 
\[
\sum_{n=0}^{N} \ell^{(n)} = (0,k_2) 
\quad \quad \text{and} \quad \quad 
\mu_1\bigl(\ell^{(0)}\bigr)\prod_{n=1}^{N} \mu\bigl(\ell^{(n)}\bigr) > 0.
\]
Hence, using Assumption (A2) we conclude  that for any $(k_1,k_2)\in\Z^2_+$ with $k_1 > N$, 
\[
\P_{(k_1,0)}\bigl(Z(n(k_2)) = (k_1,k_2), \; \tau_0 > 0\bigr) \geq \mu_1(\ell^{(0)})\prod_{n=1}^{N} \mu(\ell^{(n)}) > 0.
\]
When combined with \eqref{eq2-lemma2-proof-theorem3} the last relation proves the first inequality of \eqref{eq1-lemma2-proof-theorem3} with 
\[
c_1 = \mu_1(\ell^{(0)})\prod_{n=1}^{N} \mu(\ell^{(n)}). 
\]
The proof of the second inequality of  \eqref{eq1-lemma2-proof-theorem3} is quite similar. 
\end{proof} 
\medskip

\noindent 
{\em Now we are ready to complete proof of Theorem~\ref{theorem3}}. Throughout our proof, the starting point $j{\in}\Z^2_+{\setminus} E_0$ will be given. 

Since by Lemma~\ref{lemma2-proof-theorem3} for any $k_2\in\Z_+$, the sequence $(g(j, (n,k_2))/g(j, (n, 0)), \, n\geq N_n)$ is bounded below and above by some  positive constants, to get \eqref{eq0-proof-theorem3} it is sufficient to show that for any subsequence $(N_n)$ of the sequence of non-zero natural numbers $(n)$, for which the sequence of functions 
\be\label{eq3-proof-theorem3} 
k_2 \to f_n(k_2) = g(j, (N_n,k_2))/g(j, (N_n, 0))
\ee
converges point-wise in $\Z_+$, one has 
\be\label{eq4-proof-theorem3} 
\lim_n f_n(k_2) = \nu_1(k_2), \quad \forall k_2\in\Z_+. 
\ee 
Suppose now that for a subsequence $(N_n)$, the sequence of functions $(f_n)$ defined by \eqref{eq3-proof-theorem3} converges point-wise in $\Z_+$ and let $f_\infty = \lim_n f_n$. By Corollary~\ref{cor-proof-theorem3}, to get  \eqref{eq4-proof-theorem3} it is sufficient to show that the limit function $f_\infty$ satisfies the system of equations 
\be\label{eq5-proof-theorem3} 
\sum_{\ell_1=1}^{\infty} \mu_1(\ell_1, k_2) 
+ \sum_{\ell_2=1}^{n+1} f_\infty(\ell_2) \sum_{\ell_1=1}^{\infty} \mu(\ell_1, k_2-\ell_2) = f_\infty(k_2), \quad k_2\in\Z_+. 
\ee
To get this result, we consider the sequence of functions 
 $F_n:\Z\times\Z_+\to\R_+$ defined for any $(k_1,k_2)\in\Z\times\Z_+$ by 
\[
F_n(k_1,k_2) = \begin{cases} g(j, (k_1 + N_n, k_2))/ g(j, (N_n,0)) &\text{if $k_1 + N_n \geq 0$,}\\
0&\text{otherwise.}
\end{cases} 
\]
Remark that by \eqref{eq1a-proof-theorem3},  the sequence of functions $(F_n)$ also converges  point-wise in $\Z\times\Z_+$, and for any $(k_1,k_2)\in\Z\times\Z_+$, 
\be\label{eq6-proof-theorem3}
\lim_n F_n(k_1,k_2) = \lim_n f_n(k_2) = f_\infty(k_2). 
\ee
Remark moreover that for any $k=(k_1,k_2)\in\Z^2_+{\setminus}\{j\}$, 
\begin{align*}
g(j, (k_1,k_2)) &= \sum_{(\ell_1,\ell_2)\in\Z^2_+} g(j, (\ell_1,\ell_2))\P_{(\ell_1,\ell_2)}(Z(1) = (k_1,k_2), \, \tau_0 > 1) \\
&=\sum_{\ell_2=1}^{k_2+1} g(j, (0,\ell_2)) \mu_2(k_1,k_2-\ell_2) + \sum_{\ell_1=1}^{k_1+1} g(j, (\ell_1,0)) \mu_1(k_1-\ell_1, k_2)\\
&\hspace{3cm} +   \sum_{\ell_1 =1}^{k_1+1}\sum_{\ell_2=1}^{k_2+1} g(j, (\ell_1,\ell_2)) \mu(k_1-\ell_1,k_2-\ell_2).
\end{align*} 
Using this relation with $k_1 = N_n$, for any $n\in\N$ such that $N_n > j_1$,  one gets 
\begin{align*}
F_n(0,k_2) =\sum_{\ell_2=1}^{k_2+1} F_n(-N_n,\ell_2) &\mu_2(N_n,k_2-\ell_2) + \sum_{\ell_1=-1}^{N_n} F_n(-\ell_1,0)) \mu_1(\ell_1, k_2) \\  & +\sum_{\ell_1 = -1}^{N_n}\sum_{\ell_2=1}^{k_2+1} F_n(-\ell_1, \ell_2)\mu(\ell_1,k_2-\ell_2)\quad \quad \forall k_2\in\Z_+, 
\end{align*} 
and consequently,  using \eqref{eq6-proof-theorem3} we will obtain \eqref{eq5-proof-theorem3} if we prove that for any $k_2\in\Z_+$ and $\ell_2\in\{1,\ldots, k_2+1\}$, the following relations hold 
\be\label{eq7-proof-theorem3}
\lim_n  F_n(-N_n,\ell_2) \mu_2(N_n,k_2-\ell_2) = 0, 
\ee
\be\label{eq8-proof-theorem3}
\lim_n \sum_{\ell_1=-1}^{N_n} F_n(-\ell_1,0)) \mu_1(\ell_1, k_2)  =  \sum_{\ell_1=-1}^{\infty} \lim_n F_n(-\ell_1,0)) \mu_1(\ell_1, k_2)
\ee
and
\be\label{eq9-proof-theorem3}
\lim_n \sum_{\ell_1 = -1}^{N_n} F_n(-\ell_1, \ell_2)\mu(\ell_1,k_2-\ell_2) = \sum_{\ell_1 = -1}^{\infty} \lim_n F_n(-\ell_1, \ell_2)\mu(\ell_1,k_2-\ell_2).   
\ee
To get \eqref{eq7-proof-theorem3} we remark that by Corollary~\ref{cor2-proof-theorem3},  there is a constant $C > 0$ such that for any $N_n$ large enough, 
\[
F_n(-N_n,\ell_2) = g(j,(0, \ell_2))/ g(j, (N_n,0)) \leq g(j,(0, \ell_2)) C (N_n)^\gamma x_d^{N_n} 
\]
and consequently, for any $\eps > 0$, there is a constant $C_\eps > 0$ such that for any $N_n$ large enough, 
\[
F_n(-N_n,\ell_2) \leq g(j,(0, \ell_2)) C_\eps(1+\eps)^{N_n} x_d^{N_n}. 
\]
Since by Assumption (A3)(ii), the generating function $\phi_1$ is finite in a neighborhood of the point $(x_d, Y_1(x_d))$, there is $\eps > 0$ such that for any $k_2\in\Z_+$, 
\[
\lim_{n\to\infty} \mu_1(n,k_2) (1+\eps)^{n}x_d^{n} = 0. 
\]
and consequently,  for any $k_2\in\Z_+$ and $\ell_2\in\{1,\ldots, k_2+1\}$, \eqref{eq7-proof-theorem3} holds. 

To get \eqref{eq8-proof-theorem3} we use the implicit function theorem and Corollary~\ref{cor2-proof-theorem3}. By Corollary~\ref{cor2-proof-theorem3}, there is a constant $ C_j > 0$ such that for any $N_n > 0$ and $\ell_1\in\Z$ such that $-1 \leq \ell_1 \leq N_n$, 
\be\label{eq10-proof-theorem3}
F_n(-\ell_1, 0) = \frac{g(j, (N_n - \ell_1, 0))}{g(j, (N_n, 0))} \leq C_1 \frac{(N_n-\ell_1)^\gamma }{ N_n^\gamma} x_d^{\ell_1}. 
\ee
In the case when $\gamma \geq 0$, it follows that for any $N_n > 0$ and $\ell_1\in\Z$ such that $-1 \leq \ell_1 \leq N_n$, 
\[
F_n(-\ell_1, 0)  \leq  C_1  x_d^{\ell_1},
\]
and consequently, since $F_n(-\ell_1, 0) = 0$ for all $\ell_1 \geq N_n$, and since under our hypotheses, 
\[
 \sum_{\ell_1=-1}^{\infty} x_d^{\ell_1} \mu_1(\ell_1, k_2) \leq (Y_1(x_d))^{-k_2} \sum_{\ell=(\ell_1,\ell_2)\in\Z\times\Z_+}^{\infty} x_d^{\ell_1} (Y_1(x_d))^{\ell_2} \mu_1(\ell_1, \ell_2) =\phi_1(x_d, Y_1(x_d)) < +\infty, 
\]
by the implicit function theorem one gets \eqref{eq8-proof-theorem3}. In the case when $\gamma \geq 0$, \eqref{eq8-proof-theorem3} is therefore proved. 

Suppose now that $\gamma < 0$. In this case, using \eqref{eq10-proof-theorem3}, one gets that for any $N_n > 0$ and $\ell_1\in\Z$ such that $-1 \leq \ell_1 < N_n$, 
\[
F_n(-\ell_1, 0)   \leq C_1 \frac{ N_n^{|\gamma|}}{(N_n-\ell_1)^{|\gamma|} } x_d^{\ell_1} 
\leq \begin{cases} C_1 N_n^{|\gamma|} x_d^{\ell_1} \leq C_1 (2\ell_1)^{|\gamma|} x_d^{\ell_1}  &\text{if $\ell_1 > N_n/2$}\\ 
 C_1 2^{|\gamma|} x_d^{\ell_1}  &\text{if $\ell_1 \leq N_n/2$}
\end{cases} 
\]
and consequently, for any $\eps > 0$ there is a constant $C_\eps > 0$ such that $N_n > 0$ and $\ell_1\in\{-1,\ldots, N_n-1\}$, one has 
\be\label{eq11-proof-theorem3}
F_n(-\ell_1, 0)  \leq C_\eps (1+\eps)^{\ell_1} x_d^{\ell_1}. 
\ee
Since $F_n(-\ell_1, 0) = 0$ for all $\ell_1 \geq N_n$, and since under our hypotheses, for some $\eps > 0$ small enough, 
\begin{align}
 \sum_{\ell_1=-1}^{\infty} (1+\eps)^{\ell_1}x_d^{\ell_1} \mu_1(\ell_1, k_2) &\leq (Y_1(x_d))^{-k_2} \!\!\!\sum_{\ell=(\ell_1,\ell_2)\in\Z\times\Z_+}^{\infty} (1+\eps)^{\ell_1}x_d^{\ell_1} (Y_1(x_d))^{\ell_2} \mu_1(\ell_1, \ell_2)\nonumber \\ &\leq (Y_1(x_d))^{-k_2}  \phi_1((1+\eps)x_d, Y_1(x_d)) < +\infty, \label{eq12-proof-theorem3}
\end{align} 
by the implicit function theorem, it follows \eqref{eq8-proof-theorem3}. Relation  \eqref{eq8-proof-theorem3} is therefore proved. 

To get \eqref{eq9-proof-theorem3} we use  first Lemma~\ref{lemma2-proof-theorem3}. By Lemma~\ref{lemma2-proof-theorem3}, there are two constants $N(\ell_2)$ and 
$C_1 > 0$ such that for any $N_n \geq 0$ and $\ell_1\in\{-1,\ldots, N_n - N(\ell_2)\}$, 
\[
F_n(-\ell_1, \ell_2) = \frac{g(j, (N_n - \ell_1, \ell_2))}{g(j, (N_n, 0))} \leq C_1 \frac{g(j, (N_n - \ell_1, 0))}{g(j, (N_n, 0))} 
\]
and for $\ell_1\in\{N_n - N(\ell_2) + 1, \ldots, N_n-1\}$,
\[
F_n(-\ell_1, \ell_2) = \frac{g(j, (N_n - \ell_1, \ell_2))}{g(j, (N_n, 0))} \frac{C_2}{g(j, (N_n, 0))} 
\]
with $C_2 = \max\{g(j, (1, 0)), \ldots, g(j, (N(\ell_2), 0)\}$. Hence with the similar arguments as above one gets that for any $\eps > 0$ there is a constant $C_\eps > 0$ such that $N_n > 0$ and $\ell_1\in\{-1,\ldots, N_n-1\}$, \eqref{eq11-proof-theorem3} holds, and consequently, using \eqref{eq12-proof-theorem3} we conclude that \eqref{eq9-proof-theorem3}  holds. 

 Theorem~\ref{theorem3} is therefore proved.

%%%%%%%%%%%%%%%%%%%%%%%%%%%%%%%%%%%%%%%%

\section{Asymptotics  along Directions of ${\mathbb S}^1_+$}\label{section-proof-theorem4} 
This section is devoted to the proof of Theorem~\ref{theorem4}.

\subsection{Main ideas and the sketch of the proof} 
The main ideas of the proof of Theorem~\ref{theorem4} are the following:

By using Theorem~\ref{theorem1}, we first get an integral representation 
\begin{multline}\label{eq1-integral-representation-lemma} 
 g(j,k) =\\ \int_{|x|=\hat{x}} \int_{|y|=\hat{y}}\frac{L_j(x,y) + (\phi_1(x,y){-}P(x,y))H_j(x,0){+} (\phi_2(x,y){-}P(x,y)) H_j(0,y)}{(2\pi i)^2 \, x^{k_1+1}y^{k_2+1}(1{-}P(x,y))} dx dy 
 \end{multline} 
for any $(\hat{x}, \hat{y})\in{\Gamma}$ with $\hat{x} < x_d$ and $\hat{y} < y_d$. Next we show that the set $\{(x,y)\in \inter{D} : x < x_d, y < y_d\}$ is non empty. By the definition of the set ${\Gamma}$, the set $\inter{D}$ is included to the set ${\Gamma}$, with this result  we will be able to consider the integral representation \eqref{eq1-integral-representation-lemma}  with $\hat{x} < x_d$, $\hat{y} < y_d$ such that  $(\hat{x}, \hat{y})\in \inter{D}$.

To prove the first tree assertions of Theorem~\ref{theorem4},  the integral representation \eqref{eq1-integral-representation-lemma} is next modified in the following way:~ since the functions $(x,y) \to L_j(x,y)$ and $(x,y) \to (1-P(x,y))^{-1}$ are  analytic in the polycircular set $\Omega(\inter{D})$, and, by Theorem~\ref{theorem1}, 
\begin{itemize}
\item  the function $(x,y) {\to} (\phi_1(x,y){-}P(x,y))H_j(x,0)$ is analytic in the polycircular set $\{(x,y){\in}\Omega(\inter{D}){:} |x|{<}x_d\}$, 
\item  the function $(x,y) {\to} (\phi_2(x,y){-}P(x,y)) H_j(0,y)$ is analytic in the polycircular set  $\{(x,y){\in}\Omega(\inter{D}){:}|y|{<}y_d\}$,
\end{itemize}
and since the set $\Omega(\inter{D})$ does not contain zeros of the function $(x,y) {\to} x^{k_1+1} y^{k_2+1} (1{-}P(x,y))$,  we can write 
\be\label{eq-decomposition}
g(j,k) = I_0(j,k) + I_1(j,k) + I_2(j,k)
\ee
with 
\be\label{eqI0-intergral-representation}
I_0(j,k) = \frac{1}{(2\pi i)^2} \int_{|x|=\hat{x}_0} \int_{|y|=\hat{y}_0} \frac{L_j(x,y)}{x^{k_1+1}y^{k_2+1}(1-P(x,y))} dx\, dy, 
\ee
\be\label{eqI1-intergral-representation}
I_1(j,k)  = \frac{1}{(2\pi i)^2} \int_{|x|=\hat{x}_1} \int_{|y|=\hat{y}_1} \frac{(\phi_1(x,y)-P(x,y))H_j(x,0) }{x^{k_1+1}y^{k_2+1}(1-P(x,y))} dx\, dy 
\ee
and 
\be\label{eqI2-intergral-representation}
I_2(j,k)  = \frac{1}{(2\pi i)^2} \int_{|x|=\hat{x}_2} \int_{|y|=\hat{y}_2} \frac{(\phi_2(x,y)-P(x,y)) H_j(0,y)}{x^{k_1+1}y^{k_2+1}(1-P(x,y))} dx\, dy 
\ee
for any $(\hat{x}_0, \hat{y}_0) \in \inter{D}$, $(\hat{x}_1,\hat{y}_1)\in \{(x,y)\in \inter{D}: x < x_d\}$, and $(\hat{x}_2,\hat{y}_2)\in \{(x,y)\in \inter{D} : y < y_d\}$. 

To prove the first assertion of Theorem~\ref{theorem4}, it is sufficient to get the asymptotic behavior \eqref{simple-pole-x-asymptotics} when $\min\{k_1,k_2\}\to+\infty$ and  $w_k = k/\|k\|\to w$ for any $w\in {\mathcal W}_1$. In order to get this result, we identify the asymptotic behavior of $I_1(j,k)$ by using  the residue theorem (applied first for the integral with respect to $x$ and next for the integral with respect to $y$), and using next 
 large deviation estimates of  $I_0(j,k)$ and $I_2(j,k)$ we prove that the terms $I_0(j,k)$ and $I_3(j,k)$ are negligible with respect to $I_1(j,k)$. 

The proof of the second assertion  of Theorem~\ref{theorem4} is exactly the same as the proof of the first assertion, it is sufficient to exchange the roles of the first and the second coordinates of the points $j,k\in\Z^2_+$. 

To prove the  assertions (iii){-}(v) of our theorem,  we show that  the term $I_0(j,k)$ is negligible with respect to $I_1(j,k)+ I_2(j,k)$ and we identify (in the same way as in the proof of the first assertion of our theorem) the asymptotic behavior of each term $I_1(j,k)$ and $I_2(j,k)$. 

Finally, the last assertion of Theorem~\ref{theorem4}, is  obtained as a  consequence of Proposition~1 of the paper~\cite{Ignatiouk-2023-cone}. 

The proof of Theorem~\ref{theorem4} is organized as follows: 

 The integral representations \eqref{eq1-integral-representation-lemma} and \eqref{eq-decomposition} of the Green function $g(j,k)$ are obtained respectively in Lemma~\ref{integral-representation-lemma1} and Corollary~\ref{integral-representation-cor1} of Section~\ref{integral-representation-section}. 

Section~\ref{sec-large-deviation-estimates} is devoted to the upper large deviation estimates for the integrals $I_0(j,k)$, $I_1(j,k)$ and $I_2(j,k)$. These large deviation estimates will be used in order to identify the dominant terms of the sum $I_0(j,k){+}I_0(j,k){+}I_0(j,k)$ of the right-hand side of~\eqref{eq-decomposition}. 

In Section~\ref{section-exact-asymptotics} we obtain exact asymptotics of $I_1(j,k)$ (resp $I_2(j,k)$) as $\min\{k_1,k_2\}\to+\infty$ and $k/\|k\|\to w$ for those $w=(u,v)\in{\mathbb S}^1_+$ for which $u > u_D(x_d, Y_2(x_d))$ (resp. for which $v > v_D(X_2(y_d), y_d)$). This is a subject of Lemma~\ref{exact-asymptotics-I2-lemma} and Lemma~\ref{exact-asymptotics-I1-lemma} below. Remark that we do not need the exact asymptotics of $I_1(j,k)$ (resp $I_2(j,k)$)  when   $k/\|k\|\to w=(u,v)\in{\mathbb S}^1_+$ and $u \leq u_D(x_d, Y_2(x_d))$ (resp. $v \leq v_D(X_2(y_d), y_d)$): the large deviation asymptotics obtained in Section~\ref{sec-large-deviation-estimates} will be in this case sufficient. 

In Section~\ref{proof-first-assertion-theorem4} the proof of the first assertion of our theorem is completed. It will be proved there that for any $w\in {\mathcal W}_1$, the inequality $u > u_D(x_d, Y_2(x_d))$ holds, and consequently that the results of Section~\ref{section-exact-asymptotics} provide the exact asymptotic for $I_1(j,k)$ as  $\min\{k_1,k_2\}\to+\infty$ and $k/\|k\|\to w\in{\mathcal W}_1$. A comparison of the exact asymptotics for $I_1(j,k)$ as  $\min\{k_1,k_2\}\to+\infty$ and $k/\|k\|\to w\in{\mathcal W}_1$ with the asymptotics of $I_2(j,k)$ and $I_0(j,k)$ will show that the terms $I_0(j,k)$ and $I_2(j,k)$ are negligible with respect to $I_1(j,k)$. 

The second assertion of Theorem~\ref{theorem4} is obtained by using the arguments of the symmetry: to get this statement, the same arguments as in the proof of the first assertion can be applied if one exchanges the roles of $x$ and $y$. 

In Section~\ref{proof-third-assertions-theorem4} and~\ref{proof-four-fifth-assertions-theorem4}, we complete the proof of the assertions (iii){-}(v) of Theorem~\ref{theorem4}. It will be shown in these cases,  the exact asymptotics of $I_1(j,k)$ and $I_2(j,k)$ are given by the results of Section~\ref{section-exact-asymptotics} and that  the term $I_0(j,k)$ is negligible with respect to the terms $I_1(j,k)$ and $I_2(j,k)$. The exact asymptotic of the Green function $g(j,k)$ will be obtained  by comparing the exact asymptotics of  $I_1(j,k)$ and $I_2(j,k)$.

\subsection{Integral representation of the coefficients $g(j,k)$}\label{integral-representation-section} 
 
 \begin{lemma}\label{integral-representation-lemma1} Under the assumptions (A1){-}(A4), the set $\{(x,y)\in\inter{D}: x < x_d, \; y< y_d\}$ is non empty and for any $(\hat{x}, \hat{y})\in \inter{D}$ such that $\hat{x} < x_d$ and $\hat{y} < y_d$, \eqref{eq1-integral-representation-lemma}  holds. 
  \end{lemma} 
 \begin{proof} By the first assertion of Theorem~\ref{theorem1}, the series \eqref{series-h} (and consequently also the series \eqref{eq-generating-functions})   converge on  the polycircular set $\{(x,y){\in}\Omega({\Gamma}):  |x|{<}x_d, |y|{<} y_d\}$. The function $(x,y){\to}H_j(x,y)$ is therefore analytic in
$\{(x,y){\in} \Omega({\Gamma}):  |x|{<}x_d, |y| {<} y_d\}$  and  for any $(\hat{x},\hat{y}){\in}{\Gamma}$ with $\hat{x}{<}x_d$ and $\hat{y}{<} y_d$, one has 
 \be\label{eq1-integral-representation-lemma1} 
 g(j,k) = \frac{1}{(2\pi i)^2} \int_{|x|=\hat{x}} \int_{|y|=\hat{y}} \frac{H_j(x,y)}{x^{k_1+1}y^{k_2+1}} dx\, dy,  \quad \forall j, k{\in}\Z^2_+ 
  \ee
 Remark now that under our assumptions, the set $\{(x,y)\in \inter{D} : x < x_d, y < y_d\}$ is non empty:~ 
 \begin{itemize}
 \item when one of the assertions (B0){-}(B2) is valid, this is a consequence of the first assertion of Lemma~\ref{lemma1-proof-theorem1} 
 \item when one of the assertions (B3) or (B4) holds, the set $\{(x,y)\in \inter{D} : x < x_d, y < y_d\}$ is non empty because in this case and $x^*_P < x_d = x^{**} < x^{**}_P$, $y_d > Y_1(x_d)$  and by Lemma~\ref{preliminary-lemma1}, any point of the line segment $[(x_d, Y_1(x_d)), (x_d, Y_2(x_d))]$, aside of the ends points  $(x_d, Y_1(x_d))$ and $(x_d, Y_2(x_d))$ belongs to the interior $\inter{D}$ of the set $D$; 
 \item similarly, when one of the assertions (B5) or (B6) holds, the set $\{(x,y){\in} \inter{D} : x {<} x_d, y {<} y_d\}$ is non empty because $y^*_P {<} y_d{=}y^{**}{<}y^{**}_P$, $x_d > X_1(y_d)$ and by Lemma~\ref{preliminary-lemma1}, any point of line segment $[(X_1(y_d), y_d), (X_2(y_d), y_d)]$, aside of the end points $(X_1(y_d), y_d)$ and  $(X_2(y_d), y_d)$ belongs to the interior of the set $D$. 
 \end{itemize} 
By the definition of the set ${\Gamma}$, the set $\inter{D}$ is included to ${\Gamma}$,  it follows that \eqref{eq1-integral-representation-lemma1} holds also for any $(\hat{x}, \hat{y})\in\inter{D}$ with $\hat{x} < x_d$ and $\hat{y} < y_d$. 

Furthermore, by the second assertion of Theorem~\ref{theorem1}, for any $j\in\Z^2_+$, on the set $\{(x,y){\in}\Omega({\Gamma}) :~ |x|{<}x_d, \;  |y|{<}y_d\}$, the function $h_j$ satisfies  the identity \eqref{extended-functional-equation}. Since the set $\{(x,y)\in \Omega(\inter{D}) : x < x_d, y < y_d\}$  has an empty intersection with the sets $\{(x,y)\in\C^2: x = 0\}$ and $\{(x,y)\in\C^2:~y=0\}$, it follows that on the set $\{(x,y){\in}\Omega(\inter{D}) : x {<} x_d, y {<} y_d\}$, for any $j\in\Z^2_+$,  the relation 
\begin{align*}
(1-P(x,y))(H_j(x,y) - H_j(x,0) - &H_j(0,y)) = Q(x,y) h_j(x,y) \\
&= L_j(x,y) + \psi_1(x,y) h_{1j}(x) + \psi_2(x,y) h_{2j}(y) \\
&= L_j(x,y) + (\phi_1(x,y) - 1) H_j(x,0) + (\phi_2(x,y)-1) H_j(0,y),
\end{align*} 
holds, or equivalently,
\[
(1-P(x,y))H_j(x,y) = L_j(x,y) + (\phi_1(x,y) - P(x,y)) H_j(x,0) + (\phi_2(x,y)-P(x,y)) H_j(0,y).
\]
Since the function $(x,y)\to 1/(1-P(x,y))$ is analytic in the set $\{(x,y)\in \C^2:~(|x|,|y|)\in\inter{D}\}$,  this implies that, on  the set $\{(x,y)\in\C^2:~(|x|, |y|)\in \inter{D}, \; |x| < x_d, \; |y| < y_d\}$, we have also the identity 
\[
H_j(x,y) = \frac{L_j(x,y) + (\phi_1(x,y)-P(x,y))H_j(x,0) + (\phi_2(x,y)-P(x,y)) H_j(0,y)}{1-P(x,y)}.
\]
The last identity combined with \eqref{eq1-integral-representation-lemma1} proves \eqref{eq1-integral-representation-lemma}. 
 \end{proof} 

Since the functions $(x,y) \to L_j(x,y)$ and $(x,y) \to x^{-k_1-1} y^{k_2-1}(1-P(x,y))^{-1}$ are  analytic in the polycircular set $\Omega(\inter{D})$, and by Theorem~\ref{theorem1}, 
 the functions $(x,y) \to (\phi_1(x,y)-P(x,y))H_j(x,0)$ and $(x,y) \to (\phi_2(x,y)-P(x,y)) H_j(0,y)$  are analytic respectively on $\{(x,y){\in} \Omega(\inter{D}): |x|{<}x_d\}$ and  $\{(x,y)\in \Omega(\inter{D}): \, |y| < y_d\}$, as a straightforward consequence of Lemma~\ref{integral-representation-lemma1} we obtain 

\begin{cor}\label{integral-representation-cor1} Under the assumptions (A1){-}(A4), for any $j{\in}\Z^2_+$ and $k{=}(k_1,k_2){\in}\Z^2_+{\setminus}\{(0,0)\}$, 
 relation \eqref{eq-decomposition} holds with $I_0(j,k)$,  $I_1(j,k)$ and  $I_2(j,k)$ defined respectively by \eqref{eqI0-intergral-representation}, \eqref{eqI1-intergral-representation} and \eqref{eqI2-intergral-representation}, for any  $(\hat{x}_0, \hat{y}_0) {\in} \inter{D}$, $(\hat{x}_1,\hat{y}_1){\in} \{(x,y){\in}\inter{D}: x {<} x_d\}$, and $(\hat{x}_2,\hat{y}_2){\in} \{(x,y)\in \inter{D} : y {<} y_d\}$.
\end{cor}

\subsection{Large deviation estimates of the quantities $I_0(j,k)$, $I_1(j,k)$ and  $ I_2(j,k)$}\label{sec-large-deviation-estimates}

\begin{lemma}\label{large-deviation-estimates} Under the assumptions (A1){-}(A4), for any  $w=(u, v)\in{\mathbb S}^1_+$, 
and $j, k\in\Z^2_+$, as $\|k\|\to \infty$ and $k/\|k\|\to w$, 
\be\label{eq1-large-deviation-estimates}
\limsup_{k}  \|k\|^{-1} \log \left|  I_0(j,k)\right| \leq - \max_{(\hat{x}, \hat{y})\in D} \Bigl( u\ln(\hat{x}) + v \ln(\hat{y}) \Bigr), 
\ee 
\be\label{eq2b-large-deviation-estimates-I1}
\limsup_{k} \|k\|^{-1} \log \left| I_1(j,k)\right|  \leq - \max_{(\hat{x}, \hat{y})\in D, \, \hat{x} \leq x_d} \!\Bigl( u\ln(\hat{x}) + v \ln(\hat{y}) \Bigr).
\ee
and
\be\label{eq2b-large-deviation-estimates-I2}
\limsup_{k} \|k\|^{-1} \log \left| I_2(j,k)\right|  \leq - \max_{(\hat{x}, \hat{y})\in D, \, \hat{y} \leq y_d} \!\Bigl( u\ln(\hat{x}) + v \ln(\hat{y}) \Bigr).
\ee
\end{lemma} 
\begin{proof} 
Indeed,  the definition of $I_2(j,k)$ and  Cauchy's inequality give,  for any $j, k=(k_1,k_2)\in\Z^2_+$ and $(\hat{x},\hat{y})\in\inter{D}$ with $\hat{y} < y_d$, 
\[
|I_2(j,k) | \leq M(\hat{x}, \hat{y}) \hat{x}^{-k_1}\hat{y}^{-k_2}, 
\]
with 
\[
M(\hat{x}, \hat{y}) = \max\{|(\phi_2(x,y) - P(x,y)) H_j(0, y)(1-P(x,y))^{-1}|:~ (x,y)\in\C^2, |x| = \hat{x}, |y| = \hat{y}\}.
\]
This proves that  as $\|k\|\to +\infty$ and $k/\|k\|\to w = (u,v)$, 
\begin{multline*}
\limsup_k  \|k\|^{-1} \log \left| I_2(j,k) \right| \leq  - \sup_{(\hat{x}, \hat{y})\in \inter{D}: \hat{y} < y_g} \bigl( u\log(\hat{x}) + v \log(\hat{y}) \bigr) \\ = - \max_{(\hat{x}, \hat{y})\in D, \, \hat{y} \leq y_d} \bigl( u\log(\hat{x}) + v \log(\hat{y}) \bigr)
\end{multline*}
where the second relation holds because the function $(x,y)\to u\ln x + v\ln y$ is continuous on $D$.  Relation \eqref{eq2b-large-deviation-estimates-I2} is therefore proved. 
 The proof of \eqref{eq1-large-deviation-estimates} and \eqref{eq2b-large-deviation-estimates-I1} is quite similar.
\end{proof}

\subsection{Exact asymptotic behavior of $I_1(jk)$ as $\min\{k_1,k_2\}\to \infty$}\label{section-exact-asymptotics}
To get the exact asymptotic of $I_1(jk)$ as $\min\{k_1,k_2\}\to \infty$ we consider first the following preliminary results. 

\begin{lemma}\label{preliminary-lemma2-interior-derections} Under the assumption~(A1), for any $w=(u,v)\in{\mathbb S}^1_+$, the function 
$
x \to u \ln(x) + v \ln(Y_2(x))
$
is strictly increasing in the line segment $[x^*_P, x_D(w)]$ and strictly decreasing in the line segment $[x_D(w), x^{**}_P]$.  
\end{lemma} 
\begin{proof} Indeed, recall that under our assumptions, the function $(\alpha,\beta)\to \tilde{P}(\alpha,\beta) = P(e^\alpha, e^\beta)$ is strictly convex and finite in a neighborhood of the set $\tilde{D} = \{(\alpha, \beta)\in\R^2:~P(e^\alpha, e^\beta) \leq 1\}$, and that for any $w=(u,v)\in{\mathbb S}^1$, the point $(\alpha_D(w), \beta_D(w) = (\ln(x_D(w)), \ln(y_D(w)))$ (see Definition~\ref{def-homeomorphism-sphere-boundary}) is an only point on the boundary of the set $\tilde{D}$, where the function $(\alpha,\beta)\to u \alpha + v \beta$ achieves its maximum over $\tilde{D}$. 

Consider the line segment $[\alpha^*_P, \alpha^{**}_P] = \{\alpha\in\R:~\inf_{\beta\in\R} \tilde{P}(\alpha, \beta) \leq 1\}$. We have, see the proof of Lemma~\ref{preliminary-lemma1},  $\alpha^*_P{=}x \ln(x^*_P)$, $\alpha^{**}_P {=} \ln(x^{**}_P)$, and that for any $\alpha \in ]\alpha^*_P, \alpha^{**}_P[$ and $\beta_2(\alpha) = \ln(Y_2(e^\alpha))$, we have 
\[
 \tilde{P}(\alpha, \beta_2(\alpha)) = 1 \quad \text{and} \quad \partial_\beta \tilde{P}(\alpha,\beta_2(\alpha)) > 0. 
\]
Since the function $\tilde{P}$ is strictly convex, by the implicit function theorem,  it follows that the function $\beta_2$ is also strictly convex on the line segment $[\alpha^*_P, \alpha^{**}_P]$.

Under our assumptions, the definition~\ref{def-homeomorphism-sphere-boundary} of the mapping $w{\to}(\alpha_D(w), \beta_D(w)){=}(\ln(x_D(w)), \ln(y_D(w)))$ gives that, if $w=(u,v)\in {\mathbb S}^1_+$, 
\[
 \partial_\beta \tilde{P}(\alpha_D(w), \beta_D(w)) \geq 0. 
\]
By the definition of the function $\alpha \to \beta_2(\alpha) = \ln(Y_2(e^\alpha))$, see the proof of Lemma~\ref{preliminary-lemma1},  it follows that, 
\[
\beta_D(w) = \beta_2(\alpha_D(w)),
\]
and that $\alpha_D(w)$ is an only point in the line segment $[\alpha^*_P, \alpha^{**}_P]$ where the function 
\be\label{function-boundary}
\alpha \to u\alpha + v \beta_2(\alpha) = u\alpha + v \ln(Y_2(e^\alpha)) 
\ee
achieves its maximum over $[\alpha^*_P, \alpha^{**}_P]$. Since the function $\beta_2$ is strictly convex on $[\alpha^*_P, \alpha^{**}_P]$, the function \eqref{function-boundary} is also strictly convex on $[\alpha^*_P, \alpha^{**}_P]$, and consequently, it is strictly increasing on the segment $[\alpha^*_P, \alpha_D(w)]$ and strictly decreasing on the line segment $[\alpha_D(w), \alpha^{**}_P]$. Since the function $x\to \ln(x)$ is strictly increasing on $]0,+\infty[$, this proves that the function 
$x\to u\ln(x) + v \ln(Y_2(x))$ is strictly increasing on the line segment $[x^*_P, x_D(w)]$ and strictly decreasing on the line segment $[x_D(w), x^{**}_P]$. 
\end{proof} 

 \begin{lemma} \label{preliminary-lemma3-interior-derections} If condition (A1) is satisfied and let, for some $x_0\in]x^*_P, x^{**}_P[$, $\eps_1 > 0$ and $\eps_2 > 0$, a function $(x,y)\to F(x,y)$ be analytic in  the  polycircular set $$\{(x,y)\in\C^2:~ |x_0{-}|x|| < \eps_1, |Y_2(x_0){-} |y|| < \eps_2\}$$ and do not vanish at the point  $(x_0, Y_2(x_0))$.  Then for any $w=(u,v)\in{\mathbb S}^1_+$ such that $x_D(w) > x_0$ and $(\hat{x},\hat{y})\in \inter{D}$ such that $x_0 - \eps_1 < \hat{x} < x_0$ and  $|Y_2(x_0) - \hat{y}| < \eps_2$,     as $\min\{k_1, k_2\}\to+\infty$ and $k/\|k\|\to w$, 
 \be\label{eq3-preliminary-lemma-interior-derections} 
\frac{1}{(2\pi i)^2} \int_{|x|=\hat{x}}\int_{\|y\| = \hat{y}} \frac{F(x,y)}{x^{k_1}y^{k_2}(x_0-x) (1-P(x,y))} \, dx \, dy  ~\sim~  \frac{C_1}{x_0^{k_1} (Y_2(x_0))^{k_2}} 
 \ee 
 with
 \be\label{eq-def-C1-preliminary-lemma3-interior-derections}
 C_1 = \left.\frac{F(x, y)}{\partial_y P(x, y)} \right|_{(x,y)=(x_0,Y_2(x_0))} \not= 0
 \ee
 and
 \begin{multline}\label{eq3b-preliminary-lemma-interior-derections} 
\frac{1}{(2\pi i)^2} \int_{|x|=\hat{x}}\int_{\|y\| = \hat{y}} \frac{F(x,y)}{x^{k_1}y^{k_2}(x_0-x)^2 (1-P(x,y))} \, dx \, dy\\ ~\sim~ \frac{k_1 C_1}{x_0^{k_1+1}(Y_2(x_0))^{k_2}}  - \frac{k_2 {C}_2}{x_0^{k_1}(Y_2(x_0))^{k_2+1}} 
 \end{multline}
 with $C_1$ given by \eqref{eq-def-C1-preliminary-lemma3-interior-derections} and 
 \be\label{eq-def-C2-preliminary-lemma3-interior-derections}
 C_2 = \left.\frac{F(x,y) \partial_x P(x, y)}{\left(\partial_y P(x, y)\right)^2} \right|_{(x,y)=(x_0,Y_2(x_0))}. 
 \ee
  \end{lemma} 
 \begin{proof} Before proving this lemma, remark that under our assumptions, by Lemma~\ref{preliminary-lemma1},  $Y_2(x_0) > Y_1(x_0)$ and consequently, without any restriction of generality, we can assume throughout our proof that 
 \[\eps_2 < Y_2(x_0) - Y_1(x_0).\]
Because of Assumption~(A1), the function $(x,y)\to 1 - P(x,y)$ is analytic in a neighborhood ${\mathcal V}$ of the set $\Omega(D)$  and does 
not vanishes in its interior $\Omega(\inter{D})$. Hence, for any $m\in\N^*$, the function 
\[
(x,y) \to  F(x,y)(x_0-x)^{-m} x^{-k_1}y^{-k_2} (1-P(x,y))^{-1} 
\] 
is analytic in the polycircular sets $\{(x,y)\in\Omega(\inter{D}), \, x_0 - \eps_1 < |x| < x_0, \, |Y_2(x_0) - y| < \eps_2\}$ and  $\{(x,y)\in\Omega(\inter{D}), \, x_0 < |x| < x_0 +\eps_1, \, |Y_2(x_0) - y| < \eps_2\}$, and consequently,  the function 
\[
(\hat{x}, \hat{y}) \to J_{m,k}(\hat{x}, \hat{y}) = \frac{1}{(2\pi i)^2} \int_{|x|=\hat{x}}\int_{\|y\| = \hat{y}} \frac{F(x,y)\, dx \, dy }{x^{k_1}y^{k_2}(x_0-x)^m (1-P(x,y))} 
\]
is constant on the set  $A_- = \{(x,y)\in\inter{D}:~ x_0 - \eps_1 < x < x_0, \, |Y_2(x_0) - y| < \eps_2\}$ and on the set $A_+ = \{(x,y)\in\inter{D}:~ x_0 < x < x_0 +\eps_1, \, |Y_2(x_0) - y| < \eps_2\}$. We denote  
\begin{align*}
J_m(k) &=  J_{m,k}(\hat{x}, \hat{y}) \quad \text{ for $(\hat{x}, \hat{y})\in A_-$},\\
 \hat{J}_m(k) &=  J_{m,k}(\hat{x}, \hat{y}) \quad \text{ for $(\hat{x}, \hat{y})\in A_+$.} 
\end{align*} 
Remark now that  for any $\hat{y}\in]Y_2(x_0)-\eps_2, Y_2(x_0)[$ and $ \delta > 0$ small enough, the point $(x_0-\delta, \hat{y})$ belongs to the set $A_-$,  the point $(x_0+\delta, \hat{y})$ belongs to the set $A_+$ and by the residue theorem, for any $y\in\C$ with $|y|=\hat{y}$, 
\[
\int_{|x|= x_0+\delta }\! \frac{F(x,y) \, dx}{x^{k_1}(x_0-x) (1-P(x,y))}  = \int_{|x|= x_0-\delta } \,\frac{F(x,y)\, dx}{x^{k_1}(x_0-x) (1-P(x,y))}  - 2\pi i \frac{F(x_0,y)}{x_0^{k_1}(1-P(x_0,y))}
\]
and
\begin{multline*}
\int_{|x|= x_0+\delta } \frac{F(x,y) \, dx}{x^{k_1}(x_0-x)^2 (1-P(x,y))}  = \int_{|x|= x_0-\delta } \,\frac{F(x,y)\, dx}{x^{k_1}(x_0-x)^2 (1-P(x,y))}  \\ 
+ 2\pi i \left( \frac{\partial_x F(x_0,y)}{x_0^{k_1}(1-P(x_0,y))} - \frac{k_1F(x_0,y)}{x_0^{k_1+1}(1-P(x_0,y))}  + \frac{F(x_0,y)\partial_xP(x_0,y)}{x_0^{k_1}(1-P(x_0,y))^2} \right) 
\end{multline*}  
Hence, for any $\hat{y}\in]Y_2(x_0)-\eps_2, Y_2(x_0)[$, 
\be\label{eq4-preliminary-lemma-interior-derections} 
J_1(k) = \hat{J}_1(k) + \frac{1}{2\pi i \, x_0^{k_1}} \int_{\|y\| = \hat{y}} \frac{F(x_0,y)}{ y^{k_2} (1-P(x_0,y))} dy
\ee
and
\begin{multline}\label{eq4b-preliminary-lemma-interior-derections}
J_2(k)  = \hat{J}_2(k) - \frac{1}{2\pi i x_0^{k_1}} \int_{\|y\| = \hat{y}} \frac{\partial_x F(x_0,y)}{y^{k_2}(1-P(x_0,y))} dy \\ + \frac{k_1}{(2\pi i)x_0^{k_1+1}} \int_{\|y\| = \hat{y}} \frac{F(x_0,y)}{y^{k_2}(1-P(x_0,y))} dy  - \frac{1}{2\pi i x_0^{k_1}} \int_{\|y\| = \hat{y}} \frac{F(x_0,y)\partial_xP(x_0,y)}{y^{k_2}(1-P(x_0,y))^2} dy. 
\end{multline} 
Due to  Assumption (A1)(ii), the function $y\to P(x_0,y)$ is analytic in a neighborhood of the closed annulus $\{y\in\C:~Y_1(x_0)\leq |y| \leq Y_2(x_0)\}$, satisfies the inequality 
\[
|P(x_0, y)| \leq P(x_0, |y|) < 1
\]
on the annulus $\{y\in\C:~Y_1(x_0) < |y| < Y_2(x_0)\}$, and because of Assumption~(A1)~i) (we use here Proposition P7.5 of~\cite{Spitzer}), for any $y\in\C$ with $|y| = Y_2(x_0)$, 
\[
|P(x_0, y)| < P(x_0, Y_2(x_0)) = 1 \quad \text{whenever $y\not= Y_2(x_0)$}. 
\]
Since  by Lemma~\ref{preliminary-lemma1}, 
\[
 \left.\partial_y P(x,y)\right|_{(x,y)=(x_0,Y_2(x_0))} > 0,
\]
this proves that for some $\delta > 0$, the point $Y_2(x_0)$ is an only and simple zero of the function $y\to P(x_0,y)$ in the annulus $\{y\in\C:~Y_1(x_0) < |y| < Y_2(x_0)+\delta)$. Since  we assumed that $Y_2(x_0)-\eps_2 < Y_1(x_0)$, and the functions $y\to F(x_0, y)$, \, $y\to \partial_x F(x_0,y)$, $y\to F(x_0,y)$ are analytic in the annulus $\{y\in\C:~Y_2(x_0) -\eps_2 < |y| < Y_2(x_0)+\eps_2\}$, this implies that  for $\delta > 0$ small enough, the functions 
\[
y\to \frac{F(x_0,y)}{y^{k_2} (1-P(x_0,y))} \quad \text{and} \quad y\to \frac{\partial_x F(x_0,y)}{y^{k_2} (1-P(x_0,y))} 
\]
are analytic in the set $\{y\in\C:~Y_2(x_0) - \eps_2 < |y| < Y_2(x_0)+\delta, \, y\not= Y_2(x_0)\}$  and have at the point $Y_2(x_0)$, a simple pole with the residue equal respectively to 
\[
\frac{- C_1}{(Y_2(x_0))^{k_2}} \quad \text{and} \quad  \frac{- \tilde{C}_1}{(Y_2(x_0))^{k_2}}.
\]
with $C_1$ given by \eqref{eq-def-C1-preliminary-lemma3-interior-derections} and 
\[
\tilde{C}_1 = \left.\frac{ \partial_x F(x,y)}{\partial_y P(x,y)}\right|_{(x,y)=(x_0,Y_2(x_0))}.
\]
Similarly,  the function 
\[
y\to \frac{F (x_0,y)\partial_xP(x_0,y)}{y^{k_2}(1-P(x_0,y))^2}
\]
is analytic in the set $\{y\in\C:~Y_2(x_0) - \eps_2 < |y| < Y_2(x_0)+\delta, \, y\not= Y_2(x_0)\}$  and has at the point $Y_2(x_0)$, a  pole of the second order with the residue equal to
\[
 \frac{C}{(Y_2(x_0))^{k_2} } - \frac{k_2 C_2}{(Y_2(x_0))^{k_2+1}}
\]
with $C_2$ given by \eqref{eq-def-C2-preliminary-lemma3-interior-derections} and some constant $C\in \C$ does not depending on $k_2$. By the residue theorem,  it follows  that for any $Y_2(x_0) - \eps_2 < \hat{y} < Y_2(x_0)$  and $Y_2(x_0) < y_\delta < Y_2(x_0) + \delta)$ 
\[
\int_{\|y\| = \hat{y}} \frac{F(x_0,y)}{ y^{k_2} (1-P(x_0,y))} dy =   \frac{2\pi i C_1}{(Y_2(x_0))^{k_2}}   + \int_{\|y\| = y_\delta } \frac{F(x_0,y)}{ y^{k_2} (1-P(x_0,y))} dy, 
\]
\[ 
\int_{\|y\| = \hat{y}} \frac{\partial_x F(x_0,y)}{ y^{k_2} (1-P(x_0,y))} dy =    \frac{ 2\pi i\tilde{C}_1}{(Y_2(x_0))^{k_2}} + \int_{\|y\| = y_\delta } \frac{\partial_x F(x_0,y)}{ y^{k_2} (1-P(x_0,y))} dy 
\] 
and 
\[ 
 \int_{\|y\| = \hat{y}} \frac{F(x_0,y)\partial_xP(x_0,y)}{y^{k_2}(1-P(x_0,y))^2} dy  ~=~ - 2\pi i \left(\frac{C}{(Y_2(x_0))^{k_2} } - \frac{k_2 C_2}{(Y_2(x_0))^{k_2+1}}\right)
 + \int_{\|y\| = y_\delta} \frac{F(x_0,y)\partial_xP(x_0,y)}{y^{k_2}(1-P(x_0,y))^2} dy  
\]
Using these relations at the right hand side of \eqref{eq4-preliminary-lemma-interior-derections} and \eqref{eq4b-preliminary-lemma-interior-derections} we obtain 
\be\label{eq5-preliminary-lemma-interior-derections} 
J_1(k) = \frac{C_1}{x_0^{k_1}(Y_2(x_0))^{k_2}}  + \frac{1}{(2\pi i)} \int_{\|y\| = y_\delta } \frac{F(x_0,y)}{x_0^{k_1} y^{k_2} (1-P(x_0,y))} dy + \hat{J}_1(k)  
\ee 
and
\begin{align}
J_2(k) &= \frac{k_1 C_1}{x_0^{k_1+1}(Y_2(x_0))^{k_2}} - \frac{k_2 C_2}{x_0^{k_1}(Y_2(x_0))^{k_2+1}} + \frac{\tilde{C}}{x_0^{k_1}(Y_2(x_0))^{k_2} }  \label{eq4c-preliminary-lemma-interior-derections}  \\
& - \frac{1}{2\pi i x_0^{k_1}} \int_{\|y\| = y_\delta} \frac{\partial_x F(x_0,y)}{y^{k_2}(1-P(x_0,y))} dy + \frac{k_1}{(2\pi i)x_0^{k_1+1}} \int_{\|y\| = y_\delta} \frac{F(x_0,y)}{y^{k_2}(1-P(x_0,y))} dy \nonumber\\ & - \frac{1}{2\pi i x_0^{k_1}} \int_{\|y\| = y_\delta} \frac{F(x_0,y)\partial_xP(x_0,y)}{y^{k_2}(1-P(x_0,y))^2} dy + \hat{J}_2(k) \nonumber
\end{align} 
with some constant $\tilde{C}$ does not depending on $k$. Remark now  that  by the Cauchy inequality,  for $k=(k_1,k_2)\in\Z^2_+$, as $\min\{k_1,k_2\} \to\infty$ and $k/\|k\|\to w=(u,v)$,  
\begin{align*} 
\limsup_k \|k\|^{-1} \log \left| \frac{1}{(2\pi i)} \int_{\|y\| = y_\delta } \frac{F(x_0,y)}{x_0^{k_1} y^{k_2} (1-P(x_0,y))} dy \right| &\leq - \bigl(u \log(x_0) + v \log(y_\delta) \bigr) \\ 
& \hspace{1cm} <  - (u \ln(x_0) + v \ln( Y_2(x_0))),
\end{align*} 
and with the same argument as in the proof of Lemma~\ref{large-deviation-estimates},  
\[
\limsup_k \|k\|^{-1} \log \bigl| \hat{J}_1(k) \bigr| \leq - \sup_{x_0 \leq \hat{x} \leq x_0+\eps_1}  \bigl(u \ln \hat{x} + v\ln(Y_2(\hat{x}))\bigr) < - (u \ln(x_0) + v \ln( Y_2(x_0)))
\]
where the last relation holds because under our assumptions, $x^*_P < x_0 < x_D(w)$ and by Lemma~\ref{preliminary-lemma2-interior-derections}, the function $x \to u\ln x + v \ln(Y_2(x))$ is strictly increasing on $[x^*_P, x_D(w)]$. Since for $k=(k_1,k_2)\in\Z^2_+$, as $\|k\|\to\infty$ and $k/\|k\|\to w$, 
\[
\lim \|k\|^{-1} \log \left|\frac{F(x_0, Y_2(x_0))}{x_0^{k_1}(Y_2(x_0))^{k_2} \partial_y P(x_0, Y_2(x_0))} \right| = - \bigl(u \log(x_0) + v \log(Y_2(x_0)) \bigr)
\]
this proves that as  $\min\{k_1, k_2\}\to+\infty$ and $k/\|k\|\to w$, the second and the third terms in the right hand side of \eqref{eq5-preliminary-lemma-interior-derections} are negligible with respect to the first one.  Hence \eqref{eq3-preliminary-lemma-interior-derections} is verified, and with the same arguments, from \eqref{eq4c-preliminary-lemma-interior-derections} one gets \eqref{eq3b-preliminary-lemma-interior-derections}.
 \end{proof} 

As a consequence of Lemma~\ref{preliminary-lemma3-interior-derections} and Theorem~\ref{theorem2}, we obtain the following result.
\begin{lemma}\label{exact-asymptotics-I1-lemma} Under the assumptions (A1){-}(A3),  for any $j\in\Z^2_+{\setminus} E_0$  and $w=(u,v)\in{\mathbb S}^1_+$ with  $u > u_D(x_d, Y_2(x_d))$, the following assertions hold 

 i) \, If either one of the assertions (B0), (B1), (B3),(B4) holds or (B2) holds with  $x_d < x^{**}_P$, then for any $k=(k_1,k_2)\in\Z^2_+$, as $\min\{k_1,k_2\}\to \infty$ and $k/\|k\|\to w$, 
\be\label{eq1-exact-asymptotics-I1-lemma}
I_1(j,k) ~\sim~ {c}_1 \, a_1 \, \varkappa_1(j) x_d^{- k_1-1}(Y_2(x_d))^{-k_2-1} 
\ee
where 
\be\label{eq1b-exact-asymptotics-I1-lemma}
c_1 = \left. \frac{\phi_1(x, y) - 1}{\partial_y P(x,y))}\right|_{(x,y)=(x_d,Y_2(x_d))}  > 0
\ee
 and $a_1 > 0$ is given by \eqref{eq2b-cases-B0-B4-theorem2}. 

 ii) \,  If (B5) holds with  $x_d < x^{**}_P$, then for  $k=(k_1,k_2)\in\Z^2_+$, as $\min\{k_1,k_2\}\to \infty$ and $k/\|k\|\to w$, 
\be\label{eq2-exact-asymptotics-I1-lemma}
I_1(j,k) ~\sim~ \frac{c_1 \, a_3 \varkappa_2(j) k_1}{x_d^{k_1+2}(Y_2(x_d))^{k_2+1}} + \frac{\tilde{c}_1 \, a_3 \varkappa_2(j)k_2}{x_d^{k_1+1}(Y_2(x_d))^{k_2+2}}
\ee
where $c_1 > 0$ is given by \eqref{eq1b-exact-asymptotics-I1-lemma}, 
\[ 
\tilde{c}_1 = \left. \frac{ \partial_x P(x_d, Y_2(x_d))}{ \left(\partial_y P(x_d, Y_2(x_d))\right)^2}  \right|_{(x,y)=(x_d,Y_2(x_d))} 
\]
and $a_3 > 0$ is given by \eqref{eq-analytical-structure-IVb}.

 iii) \, If (B6) holds, then for  $k=(k_1,k_2)\in\Z^2_+$, as $\min\{k_1,k_2\}\to \infty$ and $k/\|k\|\to w$, 
\be\label{eq3-exact-asymptotics-I1-lemma}
I_1(j,k) ~\sim~ {c}_1 \, a_5 \, \varkappa_2(j) x_d^{- k_1-1}(Y_2(x_d))^{-k_2-1} 
\ee
where  $c_1 > 0$ is given by \eqref{eq1b-exact-asymptotics-I1-lemma} and $a_5 > 0$ is given by \eqref{eq-analytical-structure-IIIb}. 
\end{lemma} 
\begin{proof} Indeed,  when either, one of the conditions (B0), (B1), (B3) or (B4) is satisfied, or (B2) holds with  $x_d < x^{**}_P$, by the first assertion of Theorem~\ref{theorem2} and using Proposition~\ref{harmonic-functions}, one gets that for some $\eps > 0$, the function $x\to H_j(x,0)$ can be extended as an analytic function to the set $B(0, x_d+\eps){\setminus}\{x_d\}$ and has a simple pole at the point $x_d$ with the residue 
\[
-a_1 \varkappa_1(j) =  - \varkappa_1(j) \left(\frac{d}{dx}\phi_1(x_d,Y_1(x_d))\right)^{-1} < 0. 
\]
 Since the function $(x,y)\to \phi_1(x,y)-P(x,y)$ is analytic in a neighborhood of the set $\Omega(D)$, this implies  that the function $(x,y) \to F(x,y) = (x_d-x) (\phi_1(x,y)-P(x,y)) H_j(x,0)$ can be extended as an analytic function to a neighborhood of the set $\{(x,y)\in \Omega(D): \, |x| < x_d + \eps\}$ by letting 
 \[F(x_d, y) = (\phi(x_d,y) - P(x_d,y)) a_1 \varkappa_1(j). 
 \]
When either, one of the conditions (B0), (B1), (B3) or (B4) is satisfied, or (B2) holds with  $x_d < x^{**}_P$, one has always $x^*_p < x_d = x^{**} < x^{**}_P$, and consequently, by Lemma~\ref{preliminary-lemma1}, $Y_2(x_d) > Y_1(x_d)$ and by Lemma~\ref{preliminary_cor2}, $\phi_1(x_d, Y_1(x_d))=1$.  Since the function $y\to \phi_1(x_d, y)$ is strictly increasing,  it follows that $\phi_1(x_d, Y_2(x_d))  > 1$. Since by Theorem~\ref{theorem2}, $a_1 > 0$, we obtain therefore 
\begin{align*}
F(x_d, Y_2(x_d)) &= (\phi_1(x_d, Y_2(x)) - P(x_d, Y_2(x_d)) a_1 \varkappa_1(j) = (\phi_1(x_d, Y_2(x)) - 1) a_1 \varkappa_1(j) > 0. 
\end{align*} 
By Lemma~\ref{preliminary-lemma3-interior-derections},  it follows that for any  $w=(u,v)\in{\mathbb S}^1_+$ such that  $x_D(u) > x_d$, and as $\min\{k_1, k_2\}\to +\infty$, \eqref{eq1-exact-asymptotics-I1-lemma} holds with $c_1>0$ given by \eqref{eq1b-exact-asymptotics-I1-lemma}. 
Since for $w=(u,v)\in{\mathbb S}^2_+$ the inequality $x_D(u) > x_d$ is equivalent to the inequality $u > u_D(x_d, Y_2(x_d))$,  the first assertion of our lemma is therefore proved.

Suppose now that (B5) holds with $x_d < x^{**}_P$. Then,  by Theorem~\ref{theorem2} and using Proposition~\ref{harmonic-functions}, one gets that for some $\eps > 0$, the function $x\to H_j(x,0)$ can be extended as an analytic function to the set $B(0, x_d+\eps){\setminus}\{x_d\}$ and has a pole of the second order at the point $x_d$ with 
\[
\lim_{x\to x_0} (x_d-x)^2 H_j(x,0) = a_3 \varkappa_2(j) > 0
\]
where $a_3 > 0$ is given by \eqref{eq-analytical-structure-IVb}. 
With the same arguments as above,  it follows that the function $(x,y)\to F(x,y) = (x_d-x)^2 (\phi_1(x,y)-1) H_j(x,0)$ can be extended as an analytic function to a neighborhood of the set $\{(x,y)\in \Omega(D): \, |x| < x_d + \eps\}$ by letting $F(x_d, y) = (\phi_1(x_d,y)-1)  a_3 \varkappa_2(j)$ and that the extended function $F$
satisfies the assumptions of Lemma~\ref{preliminary-lemma3-interior-derections} with $x_0 = x_d$ and $F(x_0,Y_2(x_0)) = a_3 \varkappa_2(j)(\phi_1(x_d, Y_2(x_d))-1) > 0$. Hence, using \eqref{eq3b-preliminary-lemma-interior-derections} we obtain \eqref{eq2-exact-asymptotics-I1-lemma} with $c_1>0$ given by \eqref{eq1b-exact-asymptotics-I1-lemma}. The second assertion of our lemma is therefore also proved. 

The proof of the third assertion of our lemma is exactly the same as the proof the  first assertion, with an only difference that now, one should use the fifth assertion of Theorem~\ref{theorem2} instead of the first one. 
\end{proof} 

Remark finally that if we exchange the roles of $x$ and $y$, then with the same arguments as above we obtain the following result. 
\begin{lemma}\label{exact-asymptotics-I2-lemma} Under the assumptions (A1){-}(A3),  for any $j\in\Z^2_+{\setminus} E_0$  and $w=(u,v)\in{\mathbb S}^1_+$ with  $v > v_D(X_2(y_d), y_d)$, the following assertions hold 
\begin{enumerate}[label=\roman*)]
 \item \, If either one of the assertions (B0), (B1), (B5),(B6) holds or (B2) holds with  $y_d < y^{**}_P$, then for any $k=(k_1,k_2)\in\Z^2_+$, as $\min\{k_1,k_2\}\to \infty$ and $k/\|k\|\to w$, 
\be\label{eq1-exact-asymptotics-I2-lemma}
I_2(j,k) ~\sim~ {c}_2 \, b_2 \varkappa_2(j) (X_2(y_d))^{- k_1-1}(y_d)^{-k_2-1} 
\ee
where 
\be\label{eq1b-exact-asymptotics-I2-lemma}
c_2 = \left.(\phi_2(x,y) - 1) \left(\partial_x P(x,y) \right)^{-1}\right|_{(x,y)=(X_2(y_d),y_d)} > 0 
\ee
and 
\be\label{eq1bb-exact-asymptotics-I2-lemma}
 a_2 = \left.\left(\frac{d}{dy}\phi_2(X_1(y), y) \right)^{-1} \right|_{y=y_d} 
\ee

 \item \,  If (B3) holds with  $y_d < y^{**}_P$, then for  $k=(k_1,k_2)\in\Z^2_+$, as $\min\{k_1,k_2\}\to \infty$ and $k/\|k\|\to w$, 
\be\label{eq2-exact-asymptotics-I2-lemma}
I_2(j,k) ~\sim~ \frac{c_2 \, b_3 \varkappa_1(j) k_2}{(X_2(y_d))^{k_1+1}y_d^{k_2+2}} + \frac{\tilde{c}_2 \, b_3 \varkappa_1(j)k_1}{(X_2(y_d))^{k_1+2}{y_d}^{k_2+1}}
\ee
with $c_2 > 0$ given by \eqref{eq1b-exact-asymptotics-I2-lemma},
\be\label{eq2b-exact-asymptotics-I2-lemma}
a_3' =  \left.(\phi_1(x, y)-1) \left(\frac{d}{dy} \phi_2(X_1(y), y)\frac{d}{dx} \phi_1(x, Y_1(x)) \frac{d}{dy} X_1(y)\right)^{-1} \right|_{(x,y)=(X_2(y_d),y_d)} > 0
\ee
and  some $\tilde{c}_2 \in\R$ do not depending on $j\in\Z^2_+{\setminus} E_0$. 

 \item \, If (B4) holds, then for  $k=(k_1,k_2)\in\Z^2_+$, as $\min\{k_1,k_2\}\to \infty$ and $k/\|k\|\to w$, 
\be\label{eq3-exact-asymptotics-I2-lemma}
I_2(j,k) ~\sim~ {c}_2 \, a'_5 \varkappa_1(j) (X_2(y_d))^{- k_1-1}(y_d)^{-k_2-1} 
\ee
with $c_2$ given by \eqref {eq1b-exact-asymptotics-I2-lemma} and 
\be\label{eq3b-exact-asymptotics-I2-lemma}
a'_5 =     \left.(\phi_1(x,y)-1)\left((1 - \phi_2(x, y))\frac{d}{dx} \phi_1(x, Y_1(x))  \frac{d}{dy} X_1(y)\right)^{-1} \right|_{(x,y)=(X_2(y_d),y_d)} > 0
\ee
\end{enumerate}
\end{lemma} 

\subsection{Proof of the assertions i) and ii) of Theorem~\ref{theorem4}}\label{proof-first-assertion-theorem4}
For $I_1(j,k)$ the exact asymptotics as $\min\{k_1,k_2\}\to+\infty$ and $k/\|k\| \to w=(u,v) \in {\mathbb S}^1_+$ were obtained only in the case where $u {>} u_D(x_d,Y_2(x_d))$. In the following lemma, we show that this inequality always holds when $w\in{\mathcal W}_1$.  The definitions \ref{Wdef-B0}{-}\ref{Wdef-B3-B6} give that the set ${\mathcal W}_1$ is empty if either, one of the cases (B5), or (B6) occurs or, (B2) and $x_d{=}x^{**}_P$ hold, we can assume that either one of cases (B0), (B1), (B3) or (B4) occurs or (B2) and $x_d < x^{**}_P$ hold. 

\begin{lemma}\label{large-deviation-estimates-lemma2} If conditions (A1){-}(A3) hold and if either, one of cases (B0), (B1), (B3) or (B4) holds, or (B2) and $x_d{<} x^{**}_P$ hold, then for any $w\in{\mathcal W}_1$, 
\be\label{eq0a-large-deviation-estimates-lemma2} 
u > u_D(x_d, Y_2(x_d)). 
\ee
\end{lemma} 
\begin{proof} Suppose first that (B0) holds. Then by Proposition~\ref{cases} and the definition of the points $x_d$ and $y_d$ (see \eqref{eq2-def-xd} and \eqref{eq2-def-yd}), 
\be\label{eq1-large-deviation-estimates-lemma2-B0} 
X_1(y_d) < x_d =  x^{**} < X_2(y_d), \quad Y_1(x_d) < y_d = y^{**} < Y_2(x_d). 
\ee
In this case ${\mathcal W}_1 = \{w=(u,v)\in{\mathbb S}^1_+: u > u_c\}$ where $w_c=(u_c,v_c)$ is the only point in ${\mathbb S}^1_+$ such that 
\be\label{eq1a-large-deviation-estimates-lemma2-B0}
u_c \ln(x_d) + v_c \ln(Y_2(x_d)) = u_c \ln(X_2(y_d)) + v_c \ln(y_d),
\ee
or equivalently, such that 
\be\label{eq1b-large-deviation-estimates-lemma2-B0}
(\ln(X_2(y_d)) - \ln(x_d)) u_c = (\ln(Y_2(x_d)) - \ln(y_d)) v_c.
\ee
Since  by \eqref{eq1-large-deviation-estimates-lemma2-B0} , $X_2(y_d) > x_d$ and $Y_2(x_d) > y_d$, the last relation implies that 
$u_c > 0$ and $ v_c >0$, and consequently, if $u_D(x_d, Y_2(x_d)) \leq 0$, one gets 
\[
u_D(x_d, Y_2(x_d)) \leq u_c < u, \quad \forall \; w=(u,v)\in{\mathcal W}_1. 
\]
When (B0) holds and $u_D(x_d, Y_2(x_d)) \leq 0$, our lemma is therefore proved. 

Consider now the case when (B0) holds and $u_D(x_d, Y_2(x_d)) > 0$. For $\hat{w} = (\hat{u},\hat{v}) = (u_D(x_d, Y_2(x_d)), v_D(x_d, Y_2(x_d)))$, by the definition of the mapping $w \to (x_D(w), y_d(w))$, the point $(x_d, Y_2(x_d))$ is the only point in the set $D$ where the function $(x,y)\to \hat{u}\ln(x)  + \hat{v}\ln(y) $ achieves its maximum over the set $D$. Since in the case (B0), $(x_d, Y_2(x_d))\not= (X_2(y_d),y_d)$,  it follows that for $\hat{w} = (\hat{u},\hat{v}) = (u_D(x_d, Y_2(x_d)), v_D(x_d, Y_2(x_d)))$, 
\[
\hat{u} \ln(x_d) + \hat{v} \ln(Y_2(x_d)) >  \hat{u} \ln(X_2(y_d)) + \hat{v} \ln(y_d)
\]
or equivalently that 
\[
(\ln(X_2(y_d)) - \ln(x_d)) \hat{u} < (\ln(Y_2(x_d)) - \ln(y_d)) \hat{v}.
\]
Since $X_2(y_d) > x_d$ and $Y_2(x_d) > y_d$, and since we assumed that $\hat{u} = u_D(x_d, Y_2(x_d)) > 0$, the last inequality shows that $\hat{v}= v_D(x_d, Y_2(x_d)) > 0$ and consequently, $\hat{v} = \sqrt{1- \hat{u}^2}$ and 
\[
(\ln(X_2(y_d)) - \ln(x_d)) \hat{u} < (\ln(Y_2(x_d)) - \ln(y_d)) \sqrt{1- \hat{u}^2}.
\]
Finally, we have already proved that $u_c > 0$ and $v_c > 0$, and  in this case, relation~\eqref{eq1b-large-deviation-estimates-lemma2-B0} is equivalent to 
\[
(\ln(X_2(y_d)) - \ln(x_d)){u}_c = (\ln(Y_2(x_d)) - \ln(y_d)) \sqrt{1- {u}_c^2}.
\]
Since  the function $u\to (\ln(X_2(y_d)) - \ln(x_d)){u}$ is strictly increasing and the function  $u\to (\ln(Y_2(x_d)) - \ln(y_d)) \sqrt{1- {u}_c^2}$ is decreasing on the segment $[0,1]$, the last two relations show that $u_c > \hat{u} = u_D(x_d, Y_2(x_d))$, and consequently using the definition of ${\mathcal W}_1$ one gets $u > u_c > u_D(x_d, Y_2(x_d))$ for all $u\in{\mathcal W}_1$. If (B0) holds, Lemma~\ref{large-deviation-estimates-lemma2} is therefore proved. 

When either, (B1) holds, or (B2) and $x_d{<}x^{**}_P$ hold, relation~\eqref{eq0a-large-deviation-estimates-lemma2} follows directly from the definition of the set ${\mathcal W}_1$ (see Definition~\ref{Wdef-B1} and Definition~\ref{Wdef-B2}). 

Consider now the case when  (B3) holds. With  definition~\ref{Wdef-B3-B6}, ${\mathcal W}_1=\{w=(u,v)\in{\mathbb S}^1_+ :~u > 0\}$ and by Proposition~\ref{cases} and the definition of the points $x_d$ and $y_d$ (see \eqref{eq2-def-xd} and \eqref{eq2-def-yd}), one has $(x_d, Y_2(x_d)) \in {\mathcal S}_{12}$.  By Lemma~\ref{preliminary-lemma1} and the definition of the mapping $(x,y)\to w_D(x,y)$,  it follows that 
\[
 u_D(x_d, Y_2(x_d)) = \frac{x_d\partial_x P(x_d,y_d)}{\sqrt{(x_d\partial_x P(x_d,y_d))^2 + (y_d\partial_y P(x_d,y_d))^2}}  \leq 0. 
\]
and consequently, for any  $w=(u,v)\in{\mathcal W}_1$, \eqref{eq0a-large-deviation-estimates-lemma2} holds.

Suppose now that (B4) holds. Then, with the Definition~\ref{Wdef-B3-B6}, ${\mathcal W}_1 = {\mathbb S}^1_+$ and by Proposition~\ref{cases} and the definition of the points $x_d$ and $y_d$ (see \eqref{eq2-def-xd} and \eqref{eq2-def-yd}) and using \eqref{eq-critical-points} one gets 
\be\label{eq0-proof-theorem3-(B4)} 
 X_1(y_d) = x_d =  x^{**}  < x^{**}_P, \quad   y^*_P < y^* \leq 1 < y_d = Y_2(x_d) <  y^{**} \leq y^{**}_P,   \quad (x_d,  y_d) \in {\mathcal S}_{12}.
\ee
By Lemma~\ref{preliminary-lemma1} and  the definition of the mapping $(x,y)\to w_D(x,y)$,  it follows that  for any $w=(u,v)\in{\mathcal W}_1$, 
\[
u_D(x_d, Y_2(x_d) = u_D(x_d,y_d)  = u_D(X_1(x_d), y_d) = \frac{x_d\partial_x P(x_d,y_d)}{\sqrt{(x_d\partial_x P(x_d,y_d))^2 + (y_d\partial_y P(x_d,y_d))^2}} < 0 \leq u, 
\]
and consequently, \eqref{eq0a-large-deviation-estimates-lemma2} holds.
\end{proof} 

As a straightforward consequence of Lemma~\ref{exact-asymptotics-I1-lemma} and Lemma~\ref{large-deviation-estimates-lemma2}, one gets 

\begin{cor}\label{cor-exact-asymptotics-I1} Under the assumptions of Theorem~\ref{theorem4}, for any $j\in\Z^2_+{\setminus} E_0$ and $w\in{\mathcal W}$, as $\min\{k_1,k_2\}\to+\infty$ and $k/\|k\|\to w$,  \eqref{eq1-exact-asymptotics-I1-lemma} holds. 
\end{cor} 

Now we will show that for any $w=(u,v)\in{\mathcal W}$, as $\min\{k_1,k_2\}\to+\infty$ and $k/\|k\|\to w$, the terms $I_0(j,k)$ and $I_2(j,k)$ are negligible with respect to $I_1(j,k)$. For this we will use the large deviation estimates for $I_0(k,j)$, and 
\begin{itemize}
\item[--] the large deviation estimates of $I_2(j,k)$  when 
\[
\limsup_k \frac{1}{\|k\|} \ln I_2(j,k) < \lim \frac{1}{\|k\|} \ln I_1(j,k), 
\]
\item[--] the exact asymptotic of $I_2(j,k)$  when 
\[
\limsup_k \frac{1}{\|k\|} \ln I_2(j,k) = \lim \frac{1}{\|k\|} \ln I_1(j,k). 
\]
\end{itemize} 
To compare the limit $\lim_k \ln I_1(j,k)/\|k\|$ with the large deviation estimates for $\limsup_k \ln I_0(j,k)/\|k\|$ and $\limsup_k\ln I_2(j,k)/\|k\|$, the following lemma will be useful. 
\begin{lemma}\label{large-deviation-estimates-lemma3} If conditions (A1){-}(A3) hold and let $w{=}(u,v){\in}{\mathcal W}_1$, if  one of the assertions holds,
\begin{itemize} 
\item one of the conditions (B0), (B1) or (B3) holds,
\item (B2) and $x^{**}{<}x^{**}_P$ hold,
\item (B4) holds and $u > 0$,
\end{itemize} 
then 
\be\label{eq0-large-deviation-estimates-lemma2} 
\max_{(x,y)\in D, \, {y} \leq y_d} \!\Bigl( u\ln({x}) + v \ln({y}) \Bigr)  >  u\ln(x_d) +  v\ln(Y_2(x_d)). 
\ee
\end{lemma} 
 \begin{proof}  Consider first the case when (B0) holds. We have, see the proof of Lemma~\ref{large-deviation-estimates-lemma3}, $X_2(y_d) > x_d$ and $Y_2(x_d) > y_d$, the point $(w_c=(u_c,v_c)\in{\mathbb S}^2_+$ satisfies \eqref{eq1b-large-deviation-estimates-lemma2-B0} and   ${\mathcal W}_1 = \{w=(u,v)\in{\mathbb S}^1_+: u > u_c\}$.  For $w=(u,v)\in{\mathbb S}^1_+$ with $u > u_c$ (and consequently also with $v < v_c$),  it follows that  
\[
(\ln(X_2(y_d)) - \ln(x_d)) u >  (\ln(Y_2(x_d)) - \ln(y_d)) v
\]
or equivalently,
\be\label{eq2-large-deviation-estimates-lemma2-B0} 
u\ln(X_2(y_d)) + v\ln(y_d)  > u\ln(x_d) + v \ln(Y_2(x_d)). 
\ee
Since the point $(X_2(y_d), y_d)$ belongs to the set $\{(x,y)\in D:~y\leq y_d\}$, this proves that for any $w=(u,v)\in{\mathbb S}^1_+$ with $u > u_c$, 
\be\label{eq3-large-deviation-estimates-lemma2-B0} 
\max_{({x}, {y})\in D, \, {y} \leq y_d} \!\Bigl( u\ln({x}) + v \ln({y}) \Bigr) \geq u\ln(X_2(y_d)) + v\ln(y_d)  > u\ln(x_d) + v \ln(Y_2(x_d)),
\ee
and consequently, \eqref{eq0-large-deviation-estimates-lemma2} holds. 

\medskip 

Consider now the case when (B1) holds, we have the relation
\[
{\mathcal W}_1 = \{w=(u,v)\in{\mathbb S}^1_+:~u > u_D(x_d, Y_2(x_d))\}, 
\]
and,  by Proposition~\ref{cases} and the definition of the points $x_d$ and $y_d$ (see \eqref{eq2-def-xd} and \eqref{eq2-def-yd}), 
\be\label{eq4-large-deviation-estimates-lemma2} 
X_1(y_d) < x_d =  x^{**} = X_2(y_d), \quad Y_1(x_d) < y_d = y^{**} = Y_2(x_d) 
\ee
and  $(x_d, Y_2(x_d)) = (X_2(y_d), y_d) \in {\mathcal S}_{22}$. Hence, for any $w=(u,v)\in{\mathcal W}_1$, 
\be\label{eq5-large-deviation-estimates-lemma2} 
x_D(w) > x_d = X_2(y_d) \quad \; \text{and} \; \quad y_D(w) < y_d = Y_2(x_d).  
\ee
The last relations show that the point $(x_D(w), y_D(w))$ belongs to the set $\{(x,y)\in D:~y\leq y_d\}$ and is not equal to $(x_d, Y_2(x_d))$. Since the point $(x_D(w), y_D(w))$ is an only point in $D$ where the function $(x,y)\to u\ln(x) + x \ln(y)$ achieves its maximum over $D$,  it follows that 
\be\label{eq6-large-deviation-estimates-lemma2} 
\max_{(\hat{x}, \hat{y})\in D, \, \hat{y} \leq y_d} \!\Bigl( u\ln(\hat{x}) + v \ln(\hat{y}) \Bigr) =  u\ln(x_D(w)) + v\ln(y_D(w))  > u\ln(x_d) + v \ln(Y_2(x_d)), 
\ee
and consequently, when the assertion (B1) holds, relation \eqref{eq0-large-deviation-estimates-lemma2}  is also proved.

Consider now the case when (B2) and $x_d < x^{**}_P$ hold. By Proposition~\ref{cases} and  the definition of the points $x_d$ and $y_d$ (see \eqref{eq2-def-xd} and \eqref{eq2-def-yd}), 
\be\label{eq1-large-deviation-estimates-lemma2-B2} 
 x_d =  x^{**} > X_2(y_d), \quad  y_d = y^{**} > Y_2(x_d),   \quad (x_d, Y_2(x_d)), \, (X_2(y_d), y_d) \in {\mathcal S}_{22}, 
\ee
and with  the definition of ${\mathcal W}_1$,
\[
{\mathcal W}_1 = \{w=(u,v)\in{\mathbb S}^1_+:~ u > u_D(x_d, Y_2(x_d))\}. 
\]
Hence, in this case, for any $w=(u,v)\in{\mathcal W}_1$, 
\[
x_D(w) > x_d > X_2(y_d) \quad \; \text{and} \; \quad y_D(w) < Y_2(x_d) < y_d. 
\]
These relations show that the point $(x_D(w), y_D(w))$ belongs to the set $\{(x,y)\in D:~y\leq y_d\}$ and is not equal to $(x_d, Y_2(x_d))$, and consequently, using exactly the same arguments as in the previous case, one gets \eqref{eq0-large-deviation-estimates-lemma2}.

Suppose now that (B3) and $y_d < y^{**}_P$ hold. By Proposition~\ref{cases} and  the definition of the points $x_d$ and $y_d$ (see \eqref{eq2-def-xd} and \eqref{eq2-def-yd}), we have
\be\label{eq1-large-deviation-estimates-lemma2-B3a} 
 X_1(y_d) = x_d =  x^{**}  < x^{**}_P, \quad  y_d = y^{**} = Y_2(x_d),   \quad (x_d,  y_d) = (x_d, Y_2(x_d)) \in {\mathcal S}_{12} 
\ee
and 
\[
{\mathcal W}_1 = \{w=(u,v)\in{\mathbb S}^1_+:~ u > 0\}.
\]
Since we assume that $y_d < y^{**}_P$ and by \eqref{eq-critical-points}, $y^{**} > y^* \geq y^{*}_P$,  by Lemma~\ref{preliminary-lemma1} and the definition of the mapping $(x,y)\to w_D(x,y) = (u_D(x,y), v_D(x,y))$, from \eqref{eq1-large-deviation-estimates-lemma2-B3a} it follows that 
\be\label{eq2-large-deviation-estimates-lemma2-B3a} 
x_d = X_1(y_d) < X_2(y_d) \quad \text{and} \quad u_D(x_d, Y_2(x_d)) = u_D(x_d, y_d)  < 0. 
\ee
Hence, in this case, for any $w=(u,v)\in{\mathcal W}_1$,  
\begin{multline}\label{eq4-large-deviation-estimates-lemma2-B3a} 
 u\ln(x_d) + x\ln(Y_2(x_d)) = u\ln(x_d) + v\ln(y_d)   <  u\ln(X_2(y_d))  + v\ln(y_d) \\ \leq \max_{({x}, {y})\in D, \, {y} \leq y_d} \bigl( u\ln({x}) + v \ln({y}) \bigr)
\end{multline}
 and consequently, \eqref{eq0-large-deviation-estimates-lemma2} holds.

Suppose now that (B3) holds and $y_d = y^{**}_P$. Then ${\mathcal W}_1 = \{w=(u,v)\in{\mathbb S}^1_+:~ u > 0\}$ and \eqref{eq1-large-deviation-estimates-lemma2-B3a} holds,   but now, instead of \eqref{eq2-large-deviation-estimates-lemma2-B3a}  one has 
\[
x_d = X_1(y_d) = X_2(y_d), \quad y_d = y^{**} = y^{**}_P \quad \text{and} \quad \quad w_D(x_d, y_d) = (0,1). 
\]
Hence, in this case, for any $w=(u,v)\in{\mathcal W}_1$, 
\be\label{eq5-large-deviation-estimates-lemma2-B3a} 
x_D(w) > x_d \quad \; \text{and} \; \quad y_D(w) < y_d.
\ee
These relations show that the point $(x_D(w), y_D(w))$ belongs to the set $\{(x,y)\in D:~y\leq y_d\}$ and is not equal to $(x_d, Y_2(x_d))$  and consequently, using exactly the same arguments as in the  case (B2), one gets \eqref{eq0-large-deviation-estimates-lemma2}.

Consider finally the case when (B4) holds. In this case, ${\mathcal W}_1 = {\mathbb S}^1_+$ and by Proposition~\ref{cases} and  the definition of the points $x_d$ and $y_d$ (see \eqref{eq2-def-xd} and \eqref{eq2-def-yd}) 
\be\label{eq00-proof-theorem3-(B4)} 
 X_1(y_d) = x_d =  x^{**}  < x^{**}_P, \quad   y^* \leq 1 < y_d = Y_2(x_d) <  y^{**} \leq y^{**}_P,   \quad (x_d,  y_d) \in {\mathcal S}_{12}.
\ee
Hence,  in this case, \eqref{eq2-large-deviation-estimates-lemma2-B3a} holds and consequently, using  exactly the same arguments as above, one gets \eqref{eq0-large-deviation-estimates-lemma2} for any $w=(u,v)\in{\mathcal W}_1$ with $u > 0$. 
\end{proof}

\bigskip

{\em Now we are ready to complete the proof of the first assertion of Theorem~\ref{theorem4}}: Since the right hand side of \eqref{simple-pole-x-asymptotics} does not depend on $w \in{\mathcal W}_1$, it is sufficient to show that for any $j\in\Z^2_+{\setminus} E_0$ and $w=(u,v)\in{\mathcal W}_1$,   \eqref{simple-pole-x-asymptotics}  holds as $\min\{k_1,k_2\}\to+\infty$ and $k/\|k\|\to w$. 

Consider first the case when one of the following assertion holds
\begin{itemize} 
\item one of the conditions (B0), (B1) or (B3) holds 
\item (B2) and $x^{**} < x^{**}_P$ hold,
\item (B4) holds and $u > 0$. 
\end{itemize} 
By  using  Lemma~\ref{large-deviation-estimates} and Lemma~\ref{large-deviation-estimates-lemma3}, as $\|k\|\to+\infty$ and $k/\|k\|\to w$, one gets 
\begin{align} 
\limsup_k \|k\|^{-1} \log \left|  I_0(j,k) + I_2(j,k) \right|& \leq - \max_{(\hat{x}, \hat{y})\in D, \, \hat{y} \leq y_d} \Bigl( u\ln(\hat{x}) + v \ln(\hat{y}) \Bigr) \nonumber \\
&< - u\ln(x_d) +  v\ln(Y_2(x_d)). \label{eq7-proof-theorem5-(B0)}
\end{align}
By Corollary~\ref{cor-exact-asymptotics-I1}, for any $w=(u,v)\in{\mathcal W}_1$ and $j\in\Z^2_+{\setminus} E_0$,  as $\min\{k_1,k_2\}\to+\infty$ and $k/\|k\|\to w$, 
\be\label{eq6-proof-theorem5-(B0)} 
I_1(j,k) ~\sim~ {c}_1 \, a_1 \, \varkappa_1(j) x_d^{- k_1-1}(Y_2(x_d))^{-k_2-1} 
\ee
where $c_1 > 0$ is given by \eqref{eq1b-exact-asymptotics-I1-lemma} and  $a_1 > 0$ is given by \eqref{eq2b-cases-B0-B4-theorem2}. 

Comparison of \eqref{eq6-proof-theorem5-(B0)}  with \eqref{eq7-proof-theorem5-(B0)} shows that  the terms $I_0(j,k)$ and $I_2(j,k)$ in \eqref{eq-decomposition}  are negligible with respect to $I_1(j,k)$, and consequently, from the integral representation \eqref{eq-decomposition} and using \eqref{eq6-proof-theorem5-(B0)}  one gets \eqref{simple-pole-x-asymptotics}. 

\medskip
The set of directions ${\mathcal W}_1$ is empty in each of the following cases:
\begin{itemize}
\item (B2) and $x_d{=} x^{**}_P$ hold; 
\item (B5) or (B6) holds. 
\end{itemize} 
\noindent 
Hence, to complete the proof of the first assertion of Theorem~\ref{theorem4} it is sufficient now to prove \eqref{simple-pole-x-asymptotics}   when (B4) holds and $w{=}(0,1)$. By Proposition~\ref{cases} and  the definition of the points $x_d$ and $y_d$ (see \eqref{eq2-def-xd} and \eqref{eq2-def-yd}) one has 
\be\label{eq10-proof-theorem4}
X_1(y_d) = x_d =  x^{**}  < x^{**}_P, \quad   y^* \leq 1 < y_d = Y_2(x_d) <  y^{**} \leq y^{**}_P,   \quad (x_d,  y_d) \in {\mathcal S}_{12}.
\ee
and  the definition of the point $y^{**}_P$ and the mapping $w\to (x_D(w), y_D(w))$ give, for $w=(0,1)$, 
\be\label{eq11-proof-theorem4}
(x_D(0,1), y_D(0,1)) = (X_1(y^{**}_P), y^{**}_P). 
\ee
Hence, in this case, $y_D(0,1) = y^{**}_P > y_d$, and consequently, also $1 > v_D(X_2(y_d), y_d)$. By Lemma~\ref{exact-asymptotics-I2-lemma} applied with $w=(0,1)$, it follows that there is $C > 0$ such that for any $j\in\Z^2_+{\setminus} E_0$ and $k=(k_1,k_2)\in\Z^2_+$, and as $\min\{k_1,k_2\}\to+\infty$, $k/\|k\|\to (0,1)$ 
\be\label{eq30-exact-asymptotics-I2-lemma}
I_2(j,k) ~\sim~ {C} \, \varkappa_1(j) (X_2(y_d))^{- k_1-1}(y_d)^{-k_2-1}.  
\ee
With Lemma~\ref{large-deviation-estimates} applied with $w=(0,1)$, the definition of the point $y^{**}_P$, and using relation~\eqref{eq10-proof-theorem4}, 
\be\label{eq6000-proof-theorem5-(B0)} 
\limsup_k \|k\|^{-1} \ln \bigl| I_0(j,k)\bigr| \leq - \max_{(x,y)\in D} \ln(y) = - \ln(y^{**}_P) < - \ln(y_d).
\ee
Remark that by relation~\eqref{eq10-proof-theorem4} and Corollary~\ref{cor-exact-asymptotics-I1} applied with  $w=(0,1)\in{\mathcal W}_1$, for any  $j\in\Z^2_+{\setminus} E_0$,  as $\min\{k_1,k_2\}\to+\infty$ and $k/\|k\|\to (0,1)$, 
\be\label{eq60-proof-theorem5-(B0)} 
I_1(j,k) ~\sim~ {c}_1 \, a_1 \, \varkappa_1(j) x_d^{- k_1-1}(Y_2(x_d))^{-k_2-1} = {c}_1 \, a_1 \, \varkappa_1(j) x_d^{- k_1-1} y_d^{-k_2-1} 
\ee
where $c_1 > 0$ is given by \eqref{eq1b-exact-asymptotics-I1-lemma} and  $a_1 > 0$ is given by \eqref{eq2b-cases-B0-B4-theorem2}. 
 
A comparison of relation~\eqref{eq6000-proof-theorem5-(B0)} with relations~\eqref{eq30-exact-asymptotics-I2-lemma} and \eqref{eq60-proof-theorem5-(B0)}  shows that for any $j\in\Z^2_+{\setminus} E_0$ and $k=(k_1,k_2)\in\Z^2_+$, and as $\min\{k_1,k_2\}\to+\infty$, $k/\|k\|\to (0,1)$, the term $I_0(j,k)$ is negligible with respect to $I_1(j,k) + I_2(j,k)$ in \eqref{eq-decomposition}, and consequently, using~\eqref{eq30-exact-asymptotics-I2-lemma} and \eqref{eq60-proof-theorem5-(B0)} one gets 
\[
g(j,k) ~\sim~ I_1(j,k) + I_2(j,k) ~\sim~ {c}_1 \, a_1 \, \varkappa_1(j) x_d^{- k_1-1}y_d^{-k_2-1}  + {C} \, \varkappa_1(j) (X_2(y_d))^{- k_1-1}(y_d)^{-k_2-1}. 
\]
Since by \eqref{eq10-proof-theorem4}, $x_d = X_1(y_d) < X_2(y_d)$, this proves  \eqref{simple-pole-x-asymptotics} with $b_1 = c_1 a_1$.

The first assertion of Theorem~\ref{theorem4} is therefore proved. The proof of the second assertion is the same by exchanging the roles of $x$ and $y$.

\subsection{Proof of the assertion iii) of Theorem~\ref{theorem4}}\label{proof-third-assertions-theorem4}

Suppose first that (B0) holds. In this case, see the proof of Lemma~\ref{large-deviation-estimates-lemma2},
\be\label{eq0-proof-IV-theorem4}
u_c > u_D(x_d, Y_2(x_d)) \quad \text{and} \quad v_c > v_D(X_2(y_d), y_d). 
\ee
Consider $j\in\Z^2_+{\setminus} E_0$ and let $\min\{k_1,k_2\}\to+\infty$ and $k/\|k\|\to w_c$. Then using the first relation of \eqref{eq0-proof-IV-theorem4}, by  Lemma~\ref{exact-asymptotics-I1-lemma} applied for $w=w_c =(u_c,v_c)$,  we get 
\be\label{eq1aaa-proof-IV-theorem4}
I_1(j,k) ~\sim~ {c}_1 \, a_1 \varkappa_1(j) x_d^{- k_1-1}(Y_2(x_d))^{-k_2-1}, 
\ee
where $c_1 > 0$ and $a_1 > 0$  given respectively by \eqref{eq1b-exact-asymptotics-I1-lemma} and \eqref{eq2b-cases-B0-B4-theorem2}, and  using the second relation of \eqref{eq0-proof-IV-theorem4}, by Lemma~\ref{exact-asymptotics-I2-lemma} applied for $w=w_c =(u_c,v_c)$, we obtain 
\be\label{eq2bbb-proof-IV-theorem4}
I_2(j,k) ~\sim~ c_2 \, a_2 \varkappa_2(j) (X_2(y_d))^{-k_1-1} y_d^{- k_2-1} 
\ee
where $c_2 > 0$ and $a_2 >0$ are defined respectively by \eqref{eq1b-exact-asymptotics-I2-lemma} and \eqref{eq1bb-exact-asymptotics-I2-lemma}. 
Comparison of \eqref{eq1aaa-proof-IV-theorem4} and \eqref{eq2bbb-proof-IV-theorem4} with \eqref{eq1-large-deviation-estimates} shows that the term $I_0(j,k)$ in \eqref{eq-decomposition}  is negligible with respect to $I_1(j,k)+I_1(j,k)$, and consequently, using \eqref{eq-decomposition} together with \eqref{eq1aaa-proof-IV-theorem4} and \eqref{eq2bbb-proof-IV-theorem4}  one gets \eqref{competition-simple-poles-asymptotics} with $b_1 = c_1 a_1$ and $b_2 = c_2 a_2$. . 

\subsection{Proof of the assertions  iv)  and v) of Theorem~\ref{theorem4}} \label{proof-four-fifth-assertions-theorem4}
Suppose that (B3) holds, $y_d < y^{**}_P$ and let $w=(0,1)$. By Proposition~\ref{cases} and the definition of the points $x_d$ and $y_d$, one has $x^*_P < x_d = x^{**} < x^{**}_P$, $(x_d,y_d) = (x_d, Y_2(x_d)) = (X_1(y_d), y_d)\in{\mathcal S}_{12}$ and $y^*_P < y_d = y^{**} < y^{**}_P$.  By Lemma~\ref{preliminary-lemma1} and  the definition of the mapping $(x,y)\to w_D(x,y) = (u_D(x,y), v_D(x,y))$, it follows that 
\be\label{eq0-proof-V-theorem4} 
x_d = X_1(y_d) < X_2(y_d),  
\ee
\be\label{eq1-proof-V-theorem4} 
u_D(x_d, Y_2(x_d)) < 0 \quad \text{and} \quad v_D(X_2(y_d), y_d) < 1. 
\ee
Consider now $j\in\Z^2_+{\setminus} E_0$ and let $k\in\Z^2_+$, $\min\{k_1,k_2\}\to+\infty$ and $k/\|k\| \to w=(1,0)$. Then by by Lemma~\ref{large-deviation-estimates} and since $y_d < y^{**}_P$, one gets 
\be\label{eq2-proof-V-theorem4}
\limsup_k \|k\|^{-1} \ln \bigl| I_0(j,k)\bigr| \leq - \max_{(x,y)\in D} \ln(y) = - \ln(y^{**}_P) < - \ln(y_d),
\ee
 by Lemma~\ref{exact-asymptotics-I1-lemma}, from the first relation of  \eqref{eq1-proof-V-theorem4}  it follows 
 \be\label{eq3-proof-V-theorem4}
I_1(j,k) ~\sim~ {c}_1 \, a_1 \varkappa_1(j) x_d^{- k_1-1}(Y_2(x_d))^{-k_2-1} = {c}_1 \, a_1 \varkappa_1(j) x_d^{- k_1-1} y_d^{-k_2-1}  , 
\ee
where $c_1 > 0$ and $a_1 > 0$ are  given respectively by \eqref{eq1b-exact-asymptotics-I1-lemma} and  \eqref{eq2b-cases-B0-B4-theorem2}, 
and by the second assertion of Lemma~\ref{exact-asymptotics-I2-lemma}, from the second relation of \eqref{eq1-proof-V-theorem4}  it follows that there exists $\tilde{c}_2 \in\R$ do not depending on $j\in\Z^2_+{\setminus} E_0$ such that 
\be\label{eq4-proof-V-theorem4}
I_2(j,k) ~\sim~ \frac{c_2 \, a'_3 \varkappa_1(j) k_2}{(X_2(y_d))^{k_1+1}y_d^{k_2+2}} + \frac{\tilde{c}_2 \, a'_3 \varkappa_1(j)k_1}{(X_2(y_d))^{k_1+2}{y_d}^{k_2+1}}
\ee
where $c_2 > 0$ and $a_3' > 0$ are given respectively by  \eqref{eq1b-exact-asymptotics-I2-lemma}  and \eqref{eq2b-exact-asymptotics-I2-lemma}. 

Comparison of \eqref{eq2-proof-V-theorem4} with \eqref{eq3-proof-V-theorem4}  shows that the term $I_0(j,k)$ is negligible with respect to $I_1(j,k)+I_2(j,k)$ in \eqref{eq-decomposition}. Using therefore \eqref{eq-decomposition} together with \eqref{eq3-proof-V-theorem4} and \eqref{eq4-proof-V-theorem4}  we obtain 
\begin{align}
g(j,k) &~\sim~ I_1(j,k) + I_2(j,k) \nonumber\\
&~\sim~ \frac{{c}_1 \, a_1 \varkappa_1(j)}{ x_d^{k_1+1} y_d^{k_2+1}}  +  \frac{c_2 \, a'_3 \varkappa_1(j) k_2}{(X_2(y_d))^{k_1+1}y_d^{k_2+2}} + \frac{\tilde{c}_2 \, b_3 \varkappa_1(j)k_1}{(X_2(y_d))^{k_1+2}{y_d}^{k_2+1}} \label{eq4p-proof-V-theorem4} 
\end{align} 
Finally,  by relation~\eqref{eq0-proof-V-theorem4} we have $x_d{<}X_2(y_d)$, hence, in relation~\eqref{eq4p-proof-V-theorem4}, the last term 
is negligible with respect to the first one and consequently, relation~\eqref{competition-simple-double-poles-asymptotics} holds with $b_1 = c_1 a_1$ and $b_3 = c_2 \, a'_3$. 

The assertion  iv)  of Theorem~\ref{theorem4} is therefore also proved. The proof of the assertion  v)  is exactly the same (it is sufficient to exchange the roles of the first and the second coordinates of the random walk $(Z(n))$). 

\subsection{Proof of the assertion  vi)  of Theorem~\ref{theorem4}} 

To prove the last assertion of Theorem~\ref{theorem4} it is sufficient to show that for any $w\in{\mathcal W}_0$, \eqref{Woess-asymptotics} holds uniformly with respect to $w_m=m/\|m\|$ in some neighborhood of $w$. For this we use Proposition~1 of \cite{Ignatiouk-2023-cone}.  In our setting, this result gives the following lemma. 
\begin{lemma}\label{lemma-Woess} Suppose that the Assumption~(A1) is satisfied and let  $w=(u,v){\in}{\mathbb S}^1_+$ and  $\eps{>} 0$ and a function $(x,y){\to}F(x,y)$ be analytic in  the  polycircular set 
 \be\label{eq2-lemma-Woess}
 \{ (x,y)\in\C^2:~ \bigl||x|-x_D(w)\bigr| < \eps, \; \bigl||y|-y_D(w)\bigr| < \eps\},
 \ee
 and not vanishing at the point $(x_D(w), y_D(w))$. 
 Then  the integrals
 \be\label{eq3-lemma-Woess}
 I(m) = \frac{1}{(2\pi i)^2} \int_{|x| = \hat{x}} \int_{|y|= \hat{y}} \frac{F(x,y)}{x^{m_1+1}y^{m_2+1}(1-P(x,y))} \, dx dy, \quad  m=(m_1,m_2)\in\Z^2_+,
 \ee
 are well defined and does not depend on the point $(\hat{x}, \hat{y})$ on the set 
 \be\label{eq4-lemma-Woess}
 \{ (\hat{x},\hat{y})\in \inter{D}:~ |\hat{x} -x_D(w)| < \eps, \quad  |\hat{y} - y_D(w)| < \eps \},
 \ee
 and as $\|m\|\to+\infty$, uniformly with respect to $w_m=m/\|m\|$ in some neighborhood of $w$, 
 \be\label{eq3p-lemma-Woess}
 I(m) ~\sim~ \frac{F(x_D(w_k), y_D(w_k)) \sqrt{w_k^\perp \cdot {\mathcal Q}(w_k)w_k^\perp }}{\bigl(2\pi\|k\|)^{1/2} \|{\mathfrak m}(w_k) \|^{-1}(x_D(w_k))^{k_1} (y_D(w_k))^{k_2}}  \ee
 \end{lemma} 
 
 \bigskip 
By Corollary~\ref{integral-representation-cor1}, for any $(\hat{x}, \hat{y})\in\inter{D}$, with $\hat{x} < x_d$ and $\hat{y} < y_d$, we have 
 \be\label{eq0-proof-I-theorem4}
 g(j,k) =  \frac{1}{(2\pi i)^2} \int_{|x| = \hat{x}} \int_{|y|= \hat{y}} \frac{F_j(x,y)}{x^{k_1+1}y^{k_2+1}(1-P(x,y))} \, dx dy. 
 \ee
 with 
 \[
 F_j(x,y) = L_j(x,y) + (\phi_1(x,y)-P(x,y)) H_j(x,0) + (\phi_2(x,y)-P(x,y)) H_j(0,y). 
 \]
 Remark moreover that the set of directions ${\mathcal W}_0$ is non-empty if and only if (B2) holds, and that in this case, for any $w=(u,v)\in{\mathcal W}_0$, one has 
  \be\label{eq2-proof-I-theorem4}
  x_D(w) < x_d, \quad \quad y_D(w) < y_d.  
  \ee
 By Theorem~\ref{theorem1},  it follows that for some neighborhood $V(x_D(w), y_D(w))$ of the point $(x_D(w), y_D(w))$ in $\R^2$,  the function $(x,y) \to F_j(x,y) = L_j(x,y) + (\phi_1(x,y)-1) H_j(x,0) + (\phi_2(x,y)-1) H_j(0,y)$ is analytic in  the polycircular set $\{(x,y)\in\C^2:~ (|x|, |y|)\in V(x_D(w), y_D(w))\}$,  and  by Proposition~\ref{harmonic-functions}, 
 \[
 F_j(x_D(w),y_D(w)) = \varkappa_{(x_D(w),y_D(w))}(j) > 0, \quad \forall j\in\Z^2_+{\setminus} E_0. 
 \]
For any $w{\in}{\mathcal W}_0$ and $j{\in}\Z^2_+{{\setminus}}E_0$, the conditions of Lemma~\ref{lemma-Woess} are therefore satisfied with $F{=}F_j$, and consequently relation~\eqref{Woess-asymptotics} holds uniformly with respect to $w_k{=}k/\|k\|$ in some neighborhood of~$w$.

%%%%%%%%%%%%%%%%%%%%%%%%%%%%%

%%%%%%%%%%%%%%%%%%%%%%%%%%%%%%

\providecommand{\bysame}{\leavevmode\hbox to3em{\hrulefill}\thinspace}
\providecommand{\MR}{\relax\ifhmode\unskip\space\fi MR }
\providecommand{\MRhref}[2]{%
  \href{http://www.ams.org/mathscinet-getitem?mr=#1}{#2}
}
\providecommand{\href}[2]{#2}

\appendix
\section{Relations with Positive Recurrence and Transience Conditions}\label{section-conditions} 
This section discusses the conditions (B0){-}(B7) defining our classification with transience and positive recurrence conditions of  of these random walks. As we will see, the regions of our classification are in fact not defined in terms of positive recurrence or transience.  Transience and positive recurrence are possible in several regions.

Throughout this section, $\mu_0$, $\mu_1$ or $\mu_2$ are assumed to be stochastic, i.e. probability distributions. 
Before getting these results, note that Assumption~(A1)`(ii) is equivalent to the usual Cramer's condition for the distribution of the jumps of $(S(n))$. 
\[
(\alpha,\beta)\to \tilde{P}(\alpha,\beta) \steq{def} P(e^\alpha, e^\beta) = \sum_{k=(k_1,k_2)} \mu(k) e^{\alpha k_1 + \beta k_2}.
\]
It is satisfied if and only if the function $\tilde{P}$ is finite in a neighborhood of the set $\tilde{D} \steq{def} \{(\alpha,\beta) {\in}\R^2:~P(e^\alpha, e^\beta){\leq}1\}$. Similarly,  Assumption~(A3)~(ii) is satisfied if and only if the generating functions 
\[
\tilde\phi_i(\alpha,\beta) \steq{def} \phi_1(e^\alpha, e^\beta) = \sum_{k=(k_1,k_2)} \mu_i(k) e^{\alpha k_1 + \beta k_2}, \quad i{\in}\{0,1,2\},
\]
are finite in a neighborhood of $\tilde{D}$. Since $(0,0){\in}\tilde{D}$, all jumps of the random walk $(Z(n))$ are in particular integrable.  

\noindent
Remark furthermore  that because of Assumptions (A1)~(ii) and (A3)~(iv), the mean jumps 
\[
(M_1,M_2) = \sum_{j\in\Z^2} j\mu(j),  
\;
(M^1_1,M^1_2) = \sum_{j\in\Z^2} j\mu_1(j) 
\quad  
and 
\quad 
(M^2_1, M^2_2)  = \sum_{j\in\Z^2} j\mu_1(j) 
\]
are non zero, $M^1_2 {>} 0$ and $M^2_1 {>} 0$. Moreover, since with Assumption~(A1)~(iii), the sets $D{\cap}D_1$ and $D{\cap}D_2$ have a non-empty interior, one has also $M_1M^1_2 {\not=} M_2 M^1_1$ and  $M_2M^2_1 {\not=} M_1 M^2_2$. 

Under our assumptions, the necessary and sufficient conditions of  positive recurrence and transience for  the Markov chain $(Z(n))$, are given in the next proposition.  See Theorem~3.3.1 of~\cite{FMM} for example. 

\begin{prop}\label{recurrence_transience} Under the assumptions (A1){-}(A3) then, the following assertions hold for the Markov chain $(Z(n))$ on $\Z_+^2$.
\begin{enumerate}
\item  Positive recurrence.  If and only if one of the following conditions is satisfied:
\begin{enumerate}
\item[(R0)] $M_1{<}0$, $M_2{<}0$,  $M_1M^1_2{<}M_2 M^1_1$  and $M_2M^2_1{<}M_1 M^2_2$; 
\item[(R1)]  $M_2{<}0$,  $M_1{\geq} 0$ and $M_1 M^1_2{<}M_2 M^1_1$;
\item[(R2)] $M_1{<}0$,  $M_2{\geq} 0$ and $M_2M^2_1{<}M_1 M^2_2$.
\end{enumerate} 
\item  Transience. If  one of the following conditions is satisfied:
\begin{enumerate}
\item[(T0)] $M_1{>}0$ and $M_2{>}0$; 
\item[(T1)]  $M_2{<}0$   and   $M_1 M^1_2{>}M_2 M^1_1$; 
\item[(T2)] $M_1{<}0$  and  $M_2M^2_1{>}M_1 M^2_2$; 
%\item[(T3)] conditions (T1) and (T2) are satisfied simultaneously.
\end{enumerate}
\end{enumerate}
\end{prop} 
%In the case, when (T0) holds, the Markov chain $(Z(n))$ escapes to the infinity along the direction of the mean jump $M$. If (T1) holds, the Markov chain $(Z(n))$ escapes to the infinity along the axis $\{j=(j_1,0), \; j_1\in\Z_+\}$. Similarly, when (T2) holds, the Markov chain $(Z(n))$ escapes to the infinity along the axis $\{j=(0,j_2), \; j_2\in\Z_+\}$. And finally, if (T3) holds, i.e. if the both conditions (T1) and (T2) are satisfied, then the Markov chain $(Z(n))$ escapes to the infinity along the both axes $\{j=(j_1,0), \; j_1\in\Z_+\}$ and $\{j=(0,j_2), \; j_2\in\Z_+\}$. 

The proposition below studies the relation between  conditions of transience and positive recurrence for the Markov chain $(Z(n))$  and the location of the points defining the regions (B0){-}(B7). We formulate the conditions (R0){-}(R2) and (T0){-}(T3) in terms of the location of the points $(x^*,Y_i(x^*))$, $(X_i(y^*),y^*)$, and  $(x^{**},Y_i(x^{**}))$, $(X_i(y^{**}),y^{**})$, $i{\in}\{1,2\}$.
In Proposition~\ref{xd-yd-strictly stochastic-case}, the relation with the location of the  dominant singularities $x_d$ and $y_d$ is analyzed.

We first establish a technical lemma. 
\begin{lemma}\label{comparison-lemma} Under the assumptions (A1){-}(A3), then
\begin{enumerate}[label=(\roman*)]
\item \label{ER1}  $M_2{<} 0$ if and only if $(1,1){\in}{\mathcal S}_{11}{\cup}{\mathcal S}_{21}$ and $x^*_P{<}1{<}x^{**}_P$.\\  In this case,  $Y_1(1){=}1$ and 
\be
1 = 
\begin{cases}
x^* & \text{ if } M_1 M^1_2< M_2 M^1_1, \\ 
x^{**} &\text{ if } M_1 M^1_2> M_2 M^1_1. 
\end{cases} 
\ee

\item   $M_1{<}0$ if and only if $(1,1){\in}{\mathcal S}_{11}{\cup}{\mathcal S}_{12}$ and $y^*_P{<} 1 {<} y^{**}_P$.\\
In this case,  $X_1(1){=} 1$ and 
\be
1 = 
\begin{cases}
y^* & \text{ if } M_2M^2_1 < M_1 M^2_2, \\ 
y^{**} &\text{ if } M_2M^2_1 > M_1 M^2_2. 
\end{cases} 
\ee
\end{enumerate}
\end{lemma} 
\begin{proof} To prove~\ref{ER1}, it is sufficient to remark that 
 $M_2 {=} \partial_y P(1,1)$ and that, according to the definition of the curves  $S_{ij}$, $i$ ,$j{\in}\{1,2\}$, 
\[
{\mathcal S}_{11} \cup{\mathcal S}_{21} = \{(x,y)\in\partial D:~\partial_y P(x,y) \leq 0\}
\]
and that by Lemma~\ref{preliminary-lemma1}, the points $(x^*_P, Y_1(x^*_P)){=}(x^*_P, Y_2(x^*_P))$ and $(x^{**}_P, Y_1(x^{**}_P)){=} (x^{**}_P, Y_1(x^{**}_P))$ are the only points on the boundary $\partial D$ of the set $D$ satisfying the relation $\partial_y P(x,y){=}0$. 

Moreover,  by relations~\eqref{eq2_S_11} and~\eqref{eq2_S_21}, we have ${\mathcal S}_{11}{\cup}{\mathcal S}_{21}{=}\{(x,y)\in\partial D:~ y {=} Y_1(x)\}$. 
Hence,  when $M_2{<}0$, one has $Y_1(1){=}1$ and therefore  $\phi_1(1,Y_1(1)){=} \phi_1(1,1) = 1$. By Corollary~\ref{preliminary_cor2}, it follows that the point $1$ is an end point of the line segment $[x^*, x^{**}]$, and, consequently, with relations~\eqref{eq-critical-points} and~\eqref{eq2-critical-points}, we conclude that one and only one of the following cases occurs:
\begin{itemize}
\item $x^*_P < x^* = 1  < x^{**}\leq x^{**}_P$;
\item  $x^*_P\leq x^* < 1 = x^{**} < x^{**}_P$.
\end{itemize} 
By the implicit function theorem we have, for any $x{\in}]x^*_P, x^{**}_P[$, 
\[
\frac{d}{dx} \phi_1(x, Y_1(x)) = \left.\partial_x \phi_1(x, y) + \partial_y\phi_1(x,y) \partial_x P(x,y)/\partial_y P(x,y)\right|_{(y=Y_1(x))},
\]
hence,  when $M_2{<} 0$, the relation
\begin{align*}
\left. \frac{d}{dx} \phi_1(x, Y_1(x)) \right|_{x=1} &= \partial_x \phi_1(1,1) + \partial_y\phi_1(1,1) \partial_x P(1,1)/\partial_y P(1,1) ~=~ M^1_1 - M^1_2 M_1/M_2,
\end{align*} 
holds, and, consequently,  in a neighborhood of $x{=}1$, the function $x{\to}\phi_1(x, Y_1x))$ is 
\begin{itemize} 
\item non-decreasing if $M^1_1{-}M^1_2 M_1/M_2 > 0$, 
\item non-increasing if $M^1_1{-}M^1_2 M_1/M_2 < 0$.
\end{itemize} 
By Corollary~\ref{preliminary_cor2}, we have the relation $\phi_1(x, Y_1(x)){<}1$, for all $x{\in}]x^*, x^{**}[$, and $\phi_1(x, Y_1(x)){>} 1$, for $x{\in}[x^*_P, x^*[\,\cup\,]x^{**}, x^{**}_P]$. Hence, when $x^*_P {<} x^*{=} 1 {<} x^{**}{\leq} x^{**}_P$, the function $x{\to}\phi_1(x, Y_1(x))$ is non-decreasing in a neighborhood of $x{=}1$ and $M^1_1 M_2{-}M^1_2 M_1{>} 0$, and, when $x^*_P{\leq} x^*{<} 1{=} x^{**}{<} x^{**}_P$, the function $x{\to}\phi_1(x, Y_1(x))$ is non-decreasing in a neighborhood of $x{=}1$ and $M^1_1 M_2{-}M^1_2 M_1{<}0$. 

The first assertion of our lemma is  proved. The second assertion is symmetrical by  exchanging the roles of $x$ and $y$. 
\end{proof} 

\begin{prop}\label{cases-recurrence-transience} Under Assumptions~(A1){-}(A3), then
\begin{enumerate}[label=(\roman*)] 
\item (R0) holds if and only if  $(x^*,Y_1(x^*)){=}(X_1(y^*),y^*){=} (1,1)$, $x^*_P{<}1$ and $y^*_P{<}1$;
\item (R1) \phantom{holds if}"\phantom{and only if}  $(x^*,Y_1(x^*)){=} (1,1)$ and $ x^{*}_P{<}1$;
\item (R2) \phantom{holds if}"\phantom{and only if} $(X_1(y^*),y^*){=} (1,1)$ and $y^*_P{<}1$.
\item (T0) \phantom{holds if}"\phantom{and only if} $(1,1){\in}{\mathcal S}_{22}$, $1{<}x^{**}$ and $1{<}y^{**}$.\\ In this case, $Y_2(x^{**}){<}1$ and $X_2(y^{**}){<}1$.
\item (T1) holds if and only if $(1,1){=} (x^{**}, Y_1(x^{**}))$ and $1{<}x^{**}_P$.
\item (T2) \phantom{holds if}"\phantom{and only if} $(1,1){=} (X_1(y^{**}), y^{**})$ and $1{<}y^{**}_P$.
%\item (T3)  holds if and only if $(1,1) = (X_1(y^{**}), y^{**}) = (x^{**}, Y_1(x^{**}))$,  $1{<}x^{**}_P$ and $1{<}y^{**}_P$. 
\end{enumerate} 
\end{prop}
\begin{proof}
The assertions (i){-}(iii) and (v){-}(vi) follow directly from Lemma~\ref{comparison-lemma}. 

We have only to establish~(iv). By using  again the relations $M_1{=}\partial_x P(1,1)$ and $M_2{=}\partial_y P(1,1)$,  when $(1,1){\in}{\mathcal S}_{22}$,  by relations~\eqref{eq-critical-points}, \eqref{eq2-critical-points} and~\eqref{eq2_S_22}, and with Lemma~\ref{preliminary-lemma1},  we get the identities $(1,1){=}(1, Y_2(1)){=}(X_2(1), 1)$ and the relations
\be\label{eq2-comparison-prop}
x^*_P < X_1(y^{**}_P)  \leq 1 \leq x^{**}\leq x^{**}_P \quad \text{and} \quad y^*_P <  Y_1(x^{**}_P) \leq 1 \leq y^{**}\leq y^{**}_P, 
\ee
and
\begin{align}
M_1{>} 0 &\; \Leftrightarrow\; (1,1)\not= (X_1(y^{**}_P), y^{**}_P) \;\Leftrightarrow \;  1 < y^{**}_P,\label{eq3-comparison-prop}\\
M_2{>}0 &\; \Leftrightarrow \; (1,1){\not=} (x^{**}_P, Y_1(x^{**}_P))\quad \Leftrightarrow\;  1 < x^{**}_P. \label{eq4-comparison-prop}
\end{align}
Hence,  when $(1,1){\in}{\mathcal S}_{22}$, we have $1{<}x^{**}$ and $1{<} y^{**}$, the condition (T0) holds. 

Conversely, if (T0) holds, then, according to the definition of the curve ${\mathcal S}_{22}$, 
\be\label{eq5-comparison-prop}
 (1,1) = (1, Y_2(1)) = (X_2(1),1)\in {\mathcal S}_{22}
\ee 
and, by using  relations~\eqref{eq3-comparison-prop} and \eqref{eq4-comparison-prop}, we get
\[
1 < x^{**}_P, \quad \text{and} \quad  1 < y^{**}_P.
\]
By relation~\eqref{eq2-comparison-prop}, we have always $1{\leq}x^{**}$ and, if we assume that  $x^{**}{=}1$, then, by relations~\eqref{eq5-comparison-prop} and~\eqref{eq2-comparison-prop}, we obtain  $x^*_P{<} 1{=} x^{**}{<} x^{**}_P$. Consequently,   by Lemma~\ref{preliminary-lemma1} and Corollary~\ref{preliminary_cor2}, we have $Y_1(1){<}Y_2(1)$ and $\phi_1(1, Y_1(1)){=} 1$.

Since under our assumptions (A3)~(ii) and (A3)~(v), the function $y{\to}\phi_1(1, y)$ is strictly increasing on the line segment $[0, Y_1(1)]$, these relations imply that $\phi_2(1, Y_2(1)){>}1$ and,  consequently, $(1,Y_1(1)){\not=}(1,1)$. Since this last relation contradicts  relation~\eqref{eq5-comparison-prop}, we conclude  that $x^{**}{>}1$.

By exchanging  the roles of $x$ and $y$, the same arguments prove also that when (T0) holds, the relation $y^{**}{>}1$ holds. By relation~\eqref{eq2_S_22}, we have
\[
  {\mathcal S}_{22}{=}\{(x, Y_2(x)):~x{\in} [X_1(y^{**}_P), x^{**}_P]\}{=}\{(X_2(y), y):~y{\in}[Y_1(x^{**}_P), y^{**}_P]\},
  \]
  and,  by Lemma~\ref{preliminary-lemma1}, the functions
  \[
  X_2: [Y_1(x^{**}_P), y^{**}_P]{\to}[X_1(y^{**}_P), x^{**}_P]\text{ and } Y_2: [X_1(y^{**}_P), x^{**}_P]{\to}[Y_1(x^{**}_P), y^{**}_P]
  \]
  are strictly decreasing. It follows that $Y_2(x^{**}) {<} Y_2(1){=}1$ holds and, similarly, $X_2(y^{**}){<}X_2(1){=}1$. The  assertion~(iv) of Proposition~\ref{cases-recurrence-transience} is  proved. 
\end{proof}

The following statement give relations between  the conditions (B0){-}(B7) and the conditions (R0){-}(R2) and (T0){-}(T2). 
\begin{prop}\label{comparison-cor} Under the assumptions (A1){-}(A3), then
\begin{enumerate}[label=(\arabic*)] 
\item both conditions (T1) and (T2) hold if and only if (B7) holds;
\item  (B2) holds if (T0) holds. 

\item If either (B3) or (B4) holds, then  
\begin{itemize}
\item (T2) is not possible;
\item (T1) holds  if and only if $x^{**}{=}1$, and in this case, $Y_1(1){=}1$;
\item[--] either (R0) or (R1) or (R2) holds whenever $x^{**}{>}1$.
\end{itemize} 

\item If either (B5) or (B6) holds, then 
\begin{itemize}
\item (T1) is not possible;
\item  (T2) holds if and only if $y^{**}{=}1$, and in this case, $X_1(1){=}1$;
\item   either (R0) or (R1) or (R2) holds whenever $y^{**}{>}1$. 
\end{itemize}
\end{enumerate}
\end{prop} 
\begin{proof}
The first and the second assertions of this statement follow from directly  Proposition~\ref{cases} and Proposition~\ref{cases-recurrence-transience}.

If either (B3) or (B4) holds, by Proposition~\ref{cases}, $1{<} Y_2(x^{**}){\leq}y^{**}$ holds, and, consequently, by Proposition~\ref{cases-recurrence-transience}, the case (T2) is not possible.

Still under the condition that either (B3) or (B4) holds.  By Proposition~\ref{cases} and with relations~\eqref{eq-critical-points} and~\eqref{eq2-critical-points}, we have $1{\leq} x^{**}{<} x^{**}_P$. Therefore, by Proposition~\ref{cases-recurrence-transience}, (T1) holds if and only if $(1,1){=}(x^{**}, Y_1(x^{**}))$ holds. Since  $P(1,1){=}1$, by Lemma~\ref{preliminary-lemma1}, either $Y_1(1){=}1$ or $Y_2(1){=}1$ holds, since, in our case, we have $1{<}Y_2(x^{**})$ by Proposition~\ref{cases}, it follows that when either (B3) or (B4) holds, the condition (T1) is satisfied if and only if $x^{**}{=}1$. 

Under the condition that either (B3) or (B4) holds and that $x^{**}{>}1$ holds. We have that either $M_2{\leq}0$ or $M_1{\leq}0$, because otherwise (T0) and, consequently, (B2) holds.  By Proposition~\ref{cases} and with relation~\eqref{eq-critical-points}, we get the relations $1{<} Y_2(x^{**}) {\leq} y^{**}{\leq} y^{**}_P$ and $1 {<} x^{**} {\leq }x^{**}_P$, and, consequently, $(1,1){\not=}(x^{**}_P, Y_1(x^{**}_P))$ and $(1,1){\not=} (X_1(y^{**}_P), y^{**}_P)$. Since the point $(x,y)=(x^{**}_P, Y_1(x^{**}_P))$, resp. $(x,y){=}(X_1(y^{**}_P), y^{**}_P)$ ) is the only point on the boundary $\partial D$ of $D$ for which $\partial_xP(x,y) > 0$ and $\partial_yP(x,y)=0$, resp. $\partial_xP(x,y){=}0$ and $\partial_yP(x,y){>} 0$) and, since $M_1{=}\partial_x P(1,1)$ and $M_2{=}\partial_y P(1,1)$, we conclude therefore that,  either $M_1{<}0$ or $M_2{<} 0$.  If $M_2{<} 0$, then, by Lemma~\ref{comparison-lemma}, and the assumption $x^{**}{>}1$, we get $x^*{=}1$ and $M_1 M^1_2{<} M_2 M^1_1$. Similarly, if $M_1{<}0$, by Lemma~\ref{comparison-lemma} and since, by Proposition~\ref{cases}, $1{<} Y_2(x^{**}){\leq}y^{**} $ holds, we obtain $y^*{=}1$ and $M_2M^2_1{<} M_1 M^2_2$. We conclude that in the case when either (B3) or (B4) holds and  $x^{**}{>}1$, one of the conditions (R0), (R1) or (R2) is satisfied. 

The third assertion of Proposition~\ref{comparison-cor} is therefore also proved. The last assertion follows by exchanging  the roles of $x$ and $y$. 
\end{proof}

We can now establish the relations with the locations of the dominant singularities $x_d$ and $y_d$. 
\begin{prop}\label{xd-yd-strictly stochastic-case} Under the assumptions (A1){-}(A3) the following assertions hold: 

\begin{enumerate}[label=(\roman*)] 
\item If the Markov chain $(Z(n))$ is positive recurrent then $x_d{>}1$ and $y_d{>}1$. 

\item  If  (T0) is satisfied then  (B2),  $x_d{=} x^{**}{>}1$ and $y_d {=} y^{**}{>}1$ hold.

\item If  (T1) is satisfied then 
\begin{itemize} 
\item in  the cases (B0){-}(B2), relations  $x_d {=} x^{**} {=} 1{<}x^{**}_P$ and $y_d {=} y^{**}{>}Y_1(1) {=} 1$ hold;
\item in the cases (B3) and (B4), relations $x_d {=} x^{**} {=} 1{<}x^{**}_P$ and $y_d {=} Y_2(1){>}1$ holds;
\item (B5) and (B6) do not hold.
\end{itemize} 

\item If (T2) is satisfied then 
\begin{itemize} 
\item in the cases (B0){-}(B2),  relations $y_d {=} y^{**} {=} 1{<} y^{**}_P$ and $x_d {=} x^{**}{>}X_1(1) {=} 1$ hold;
\item and (B3) and (B4) do not hold;
\item in the cases (B5) and (B6), relations $y_d {=} y^{**} {=} 1{<}y^{**}_P$ and $x_d {=} X_2(1){>}1$ hold.
\end{itemize}
\end{enumerate}
\end{prop} 
\begin{proof}
  Consider first the case when one of the cases (B0){-}(B2) occurs. Then by Proposition~\ref{cases} and according to the definition of $x_d$ and $y_d$, 
\be\label{eq1-xd-yd-strictly stochastic-case}
x_d = x^{**} > X_1(y^{**}) \quad  \text{and}  \quad y_d = y^{**} > Y_1(x^{**}). 
\ee
Hence, if we suppose moreover that (R0) holds,  by Lemma~\ref{comparison-lemma} and using \eqref{eq-critical-points},  we obtain 
\[
1 = x^* < x^{**} = x_d \quad \text{and} \quad 1 = y^* < y^{**} = y_d.
\]
If we suppose that (R1) holds, then by Proposition~\ref{cases-recurrence-transience} and using \eqref{eq-critical-points}, and \eqref{eq1-xd-yd-strictly stochastic-case}, 
\[
1 = x^*  < x^{**} = x_d, \quad \text{and} \quad 1 = Y_1(x^*) < y_d.
\]
And similarly, if we suppose that (R2) holds, then by Proposition~\ref{cases-recurrence-transience} and using \eqref{eq-critical-points}, and \eqref{eq1-xd-yd-strictly stochastic-case}, 
\[
1 = y^*  < y^{**} = y_d, \quad \text{and} \quad 1 = X_1(y^*) < x_d.
\]
In the case when one of the cases (B0){-}(B2) holds, the first assertion of out statement is therefore proved. 

Suppose now that either (B3) or (B4) holds. In this case, by Proposition~\ref{cases} and according to the definition of $x_d$ and $y_d$, we have
\be\label{eq2-xd-yd-strictly stochastic-case}
x_d = x^{**} < x^{**}_P \quad \text{and} \quad  y_d = Y_2(x^{**}) > 1. 
\ee
Moreover, by relation~\eqref{eq2-critical-points}, the inequality $1 \leq x^{**}$ always holds, and consequently, in this case, by Proposition~\ref{comparison-cor}, 
one of the conditions (R0){-}(R2) is satisfied if and only if $1 < x^{**} = x_d$. Hence,  when either (B3) or (B4) holds, the first assertion of Proposition~\ref{xd-yd-strictly stochastic-case} is also proved. The same arguments (it is sufficient to exchange the roles of $x$ and $y$) prove the first assertion of our proposition in the case when either (B5) or (B6) holds. 

Furthermore, if the condition (T0) is satisfied, then by Proposition~\ref{comparison-cor}, (B2) holds and consequently, by Proposition~\ref{cases-recurrence-transience} and using \eqref{eq1-xd-yd-strictly stochastic-case} one gets $x_d = x^{**} > 1$ and $y_d = y^{**} > 1$.  The second assertion of Proposition~\ref{xd-yd-strictly stochastic-case} is therefore also proved.

Suppose now that the condition (T1) is satisfied.  Then by Proposition~\ref{comparison-cor}, the cases (B5) and (B6) are impossible, and by Proposition~\ref{cases-recurrence-transience}, 
\be\label{eq3-xd-yd-strictly stochastic-case}
x^{**} = 1 \quad \text{ and } \quad Y_1(x^{**}) = 1.
\ee
Moreover, if one of the cases (B0){-}(B2) holds, then by \eqref{eq1-xd-yd-strictly stochastic-case}, one gets $y_d = y^{**} > Y_1(x^{**}) = 1$, and if either (B3) or (B4) holds, then by \eqref{eq2-xd-yd-strictly stochastic-case} and using the second relation of \eqref{eq3-xd-yd-strictly stochastic-case}, we obtain $y_d = Y_2(x^{**}) > Y_1(x^{**}) = 1$. The third assertion of Proposition~\ref{xd-yd-strictly stochastic-case} is therefore also proved. The proof of fourth assertion  is the same, it is sufficient to exchange the roles of $x$ and $y$. 
\end{proof}
\section{Irreducibility Properties of the killed Markov Chain}\label{pfirr}

\subsection*{Proof of Lemma~\ref{irreductibility-lemma}}
Since by the assumption (A1), the homogeneous random walk $(S(n))$ with transition probabilities $\P_j(S(1)=k) = \mu(k-j)$ is irreducible on $\Z^2$, for any $k\in\Z^2$, there is  a sequence of points $\ell^k(1)=(\ell_1^k(1),\ell_2^k(1)), \ldots \ell^k(N_{k}) = (\ell_1^k(N_{k}),\ell_2^k(N_{k}))\in\Z^2$ such that $\ell^k(1)+ \cdots + \ell^k(N_{k}) = k$ and $\P_0(S(N_k) = k) \geq \mu(\ell_1^k)\times\cdots\times\mu(\ell^k(N_k)) > 0$. 

Let $N_1 = 1 + \max\{N_{(-1,0)}, N_{(1,0)}, N_{0,-1)}, N_{(0,1)}\}$. Then because of Assumption (A2), for any  $j=(j_1,j_2)\in\Z^2_+$ with $j_1,j_2\geq N_1$ and $k\in \{(-1,0), (1,0), (0,-1), (0,1)\}$ the sequence of points $m^k(1)= j + \ell^k(1), \ldots , m^k(N_k) = j + \ell^k(1) + \cdots + \ell^k(N_{k})$ does not exit from the set  $(\N^*)^2$, and consequently 
\[
\P_j\bigl(Z(n) = j + k, \; \text{for some $n < \tau$}\bigr) \geq \P_j\bigl( Z(n) = m^k(n), \quad \forall n\in\{1,\ldots, N_k\}\bigr) > 0. 
\]
This proves that  
\be\label{eq1-irreductibility-lemma}
g(j,k) \geq \P_j(Z(n) = k, \; \text{for some $n < \tau$}) > 0 \quad \text{whenever \; $j_1,j_2, k_1, k_2\geq N_1$}.
\ee

Consider now a random walk $(\hat{S}(n))$ on $\Z\times\N$ with transition probabilities 
\[
\P_{j}(\hat{S}(1) = k) =\begin{cases} \mu(k-j) &\text{ for all $j= (j_1,j_2), k\in\Z\times\N$ with $j_2 > 0$,}\\
\mu_1(k-j) &\text{ for all $j= (j_1,j_2), k\in\Z\times\N$ with $j_2 = 0$.}
\end{cases} 
\]
Remark that under our hypotheses (see Assumptions (A1) and (A3) (v) and (vi)) such a random walk is irreducible on $\Z\times\N$, and consequently, for any $j_2 \in\N$ and $k=(k_1,k_2)\in\Z^2$ such that $(k_1, k_2 +j_2)\in\Z\times\N$, there is a sequence of points $\ell^{j_2,k}(1)=(\ell_1^{j_2,k}(1),\ell_2^{j_2,k}(1)), \ldots \ell^{j_2,k}(N_{j_2,k}) = (\ell_1^{j_2,k}(N_{j_2,k}),\ell_2^{j_2,k}(N_{j_2,k}))\in\Z^2$ such that for any $n\in\{1,\ldots, N_{j_2,k}\}$,
\[
(0,j_2) + \ell^{j_2,k}(1) +\cdots + \ell^{j_2,k}(n) \in\Z\times\N,
\]
and 
\[
\P_{(0,j_2)} \bigl(\hat{S}(N_{j_2,k}) = (0,j_2) + k \bigr) = \P_{(0,j_2)}\bigl( \hat{S}(n) - \hat{S}(n-1) = \ell^{j_2,k}(n), \; \forall n\in\{1,\ldots, N_{j_2,k}\}\bigr) > 0. 
\]
Letting
\[
N_2 = 1 + \max\{N_{j_2,k}:~ \; j_2 \leq N_1, \,  (k_1, k_2)\in \{(-1,0), (1,0), (0,-1), (0,1)\}, \;  (k_1, k_2 +j_2)\in\Z\times\N \}
\]
and using Assumptions (A2) and (A3) (iv), we obtain  that for any $j=(j_1,j_2)\in \Z^2_+$ with $j_1 \geq N_2$, and any $k=(k_1, k_2)\in \{(-1,0), (1,0), (0,-1), (0,1)\}$ such that $j+k\in \Z\times\N$, the sequence of points $m^{j,k}(1) = j+ \ell^{j_2,k}(1), \ldots, m^{j,k}(N_{j_2,k}) = j+ \ell^{j_2,k}(1) +\cdots + \ell^{j_2,k}(N_{j_2,k}) = (\ell_1^{j_2,k}(N_{j_2,k}),\ell_2^{j_2,k}(N_{j_2,k}))$ does not exist from the set $\N^*\times\N$, and consequently, 
\[
\P_j\bigl( Z(n) = j+k \; \text{for some $n < \tau$}\bigr) \geq \P_j\bigl(\hat{Z}(n) = m^{j_2,k}(n) \; \forall n\in\{1,\ldots, N_{j_2,k} \}\bigr) > 0. 
\]
This proves that 
\be\label{eq2-irreductibility-lemma}
g(j,k)  > 0 \quad \text{whenever \, $j_2\leq N_1,  k_2\leq N_1+1$ and $j_1, k_1 \geq N_2$},
\ee
and with exactly the same arguments (it is sufficient to exchange the roles of the first and the second coordinates of the points in $\Z^2_+$), one gets that for some $N_3 > 0$,
\be\label{eq3-irreductibility-lemma}
g(j,k)  > 0 \quad \text{whenever \, $j_1\leq N_1,  k_1\leq N_1+1$ and $j_2, k_2 \geq N_3$}. 
\ee
When combined together, relations \eqref{eq1-irreductibility-lemma}, \eqref{eq2-irreductibility-lemma} and \eqref{eq3-irreductibility-lemma} show that for some $N_0 \geq 0$,
\be\label{eq4-irreducibility-lemma}
g(j,k)  > 0 \quad \text{for all \; $j,k\in\Z^2_+$ with $\|j\|\geq N_0$ and  $\|k\|\geq N_0$}. 
\ee
Remark now that if $\ell\in\Z^2_+$ is such that  $g(\ell,j_0) > 0$ for some  $j_0\in\Z^2_+$ with $\|j_0\| \geq N_0$, then also $\P_\ell(Z(n) = j_0, \; \tau > n) > 0$ for some $n\in\N$, and consequently, by \eqref{eq4-irreducibility-lemma} and using the inequality 
\be\label{eq5-irreducibility-lemma}
g(\ell, k) \geq \sum_{j\in\Z^2_+\setminus\{0\}}\P_\ell(Z(n) = j, \; \tau > n) g(j,k) ~\geq~ \P_\ell(Z(n) = j_0, \; \tau > n) g(j_0,k)
\ee
we obtain 
\[
g(\ell, k) \geq \P_\ell(Z(n) = j, \; \tau > n)  g(j,k) > 0 \quad \text{for all $k\in\Z^2_+$ with $\|k\| \geq N_0$}. 
\]
Hence, for any $\ell\in\Z^2_+$  we have either  $g(\ell, k) > 0$ for all $k\in\Z^2_+$ with $\|k\|\geq N_0$, or $g(\ell,k) = 0$ also for all $k\in\Z^2_+$ with $\|k\|\geq N_0$, and in the last case, because of \eqref{eq4-irreducibility-lemma}, $\|\ell\| < N_0$. Letting therefore $E_0=\{\ell\in\Z^2_+:~g(\ell, k) = 0 \; \text{for all $ k\in\Z^2_+$ with $\|k\| \geq N_0$}\}$ we obtain a finite subset of $\Z^2_+$ satisfying \eqref{eq-set-E0-2}. Remark moreover that this set satisfies also \eqref{eq-set-E0-1} because if suppose that $g(\ell, j) > 0$ for some $\ell\in E_0$ and $j\in\Z^2_+{\setminus} E_0$, then using again the inequality \eqref{eq5-irreducibility-lemma} and the same arguments as above we would get $g(\ell, k) > 0$ for all $k\in\Z^2_+$ with $\|k\| \geq N_0$. Now, to complete the proof of our lemma, it is sufficient  to notice that $g(0,k) > 0$ for any $k\in\Z^2_+{\setminus}\{0\}$, because under our assumptions, the random walk $(Z(n))$ is irreducible on $\Z^2_+$, and consequently, the point $0=(0,0)$ does not belong to the set $E_0$.

\newpage

\section{Figures} \label{FigSec}
\begin{figure}[ht]
\resizebox{10cm}{5cm}{\includegraphics{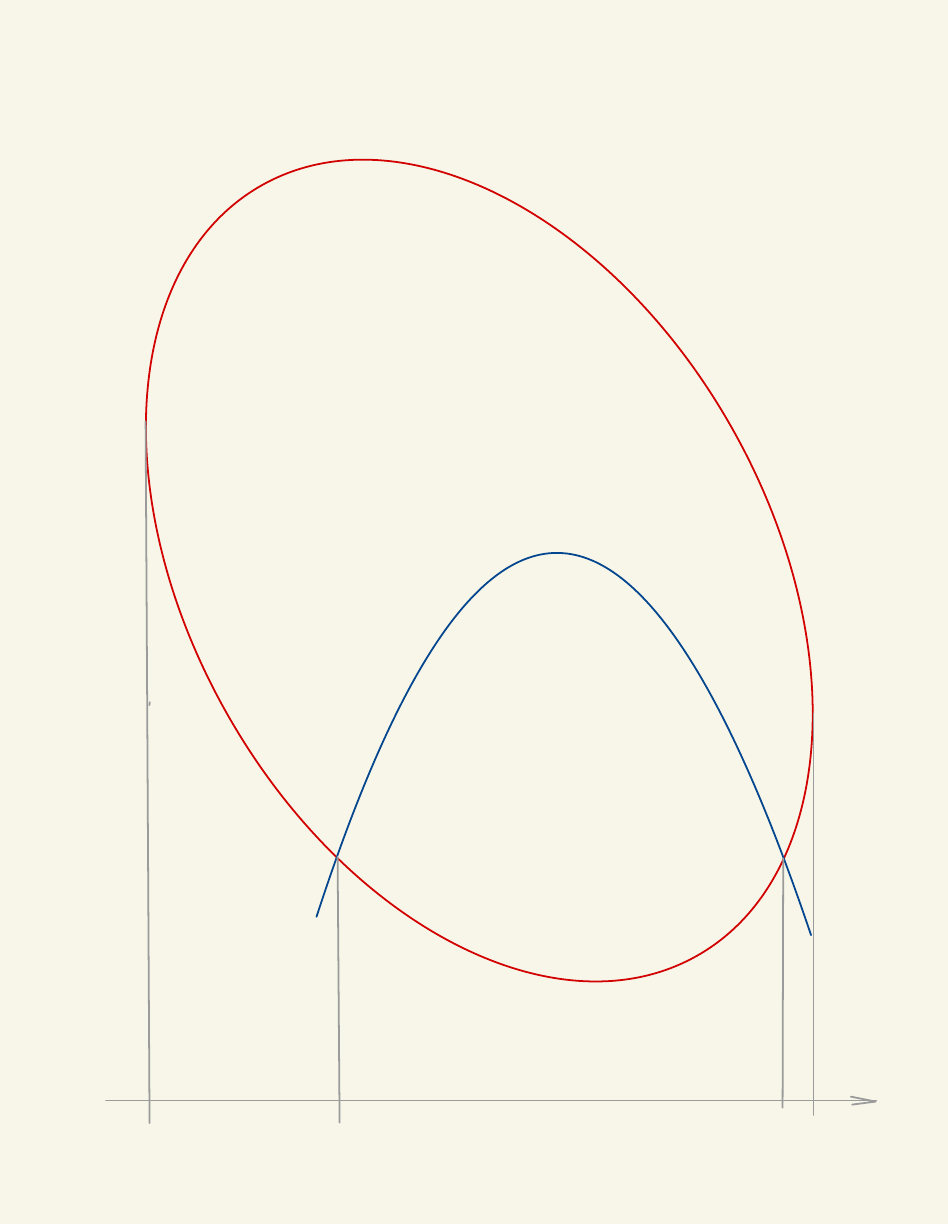}}
\put(-122,47){$D\cap D_1$}
\put(-155,85){$D$} 
\put(-80,6){$(x^{**}, 0)$}
\put(-37,6){$(x^{**}_P, 0)$}
\put(-180,6){$(x^{*}, 0)$}
\put(-260,6){$(x^{*}_P, 0)$}
\caption{ }\label{F0a} 
\end{figure}   

\begin{figure}[ht]
\resizebox{10cm}{5cm}{\includegraphics{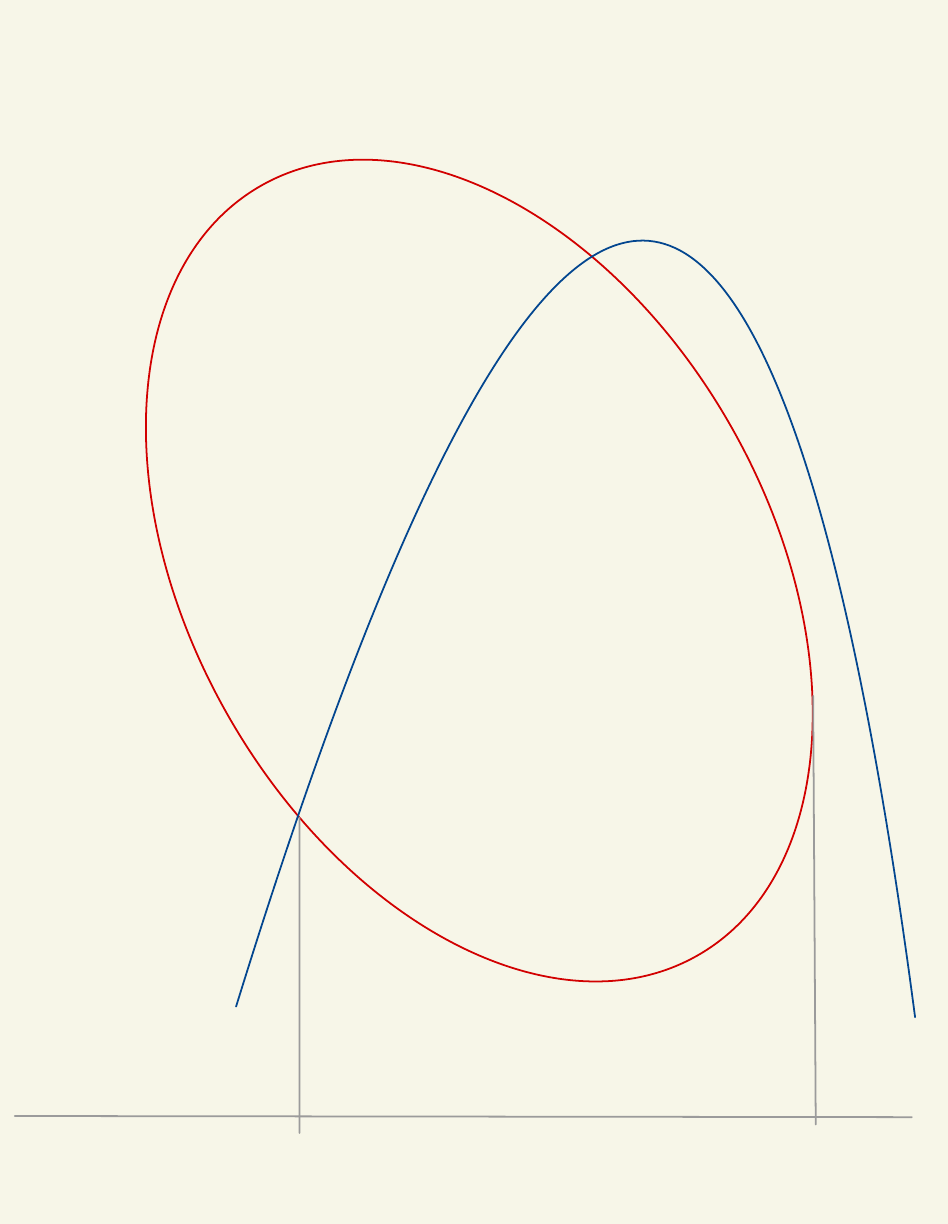}}
\put(-122,47){$D\cap D_1$}
\put(-165,95){$D$} 
\put(-80,4){$(x^{**}, 0)=(x^{**}_P,0)$}
\put(-195,4){$(x^{*}, 0)$}
\caption{ }\label{F0b} 
\end{figure}

 \begin{figure}[ht]
\resizebox{10cm}{5cm}{\includegraphics{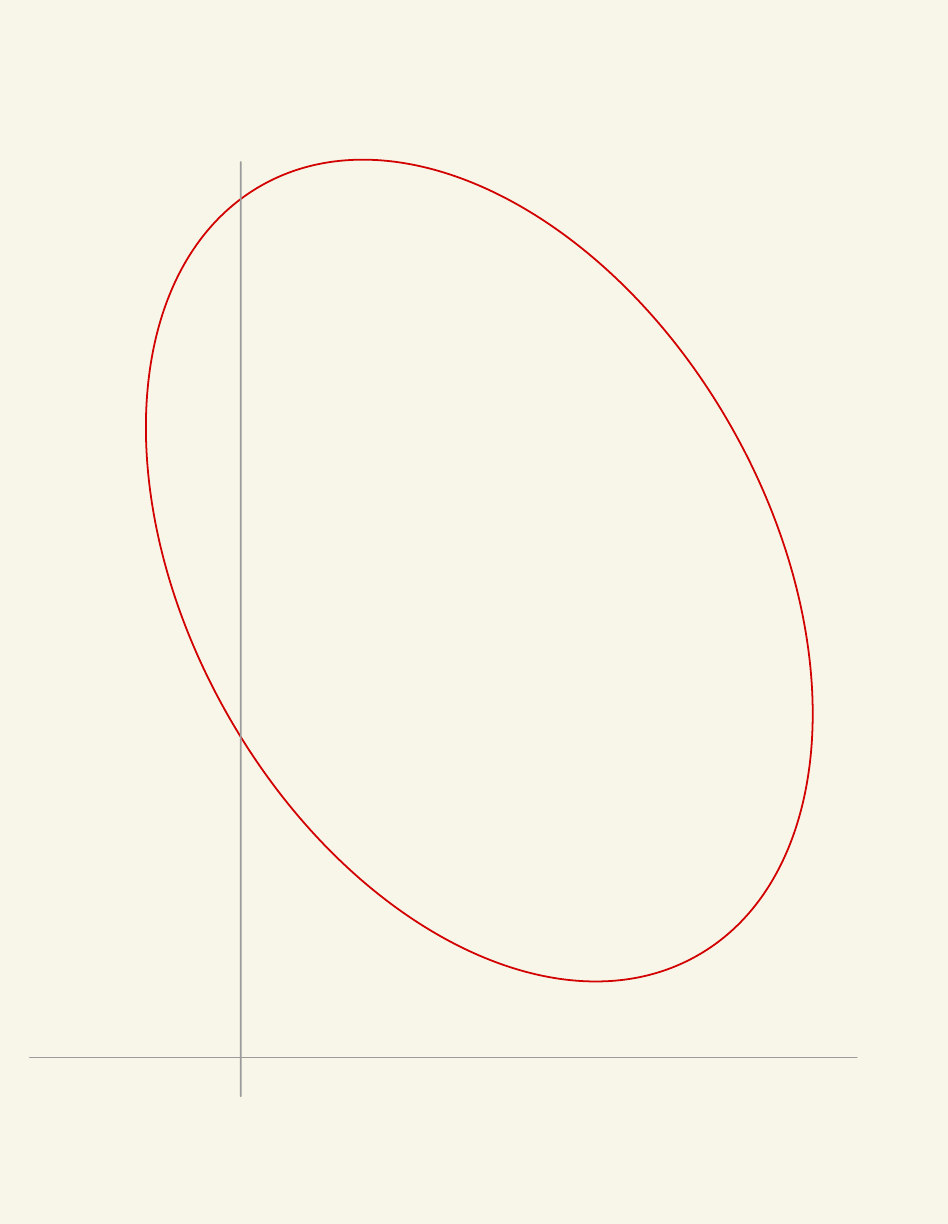}}
%\put(-102,55){$(x^{**}_P, Y_1(x^{**}_P))$}
%\put(-122,17){$(X_1(y^*_P), y^{*}_P)$}
\put(-145,75){$D$} 
%\put(-283,90){$(x^*_P, Y_1(x^*_P))$}
\put(-210,58){$(x, Y_1(x))$}
\put(-210,111){$(x, Y_2(x))$}
\put(-210,6){$(x, 0)$}
%\put(-189,129){$(X_1(y^{**}_P),y^{**}_P)$} 
\caption{ }\label{F1} 
\end{figure}

\begin{figure}[ht]
\resizebox{10cm}{5cm}{\includegraphics{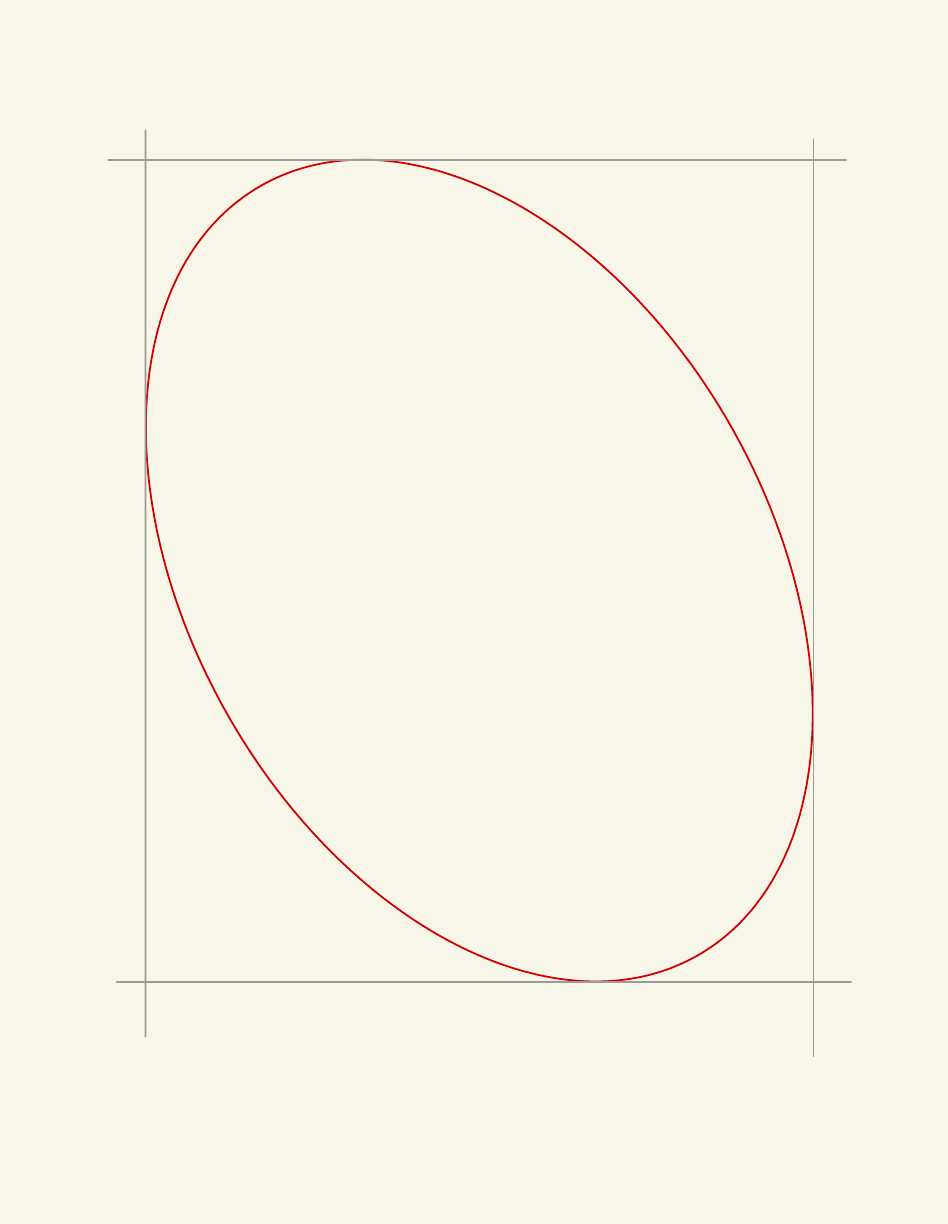}}
\put(-66,98){${\mathcal S}_{22}$}
\put(-56,31){${\mathcal S}_{21}$}
\put(-140,70){$D$} 
\put(-222,51){${\mathcal S}_{11}$}
\put(-230,105){${\mathcal S}_{12}$} 
\caption{ }\label{F2} 
\end{figure}  

\begin{figure}[ht]
\resizebox{15cm}{10cm}{\includegraphics{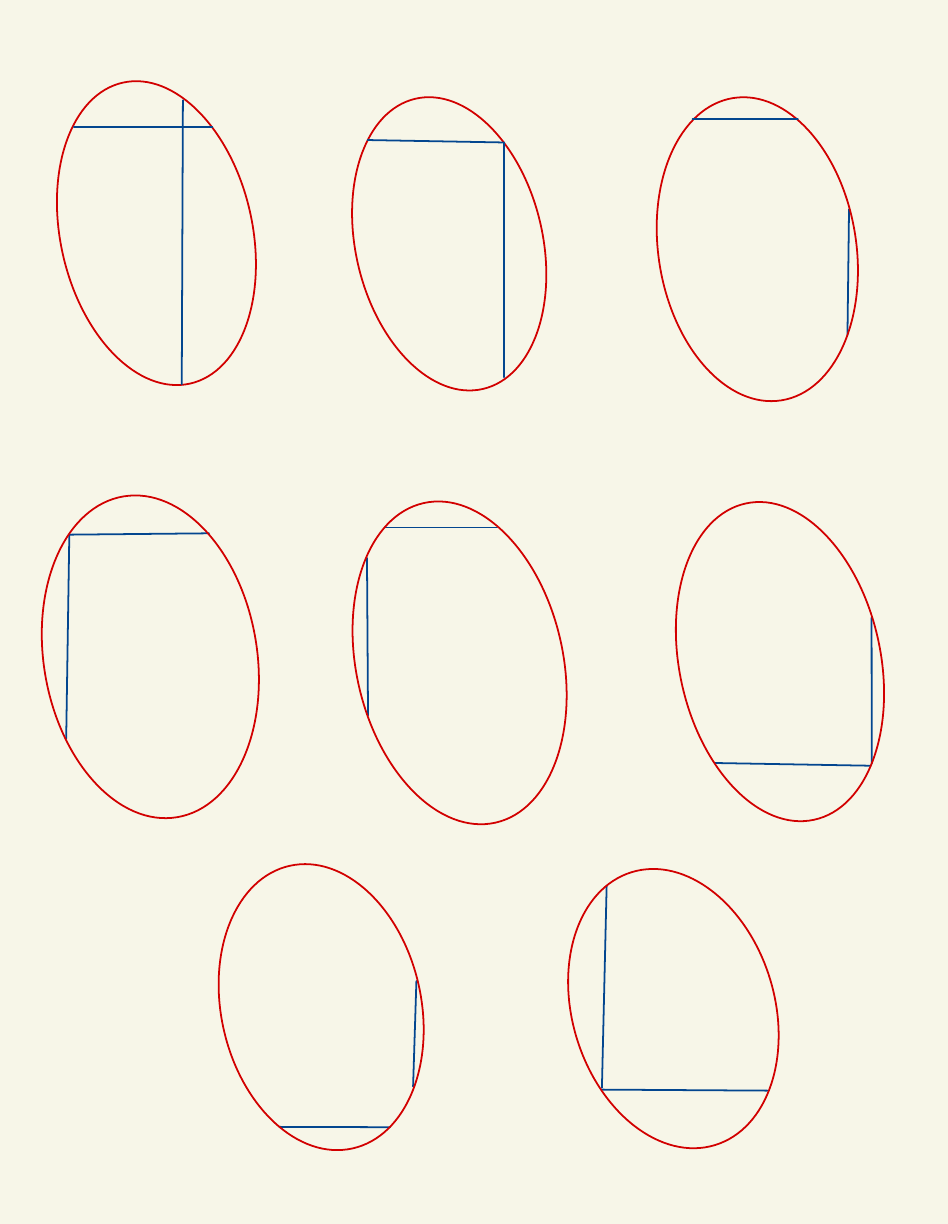}}
\put(-400,200){(B0)}
\put(-402,255){$a$}
\put(-330,255){$b$}
\put(-345,265){$d$}
\put(-350,187){$c$}
\put(-270,200){(B1)}
\put(-267,255){$a$}
\put(-202,255){$b=d$}
\put(-202,187){$c$}
\put(-152,200){(B2)}
\put(-127,255){$a$}
\put(-61,255){$b$}
\put(-45,200){$c$}
\put(-45,240){$d$}
\put(-410,95){(B3)}
\put(-420,162){$a=d$}
\put(-330,162){$b$}
\put(-403,107){$c$}
\put(-273,95){(B4)}
\put(-260,163){$a$}
\put(-199,163){$b$}
\put(-269,155){$d$}
\put(-267,113){$c$}
\put(-145,95){(B5)}
\put(-33,143){$d$}
\put(-113,103){$a$}
\put(-33,103){$b=c$}
\put(-330,10){(B6)}
\put(-237,29){$c$}
\put(-302,25){$a$}
\put(-252,25){$b$}
\put(-240,63){$d$}
\put(-170,10){(B7)}
\put(-163,80){$d$}
\put(-185,27){$a=c$}
\put(-77,27){$b$}
\caption{Possible cases for the location of the line segments  $[a,b] = [(X_1(y^{**}),y^{**}), (X_2(y^{**}),y^{**})]$ and $[cd] = [(x^{**}, Y_1(x^{**})), (x^{**}, Y_2(x^{**}))]$ }\label{F3} 
\end{figure}  

\begin{figure}[ht]
\resizebox{15cm}{10cm}{\includegraphics{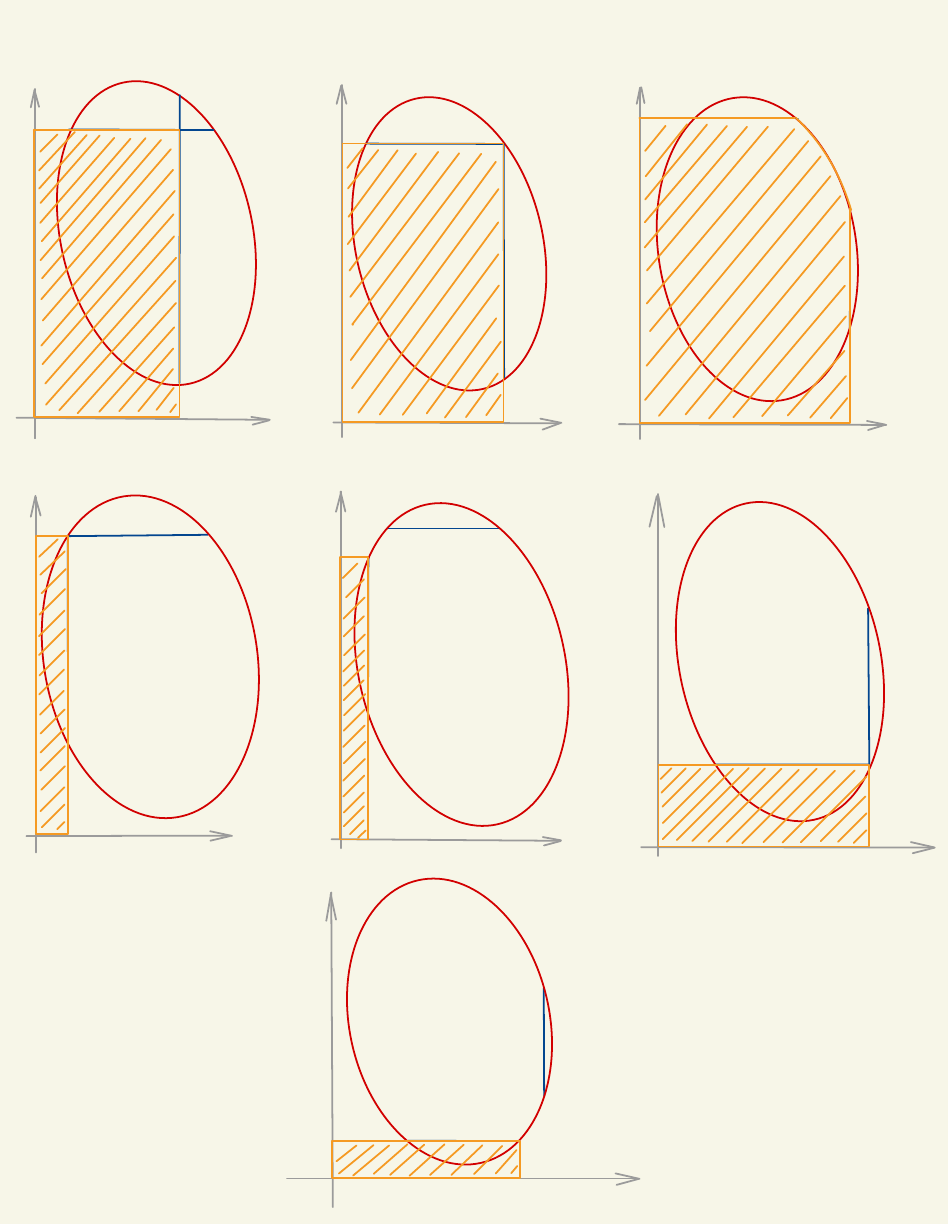}}
\put(-410,190){(B0)}
\put(-402,255){$a$}
\put(-330,255){$b$}
\put(-345,265){$d$}
\put(-350,187){$c$}
\put(-295,190){(B1)}
\put(-267,255){$a$}
\put(-202,255){$b=d$}
\put(-202,187){$c$}
\put(-160,190){(B2)}
\put(-127,255){$a$}
\put(-61,255){$b$}
\put(-45,200){$c$}
\put(-45,240){$d$}
\put(-410,77){(B3)}
\put(-420,162){$a=d$}
\put(-330,162){$b$}
\put(-403,107){$c$}
\put(-295,95){(B4)}
\put(-260,163){$a$}
\put(-199,163){$b$}
\put(-269,155){$d$}
\put(-267,113){$c$}
\put(-153,95){(B5)}
\put(-33,143){$d$}
\put(-113,103){$a$}
\put(-33,103){$b=c$}
\put(-300,20){(B6)}
\put(-180,27){$c$}
\put(-247,22){$a$}
\put(-197,22){$b$}
\put(-180,53){$d$}
\caption{All possible configurations for  the line segments  $[a,b] = [(X_1(y^{**}),y^{**}), (X_2(y^{**}),y^{**})]$ and $[cd] = [(x^{**}, Y_1(x^{**})), (x^{**}, Y_2(x^{**}))]$ and  the trace of the set $\{(x,y)\in{\mathcal D}: \; |x| < x_d, \; |y| < y_d\}$. Cases (B0){-}(B6). }\label{F4} 
\end{figure}  

\begin{figure}[ht]
\resizebox{14cm}{5cm}{\includegraphics{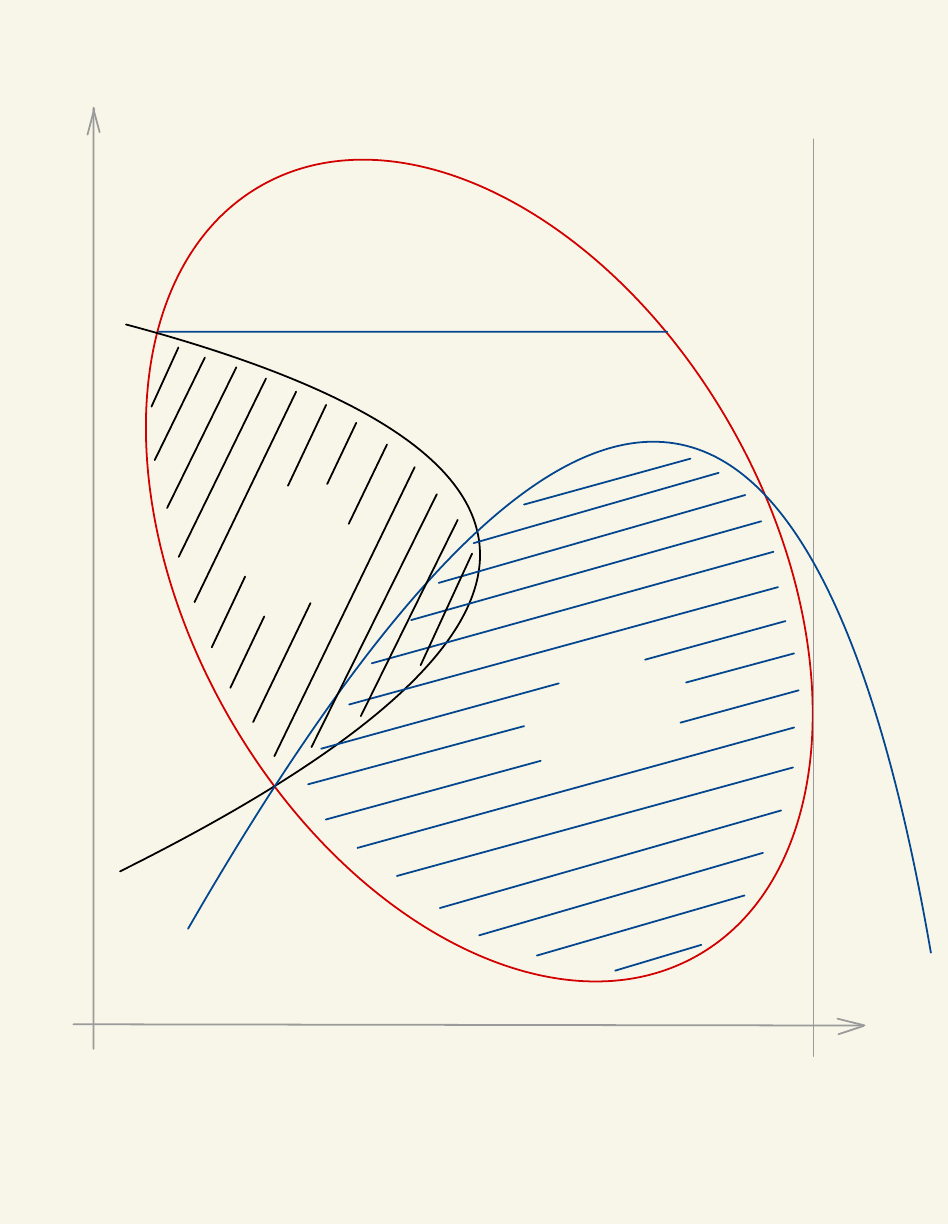}}
\put(-285,75){$D\cap D_2$} 
\put(-290,40){$(1,1)$} 
\put(-160,55){$D\cap D_1$} 
\put(-190,90){$D$} 
\put(-345,107){$(X_1(y_d),y_d)$} 
\put(-110,103){$(X_2(y_d),y_d)$} 
\put(-380,10){$(0,0)$}
\put(-70,10){$(x_d,0)$}
\put(-55,57){$(x_d,Y_1(x_d))$}
\caption{Case (B2) with $x_d=x^{**}=x^{**}_P$. In this case, the line segment $[(x^{**}, Y_1(x^{**})), (x^{**}, Y_2(x^{**}))]$ is a single point and $y_d=y^{**} \leq y^{**}_P$.}\label{F7} 
\end{figure} 

\begin{figure}[ht]
\resizebox{14cm}{5cm}{\includegraphics{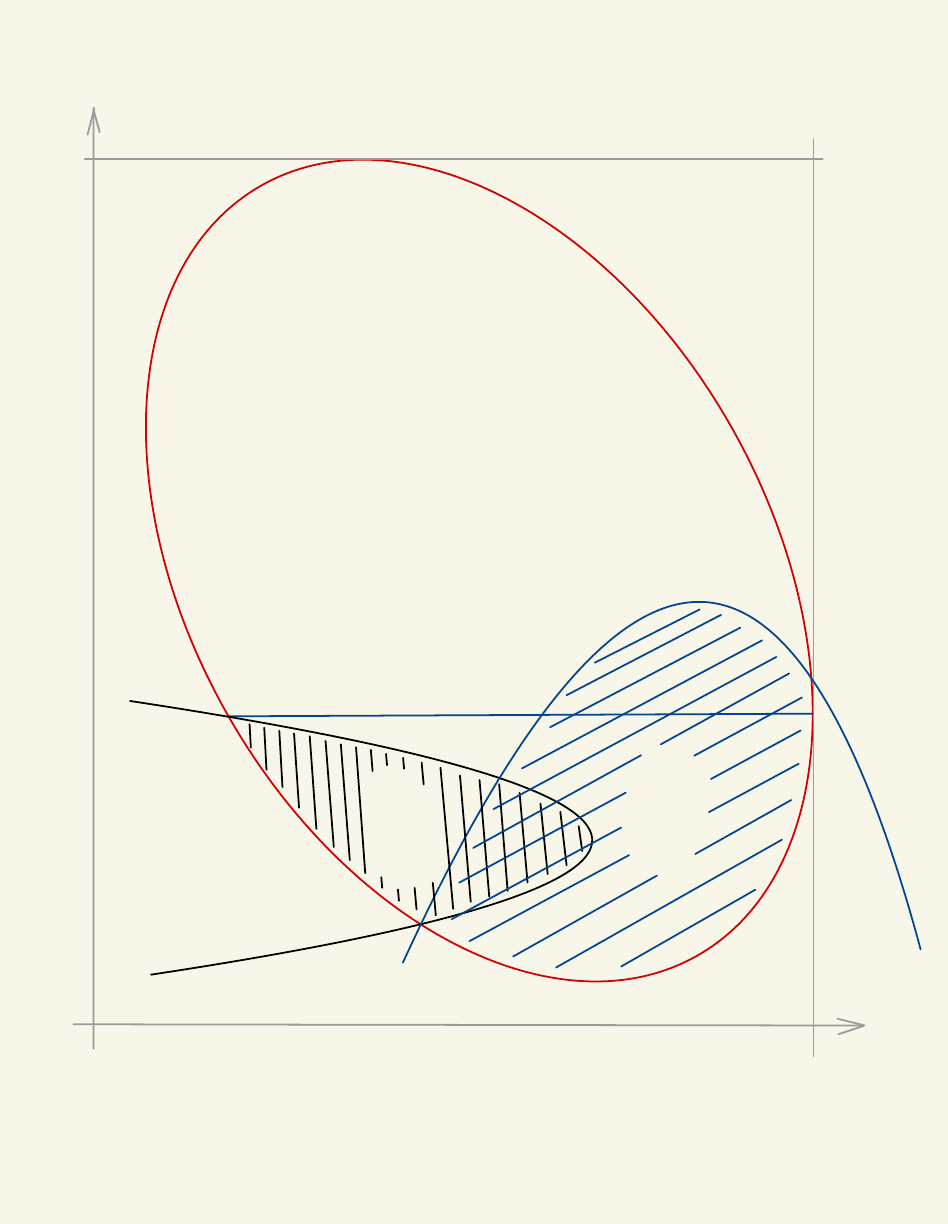}}
\put(-393,120){$(0,y^{**}_P)$}
\put(-180,70){$D$} 
\put(-244,42){$D\cap D_2$} 
\put(-135,45){$D\cap D_1$} 
\put(-355,53){$(X_1(y_d),y_d)$} 
\put(-380,10){$(0,0)$}
\put(-70,10){$(x_d,0)$}
\put(-55,57){$(x_d,y_d)$}
\caption{Case (B5) with $x_d=x^{**}=x^{**}_P$. In this case, the line segment $[(x^{**}, Y_1(x^{**})), (x^{**}, Y_2(x^{**}))]$ is a single point and $y_d=y^{**} = Y_1(x_d) < y^{**}_P$.}\label{F6} 
\end{figure} 

\begin{figure}[ht]
\resizebox{14cm}{5cm}{\includegraphics{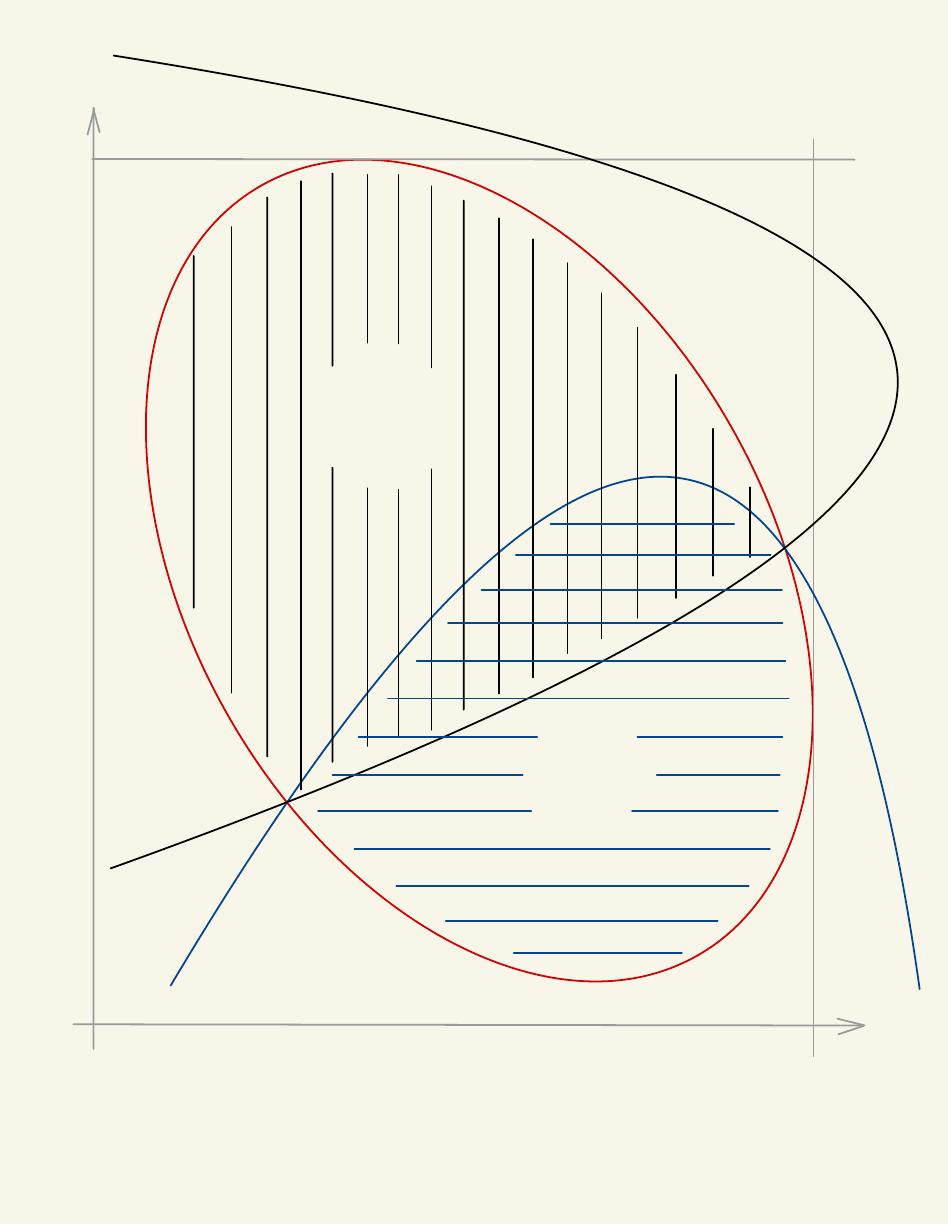}}
\put(-250,90){$D\cap D_2$} 
\put(-290,40){$(1,1)$} 
\put(-165,50){$D\cap D_1$} 
\put(-380,10){$(0,0)$}
\put(-388,121){$(0,y_d)$}
\put(-70,10){$(x_d,0)$}
\caption{Case (B2) with $x_d=x^{**}=x^{**}_P$ and $y_d = y^{**} = y^{**}_P$. Here, each of the line segments $[(x^{**}, Y_1(x^{**})), (x^{**}, Y_2(x^{**}))]$ and $[(X_1(y^{**}), y^{**}), (X_2(y^{**}), y^{**})]$ is a single point. }\label{F8} 
\end{figure}

\begin{figure}[ht]
\resizebox{10cm}{5cm}{\includegraphics{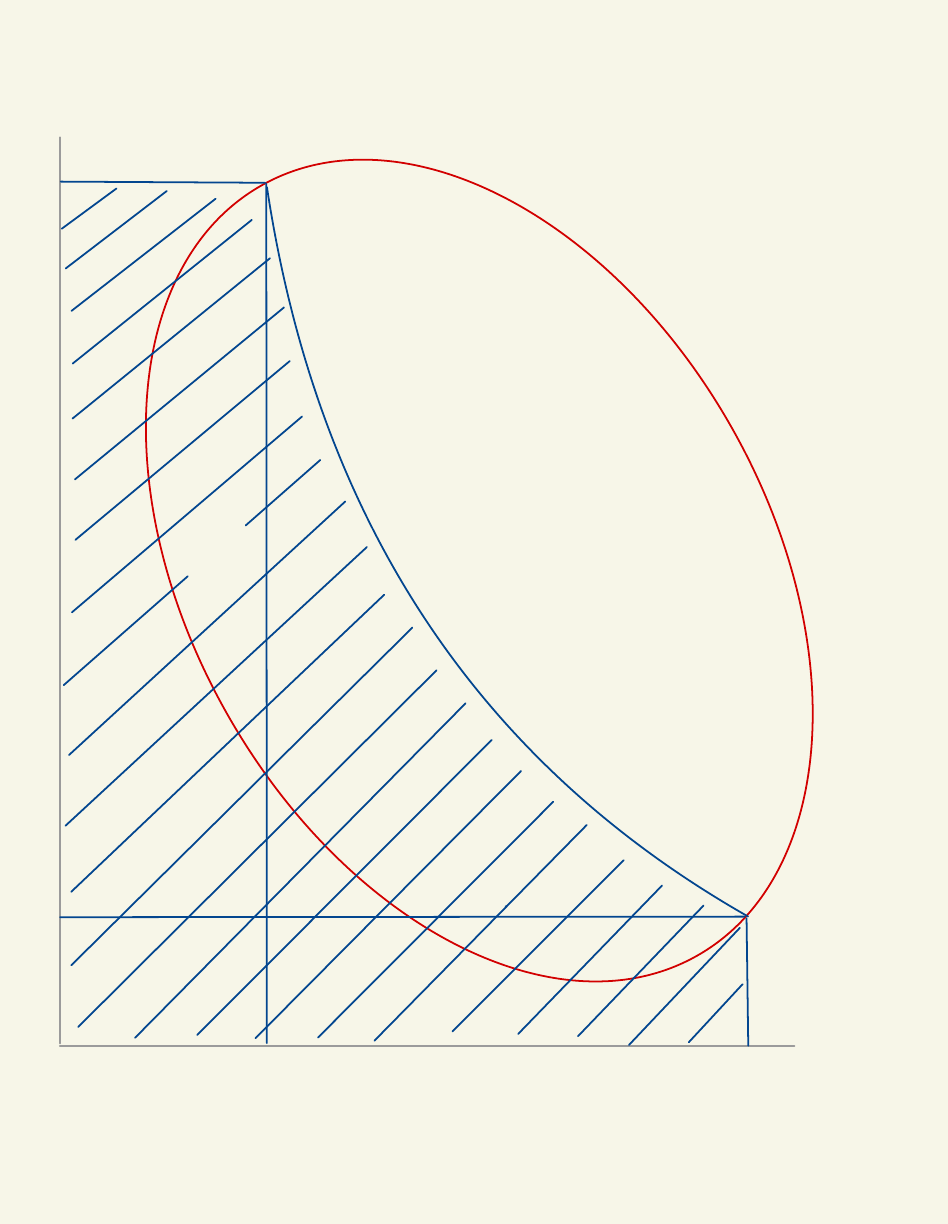}}
\put(-140,70){$D$} 
\put(-222,75){$\Theta$}
\put(-280,10){$(0,0)$}
\put(-280,40){$(0,Y_1(x^{**}))$}
\put(-220,10){$(X_1(y^{**}),0)$}
\put(-280,130){$(0,y^{**})$}
\put(-80,10){$(x^{**},0)$}
\put(-60,30){$(x^{**},Y_1(x^{**}))$}
\caption{Set $\Theta$, cases (B0){-}(B2). }\label{F5} 
\end{figure}


\begin{thebibliography}{10}

\bibitem{Alili-Doney}
L.~Alili and R.~A. Doney, \emph{Martin boundaries associated with a killed
  random walk}, Annales de l'Institut Henri Poincar\'{e}. {P}robabilit\'{e}s et
  {S}tatistiques \textbf{37} (2001), no.~3, 313--338.

\bibitem{Borovkov-Mogulskii-2001}
A.~A. Borovkov and A.~A. Mogul'ski\u{\i}, \emph{Large deviations for {M}arkov
  chains in the positive quadrant}, Uspehi Matemati\v{c}eskih Nauk \textbf{56}
  (2001), no.~5(341), 3--116.

\bibitem{Cartier}
P.~Cartier, \emph{Fonctions harmoniques sur un arbre}, Symposia {M}athematica,
  {V}ol. {IX} ({C}onvegno di {C}alcolo delle {P}robabilit\`a \& {C}onvegno di
  {T}eoria della {T}urbolenza, {INDAM}, {R}ome, 1971), Academic Press,
  London-New York, 1972, pp.~203--270.

\bibitem{Doney:02}
R.~A. Doney, \emph{The {M}artin boundary and ratio limit theorems for killed
  random walks}, Journal of the London Mathematical Society. Second Series
  \textbf{58} (1998), no.~3, 761--768.

\bibitem{Doob:02}
J.~L. Doob, \emph{Discrete potential theory and boundaries}, Journal of
  Mathematics and Mechanics \textbf{8} (1959), 433--458.

\bibitem{Dynkin}
E.~B. Dynkin, \emph{The boundary theory of {M}arkov processes (discrete case)},
  Uspehi Matemati\v{c}eskih Nauk \textbf{24} (1969), no.~2(146), 3--42.

\bibitem{FMM}
G.~Fayolle, V.~A. Malyshev, and M.~V. Menshikov, \emph{Topics in the
  constructive theory of countable {M}arkov chains}, Cambridge University
  Press, Cambridge, 1995.

\bibitem{F-Y-M-book}
Guy Fayolle, Roudolf Iasnogorodski, and Vadim Malyshev, \emph{Random walks in
  the quarter plane}, second ed., vol.~40, Springer, 2017.

\bibitem{FlajoletO}
Philippe Flajolet and Andrew Odlyzko, \emph{Singularity analysis of generating
  functions}, SIAM Journal on Discrete Mathematics \textbf{3} (1990), no.~2,
  216--240.

\bibitem{Flajolet-Sedgewick}
Philippe Flajolet and Robert Sedgewick, \emph{Analytic combinatorics},
  Cambridge University Press, Cambridge, 2009.

\bibitem{Ignatiouk-2023-cone}
Irina Ignatiouk-Robert, \emph{Asymptotics of the green function for
  non-centered random walks in a cone}, Preprint. Arxiv:
  https://arxiv.org/abs/2310.10257.

\bibitem{Ignatiouk:2008}
\bysame, \emph{Martin boundary of a killed random walk on a half-space},
  Journal of Theoretical Probability \textbf{21} (2008), no.~1, 35--68.

\bibitem{Ignatiouk:2010}
\bysame, \emph{Martin boundary of a reflected random walk on a half-space},
  Probability Theory and Related Fields \textbf{148} (2010), no.~1-2, 197--245.

\bibitem{Ignatiouk:2020}
Irina Ignatiouk-Robert, \emph{Martin boundary of a killed non-centered random
  walk in a general cone}, 2020, Arxiv preprint 2006.15870.

\bibitem{I-K-R}
Irina Ignatiouk-Robert, Irina Kourkova, and Kilian Raschel, \emph{{Reflected
  random walks and unstable Martin boundary}}, Annales de l'Institut Henri
  Poincar\'e B - {P}robabilit\'es et Statistiques (2023), To appear.

\bibitem{Kobayashi-Miyazawa-2}
Masahiro Kobayashi and Masakiyo Miyazawa, \emph{Revisiting the tail asymptotics
  of the double {QBD} process: refinement and complete solutions for the
  coordinate and diagonal directions}, Matrix-analytic methods in stochastic
  models, Springer Proc. Math. Stat., vol.~27, Springer, New York, 2013,
  pp.~145--185.

\bibitem{Kobayashi-Miyazawa}
\bysame, \emph{Tail asymptotics of the stationary distribution of a
  two-dimensional reflecting random walk with unbounded upward jumps}, Advances
  in Applied Probability \textbf{46} (2014), no.~2, 365--399.

\bibitem{Kurkova-Malyshev}
I.~A. Kurkova and V.~A. Malyshev, \emph{Martin boundary and elliptic curves},
  Markov Processes and Related Fields \textbf{4} (1998), no.~2, 203--272.

\bibitem{Kurkova-Raschel}
Irina Kurkova and Kilian Raschel, \emph{Random walks in {$({\mathbb Z}_+)^2$}
  with non-zero drift absorbed at the axes}, Bull. Soc. Math. France
  \textbf{139} (2011), no.~3, 341--387.

\bibitem{Li-Zhao-2018}
Hui Li and Yiqiang~Q. Zhao, \emph{A kernel method for exact tail asymptotics
  --- {R}andom walks in the quarter plane}, Queueing Models and Service
  Management \textbf{1} (2018), no.~1, 95--130.

\bibitem{Malyshev}
V.~A. Maly\v{s}ev, \emph{Asymptotic behavior of the stationary probabilities
  for two-dimensional positive random walks}, Sibirsk. Mat. \v{Z}. \textbf{14}
  (1973), 156--169, 238.

\bibitem{Miyazawa}
Masakiyo Miyazawa, \emph{Tail decay rates in double qbd processes and related
  reflected random walks}, Mathematics of Operations Research \textbf{34}
  (2009), no.~3, 547--575.

\bibitem{Ney-Spitzer}
P.~Ney and F.~Spitzer, \emph{The {M}artin boundary for random walk},
  Transactions of the American Mathematical Society \textbf{121} (1966),
  116--132.

\bibitem{Spitzer}
F.~Spitzer, \emph{Principles of random walks}, Van Nostrand, Princeton, New
  Jersey, 1964.

\bibitem{Woess}
Wolfgang Woess, \emph{Random walks on infinite graphs and groups}, Cambridge
  University Press, Cambridge, 2000.

\end{thebibliography}
\end{document}